\documentclass{amsart}

\newif\ifignoretikz\ignoretikzfalse
\newif\iftwist\twistfalse

\newif\iflarge\largetrue

\usepackage{amsmath}       
\usepackage{amsthm}        
\usepackage{amssymb}       
\usepackage{amsfonts}
\usepackage[all]{xy}
\usepackage{xspace}
\usepackage{calc}
\usepackage{pgfplots}
\usepgflibrary{fpu}
\usepackage{mathtools}
\usepackage{tabularx}
\usepackage{ifthen}	
\usepackage{standalone}
\usepackage{tensor}

\usepackage[margin=1.5in]{geometry}
\usepackage[multiple]{footmisc}

\tikzset{morgray/.style={black!50}}  

\makeatletter
\renewcommand\subsubsection{\@startsection{subsubsection}{3}%
  \z@{.5\linespacing\@plus.7\linespacing}{-.5em}%
  {\normalfont\bfseries}}
\renewcommand\paragraph{\@startsection{paragraph}{4}%
  \z@{.5\linespacing\@plus.35\linespacing}{-.5em}
  {\normalfont\itshape}}
\makeatother

\usepackage{graphicx}
\usepackage{array}
\usepackage{subfloat} 
\usepackage{bbm}            
\usepackage{etoolbox}  
\usepackage{verbatim}
\usepackage{stmaryrd}
\usepackage[hyperfootnotes=false,draft=false,pdftex, colorlinks=true, linkcolor=black, citecolor=black, urlcolor=black, hypertexnames=false
	]{hyperref}

\newcommand{ \ignore}[1]{}
\newcommand{\sgn}{\mathrm{sgn}}

\newcommand{\tr}{\mathrm{tr}}

\newcommand{\Iz}{\mathrm{I}}
\newcommand{\Io}{\mathnormal{1}}
\newcommand{\It}{1}

\newcommand{\xz}{\sq}
\newcommand{\xo}{\circ}
\newcommand{\xt}{\cdot}

\newcommand{\sig}{{\Gamma^\sigma}}

\newcommand{\zh}{\widehat{Z}}
\newcommand{\ZH}{\widehat{Z}^{\oplus}}

\newcommand{\sk}{\omega}

\newcommand{\conj}[1]{\overline{#1}}

\newcommand{\dm}{\mathrm{d}}
\newcommand{\dmo}{\mathrm{d}}

\newcommand{\lan}{\langle}
\newcommand{\ran}{\rangle}

\newcommand{\Lcomp}{\mathrm{LC}}
\newcommand{\Scomp}{\mathrm{SC}}

\newcommand{\tdim}[1]{\dm\!\left(\vphantom{a_x^x}[#1]\right)}
\newcommand{\odim}[1]{\dmo\!\left(\vphantom{a_x^x}[#1]\right)}
\newcommand{\fs}[1]{[#1]}
\newcommand{\Xs}[1]{[#1]}

\newcommand{\idm}{\triangledown}

\newcommand{\zero}{\mathfrak{0}}
\renewcommand{\zero}{0} 
\newcommand\B{\mathrm{B}}

\newcommand\V[1]{V^+(#1)}

\def\labeleqgap{\hspace{0.2cm}}

\newcommand\shrink[1]{\text{\scalebox{.95}{$#1$}}}
\newcommand\shrinker[2]{\text{\scalebox{#1}{$#2$}}}
\newcommand\shrinkalign[2]{\text{\centerline{\scalebox{#1}{\parbox{\linewidth}{#2}}}}}

\newcommand{\skiptocparagraph}[1]{\paragraph*{#1}
\addtocontents{toc}{\SkipTocEntry}
}

\newcommand{\swapeq}{\sim}

\DeclareRobustCommand{\SkipTocEntry}[5]{}

\makeatletter
\def\slashedarrowfill@#1#2#3#4#5{%
  $\m@th\thickmuskip0mu\medmuskip\thickmuskip\thinmuskip\thickmuskip
   \relax#5#1\mkern-7mu%
   \cleaders\hbox{$#5\mkern-2mu#2\mkern-2mu$}\hfill
   \mathclap{#3}\mathclap{#2}%
   \cleaders\hbox{$#5\mkern-2mu#2\mkern-2mu$}\hfill
   \mkern-7mu#4$%
}
\def\rightslashedarrowfill@{%
  \slashedarrowfill@\relbar\relbar\mapstochar\rightarrow}
\newcommand\hto[2][]{%
  \ext@arrow 0055{\rightslashedarrowfill@}{#1}{#2}}
\makeatother

\usepackage{tikz}
\usetikzlibrary{matrix}
\usetikzlibrary{decorations.pathreplacing}
\usetikzlibrary{decorations.markings}
\usetikzlibrary{arrows}
\usetikzlibrary{calc}
\usetikzlibrary{shapes.misc}
\usetikzlibrary{fit}
\usepgflibrary{decorations.pathmorphing}
\usepgflibrary{shapes.geometric}
\usetikzlibrary{arrows.meta}
\newenvironment{tz}[1][]{%
			\begin{tikzpicture}[baseline={([yshift=-.8ex]current bounding 					box.center)},#1] %
				}{%
			\end{tikzpicture} %
			}

\makeatletter
\pgfkeys{%
  /tikz/on layer/.code={
    \pgfonlayer{#1}\begingroup
    \aftergroup\endpgfonlayer
    \aftergroup\endgroup
  },
  /tikz/node on layer/.code={
     \gdef\node@@on@layer{%
      \setbox\tikz@tempbox=\hbox\bgroup\pgfonlayer{#1}\unhbox\tikz@tempbox\endpgfonlayer\egroup}
    \aftergroup\node@on@layer
  },
  /tikz/end node on layer/.code={
    \endpgfonlayer\endgroup\endgroup
  }
}

\def\node@on@layer{\aftergroup\node@@on@layer}
\renewcommand{\-}[0]{\nobreakdash-\hspace{0pt}}

\makeatletter
\def\calign@preamble{%
   &\hfil\strut@
    \setboxz@h{\@lign$\m@th\displaystyle{##}$}%
    \ifmeasuring@\savefieldlength@\fi
    \set@field
    \hfil
    \tabskip\alignsep@
}
\let\cmeasure@\measure@
\patchcmd\cmeasure@{\divide\@tempcntb\tw@}{}{}{}
\patchcmd\cmeasure@{\divide\@tempcntb\tw@}{}{}{}
\patchcmd\cmeasure@{\ifodd\maxfields@
  \global\advance\maxfields@\@ne
  \fi}{}{}{}    
\newenvironment{calign}
{%
  \let\align@preamble\calign@preamble
  \let\measure@\cmeasure@
  \align
}
{%
  \endalign
}  
\makeatother

\tikzset{
    master/.style={
        execute at end picture={
            \coordinate (lower right) at (current bounding box.south east);
            \coordinate (upper left) at (current bounding box.north west);
        }
    },
    slave/.style={
        execute at end picture={
            \pgfresetboundingbox
            \path (upper left) rectangle (lower right);
        }
    }
}

\tikzset{blob/.style={draw, circle, fill=white, inner sep=1pt, minimum width=15pt, font=\scriptsize, line width=0.7pt}}
\tikzset{greenregion/.style={fill=green, fill opacity=0.3, draw=none}}
\tikzset{redregion/.style={fill=red, fill opacity=0.3, draw=none}}
\tikzset{blueregion/.style={fill=blue, fill opacity=0.3, draw=none}}
\tikzset{yellowregion/.style={fill=yellow, fill opacity=0.5, draw=none}}
\tikzset{cyanregion/.style={fill=cyan, fill opacity=0.3, draw=none}}
\tikzset{orangeregion/.style={fill=orange, fill opacity=0.6, draw=none}}
\tikzset{solidgreenregion/.style={fill=green!30, fill opacity=1, draw=none}}
\tikzset{solidredregion/.style={fill=red!30, fill opacity=1, draw=none}}
\tikzset{solidblueregion/.style={fill=blue!30, fill opacity=1, draw=none}}
\tikzset{solidyellowregion/.style={fill=yellow!30, fill opacity=1, draw=none}}
\tikzset{string/.style={line width=0.7pt}}
\tikzset{zig/.style={decoration={zigzag,segment length=3, amplitude=0.5}}}
\tikzset{bnd/.style={draw,string}}   
\tikzset{projector/.style={circle, draw, font=\scriptsize, inner sep=-5pt, minimum width=0.35cm, string, fill=white}}
\tikzset{dimension/.style={font=\scriptsize, inner sep=1pt}}

\tikzset{zx/.style = {string, scale=\zxscale}}
\tikzset{zxnode/.style n args={1}{blob,scale=\zxnodescale,fill=#1}}
\tikzset{zxvertex/.style n args={1}{draw,fill=#1,circle,scale=0.75*\vertexscale}}
\tikzset{zxdown/.style={yshift=-\zxshift}}
\tikzset{zxup/.style={yshift=\zxshift}}

\tikzset{arrow data/.style 2 args={
      decoration={
         markings,
         mark=at position #1 with {\arrow[scale=0.75]{#2}}}, 
         postaction=decorate}
}

\usetikzlibrary{3d}
\def\dl{-135}
\def\dr{-45}
\def\ul{135}
\def\ur{45}
\def\dlcusp{-105}
\def\drcusp{-75}
\def\ulcusp{105}
\def\urcusp{75}
\def\stdr{0.36}
\def\stdl{0.59}
\pgfdeclarelayer{back}
\pgfdeclarelayer{front}
\pgfdeclarelayer{superfront}
\pgfdeclarelayer{superback}
\pgfdeclarelayer{fronta}
\pgfdeclarelayer{frontb}
\pgfdeclarelayer{frontc}
\pgfdeclarelayer{backa}
\pgfsetlayers{superback,back,backa,main,fronta,frontb,frontc,front,superfront}
\tikzset{xyplane/.style={canvas is yx plane at z=#1}}
\tikzset{xzplane/.style={canvas is yz plane at x=#1}}
\tikzset{yzplane/.style={canvas is xz plane at y=#1}}
\tikzset{td/.style={
					y={(0.4cm, 0.6cm)},
					x={(-1cm, 0cm)},					
                    z  = {(0cm,1cm)},
                    scale = 0.5}}
\tikzset{slice/.style={draw = gray!50, line width = 0.7pt}}
\tikzset{braid slice/.style={draw = white, double distance =0.7pt, line width =1.4pt, double = gray!50}}
\tikzset{wire/.style={black, line width=0.7pt}}
\tikzset{braid wire/.style={draw=white, double distance=0.7pt, line width=1.4pt, double=black}}
\tikzset{dot/.style={circle, scale=0.15, fill=black, thick, draw}}
\tikzset{obj/.style={scale=0.6, color=gray}}
\tikzset{omor/.style={scale=0.6, color=black}}   
\tikzset{tmor/.style={scale=0.6, color=black}}   
\def\lw{1.4pt}           
\tikzset{short/.style={shorten >=-0.26*\lw,shorten <=-0.25*\lw}}    
\def\h{3}
\tikzset{tinydash/.style={on layer=back,gray!50,line width = 0.4pt,densely dotted}}  
              
\def\maskpointheight{7pt}
\tikzset{mask point/.style={ transform shape, sloped, 
minimum width=17pt, minimum height=\maskpointheight, inner sep=0pt, ultra thin, font=\tiny}}

\tikzset{
clip even odd rule/.code={\pgfseteorule},
invclip/.style={clip,insert path=[clip even odd rule]{
   [reset cm](-\maxdimen,-\maxdimen)rectangle(\maxdimen,\maxdimen)
    }}} 
\newcommand\clipintersection[1]{(#1.north west) -- (#1.north east) -- (#1.south east) -- (#1.south west) -- (#1.north west)}

\newcommand\cliparoundthree[4]{\begin{pgfscope}\begin{scope}[overlay]
    \path [invclip] \clipintersection{#1} -- \clipintersection{#2}--\clipintersection{#3} -- (#1.north west);#4
\end{scope}\end{pgfscope}}  
\newcommand\cliparoundone[2]{\begin{pgfscope}\begin{scope}[overlay]
\path[invclip] \clipintersection{#1};#2
\end{scope}\end{pgfscope}}
\newcommand\cliparoundtwo[3]{\begin{pgfscope}\begin{scope}[overlay]
    \path [invclip] \clipintersection{#1}  \clipintersection{#2};#3
\end{scope}\end{pgfscope}}

\def\coronawidth{4pt}   
\def\coronacolor{gray!30}
\def\coronaopacity{0.5}

\newcommand{\insidepath}[3][main]{
\begin{scope}[on layer=#1] \clip #3;
\draw[\coronacolor, opacity = \coronaopacity,line width =\coronawidth] #3;
\end{scope}
\draw[slice,#2] #3;}

\tikzset{invclip/.style={clip,insert path=[clip even odd rule]{
   [reset cm](-\maxdimen,-\maxdimen)rectangle(\maxdimen,\maxdimen)
    }}}

\newcommand{\rightpathup}[5][main]{
\begin{pgfinterruptboundingbox} 
\begin{scope} [on layer=#1]
\path [clip]  #3 (#5) to (\maxdimen, \maxdimen) to (\maxdimen, -\maxdimen) to (#4);
\draw[\coronacolor, opacity = \coronaopacity,line width =\coronawidth] #3;
\end{scope}
\end{pgfinterruptboundingbox}
\draw[slice,#2] #3;}

\newcommand{\rightpathdown}[5][main]{
\begin{pgfinterruptboundingbox} 
\begin{scope} [on layer=#1]
\path [clip] #3 (#4) to (\maxdimen, -\maxdimen) to (\maxdimen, \maxdimen) to (#5);
\draw[\coronacolor, opacity =\coronaopacity,line width =\coronawidth] #3;
\end{scope}
\end{pgfinterruptboundingbox}
\draw[slice,#2] #3;}

\newcommand{\leftpathup}[5][main]{
\begin{pgfinterruptboundingbox} 
\begin{scope} [on layer=#1]
\path [clip, on layer=#1] #3 (#5) to (-\maxdimen, \maxdimen) to (-\maxdimen, -\maxdimen) to (#4);
\draw[\coronacolor, opacity =\coronaopacity,line width =\coronawidth] #3;
\end{scope}
\end{pgfinterruptboundingbox}
\draw[slice,#2] #3;}

\tikzset{halfarrowleftdown/.style={
      decoration={
         markings,
         mark=between positions 4pt and 1 step 4pt with {\arrow[scale=1,gray!50,opacity=0.5]{Straight Barb[left,angle = 90:2mm 1]}}}, 
         preaction=decorate}
}

\tikzset{halfarrowrightdown/.style={
      decoration={
         markings,
         mark=between positions 4pt and 1 step 4pt with {\arrow[scale=1,gray!50,opacity=0.5]{Straight Barb[right,angle = 90:2mm 1]}}}, 
         preaction=decorate}
}

\tikzset{halfarrowleftup/.style={
      decoration={
         markings,
         mark=between positions 2pt and 1 step 4pt with {\arrow[scale=1,gray!50, opacity=0.5]{Straight Barb[left,angle = 270:2mm 1]}}}, 
         preaction=decorate}
}
\tikzset{halfarrowrightup/.style={
      decoration={
         markings,
         mark=between positions 2pt and 1 step 4pt with {\arrow[scale=1,gray!50,  opacity=0.5]{Straight Barb[right,angle = 270:2mm 1]}}}, 
         preaction=decorate}
}

\newcommand{\Tr}{\mathrm{Tr}}

\newcommand\smod[1]{\ensuremath{\left\llbracket #1\right\rrbracket}}

\newcommand\superequals[1]{\stackrel {\raisebox{1pt}{\makebox[0pt]{\tiny #1}}} =}

\newcommand\eqgap{\hspace{5pt}}
\newcommand\planareqgap{\hspace{10pt}}
\newcommand\tdeqgap{\hspace{10pt}}
\newcommand\seqgap{\hspace{3pt}}

\def\sqt{1.414}

\renewcommand{\to}[1][]{\ensuremath{\xrightarrow{#1}}}
\newcommand{\To}[1][]{\ensuremath{\xRightarrow{#1}}}

\allowdisplaybreaks[1]

\newcommand\vc[1]{\begin{tabular}{@{}c@{}}#1\end{tabular}}

\newcommand\lix[2]{\tensor[^{#1}]{#2}{}}

\newcommand{\sI}{e}
\newcommand{\sII}{s}
\newcommand{\sIII}{\kappa}
\newcommand{\sIV} {\mu}
\newcommand{\s}{\tau}
\newcommand{\lk}{\mathrm{lk}}
\newcommand{\K}{K} 
\newcommand{\interior}[1]{i(#1)}

\theoremstyle{plain} 
\newtheorem{mainthm}{Theorem}
\newtheorem{maindef}[mainthm]{Definition}
\newtheorem{theorem}{Theorem}[subsection]
\newtheorem{lemma}[theorem]{Lemma}
\newtheorem{corollary}[theorem]{Corollary}          
\newtheorem{proposition}[theorem]{Proposition}              
          
\newtheorem{prop}[theorem]{Proposition}      
\newtheorem{theorem*}[]{Theorem} 
\newtheorem*{convention*}{Convention}

\newtheorem{apptheorem}{Theorem}[section]
\newtheorem{appcorollary}[apptheorem]{Corollary}
\newtheorem{applemma}[apptheorem]{Lemma}
\newtheorem{appprop}[apptheorem]{Proposition}

\theoremstyle{definition} 
\newtheorem{definition}[theorem]{Definition}
\newtheorem{definition*}[theorem*]{Definition} 

\theoremstyle{remark}  
\newtheorem{remark}[theorem]{Remark}
\newtheorem{example}[theorem]{Example}
\newtheorem{recollection}[theorem]{Recollection}
\newtheorem{notation}[theorem]{Notation}

\newtheorem{conjecture*}[theorem*]{Conjecture}
\newtheorem{construction}[theorem]{Construction}
\newtheorem{warning}[theorem]{{Warning}}

\newtheorem{question*}[theorem*]{{Question}}
\newtheorem*{guide*}{Guide}
\newtheorem*{outline*}{Outline}
\newtheorem*{remarkohc*}{Remark on higher categories and the cobordism hypothesis}
\newtheorem*{assumpfield*}{Assumptions on the base field}

\newtheoremstyle{special_statement} 
	{\topskip}
	{\topskip}
	{\addtolength{\leftskip}{2.5em} \itshape }
	{}
	{\bfseries}
	{:}
	{.5em}
	{}
\theoremstyle{special_statement}

\newcommand{\nid}{\noindent}
\newcommand{\ra}{\rightarrow}

\newcommand{\xra}{\xrightarrow}

\DeclareMathOperator{\Hom}{Hom}
\DeclareMathOperator{\End}{End}

\newcommand{\id}{\mathrm{id}}
\newcommand{\op}{\mathrm{op}}
\newcommand{\mpt}{\mathrm{mp}}

\newcommand{\iso}{\cong}
\let\origequiv\equiv
\renewcommand{\equiv}{\simeq}

\newcommand{\bplus}{\boxplus}

\newcommand\sq{\mathbin{\text{\scalebox{.84}{$\square$}}}}

\newcommand{\vp}{\vphantom{\frac{a}{b}}}
\newcommand{\bigboxplus}{
  \mathop{
    \vphantom{\bigoplus} 
    \mathchoice
      {\vcenter{\hbox{\resizebox{\widthof{$\displaystyle\bigoplus$}}{!}{$\boxplus$}}}}
      {\vcenter{\hbox{\resizebox{\widthof{$\bigoplus$}}{!}{$\boxplus$}}}}
      {\vcenter{\hbox{\resizebox{\widthof{$\scriptstyle\oplus$}}{!}{$\boxplus$}}}}
      {\vcenter{\hbox{\resizebox{\widthof{$\scriptscriptstyle\oplus$}}{!}{$\boxplus$}}}}
  }\displaylimits 
}

\newcommand{\tVect}{\mathrm{2Vect}}

\newcommand{\Rep}{\mathrm{Rep}}
\newcommand{\tRep}{\mathrm{2Rep}}
\newcommand{\tVep}{\mathrm{2Vep}}

\newcommand{\Aut}{\mathrm{Aut}}

\newcommand{\Vect}{\mathrm{Vect}}

\newcommand{\LMod}{\mathrm{LMod}}
\newcommand{\RMod}{\mathrm{RMod}}
\newcommand{\Mod}{\mathrm{Mod}}

\def\mfus{\oc{C}}%

\newcommand{\colim}{\mathrm{colim}}

\newcommand\oc[1]{\mathrm{#1}}
\newcommand\tc[1]{\mathcal{#1}}

\newcommand\asmod[1]{\ensuremath{\left\llangle #1\right\rrangle}}

\def\cA{\mathcal A}\def\cB{\mathcal B}\def\cC{\mathcal C}\def\cD{\mathcal D}

\def\CC{\mathbb C}

\def\LL{\mathbb L}

\def\QQ{\mathbb Q}\def\RR{\mathbb R}

\def\ZZ{\mathbb Z}

\makeatletter
\DeclareFontFamily{OMX}{MnSymbolE}{}
\DeclareSymbolFont{MnLargeSymbols}{OMX}{MnSymbolE}{m}{n}
\SetSymbolFont{MnLargeSymbols}{bold}{OMX}{MnSymbolE}{b}{n}
\DeclareFontShape{OMX}{MnSymbolE}{m}{n}{
    <-6>  MnSymbolE5
   <6-7>  MnSymbolE6
   <7-8>  MnSymbolE7
   <8-9>  MnSymbolE8
   <9-10> MnSymbolE9
  <10-12> MnSymbolE10
  <12->   MnSymbolE12
}{}
\DeclareFontShape{OMX}{MnSymbolE}{b}{n}{
    <-6>  MnSymbolE-Bold5
   <6-7>  MnSymbolE-Bold6
   <7-8>  MnSymbolE-Bold7
   <8-9>  MnSymbolE-Bold8
   <9-10> MnSymbolE-Bold9
  <10-12> MnSymbolE-Bold10
  <12->   MnSymbolE-Bold12
}{}

\let\llangle\@undefined
\let\rrangle\@undefined
\DeclareMathDelimiter{\llangle}{\mathopen}%
                     {MnLargeSymbols}{'164}{MnLargeSymbols}{'164}
\DeclareMathDelimiter{\rrangle}{\mathclose}%
                     {MnLargeSymbols}{'171}{MnLargeSymbols}{'171}
\makeatother

\definecolor{DRcolor}{rgb}{0.0,0.5,0.75}	
\definecolor{CDcolor}{rgb}{0.8,0.0,0.2}

\newcommand{\DRs}[1]{}
\newcommand{\CDs}[1]{}

\let\ccalign\calign
\let\endccalign\endcalign
\let\ttz\tz
\let\endttz\endtz

\newcommand\excludecommands{
	\renewenvironment{calign}{
	\[\framebox[4cm]{equation}
	\]
	\expandafter\comment}{\expandafter\endcomment}
	\renewenvironment{tz}{
	\expandafter\comment
	}
	{\expandafter\endcomment}
	\let\endtz\relax
	\renewcommand\smod[1]{\llbracket}
	\renewcommand\asmod[1]{\llangle }
}

\ifignoretikz
	\excludecommands
\fi

\newcommand\resetcommands{
	\let\tz\ttz
	\let\endtz\endttz
	\renewcommand\smod[1]{\ensuremath{\left\llbracket ##1\right\rrbracket}}
	\renewcommand\asmod[1]{\ensuremath{\left\llangle ##1\right\rrangle}}
	\let\calign\ccalign
	\let\endcalign\endccalign
}

\usepackage{pict2e}
\makeatletter
\newcommand{\adjunction}{\@ifstar\named@adjunction\normal@adjunction}
\newcommand{\normal@adjunction}[4]{%
  #1\colon #2%
  \mathrel{\vcenter{%
    \offinterlineskip\m@th
    \ialign{%
      \hfil$##$\hfil\cr
      \longrightharpoonup\cr
      \noalign{\kern-.3ex}
      \smallbot\cr
      \longleftharpoondown\cr
    }%
  }}%
  #3 \noloc #4%
}
\newcommand{\named@adjunction}[4]{%
  #2%
  \mathrel{\vcenter{%
    \offinterlineskip\m@th
    \ialign{%
      \hfil$##$\hfil\cr
      \scriptstyle#1\cr
      \noalign{\kern.1ex}
      \longrightharpoonup\cr
      \noalign{\kern-.3ex}
      \smallbot\cr
      \longleftharpoondown\cr
      \scriptstyle#4\cr
    }%
  }}%
  #3%
}
\newcommand{\longrightharpoonup}{\relbar\joinrel\rightharpoonup}
\newcommand{\longleftharpoondown}{\leftharpoondown\joinrel\relbar}
\newcommand\noloc{%
  \nobreak
  \mspace{6mu plus 1mu}
  {:}
  \nonscript\mkern-\thinmuskip
  \mathpunct{}
  \mspace{2mu}
}
\newcommand{\smallbot}{%
  \begingroup\setlength\unitlength{.15em}%
  \begin{picture}(1,1)
  \roundcap
  \polyline(0,0)(1,0)
  \polyline(0.5,0)(0.5,1)
  \end{picture}%
  \endgroup
}
\makeatother

\title{Fusion 2-categories and a state-sum invariant for 4-manifolds}

\author{Christopher L. Douglas}
\address{Mathematical Institute\\ University of Oxford\\ Oxford OX2 6GG\\ United Kingdom}
\email{cdouglas@maths.ox.ac.uk}
\urladdr{http://www.christopherleedouglas.com}
      	
\author{David J. Reutter}
\address{Department of Computer Science\\ University of Oxford\\ Oxford OX1 3QD\\ United Kingdom}
\email{david.reutter@cs.ox.ac.uk}
\urladdr{https://www.cs.ox.ac.uk/people/david.reutter/}

\begin{document}

\vspace*{-16pt}
\maketitle

\begin{abstract}
\vspace*{-15pt}
We introduce semisimple 2-categories, fusion 2-categories, and spherical fusion 2-categories.  For each spherical fusion 2-category, we construct a state-sum invariant of oriented singular piecewise-linear 4-manifolds.
\vspace*{15pt}
\end{abstract}

\setcounter{tocdepth}{3}
\tableofcontents


\pagebreak

\section*{Introduction}

\renewcommand{\thesubsection}{I.{\arabic{subsection}}}

One of the early successes of quantum topology was Turaev and Viro's construction of a 3-manifold invariant based on the representation theory of quantum $\mathfrak{sl}_2$~\cite{TV}, and Barrett and Westbury's generalization of this construction to an invariant based on any spherical fusion category~\cite{BW}.  These invariants are defined by a `state sum', a weighted average of numbers associated to fusion-categorical labelings of a triangulated manifold.  From a more recent cobordism-hypothesis perspective, associated to a fusion category there is a local 3-dimensional field theory~\cite{DTC}, and the classical Turaev--Viro--Barrett--Westbury invariant is obtained by restricting to closed 3-manifolds.  As the cobordism-hypothesis is non-constructive and invariants produced from it are not in general directly computable, explicit state sum constructions of invariants remain informative and useful, both mathematically and for their role in physical lattice field theories~\cite{LevinWen} and consequent relevance for condensed matter physics and topological quantum computation~\cite{KKR}.

In contrast to the situation in dimension 3, and despite a wealth of important field-theoretically-inspired 4-manifold invariants~\cite{donaldson, witten, ozsvathszabo, kronheimermrowka}, constructions of true 4-dimensional topological field theory invariants have been sparse and sporadic.  The earliest was the Crane--Yetter 4-manifold invariant based on the modular data of the representations of quantum $\mathfrak{sl}_2$~\cite{CY}, and its Crane--Yetter--Kauffman generalization using the data of any semisimple ribbon category~\cite{CYK}.  Around the same time, given the data of a finite 2-group, Yetter defined a state sum (in any dimension in fact) generalizing the Dijkgraaf--Witten invariant associated to a finite group~\cite{Yetter}; this was later generalized by Faria Martins--Porter to include a twisting cocycle~\cite{fariamartinsporter}.  Mackaay attempted to systematize the data needed for a 4-dimensional state sum in a framework of certain monoidal 2-categories with trivial endomorphism categories, but the resulting notion did not encompass either the Crane--Yetter--Kauffman invariants or the Yetter--Dijkgraaf--Witten invariants and appears to only accommodate a twisted version of classical Dijkgraaf--Witten theory~\cite{Mackaay}.  More recently, given the data of a crossed-braided spherical fusion category, Cui constructed a state sum invariant of 4-manifolds that subsumes both the Crane--Yetter--Kauffman invariant and the Yetter--Dijkgraaf--Witten invariant, but does not incorporate either the twisted Yetter--Dijkgraaf--Witten case or hypothetical other instances of the Mackaay invariant~\cite{Cui}.

The quantum topology community has long expected that all these constructions should be expressible in a unified framework that associates a 4-dimensional field theory to some sort of `spherical fusion 2-category' (analogous to the Barrett--Westbury framework for 3-dimensional theories from spherical fusion 1-categories), but the appropriate notion of fusion 2-category and of sphericality has remained unclear.  In his recent survey article, \emph{Beyond Anyons}, Wang notes, ``One problem is to formulate a higher category theory that underlies all these theories, and study their application in 3-dimensional topological phases of matter"~\cite{Wang}.  In this paper, we completely address the relevant higher category theory by introducing a general purpose notion of fusion 2-category, based on a new notion of semisimple 2-category, and providing an appropriate corresponding sphericality condition.  We define, given the data of a spherical fusion 2-category, a piecewise-linear 4-manifold invariant that specializes (for appropriate choices of the fusion 2-category) to all the aforementioned invariants, and therefore provides a unified framework for 4-dimensional semisimple topological field theory.

\pagebreak

\addtocontents{toc}{\SkipTocEntry}

\subsection{Semisimple 2-categories}

We restrict attention to $k$-linear categories and 2-categories, where $k$ is an algebraically closed field of characteristic zero.

For a monoidal linear 1-category to produce a full-fledged 3-dimensional topological field theory, it must be fully-dualizable in some 3-category of monoidal linear 1-categories.  A convenient such 3-category is the 3-category of finite tensor categories; a finite tensor category is a category equivalent to the category of finite-dimensional modules over a finite-dimensional algebra, equipped with a monoidal structure such that every object has left and right duals.  Any fully-dualizable finite tensor category must be semisimple~\cite{DTC}; its underlying linear 1-category is therefore the category of modules for a finite-dimensional semisimple algebra.  It therefore stands to reason that in building a categorical framework for 4-dimensional topological field theory, we should look for a notion of monoidal semisimple 2-category, and that we might expect the underlying semisimple linear 2-category to be the 2-category of modules for a finite semisimple tensor category, i.e.\ a ``multifusion category".
\begin{maindef}
A \emph{semisimple 2-category} is a locally semisimple 2-category, admitting adjoints for 1-morphisms, that is additive and idempotent complete.
\end{maindef}
\vspace{-1pt}
\nid This is Definition~\ref{def:ss2cat} in the main text.  Here `locally semisimple' means that the Hom categories are semisimple linear categories, and \emph{idempotent complete} is shorthand for the property that every separable algebra 1-morphism admits a separable splitting (see Section~\ref{sec:ic2cat} and Appendix~\ref{app:ic} for extensive discussion of this condition).\footnote{A separable splitting of a separable algebra 1-morphism $E: B \to B$ (i.e.\ a separable monad) in a 2-category $\tc{C}$ is an adjunction $\iota \vdash \rho$ with right invertible counit, together with an isomorphism of algebras $E \cong \iota \xo \rho$.  The condition that a separable algebra 1-morphism in a locally idempotent complete 2-category admits a separable splitting is equivalent to the condition that it admits a universal left module, that is an `Eilenberg--Moore object', and also equivalent to the condition that it admits a universal right module, that is a `Kleisli object'.}\footnote{Morrison and Walker have sketched an elegant theory of completeness for $n$-categories.  We speculate that the notion of completeness we describe for $2$-categories is, informally speaking, related to their notion of completeness in the same way that framed 2-dimensional field theory is related to oriented 2-dimensional field theory.}\footnote{In the context of a modular tensor category representing excitations of a 2-dimensional topological phase of matter, the splitting of a commutative separable algebra can be thought of as `anyon condensation'~\cite{Kong}.}
The definition of semisimple 2-category does not explicitly demand the existence of any sort of additive decomposition of objects; nevertheless the local semisimplicity and the idempotent completeness conditions combine to ensure that, as one might hope given the name, in a semisimple 2-category every object decomposes as a finite direct sum of simple objects.  A semisimple 2-category is called \emph{finite} if it is locally finite semisimple and it has finitely many equivalence classes of simple objects.\footnote{Finite semisimple 1-categories are a categorification of finite-dimensional vector spaces, and so are often referred to as (finite) `2-vector spaces'.  Similarly, finite semisimple 2-categories are a categorification of finite semisimple 1-categories, and so may be thought of as (finite) `3-vector spaces'.}  

The above definition of semisimple 2-category does indeed have the desired relation to modules for multifusion categories.
\vspace{-3pt}
\begin{mainthm}
The 2-category of finite semisimple module categories of a multifusion category is a finite semisimple 2-category.
\end{mainthm}
\vspace*{-6pt}
\begin{mainthm}
Every finite semisimple 2-category is equivalent to the 2-category of finite semisimple module categories of a multifusion category.
\end{mainthm}
\vspace{-3pt}
\nopagebreak[4]
\nid These appear as Theorems~\ref{thm:multifusiontosemisimple} and~\ref{thm:semisimplefrommultifusion}.

Since finite semisimple 2-categories are exactly the 2-categories of modules for multifusion categories, one might wonder what utility finite semisimple 2-categories provide over the existing theory of multifusion categories.  The crucial advantage becomes apparent when we add a monoidal structure to these 2-categories.  In general, an additional monoidal structure on a multifusion category $\oc{C}$ would have to be encoded, somewhat intractably, as a $(\oc{C} \boxtimes \oc{C})$--$\oc{C}$-bimodule together with further associativity structures and conditions.  By contrast, a monoidal structure on a semisimple 2-category $\tc{C}$ will be describable \emph{functorially}, that is simply as an ordinary 2-functor $\tc{C} \times \tc{C} \to \tc{C}$.  A priori, a bimodule between tensor categories induces a 2-distributor between the associated 2-categories of modules.  (A finite semisimple 2-distributor $\tc{C} \hto{} \tc{D}$ is a bilinear 2-functor $\tc{D}^{\mathrm{op}} \times \tc{C} \to \tVect$, where $\tVect$ is the 2-category of `2-vector spaces', that is finite semisimple 1-categories.)  However, it turns out that, thanks to the idempotent completeness of semisimple 2-categories, any 2-distributor between semisimple 2-categories is a 2-functor.

\begin{mainthm}
Every finite semisimple 2-distributor between finite semisimple 2-categories is equivalent to a 2-functor.
\end{mainthm}

\nid This result appears as Corollary~\ref{cor:2dist}.

As an elementary example of a semisimple 2-category, consider the 2-category \linebreak $\Mod(\Vect(\ZZ_2))$ of finite semisimple module categories for the category $\Vect(\ZZ_2)$ of $\ZZ_2$-graded vector spaces.  This 2-category has two simple objects, namely the modules $\Vect$ and $\Vect(\ZZ_2)$; both these objects have endomorphism categories $\Vect(\ZZ_2)$, and the Hom category either direction between the two objects is $\Vect$.  The 2-category may therefore be drawn as follows:
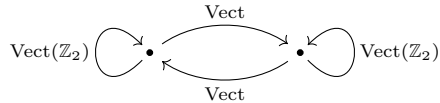
\[\begin{tz}
\node[dot,scale=1.25] at (0,0){};
\node[dot,scale=1.25] at (2,0){};
\draw[->,shorten <=0.2cm, shorten >=0.2cm] (0,0) to [out=35, in=145]node[above, sloped] {$\scriptstyle \Vect$}  (2,0);
\draw[<-,shorten <=0.2cm, shorten >=0.2cm] (0,0) to [out=-35, in=-145]node[below, sloped] {$\scriptstyle \Vect$}  (2,0);
\draw[->] (-0.1,-0.1) to [out=-135, in=135, looseness=15] node[left,xshift=.05cm] {$\scriptstyle \Vect(\ZZ_2)$} (-0.1,0.1);
\draw[->] (2.1,-0.1) to [out=-45, in=45, looseness=15] node[right,xshift=-.05cm] {$\scriptstyle \Vect(\ZZ_2)$} (2.1,0.1);
\end{tz}
\] 
\nid This and other examples are described in Section~\ref{sec:egss2cat}.  Note well that, as illustrated here and quite unlike the situation for semisimple 1-categories, in a semisimple 2-category there can be nontrivial morphisms between inequivalent simple objects.

\addtocontents{toc}{\SkipTocEntry}

\subsection{Fusion 2-categories}

Because transformations between semisimple 2-categories can be encoded functorially, a monoidal structure on a semisimple 2-category can be encoded as an ordinary `2-functorial' monoidal 2-category.

\begin{maindef}
A \emph{fusion 2-category} is a finite semisimple monoidal 2-category that has left and right duals for objects and a simple monoidal unit.
\end{maindef}

\nid This appears as Definition~\ref{def:fusion2cat}.  Examples of fusion 2-categories include 2-representations of a 2-group, 2-group-graded 2-vector spaces, modules for a braided fusion category, semisimple completions of crossed-braided fusion categories, and twisted versions thereof---see Section~\ref{sec:egfus}.  

As a simple example, consider the fusion 2-category of modules of the symmetric fusion category $\Vect(\ZZ_2)$. There are two simple objects, the identity $I$ (namely the module category $\Vect(\ZZ_2)$) and an object $X$ (namely the module category $\Vect$), with nontrivial fusion rule $X \xz X \simeq X \boxplus X$.  We may therefore depict this fusion 2-category as follows:
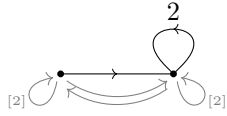
\[\begin{tz}
\node[dot,scale=1.25] at (0,0){};
\node[dot,scale=1.25] at (1.5,0){};
\draw[->,shorten <=0.1cm, shorten >=0.1cm,morgray,yshift=-1] (0,0) to [out=-35, in=-145]node[above, sloped,yshift=-1.25] {}  (1.5,0); 
\draw[<-,shorten <=0.1cm, shorten >=0.1cm,morgray,yshift=-4] (0,0) to [out=-35, in=-145]node[below, sloped,yshift=1.25] {}  (1.5,0); 
\draw[->,morgray] (-0.05,-0.1) to [out=-100, in=-170, looseness=25] node[left,xshift=1.25] {$\scriptscriptstyle [2]$} (-0.1,-0.05); 
\draw[->,morgray] (1.55,-0.1) to [out=-80, in=-10, looseness=25] node[right,xshift=-1.25] {$\scriptscriptstyle [2]$} (1.6,-0.05); 
\draw[arrow data={0.5}{>}] (0,0) to (1.5,0);
\draw[arrow data ={0.5}{>}] (1.5,0) to (1.6,0.1) to [out=45, in=135, looseness=12] node[above] {$2$} (1.4,0.1) to (1.5,0);
\end{tz}
\]
\nid Here the black directed edges represent multiplication by the object $X$ and the label indicates the multiplicity.  The gray edges record the morphism categories: an unlabeled edge indicates a rank 1 category, that is $\Vect$, and the label $[2]$ indicates a rank 2 category, that is $\Vect \boxplus \Vect$.  (In this case those rank 2 endomorphism fusion categories are $\Vect(\ZZ_2)$.)  This sort of fusion graph, where an object has no fusion product containing an identity factor, is a completely new phenomenon in fusion 2-categories---in a fusion 1-category, the product of an object and its dual always has an identity summand, but in a fusion 2-category, this need not happen thanks to the existence of nontrivial nonequivalence morphisms between simple objects.

As another example, the fusion 2-category obtained as the semisimple completion of a $\ZZ_4$-crossed-braided structure on the Ising fusion category has the following structure:
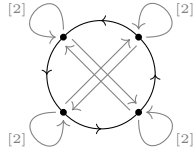
\[
\begin{tz}
\node[dot,scale=1.25] (A) at (.5,.5){};
\node[dot,scale=1.25] (B) at (-.5,.5){};
\node[dot,scale=1.25] (C) at (-.5,-.5){};
\node[dot,scale=1.25] (D) at (.5,-.5){};
\draw[arrow data={0.5}{>}] (A) to [out=135, in=45] (B);
\draw[arrow data={0.5}{>}] (B) to [out=-135, in=135] (C);
\draw[arrow data={0.5}{>}] (C) to [out=-45, in=-135] (D);
\draw[arrow data={0.5}{>}] (D) to [out=45, in=-45] (A);
\draw[->,shorten <=0.1cm, shorten >=0.1cm, morgray, xshift=1.25, yshift=-1.25] (.5,.5) to node[below] {} (-.5,-.5);
\draw[->,shorten <=0.1cm, shorten >=0.1cm,morgray, xshift=-1.25, yshift=1.25] (-.5,-.5) to node[below] {} (.5,.5);
\draw[->,shorten <=0.1cm, shorten >=0.1cm, morgray, xshift=1.25, yshift=1.25] (-.5,.5) to node[below] {} (.5,-.5);
\draw[->,shorten <=0.1cm, shorten >=0.1cm, morgray, xshift=-1.25, yshift=-1.25] (.5,-.5) to node[below] {} (-.5,.5);
\draw[->,morgray] (.6,.5) to [out=0,in=90,looseness=15] node[right] {$\scriptscriptstyle [2]$} (.5,.6);
\draw[->,morgray] (-.5,.6) to [out=90,in=180,looseness=15] node[left] {$\scriptscriptstyle [2]$} (-.6,.5);
\draw[->,morgray] (-.6,-.5) to [out=180,in=-90,looseness=15] node[left] {$\scriptscriptstyle [2]$} (-.5,-.6);
\draw[->,morgray] (.5,-.6) to [out=-90,in=0,looseness=15] node[right] {$\scriptscriptstyle [2]$} (.6,-.5);
\end{tz}
\]
\nid The fusion structure of the simple objects is just the cyclic group $\ZZ_4$; the black edges denote multiplication by a generating object.  Again the gray edges record the rank of the morphism categories.  (In this case, the rank 2 endomorphism fusion categories are all $\Vect(\ZZ_2)$.)  Note, again quite unlike what can happen in the context of fusion 1-categories, that there are only two connected components of simples in this fusion 2-category, despite the underlying order four fusion group of simple objects.

We expect that fusion 2-categories are fully-dualizable objects of an appropriate 4-category of tensor 2-categories, and therefore provide framed 4-dimensional local field theories,\footnote{The 4-category in question may be defined, using the framework developed by Johnson-Freyd--Scheimbauer~\cite{JFS}, as the $(\infty,4)$-category of algebras in the 3-category of finite semisimple 2-categories, linear 2-functors, natural transformations, and modifications.  We conjecture that fusion 2-categories are fully-dualizable objects of that 4-category.} 
but to obtain oriented field theories and therefore oriented 4-manifold invariants, we need additional structure and properties on the fusion 2-categories.  Recall that a \emph{planar pivotal 2-category} is a 2-category with a functorial involutive coherent choice of adjoint for each 1-morphism.  A \emph{monoidal planar pivotal 2-category} is a planar pivotal 2-category with a compatible monoidal structure.  A \emph{pivotal 2-category} is a monoidal planar pivotal 2-category with a compatible involutive coherent choice of dual for each object.  These definitions are given in detail in Section~\ref{sec:piv}.  (Note that what we call a `pivotal 2-category' is presumably equivalent to what Barrett--Meusburger--Schaumann call a `spatial Gray monoid'~\cite{BMS}.)

A pivotal 2-category will still not provide an oriented 4-dimensional field theory, just as a pivotal 1-category does not provide an oriented 3-dimensional field theory---what is needed is a sphericality condition.  Recall that a pivotal 1-category is called `spherical' when the left and right `circular' traces of any 1-endomorphism $f: A \to A$ agree:
\[
\begin{tz}[scale=0.5] 
\draw[wire, arrow data = {0.1}{>}, arrow data = {0.5}{>}, arrow data={0.93}{>}] (0.5,1.5) to (0.5, 1.75) to [out=up, in=up, looseness=2] (-0.5, 1.75) to (-0.5, 1.25) to [out=down, in=down, looseness=2] (0.5, 1.25) to (0.5, 1.5);
\node[dot] at (0.5,1.5){};
\node[tmor, right] at (0.5, 1.5) {$f$};
\node[obj, black, below right] at (0.45,1.15){$A$};
\end{tz}
~~ = ~~
\begin{tz}[scale=0.5,xscale=-1] 
\draw[wire, arrow data = {0.1}{>}, arrow data = {0.5}{>}, arrow data={0.93}{>}] (0.5,1.5) to (0.5, 1.75) to [out=up, in=up, looseness=2] (-0.5, 1.75) to (-0.5, 1.25) to [out=down, in=down, looseness=2] (0.5, 1.25) to (0.5, 1.5);
\node[dot] at (0.5,1.5){};
\node[tmor, left] at (0.5, 1.5) {$f$};
\node[obj, black, below left] at (0.5,1.15){$A$};
\end{tz}
\]
\nid Analogously, for a 2-endomorphism $\alpha: \Io_A \To \Io_A$ in a pivotal 2-category, there is a `front' and a `back 2-spherical trace' construction, which may be depicted graphically as follows:

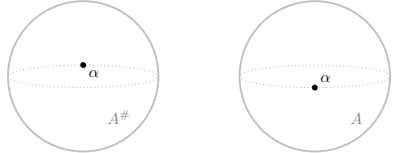
\[
\begin{tz}[scale=0.5]
  \draw[slice] (0,0) circle (2cm);
  \draw[tinydash] (-2,0) arc (180:360:2 and 0.3);
  \draw[tinydash] (2,0) arc (0:180:2 and 0.3);
  \node[dot] at (0,0.3){};
  \node[below right,tmor] at (0,0.3){$\alpha$};
     \node[obj,above left] at (\sqt, -\sqt) {$A^\#$};
\end{tz}
\hspace{30pt}
\begin{tz}[scale=0.5]
  \draw[slice] (0,0) circle (2cm);
  \draw[tinydash] (-2,0) arc (180:360:2 and 0.3);
  \draw[tinydash] (2,0) arc (0:180:2 and 0.3);
  \node[dot] at (0,-0.3){};
  \node[above right,tmor] at (0,-0.3){$\alpha$};
   \node[obj,above left] at (\sqt, -\sqt) {$A$};
\end{tz}
\]

\begin{maindef}
A \emph{spherical 2-category} is a pivotal 2-category such that the front and back 2-spherical traces agree.
\end{maindef}

\nid This appears as Definition~\ref{def:spherical} in the main text.  Examples of spherical fusion 2-categories include 2-representations of a 2-group, 2-group-graded 2-vector spaces, modules for a ribbon fusion category, and the semisimple completion of a crossed-braided spherical fusion category.\footnote{Note that Mackaay~\cite{Mackaay} used the term `spherical fusion 2-category' to refer to what we might call `circo-spherical endotrivial fusion 2-categories'---the `sphericality' condition there is the equivalence of two categorical circular traces, not two 2-spherical traces, and the endomorphism fusion category of every indecomposable object is the trivial fusion category $\Vect$.  A `circo-spherical endotrivial fusion 2-category' is a spherical fusion 2-category in our sense, but none of the aforementioned examples of spherical fusion 2-categories satisfy Mackaay's much more restrictive conditions.}  Spherical fusion 2-categories provide the desired data for constructing a state sum invariant of 4-manifolds.

\addtocontents{toc}{\SkipTocEntry}

\subsection{The state sum for combinatorial 4-manifolds}

Recall that a combinatorial $n$-sphere is a simplicial complex piecewise-linearly homeomorphic to the boundary of the standard $(n+1)$-simplex, and a combinatorial $n$-manifold is a simplicial complex such that the link of every vertex is a combinatorial $(n-1)$-sphere.  Furthermore, a singular combinatorial $n$-manifold is a simplicial complex such that the link of every vertex is a combinatorial $(n-1)$-manifold.  Given the data of a spherical fusion 1-category, Barrett and Westbury defined a state sum invariant of oriented combinatorial 3-manifolds.  They prove invariance of their state sum by explicitly showing invariance under each 3-dimensional bistellar move.  A bistellar move on a combinatorial $n$-manifold replaces a triangulation of a subcomplex isomorphic to a ball in $\partial \Delta^{n+1}$ by the complementary ball; these moves are known to generate piecewise-linear equivalence by a theorem of Pachner~\cite{pachner}.  Barrett and Westbury showed that in fact their invariant extends to singular combinatorial 3-manifolds by proving a generalization of the 3-dimensional case of Pachner's theorem: two singular combinatorial 3-manifolds are piecewise-linearly homeomorphic if and only if they are bistellar equivalent.

We will similarly show that our state sum is a piecewise-linear 4-manifold invariant by explicitly showing invariance under each 4-dimensional bistellar move.  And, in fact, this invariant will also extend to singular combinatorial 4-manifolds because of the following result.

\begin{mainthm}
Two singular combinatorial 4-manifolds are piecewise-linearly homeomorphic if and only if they are bistellar equivalent.
\end{mainthm}

\nid This appears as Theorem~\ref{thm:singular}.  The crucial fact that makes such a result possible is that the first nonshellable spheres do not appear until dimension 3.

We can now describe the state sum invariant itself, depending on a spherical fusion 2-category and a combinatorial 4-manifold.  The sum is over labelings of the 1-simplices of the manifold by simple objects of the 2-category, and of 2-simplices of the manifold by simple 1-morphisms; the numbers being summed are, roughly speaking, the 10j symbols of the fusion 2-category corresponding to the given configuration of object and 1-morphism labels.
\begin{maindef}
Given an oriented singular combinatorial 4-manifold $K$ and a spherical fusion 2-category $\tc{C}$, the \emph{state sum} is the number
\[
Z_{\tc{C}}(K) := 
\text{\scalebox{.9}{$
\sum_{\Gamma} 
\Big(\prod_{K_0} \dim(\tc{C})^{-1}\Big) 
\Big(\prod_{\sI \in K_1} \big(\dim(\Gamma(\sI)) \dim(\End_{\tc{C}}(\Gamma(\sI))) n(\Gamma(\sI))\big)^{-1}\Big) 
\Big(\prod_{\sII \in K_2} \dim(\Gamma(\sII))\Big) 
Z(\Gamma).
$}}
\]
\end{maindef}

\nid A more precise version of this definition appears as Definition~\ref{def:ss} in the main text.  Here, $K_i$ refers to the set of $i$-simplices of $K$, the sum is over the labelings $\Gamma$ as described above, the dimension $\dim(\tc{C})$ is a kind of `global dimension' of the underlying semisimple 2-category of $\tc{C}$, the dimensions $\dim(\Gamma(\sI))$ and $\dim(\Gamma(\sII))$ refer to appropriate scalar `2-spherical traces', $\dim(\End_{\tc{C}}(\Gamma(\sI)))$ is the dimension of that fusion 1-category, $n(\Gamma(\sI))$ is the number of equivalence classes of simple objects in the connected component of the object $\Gamma(\sI)$, and finally $Z(\Gamma)$ is the aforementioned 10j symbol, which may be thought of as (derivable from) a piece of structural data of the fusion 2-category.  This state sum definition is inspired by a combination of the state sums of Barrett--Westbury~\cite{BW}, Mackaay~\cite{Mackaay}, and Cui~\cite{Cui}.

Barrett and Westbury's state sum does not really need a full spherical fusion category as input---because the sum restricts attention to simple object labels, the nonsimple objects play essentially no role and therefore one need not insist that the input category have all direct sums.  Our situation is somewhat analogous---because our state sum uses only simple object and simple 1-morphism labels, the necessary information is contained in the simples, and their tensor products, and the maps between such.  In particular, we can drop the assumption that our fusion 2-category has direct sums, and even that it is idempotent complete.  This leads to a notion of `prefusion 2-category', a locally semisimple 2-category such that 1-morphisms have adjoints and every object decomposes as a direct sum of objects with simple identity, equipped with a monoidal structure for which objects have duals and the unit is simple.  Our state sum works without modification for spherical prefusion 2-categories.  This is convenient because various examples, for instance the deloop of a braided fusion category or a $G$-crossed-braided fusion category, are in the first instance prefusion 2-categories and must be additive and idempotent completed to obtain fusion 2-categories per se.

\begin{mainthm}
For an oriented singular piecewise-linear 4-manifold $M$ and a spherical prefusion 2-category $\tc{C}$, the numerical state sum $Z_{\tc{C}}(K)$ is independent of the choice of triangulation $K$ of $M$ and therefore defines an oriented piecewise-linear invariant of the singular 4-manifold $M$.
\end{mainthm}

\nid This Theorem appears, in a more precise form, as Theorem~\ref{thm:maintheorem}, and the proof, occupying all of Section~\ref{sec:statesumproof}, proceeds by direct combinatorial analysis of the effect of each bistellar move.  As mentioned previously, for appropriate choices of the spherical prefusion 2-category, our invariant specializes to the Crane--Yetter--Kauffman invariant for ribbon categories, the (twisted) Yetter--Dijkgraaf--Witten invariant for finite 2-groups, the Mackaay invariant for endotrivial fusion 2-categories, and the Cui invariant for crossed-braided fusion categories.\footnote{B\"arenz and Barrett use the handle calculus to define a smooth 4-manifold invariant given the data of a pivotal functor from a spherical fusion 1-category to a ribbon fusion 1-category~\cite{BB}.  When the target ribbon category is modularizable, this invariant can be reformulated as a state sum invariant.  We wonder whether there a fusion 2-category for which our state sum invariant gives the modularizable-target case of the B\"arenz--Barrett invariant.}

Note that 4-dimensional field theories associated to fusion 2-categories may not directly distinguish distinct smooth structures on closed 4-manifolds---for that one would presumably need nonsemisimple or derived algebraic structures, among other modifications.  Nevertheless, as fusion 1-categories and particularly their classification have proven compelling quite apart from their associated closed 3-manifold invariants, we imagine fusion 2-categories and their classification will be a worthwhile subject in its own right.

\addtocontents{toc}{\SkipTocEntry}

\subsection{Overview}

\skiptocparagraph{Guide}

Readers exclusively interested in the theory of fusion 2-categories can restrict attention to Sections 1 and 2 and Appendices A and B.  Readers primarily interested in the state sum construction will want to focus on Sections 3 and 4 and Appendix B, but will still need to at least skim Sections 1 and 2.  Many readers, whether expert or novice, will want to merely skim Section 1.3, which is both more motivational and more technical, and most readers will want to skip Appendix A (containing the proofs about idempotent completion) and Appendix B (containing the proofs of the dimension formulas).  Experts can skim or skip Sections 3.1, 3.2, and 3.4, and only the hardiest souls will want to get into the combinatorial invariance proof in Section 4.3.  The recommended fast-track is therefore Sections 1.1, 1.2, 1.4, 2.1, 2.2, 2.3, 3.3, and then only as much of the proof in 4.1, 4.2, and 4.3 as required for one's purposes.

\skiptocparagraph{Outline}

Section 1 concerns linear 2-categories; it begins by recalling the notions of linear 2-categories and their functors, and by setting our conventions and notation for compositions and identities in 2-categories.  Section 1.1 discusses zero objects and direct sums in 2-categories and defines additive 2-categories.  Section 1.2, after recalling the landscape of presemisimple and semisimple 1-categories, defines presemisimple 2-categories, analyzes the direct sum decomposition of objects in such 2-categories, and defines the dimension of a presemisimple 2-category.  Section 1.3 motivates categorified notions of idempotent and idempotent splitting, and defines idempotent complete 2-categories and an idempotent completion operation on 2-categories.  Appendix A is a companion to Section 1.3, providing further technical details and many proofs concerning idempotent completion that are omitted from the main text.  Section 1.4 defines semisimple 2-categories, proves the crucial result that semisimple 2-categories are exactly 2-categories of modules of multifusion categories, and shows that 2-distributors between semisimple 2-categories are 2-functors; it then gives a variety of explicit constructions and examples of semisimple 2-categories.

Section 2 investigates monoidal structures on linear 2-categories.  Section 2.1 gives a precise definition of a monoidal 2-category, defines prefusion and fusion 2-categories, and describes the graphical calculus of surfaces with defects that encodes 2-morphisms in a fusion 2-category; it then gives an extensive list of constructions and examples of fusion 2-categories.  Section 2.2 recalls the notion and graphical calculus of planar pivotal 2-categories, and describes further the notion and graphical calculus for pivotal 2-categories; then it defines 2-spherical traces in pivotal 2-categories.  Section 2.3 uses this notion of 2-spherical trace to define the notion of spherical 2-categories, mentions examples of spherical 2-categories, and then discusses dimensions of objects and 1-morphisms in spherical 2-categories.

Section 3 defines our state sum invariant of singular piecewise-linear 4-manifolds.  Section 3.1 recalls the relevant notions of combinatorial and singular combinatorial manifolds.  Section 3.2 recalls bistellar transformations of combinatorial manifolds and the fundamental theorem that they generate piecewise-linear equivalence, and then proves that in fact bistellar moves generate piecewise-linear equivalence of singular combinatorial 4-manifolds.  Section 3.3 motivates and discusses the technicalities of our state-sum labeling scheme, defines the 10j symbol and 10j action maps that form the core numerical information in the state sum, and finally gives the full state sum expression.  Section 3.4 discusses how the state sum specializes, given specific choices of fusion 2-categories, to the Crane--Yetter--Kauffman, Yetter--Dijkgraaf--Witten, Mackaay, and Cui invariants.

Section 4 proves the invariance of the state sum.  Section 4.1 shows that the sum is invariant of the labeling skeleton, that is of the choice of the particular representative simple objects and 1-morphisms used as labels; despite sounding innocuous, this invariance is nontrivial, in part because the dimensions of equivalent objects need not be the same, and the proof relies crucially on the local combinatorial manifold structure.  Section 4.2 shows that the state sum is invariant under changing the chosen global order on the vertices, which was used to orient all the simplices and thus define the labeling scheme for the state sum; again despite sounding like a minor matter, the terms of the state sum change substantially under vertex reordering because object and 1-morphism labels are being replaced by composites of duals and adjoints, and proving invariance here uses much of the power of the pivotal 2-categorical context.  Finally, Section 4.3 proves the state sum is invariant under changing the combinatorial structure of the 4-manifold, that is under the bistellar moves; it does so by defining a state sum for manifolds with boundary, in order to isolate the effect of a local bistellar move, and then explicitly computing each of the three 4-dimensional bistellar moves.  This last bistellar proof uses an extensive collection of formulas relating the dimensions of objects and 1-morphisms in fusion 2-categories, formulas which are established in Appendix~B.

\addtocontents{toc}{\SkipTocEntry}
\subsection*{Acknowledgments}

We are grateful to Scott Morrison and Kevin Walker, whose work in higher category theory and topological field theory has been an ongoing inspiration for us.  We thank Bruce Bartlett, Andr\'e Henriques, Chris Schommer-Pries, Noah Snyder, and Jamie Vicary for helpful discussions. We thank Patrick Hayden and the Stanford Institute for Theoretical Physics for hosting us for an extended research visit, and we especially thank Shawn Cui for sharing and explaining his recent work, which substantially influenced our research.

\newpage

\renewcommand{\thesubsection}{\thesection.{\arabic{subsection}}}

\addtocontents{toc}{\protect\vspace{8pt}}

\section{Linear $2$-categories}\label{sec:linear2cats}

Except where otherwise noted, we assume the field $k$ is algebraically closed and of characteristic zero.  Let $\Vect$ denote the $1$-category of finite-dimensional $k$-vector spaces and linear maps. In the following, a \emph{linear $1$-category} is a $\Vect$-enriched $1$-category and a \emph{linear $1$-functor} is a $\Vect$-enriched functor.

By a $2$-category, we mean a weak $2$-category, though throughout we will suppress unitors and associators from our notation.  We define a \emph{linear $2$-category} to be a $\Vect$-enriched $2$-category, that is, a $2$-category $\tc{C}$ whose $2$-morphism sets are $k$-vector spaces such that horizontal and vertical composition of $2$-morphisms are $k$-bilinear operations. A \emph{linear $2$-functor} will be a $2$-functor that is locally $k$-linear. 

Similar to~\cite[Sec 2.1]{DTC}, we will distinguish between the \emph{geometric} and the \emph{functorial} direction of composition. If $A$, $B$ and $C$ are $i$-morphisms in an $n$-category, and $f:A\to B$ and $g:B\to C$ are $(i+1)$-morphisms, we denote their composition `$f$ followed by $g$' in functorial notation as $g\circ f$ and in geometric notation as $f\otimes_B g:= g\circ f$. The object $B$ is often left implicit and $f\otimes_B g$ is simply denoted by $f\otimes g$. For monoidal $2$-categories, we will reserve the following symbols for the various composition conventions:
\[
\begin{array}{l|c|c|c}
&
\text{object}
&
\text{$1$-morphism}
&
\text{$2$-morphism}
\\ \hline
\text{functorial}
& 
\sq & \circ & \cdot
\\ \hline
\text{geometric} & \boxtimes & \otimes & \times
\end{array}
\]
For uniformity, we will always use the functorial composition.  Note that this is not the most naive categorification of the typical convention for tensor categories, which uses functorial composition for morphisms but geometric composition for objects.

We adopt the following notation:
\begin{itemize}
\item[-] Given an object $A$ in a $2$-category, we denote its identity $1$-morphism by $\Io_A$; for a $1$-morphism $f:A\to B$, we denote its identity $2$-morphism by $\It_f$. 
\item[-] Given 1-morphisms $f,f':A \to B$ and $g,g':B\to C$ and 2-morphisms $\alpha: f\To f'$ and $\beta: g\To g'$, we often abbreviate the composites $\It_g\xo \alpha$ and $\beta \xo \It_f$ by $g\xo \alpha$ and $\beta \xo f$, respectively.
\item[-] We write $A\equiv B$ for equivalent objects in a $2$-category, and $f\iso g$ for isomorphic $1$-morphisms.
\item[-] For $1$-morphisms $f:A\to B$ and $g:B\to A$ in a $2$-category, we write $f \dashv g$, and say $g$ is right adjoint to $f$ or equivalently $f$ is left adjoint to $g$, if there are $2$-morphisms $\epsilon: f \circ g \To \Io_B,~\eta: \Io_A \To g \circ f$ such that $(\epsilon \circ \It_{f}) \cdot (\It_{f} \circ \eta) = \It_{f}$ and $(\It_g \circ \epsilon) \cdot (\eta \circ \It_g) = \It_g$. In the following, we will often refer to these equations as the \emph{cusp equations}. 
\end{itemize}

\subsection{Additive $2$-categories}

\skiptocparagraph{Zero objects}

Recall that a \emph{zero object} in a linear $1$-category is an object $0$ such that $\Hom(0,0)$ is the zero vector space.  Note that an object whose identity morphism is the zero vector is necessarily a zero object.  A zero object is unique up to equivalence, and is preserved by all functors.  A $1$-category with a zero object is called \emph{pointed}. A \emph{zero $1$-morphism} in a linear $2$-category is a $1$-morphism $0_{A,B}:A\to B$ that is a zero object in the linear $1$-category $\Hom(A,B)$.  A linear $2$-category is \emph{locally pointed} if all its $1$-morphism categories are pointed, that is have zero objects.
\begin{definition}[Zero object in a 2-category]\label{def:zeroobject} 
A \emph{zero object} in a locally pointed linear $2$-category $\tc{C}$ is an object $\zero$ such that $\Hom_{\tc{C}}(\zero,\zero)$ is the terminal $1$-category.  
\end{definition}
\nid%
Note that an object $\zero$ in a locally pointed linear $2$-category is a zero object if and only if its identity $1$-morphism $\Io_\zero: \zero \to \zero$ is a zero $1$-morphism.  Also, observe that a zero object in a locally pointed linear $2$-category is uniquely determined up to equivalence, and is preserved by all linear $2$-functors.  We say that a linear $2$-category is \emph{pointed} if it is locally pointed and has a zero object.

\skiptocparagraph{Direct sums}

Recall that the direct sum of two objects $A_1$ and $A_2$ in a linear $1$-category is an object $A_1\oplus A_2$ with inclusion morphisms $i_j: A_j \to A_1\oplus A_2$, and projection morphisms $p_j: A_1\oplus A_2 \to A_j$, fulfilling the following conditions:
\begin{itemize}
\item[-] $p_j \cdot i_j=\It_{A_j}$ for $j=1,2$;
\item[-] $p_1\cdot i_2 =0 $ and $p_2\circ i_1 =0 $;
\item[-] $i_1\cdot  p_1 + i_2 \cdot p_2 = \It_{A_1\oplus A_2}$.
\end{itemize}
A linear $1$-category is \emph{additive} if it has a zero object and has (pairwise) direct sums. A \emph{locally additive} $2$-category is a linear $2$-category whose Hom-categories $\Hom(A,B)$ are additive for all objects $A$ and $B$.

\begin{definition}[Direct sum in a 2-category]\label{def:directsum} A \emph{direct sum} of two objects $A_1$ and $A_2$ in a locally additive $2$-category $\tc{C}$ is an object $A_1\bplus A_2$ together with inclusion and projection $1$-morphisms $\iota_i: A_i \to A_1\bplus A_2$ and $\rho_i: A_1\bplus A_2 \to A_i$ for $i=1,2$, satisfying the following conditions:
\begin{itemize}
\item[-] $\rho_i \circ \iota_i$ is isomorphic to $\Io_{A_i}$ for $i=1,2$; 
\item[-] $\rho_2\circ \iota_1 \in \Hom_{\tc{C}}(A_1, A_2)$ and $\rho_1\circ \iota_2 \in \Hom_{\tc{C}}(A_2, A_1)$ are zero objects;
\item[-] $\Io_{A_1\bplus A_2} \in \Hom_{\tc{C}}(A_1\bplus A_2,A_1\bplus A_2)$ is a direct sum of $\iota_1\circ \rho_1$ and $\iota_2 \circ \rho_2$.
\end{itemize}
\end{definition}
\nid Observe that direct sums in a locally additive $2$-category are uniquely determined up to equivalence and are preserved by all linear $2$-functors.

\begin{prop}[Projection and inclusion are adjoint]\label{prop:coherencedirectsums} Let $A_1\boxplus A_2$ be a direct sum with inclusion and projection $1$-morphisms $\iota_i: A_i \to A_1\bplus A_2$ and $\rho_i:A_1\bplus A_2 \to A_i$ for $i=1,2$. Then $\rho_i$ is both a left and right adjoint of $\iota_i$.
\end{prop}
\begin{proof} The proof is analogous to the proof that every equivalence in a $2$-category can be promoted to an adjoint equivalence.  We show that $\iota_i$ is right adjoint to $\rho_i$; the proof for left-adjointness is similar.  Let $\widetilde{\epsilon}_i: \rho_i \xo \iota_i \To \Io_{A_i}$, $\eta_i: \Io_{A_1\bplus A_2} \iso \iota_1\xo \rho_1 \oplus \iota_2 \xo \rho_2 \To \iota_i \xo \rho_i$, and $\overline{\eta}_i: \iota_i \xo \rho_i \To \iota_1\xo \rho_1\oplus \iota_2 \xo \rho_2 \iso \Io_{A_1\bplus A_2}$ be the 2-morphisms provided by the definition of the direct sum. Define 
\[ \epsilon_i := \widetilde{\epsilon}_i\xt\left( 1_{\rho_i} \xo \overline{\eta}_i \xo 1_{\iota_i}\right)\xt\left(1_{\rho_i\xo\iota_i} \xo \widetilde{\epsilon}_i^{-1}\right): \rho_i \xo \iota_i \To \Io_{A_i}
\]
Now observe that $\eta_i$ and $\epsilon_i$ are the unit and counit of an adjunction $\rho_i\dashv \iota_i$. Checking that $\left(\It_{\iota_i} \xo \epsilon_i\right)\xt\left(\eta_i \xo \It_{\iota_i}\right) = \It_{\iota_i}$ is straightforward. The equation $\left(\epsilon_i \xo \It_{\rho_i} \right) \xt \left( \It_{\rho_i}\xo \eta_i\right) = \It_{\rho_i}$ follows (by the same calculation that shows that after appropriately modifying the 2-morphisms of an equivalence, one obtains an adjoint equivalence) using the fact that $\sum_j \overline{\eta}_j \xt \eta_j  = \It_{\Io_{A_1\bplus A_2}}$, hence that $\It_{\rho_i} = \It_{\rho_i} \xo \left(\overline{\eta}_i\xt \eta_i\right)$ and in particular that $1_{\rho_i} \xo \eta_i$ and $1_{\rho_i} \xo \overline{\eta}_i$ are inverse. 
\end{proof}

\begin{remark}[Direct sums and zero objects in 2-categories are preserved by all 2-functors] \label{rem:absolute} 
In a linear 2-category, direct sums and zero objects are `equational' constructions, in that they may be defined in terms of the existence of certain morphisms satisfying certain equations.  It follows that they are preserved by all linear 2-functors.  Recall that an absolute (2-)colimit is a colimit preserved by all linear (2-)functors.  Both direct sums and zero objects may be expressed as universal constructions, in particular as colimits, and are therefore absolute colimits.
\end{remark}

\skiptocparagraph{Additivity and additive completion}

\begin{definition}[Additive 2-category] 
A linear $2$-category is \emph{additive} if it is locally additive, has a zero object, and has direct sums.
\end{definition}

\begin{construction}[Additive completion of a 2-category] \label{con:additivecompletion}
Any locally additive linear $2$-category $\tc{C}$ can be completed to an additive $2$-category $\tc{C}^{\boxplus}$; here $\tc{C}^{\boxplus}$ has as objects finite (possibly empty) lists of objects of $\tc{C}$, as $1$-morphisms matrices of $1$-morphisms in $\tc{C}$, and as $2$-morphisms matrices of $2$-morphisms in $\tc{C}$. Horizontal composition of $1$-morphisms in $\tc{C}^{\boxplus}$ is `matrix multiplication' with sum and product replaced by direct sum and composition in $\tc{C}$.  \looseness=-2
\end{construction}

\begin{remark}[Additive completion is idempotent] \label{rem:additivecompletion}
If $\tc{C}$ is already additive, then $\tc{C}^{\boxplus}$ is equivalent to $\tc{C}$; this is a consequence of the fact that direct sums and zero objects are absolute colimits and that $\tc{C}^{\boxplus}$ is the free cocompletion under these colimits.
\end{remark}

\subsection{Presemisimple $2$-categories}

We would like to restrict our attention to $2$-categories in which all objects are sums of a fixed finite list of `simple' objects.  We want not only that those component objects are indecomposable, but that they are in fact simple, in the technical sense of having no subobjects, and furthermore that their endomorphism categories are as simple as may reasonably be expected.

\subsubsection{Presemisimple and semisimple $1$-categories}

We recall the following $1$-categorical notions.
\begin{itemize}
\item[-] In a $1$-category, a $1$-morphism $f: A \to B$ is a \emph{monomorphism} if, for all objects $X$, the map $\Hom(X,A) \to \Hom(X,B)$ is an injection of sets.
\item[-] A \emph{subobject} of an object $B$ is an isomorphism class of monomorphisms $A \to B$.  
\item[-] A nonzero object $B$ of a linear $1$-category is \emph{simple} if its subobjects are all either zero objects or isomorphisms.  
\item[-] A linear $1$-category is \emph{presemisimple} if every object can be decomposed as a finite direct sum of simple objects, and the composition of any two nonzero morphisms between simple objects is again nonzero.\footnote{Recall that an algebra is called a domain if it has no zero divisors; the composition condition in the definition of presemisimple 1-category can be understood as insisting that the category be a many-object version of a domain.}
\item[-] A linear $1$-category is \emph{idempotent complete} if every idempotent splits, that is for any $1$-morphism $\gamma: B \to B$ such that $\gamma\gamma = \gamma$, there exists morphisms $\iota: A \to B$ and $\rho: B \to A$ such that $\rho \xt\iota = \It_A$ and $\iota \xt\rho = \gamma$.
\item[-] A linear $1$-category is \emph{semisimple} if it is presemisimple, additive, and idempotent complete.\footnote{One may alternatively define a linear 1-category to be semisimple if it is additive, idempotent complete, and the endomorphism algebra of every object is semisimple.  Note that every semisimple linear $1$-category is necessarily semisimple abelian~\cite[Lemma 2]{Jannsen}.\label{foot:ss}}
\item[-] A semisimple $1$-category is \emph{finite} if there are only finitely many isomorphism classes of simple objects.
\end{itemize}

\begin{remark}[1-categories in which objects split into simples] \label{rem:onepressdef}
Note that in a linear $1$-category, asking merely that every object can be decomposed into a finite direct sum of simple objects achieves very little. For example, the category with one object whose endomorphism algebra is the algebra of 2-by-2 matrices satisfies this condition, even though it has nontrivial idempotents and the object will decompose upon idempotent completion.  In fact, in any finite-dimensional algebra, a left-cancellative element is necessarily invertible; thus, in any category with one object and endomorphism algebra a finite-dimensional algebra, the object is simple, and so the category has the property that `every object decomposes into a finite direct sum of simples'.
\end{remark}

\begin{remark}[Endomorphism algebras in semisimple 1-categories are semisimple]
In a semisimple 1-category, the endomorphism algebra of any object is semisimple (and therefore a semisimple 1-category is semisimple in the alternative sense of Footnote~\ref{foot:ss}).  Indeed, a finite-dimensional algebra is a division algebra if and only if it is a domain.  The endomorphism algebra of a simple object in a semisimple 1-category is therefore a division algebra.  Furthermore, the fact that composites of nonzero morphisms are nonzero implies that there are no nonzero morphisms between nonisomorphic simple objects.  It follows that the endomorphism algebra of any object is semisimple.

Note that that the additive and idempotent completion of a presemisimple 1-category is indeed semisimple: in this case the additive and idempotent completion does not produce any new simple objects, and so the category remains presemisimple, as required.
\end{remark}

\begin{remark}[Semisimple endomorphism algebras implies semisimple 1-category] \label{rem:ssconditions}
A linear 1-category that is additive and idempotent complete, and in which the endomorphism algebra of every object is semisimple (cf Footnote~\ref{foot:ss}), is necessarily presemisimple (and therefore semisimple); that is, in such a category every object decomposes as a finite direct sum of simple objects, and the composite of nonzero morphisms between simple objects is nonzero.
\end{remark}

\begin{remark}[Direct sums and idempotent splittings are preserved by all 1-functors]
In a linear 1-category, both direct sums and idempotent splittings may be expressed as equational constructions.  It follows that they are not generic colimits but in fact absolute colimits, in that they are preserved by all linear 1-functors.  Thus, we do not need to restrict the sort of functors considered when we are restricting attention to semisimple 1-categories.
\end{remark}

\subsubsection{The definition of presemisimple $2$-categories}

We now discuss analogous $2$-categorical notions.  A $1$-morphism $f: A \to B$ in a $2$-category $\tc{C}$ is \emph{fully faithful} if, for all objects $X$, the functor $\Hom_{\tc{C}}(X,A) \to \Hom_{\tc{C}}(X,B)$ is fully faithful; a \emph{subobject} of an object $B$ of a $2$-category is then an equivalence class of fully faithful $1$-morphisms $A \to B$.
\begin{definition}[Simple object in a 2-category]
A nonzero object in a linear $2$-category is \emph{simple} if its subobjects are either zero objects or equivalences.
\end{definition}
\nid%
A linear $2$-category is called \emph{locally semisimple} if all of its Hom categories are semisimple, and \emph{locally finite semisimple} if all its Hom categories are finite semisimple.

\begin{definition}[Decomposable object in a 2-category]
An object in a linear $2$-category is \emph{decomposable} if it is equivalent to a direct sum of nonzero objects, and it is \emph{indecomposable} if it is nonzero and not decomposable.
\end{definition}

\nid
For conciseness we will use the following terminology:
\begin{itemize}
\item[-] A multifusion category is a finite semisimple linear monoidal $1$-category whose objects have left and right duals.
\item[-] A fusion category is a multifusion category whose tensor unit is simple.
\item[-] An infusion category is a semisimple linear monoidal $1$-category whose objects have left and right duals and whose tensor unit is simple.
\end{itemize}
Note that an infusion category may have infinitely many isomorphism classes of simple objects.  We will see later (in Corollary~\ref{cor:fusiondomain}) that infusion categories, though not quite `categorical division algebras', function as `categorical domains'.

We might expect to define a presemisimple 2-category to be a linear 2-category in which every object decomposes into simple objects and in which the composition of nonzero morphisms between simples is nonzero; in fact, we will find (see Proposition~\ref{prop:Schur}) that this composition property is implied merely asking the endomorphism categories of simples to be infusion categories, that is categorical domains.

\begin{definition}[Presemisimple 2-category] \label{def:presstwocat}
A \emph{presemisimple $2$-category} is a locally semi\-\linebreak simple $2$-category such that every $1$-morphism admits a right adjoint and a left adjoint, such that every object is decomposable as a finite direct sum of simple objects, and such that the endomorphism category of every simple object is an infusion category.
\end{definition}

\nid We will see, from Corollary~\ref{cor:pressdirectsum} below, that this definition is equivalent to the following somewhat more compact definition: \emph{A presemisimple $2$-category is a locally semisimple $2$-category such that every $1$-morphism admits a right adjoint and a left adjoint and such that every object is decomposable as a finite direct sum of objects with simple identity.}

\begin{definition}[Finite presemisimple 2-category]
A presemisimple $2$-category is \emph{finite} if its Hom-categories are finite semisimple and if it has a finite number of equivalence classes of simple objects.
\end{definition}

\begin{remark}[2-categories in which objects split into simples]
As in the 1-categorical case, cf Remark~\ref{rem:onepressdef}, it is not particularly useful to consider linear (locally semisimple) 2-categories (with adjoints) merely such that every object decomposes as a finite sum of simples.  In such a $2$-category, the endomorphism categories of simple objects can be arbitrarily complicated and supposedly simple objects may decompose after an idempotent completion operation on the 2-category.\end{remark}

\begin{remark}[Endomorphism categories are multifusion] \label{rem:endismultifusion}
Observe that in a finite presemisimple $2$-category, the endomorphism category of any object is a multifusion category.
\end{remark}

\begin{example}[The delooping of an infusion category]\label{eg:deloopfusion}
Associated to an infusion category $\oc{D}$, there is a presemisimple $2$-category $\B\oc{D}$ with a unique object $*$ and the endomorphism category $\Hom_{\B\oc{D}}(*,*) = \oc{D}$.
\end{example}

\begin{construction}[The unfolded finite presemisimple 2-category of a multifusion category]\label{con:unfold}
More generally, let $\oc{D}$ be a multifusion category, and let $I \iso \bigoplus_{i \in \mathcal{I}} I_i$ be the simple decomposition of the tensor unit of $\oc{D}$.  Let $\oc{D}_{i,j}$ be the full additive subcategory of $\oc{D}$ containing the simple objects $X$ that fulfill $I_j \otimes X \iso X \iso X \otimes I_i$.  Recall that $\oc{D}_{i,i}$ is a fusion category, $\oc{D}_{i,j}$ is a $\oc{D}_{i,i}$-$\oc{D}_{j,j}$-bimodule category, and as a linear $1$-category, $\oc{D} \iso \bigoplus_{i,j} \oc{D}_{i,j}$~\cite[Sec 2.4]{ENO}.  Associated to this multifusion category $\oc{D}$, there is a finite presemisimple $2$-category $\tc{D}$, the \emph{unfolded $2$-category} of $\oc{D}$, with objects $i \in \mathcal{I}$ and $1$-morphism categories $\Hom_{\tc{D}}(i,j) := \oc{D}_{i,j}$. (Regarding the process of `folding' and `unfolding' between tensor $1$-categories and linear $2$-categories, see Kuperberg~\cite{KuperbergFolded}.) 
\end{construction}

\begin{construction}[The folded multifusion category of a finite presemisimple 2-category]\label{con:fold}
Given a finite presemisimple $2$-category $\tc{D}$, one can conversely consider the associated \emph{folded multifusion category} $\oc{D}$; if $I$ is a set of representative simple objects of $\tc{D}$, then the folded category is defined as $\oc{D} := \oplus_{(i,j) \in I \times I} \Hom_{\tc{D}}(i,j)$.  Note well that folding and unfolding are not strictly inverse operations.  For instance, a presemisimple $2$-category $\tc{C}$ with only simple objects will have the same folding as its additive completion $\tc{C}^\boxplus$.  (Note also that the unfolding of a multifusion category is never additive.)  Nevertheless, we do expect that folding and unfolding produce inverse equivalences between the 3-category of multifusion categories and the 3-category of finite presemisimple $2$-categories, where the $1$-morphisms of multifusion categories are finite semisimple bimodule categories and the $1$-morphisms between finite presemisimple $2$-categories are finite semisimple 2-distributors (also called `$2$-profunctors', see Section~\ref{sec:completion}).  Thus we can consider giving a multifusion category as a method for providing the data of a finite presemisimple $2$-category.
\end{construction}

\subsubsection{Decomposition in presemisimple $2$-categories} \label{sec:decomppress}

Note that a presemisimple $2$-category is not assumed to be additive, that is, it need not have a zero object or direct sums of objects, and it has no $1$-morphism-level idempotent-completeness condition.  Nevertheless, presemisimple $2$-categories have a reasonably well behaved notion of decomposition of objects, as follows.

\skiptocparagraph{Simple objects and simple identities correspond}

\begin{proposition}[Decomposition with simple identities implies simple objects and simple identities correspond] \label{prop:simpleiffidentity}
Let $\tc{C}$ be a locally semisimple $2$-category such that every $1$-morphism admits a right adjoint and a left adjoint and such that every object is decomposable as a finite direct sum of objects with simple identity.  Then an object $X$ is simple if and only if the identity $1$-morphism $\Io_X \in \Hom_{\tc{C}}(X,X)$ is simple.
\end{proposition}
\begin{proof}
A simple object can only decompose as itself (otherwise it would have a nontrivial subobject), and so by the decomposition assumption, it must have simple identity.

Now suppose $X$ is an object with simple identity, and let $R: A \to X$ be a subobject, that is a fully faithful 1-morphism from a nonzero object $A$ to $X$.  Let $L: X \to A$ be a left adjoint of $R$, with unit $\eta: \Io_X \To R \xo L$ and counit $\epsilon:L \xo R \To \Io_A$.  We will show that $R$ is an equivalence (with inverse $L$), and thus $X$ is simple.  We do so by explicitly constructing inverses of the counit and unit of the adjunction.

Since by assumption $R: A \to X$ is fully faithful, the functor $R \xo -: \Hom_{\tc{C}}(A,A) \to \Hom_{\tc{C}}(A,X)$ is full and faithful.  In particular, by fullness, the composite $\eta \xo R: R \To R \xo L \xo R$ is equal to $R \xo \delta$ for some 2-morphism $\delta: \Io_A \To L \xo R$.  Postcomposing the equation $L \xo \eta \xo R = L \xo R \xo \delta$ with $\epsilon \xo L \xo R$ implies, using the cusp equation, that $\delta \xt \epsilon = \It_{L\xo R}$.  The functor $R \xo -$ sends both $\epsilon \xt \delta$ and $\It_{\Io_A}$ to $(R \xo \epsilon) \xt (R \xo \delta) = (R \xo \epsilon) \xt (\eta \xo R) = \It_R$, by the cusp equation.  Faithfulness of $R \xo -$ implies that $\epsilon \xt \delta = \It_{\Io_A}$.

Now the counit $\eta: \Io_X \To R \xo L$ is certainly nonzero, since otherwise by the cusp equation $\It_L$ would be zero, implying $L$ is zero and therefore $R$ is zero, contradicting the fact that $R$ is faithful and $A$ is nonzero.  By local semisimplicity and the simplicity of $\Io_X$, there is a 2-morphism $\alpha: R \xo L \To \Io_X$ such that $\alpha \xt \eta = \It_{\Io_X}$.  By the fullness of $R$, there is a 2-morphism $r: \Io_A \To \Io_A$ such that $R \xo r = (\alpha \xo R) \xt (R\xo \delta) \in \Hom_{\Hom_{\tc{C}}(A,X)}(R,R)$.  Since $\epsilon$ and $\delta$ are inverse, precomposing this equation with $R \xo \epsilon$ gives $\alpha \xo R = (R\xo r) \xt (R\xo \epsilon)$, and further precomposing with $\eta \xo R$ gives $R = (\alpha \xt \eta) \xo R = R \xo r$.  Faithfulness of $R$ implies $r = \It_{\Io_A}$, so $R = (\alpha \xo R) \xt (R\xo \delta)$.  That last equation is (by 1-morphism precomposing with L and then postcomposing with $\eta$, respectively by 1-morphism precomposing with R and then precomposing with $\epsilon$) equivalent to the equation $\eta \xt \alpha = \It_{R \xo L}$.
\end{proof}

\begin{corollary}[Decomposition with simple identities implies presemisimple] \label{cor:pressdirectsum}
A locally semisimple $2$-category is presemisimple if and only if every $1$-morphism admits a right adjoint and a left adjoint and every object decomposes as a finite direct sum of objects with simple identity.
\end{corollary}
\begin{corollary}[Projections and inclusions are simple] \label{cor:simpleprojection}
Let $\bigboxplus X_i$ be a finite direct sum of simple objects in a presemisimple $2$-category $\tc{C}$.  The projection and inclusion $1$-morphisms $\iota_i:X_i \leftrightarrows \bigboxplus X_i:\rho_i$ are necessarily simple $1$-morphisms.
\end{corollary}
\begin{proof} By Proposition~\ref{prop:coherencedirectsums}, we have $\End_{\tc{C}}(\iota_i) \iso \Hom_{\tc{C}}(\Io_{X_i}, \rho_i\xo\iota_i)\iso \End_{\tc{C}}(\Io_{X_i})$; hence $\iota_i$ is simple if and only if $\Io_{X_i}$ is, and $\Io_{X_i}$ in turn is simple if and only if $X_i$ is.  Since $\iota_i$ and $\rho_i$ are adjoint, taking mates induces an isomorphism $\End_{\tc{C}}(\iota_i) \iso \End_{\tc{C}}(\rho_i)$, so $\rho_i$ is also simple if and only if $\iota_i$ is.
\end{proof}

\skiptocparagraph{Simple objects and indecomposable objects correspond}

It turns out that in a presemisimple $2$-category, the notion of simple object and of indecomposable object coincide.
\begin{proposition}[Simple if and only if indecomposable]
An object in a presemisimple 2-category is simple if and only if it is indecomposable.
\end{proposition}
\begin{proof}
Given a nontrivial decomposition of an object $X$ as $A_1 \boxplus A_2$, the inclusion 1-morphism $A_1 \ra X$ is fully faithful and therefore is a nontrivial subobject; thus $X$ is not simple.  Conversely, by the definition of presemisimplicity, any object is a sum of simple objects; for an indecomposable object, this sum can only have a single factor, and so the object itself is simple.
\end{proof}

\skiptocparagraph{Categorical domain Schur's lemma}

We now show that in a presemisimple $2$-category, the decomposition of an object into a sum of simple objects is unique.  To show this we need the following `categorical domain' version of Schur's lemma: though nonzero morphisms between simple objects in a presemisimple $2$-category need not be equivalences (and therefore the endomorphism category of a simple object need not be a `categorical division algebra'), nevertheless the composite of two nonzero morphisms between simple objects in a presemisimple $2$-category cannot be zero (and therefore the endomorphism category of a simple object is a kind of `categorical domain').

\begin{definition}[Categorical domain]
A \emph{categorical domain} is a semisimple monoidal $1$-category with duals such that the tensor product of two nonzero objects is nonzero.
\end{definition}
\begin{prop}[Categorical domain Schur's lemma] \label{prop:Schur} 
If $g:A\to B$ and $f:B\to C$ are nonzero $1$-morphisms between simple objects in a presemisimple $2$-category, then the composite $f\xo g$ is also nonzero.
\end{prop}
\begin{proof}
Let $g^*: B \to A$ be a right adjoint of $g$ with counit $\epsilon: g \xo g^* \To \Io_B$.  By assumption $B$ is simple, and so by Corollary~\ref{cor:pressdirectsum} and Proposition~\ref{prop:simpleiffidentity}, the identity $\Io_B$ is simple. As a counit, $\epsilon$ must be nonzero, and as a nonzero morphism to a simple object in a semisimple category, it must have a section.  If $f \xo g$ were zero, then $f \xo g \xo g^*$ would be zero and so the morphism $\It_f \xo \epsilon$ would necessarily be zero.  Precomposing with the section would imply that $\It_f$ itself was zero, which in turn would force $f$ to be zero.
\end{proof}
\begin{corollary}[Infusion if and only if categorical domain] \label{cor:fusiondomain}
A semisimple monoidal $1$-category with duals is infusion if and only if it is a categorical domain.
\end{corollary}
\begin{proof}
By Proposition~\ref{prop:Schur} and Proposition~\ref{prop:simpleiffidentity} applied to Example~\ref{eg:deloopfusion}, an infusion category is a categorical domain.  Conversely, if a semisimple monoidal 1-category with duals has two distinct simple subobjects of its tensor unit, then the product of those objects is zero~\cite[Sec 2.4]{ENO}, preventing the category from being a domain.
\end{proof}

\nid
Note that if $g: A \to B$ and $f: B \to C$ are $1$-morphisms in a linear $2$-category and either one is a zero $1$-morphism, then their composite $f \circ g$ is also a zero $1$-morphism.  Thus, morphisms between simple objects in a presemisimple $2$-category satisfy a `two out of three property', that if any two of $f$, $g$, and $f \circ g$ are nonzero, then so is the third; in this sense, though not necessarily equivalences, nonzero morphisms between simple objects in a presemisimple $2$-category are a sort of `very weak equivalences'.

\skiptocparagraph{Uniqueness of decomposition}

\begin{prop}[Decomposition into simples is unique in presemisimple 2-categories] \label{prop:uniquedecomp}
The decomposition of any object in a presemisimple $2$-category into a finite direct sum of simple objects is unique up to permutation and equivalence.
\end{prop}
\begin{proof}
Let $\bigboxplus_{i \in I} X_i$ and $\bigboxplus_{j \in J} X'_j$ be direct sum decompositions of an object $X$ into simple objects, with inclusion and projection $1$-morphisms $\iota_i: X_i \leftrightarrows X: \rho_i$ and $\iota'_j:X_j' \leftrightarrows X:\rho'_j$. Note that $\oplus_{j \in J} \, \rho_i \circ \iota'_j \circ \rho_j' \circ \iota_i\iso \rho_i \circ \iota_i \iso \Io_{X_i}$. Since $\Io_{X_i}$ is simple, it follows that there exists a unique $f(i)\in J$ such that $\rho_i \circ \iota'_{f(i)} \circ \rho_{f(i)}' \circ \iota_i$ is nonzero (and in this case isomorphic to $\Io_{X_i}$). By Proposition~\ref{prop:Schur}, a $1$-morphism $F$ between simple objects is zero if and only if $F^* \xo F $ is zero. Using Proposition~\ref{prop:coherencedirectsums}, it follows that for every $i\in I$ there is a unique $f(i) \in J$ such that $\rho_{f(i)}'\circ \iota_i:X_i \to X_{f(i)}'$ is nonzero. The same argument applied to the decomposition $\oplus_{i \in I} \rho'_j \xo \iota_{i}\xo \rho_i \xo \iota'_j \iso \Io_{X'_j}$ shows that for every $j\in J$ there is a unique $g(j) \in I$ such that $\rho_j' \xo \iota_{g(j)}: X_{g(j)} \to X_j'$ is nonzero. Thus $f:I \to J$ is a bijection with inverse $g:J \to I$ and $\rho_{f(i)}' \xo \iota_i$ is an equivalence since
\begin{align*}
\Io_{X_i} &\iso \bigoplus_{j \in J} \rho_i \xo \iota'_{j} \xo \rho'_{j} \xo \iota_i \iso \rho_i \xo \iota'_{f(i)} \xo \rho'_{f(i)} \xo \iota_i 
\\
\Io_{X_{f(i)}'} &\iso \bigoplus_{i' \in I}   \rho'_{f(i)} \xo \iota_{i'} \xo \rho_{i'} \xo \iota'_{f(i)}\iso\rho'_{f(i)} \xo \iota_i\xo \rho_i \xo \iota'_{f(i)}.
\end{align*} 
Finally note that
\begin{align*}
\iota_{f(i)}' \xo \rho_{f(i)}' \xo \iota_i &\iso \bigoplus_{j \in J} \iota_j' \xo \rho_j' \xo \iota_i \iso \iota_i \\
\rho_{f(i)}' \xo \iota_i \xo \rho_i &\iso \bigoplus_{i' \in I} \rho_{f(i)}' \xo \iota_{i'} \xo \rho_{i'} \iso \rho_{f(i)}'.
\end{align*}
Hence there is a bijection $f:I \to J$ and equivalences $e_i: X_i \to X'_{f(i)}$ such that $\iota_{f(i)}' \xo e_i \iso \iota_i$ and $e_i \xo \rho_i \iso \rho_{f(i)}'$, as required.
\end{proof}

\skiptocparagraph{Components of presemisimple $2$-categories}

A presemisimple $2$-category may itself be `decomposable' in the sense that its set of simple objects splits into two pieces, such that there are no nonzero morphisms between the simple objects in one piece and the simple objects in the other piece.  Furthermore, each `indecomposable' collection of simples will be completely connected in the sense that there is a nonzero morphism between any two simples in the collection; we will refer to such a completely connected collection of simples as a component of the $2$-category.

\begin{definition}[Components of presemisimple 2-categories] \label{def:connectedcomponent}
Let $\tc{C}$ be a presemisimple $2$-category.  Two simple objects $A$ and $B$ in $\tc{C}$ are \emph{in the same component} if there is a nonzero $1$-morphism $A \to B$, that is if $\Hom_{\tc{C}}(A,B) \neq 0$.  The \emph{set of components} of $\tc{C}$, denoted $\pi_0 \tc{C}$, is the quotient of the set of simples by the equivalence relation of being in the same component.
\end{definition}

\nid Note that being in the same component is indeed an equivalence relation: reflexivity is clear; symmetry follows from the fact that the right adjoint of a nonzero morphism is nonzero; and transitivity is precisely the content of the categorical domain Schur's lemma.

Given an indecomposable multifusion category (that is one that is not the direct sum of two nontrivial multifusion categories), the associated unfolded $2$-category (see Construction~\ref{con:unfold}) is connected (that is has a single component).  More generally, the set of indecomposable factors of a multifusion category corresponds to the set of components of its unfolding.

\begin{remark}[Components as indecomposable summands] \label{rem:components}
Though we will not need it, there is a natural notion of direct sum of linear $2$-categories, and therefore of indecomposable linear $2$-category.  When a presemisimple $2$-category $\tc{C}$ is in fact additive, we may think of its components as the summands in the finest direct sum decomposition of $\tc{C}$ into indecomposable linear $2$-categories.
\end{remark}

\subsubsection{Dimensions of presemisimple $2$-categories}

In a finite semisimple $1$-category, there are finitely many isomorphism classes of simple objects, the endomorphism algebra of any simple object is the base field, and there are no morphisms between non-isomorphic simples.  There is therefore a natural invariant of such a category, namely the number of isomorphism classes of simple objects.  This `dimension' is of course a natural number.  We now describe the analogous notion of dimension for finite presemisimple $2$-categories.  This notion is complicated by the fact that, in a presemisimple $2$-category, the endomorphism fusion categories of simple objects are not necessarily trivial and there can be nontrivial morphisms between distinct simple objects.  In particular, as a result, the dimension of a presemisimple $2$-category will not necessarily be a natural number.


Recall the notion of the global dimension of a fusion category $\oc{C}$~\cite{Mueger}: any simple object $x \in \oc{C}$ is isomorphic to its double dual $x^{\ast\ast}$; given any isomorphism $a: x \to x^{\ast\ast}$, one uses the counit of the duality $(x \dashv x^\ast)$ and the unit of the duality $(x^\ast \dashv x^{\ast\ast})$ to form the quantum trace $\Tr(a) \in k$; the product of the quantum trace of $a$ and the quantum trace of ${}^\ast(a^{-1})$ is independent of the choice of morphism $a$ and is called the squared norm of the simple object $x$; the sum of the squared norms of a set of distinct simple objects is called the global dimension of the fusion category.

We describe the analogous notions for $1$-morphisms in an appropriate $2$-category.

\begin{prop}[Double adjunction is trivial] \label{prop:doubledual}
Any simple $1$-morphism $f$ in a finite presemisimple $2$-category is isomorphic to its double right adjoint $f^{\ast\ast}$.
\end{prop}
\begin{proof}
This is similar to the analogous result for fusion categories~\cite[Prop 2.1]{ENO}.  The $1$-morphism $f: A \ra B$ is an object of the finite semisimple $1$-category $\Hom(A,B)$.  In any finite semisimple $1$-category $\oc{C}$, for any two objects $X,Y \in \oc{C}$, there is a noncanonical isomorphism $\Hom_{\oc{C}}(X,Y) \iso \Hom_{\oc{C}}(Y,X)$.  Note that $f^{**}$ is simple if and only if $f$ is simple.  We therefore have $\Hom(f,f^{**}) \iso \Hom(\Io_B, f^* \xo f^{**}) \iso \Hom(f^* \xo f^{**},\Io_B) \iso \Hom(f^{**},f^{**}) \iso k$.  Since both $f$ and $f^{**}$ are simple, the existence of a nontrivial morphism between them implies they are isomorphic.
\end{proof}

\nid In a locally semisimple $2$-category, every $2$-endomorphism $\mu:g\To g$ of a simple $1$-morphism $g$ is proportional to the identity $2$-morphism; we denote the proportionality factor by $\langle \mu \rangle \in k$, that is $\mu = \langle \mu \rangle \It_{g}$. 

\begin{definition}[Squared norm of 1-morphism] \label{def:squarednorm} 
The \emph{squared norm} of a simple $1$-morphism $f:A\to B$ between simple objects in a finite presemisimple $2$-category is the product
\[ 
\|f\|:= \left\langle\vphantom{\frac{a}{b}}\epsilon_{f^*}\xt(\It_{f^*} \xo a)\xt\eta_f\right\rangle~~\left\langle\vphantom{\frac{a}{b}}\epsilon_f\xt(a^{-1} \xo \It_{f^*})\xt\eta_{f^*} \right\rangle \in k
\]
where $a: f \To f^{**}$ is an arbitrary $2$-isomorphism, $\eta_f$ and $\epsilon_f$ are the unit and counit of the adjunction $f \dashv f^*$, and $\eta_{f^*}$ and $\epsilon_{f^*}$ are the unit and counit of the adjunction $f^* \dashv f^{**}$.
\end{definition}
\nid Here simplicity of the objects $A$ and $B$ is necessary to ensure that $\Io_A$ and $\Io_B$ are simple, so that we can extract scalars from the two `quantum trace' endomorphisms in the formula for the squared norm.  Simplicity of $f$ ensures that the squared norm is independent of the choice of $2$-morphism $a: f \To f^{**}$, and hence only depends on the isomorphism class of $f$.

\begin{definition}[Dimension of Hom category] \label{def:dimhom}
For simple objects $A$ and $B$ in a finite presemisimple $2$-category $\tc{C}$, the \emph{dimension} of the category $\Hom_{\tc{C}}(A,B)$ is
\[
\dim(\Hom_{\tc{C}}(A,B)):=~\sum_{f:A\to B} \|f\|
\]
where the sum is over isomorphism classes of simple $1$-morphisms $f \in \Hom_{\tc{C}}(A,B)$.
\end{definition}
\nid This definition is analogous to that of the dimension of the `off-diagonal subcategories' of a multifusion category~\cite[Sec 2.4]{ENO}.  Note that for a simple object $A$, the dimension $\dim(\Hom_{\tc{C}}(A,A))$ is the usual global dimension of the fusion category $\Hom_{\tc{C}}(A,A)$; when over an algebraically closed field of characteristic zero, that dimension is always nonzero~\cite[Thm 2.3]{ENO}.

\begin{proposition}[Dimension is uniform within a component] \label{prop:sameblock} 
Let $\{A_i\}_{i \in I}$ be the simple objects of a connected component of a finite presemisimple $2$-category $\tc{C}$.  Then the categories $\Hom_{\tc{C}}(A_i,A_j)$ all have the same dimension, for $i,j \in I$.
\end{proposition}
\begin{proof}
The multifusion category $\bigoplus_{i,j \in I}\Hom_{\tc{C}}(A_i, A_j)$ is indecomposable, because \linebreak $\Hom_{\tc{C}}(A_i, A_j)\neq 0$ for simple objects in the same connected component.  By~\cite[Prop 2.17]{ENO}, it follows that the component categories $\Hom_{\tc{C}}(A_i,A_j)$ all have the same dimension.
\end{proof}

\nid%
We now have a notion of the dimension of each component of a presemisimple $2$-category (namely the dimension of any Hom category between simples in that component), and we are ready to assemble them into a notion of the dimension of the whole $2$-category.  Recall that for a $1$-groupoid, the natural notion of size (the `groupoid cardinality'~\cite{Groupoidification}) is the sum over components of the reciprocal of the size of the automorphism groups.  The dimension for a presemisimple $2$-category is analogous. 

\begin{definition}[Dimension of presemisimple 2-category] \label{def:dimension2cat}
The \emph{dimension} of a finite presemisimple $2$-category $\tc{C}$ is
\[
\dim(\tc{C}) := \sum_{[x]\in \pi_0\tc{C}} \frac{1}{\dim(\vphantom{\frac{a}{b}}\!\End_{\tc{C}}(x))} \in k.
\]
Here the sum is over components $[x]$ of $\tc{C}$, and $x$ is any simple object in the component $[x]$.
\end{definition}

\nid Of course, the dimension of a finite presemisimple 2-category is only defined when the dimensions of all its endomorphism fusion categories are nonzero; this is ensured by our standing assumption that the base field is algebraically closed of characteristic zero.  

As for fusion categories, when over an algebraically closed field of characteristic zero, the dimension of a presemisimple $2$-category cannot vanish.


\begin{prop}[Dimension is nonzero]
For $\tc{C}$ a finite presemisimple $2$-category over an algebraically closed field of characteristic zero, the dimension $\dim(\tc{C})$ is nonzero.
\end{prop}
\begin{proof}
By~\cite[Thm 2.3]{ENO}, a fusion category over $\CC$ has positive real global dimension.  Thus the dimension of a presemisimple $2$-category over $\CC$ is positive real, in particular nonzero.  The result follows by noting that any finite presemisimple $2$-category over an algebraic closed field $k$ of characteristic zero can be defined over a subfield $k'$ that is finitely generated over $\QQ$ and which can therefore be embedded in $\CC$.
\end{proof}

\begin{remark}[Nonzero characteristic]\label{rem:nonzerochar1}
We could proceed without a characteristic zero assumption, at the expense of restricting attention to \emph{non-degenerate} finite presemisimple $2$-categories $\tc{C}$, that is those for which the dimensions $\dim(\End_{\tc{C}}(x))$ are nonzero for all simple objects $x$ and for which the overall dimension $\dim(\tc{C})$ is nonzero.
\end{remark}

\subsection{Idempotent complete $2$-categories} \label{sec:ic2cat}

A semisimple $1$-category is a presemisimple $1$-category that is also additive and idempotent complete.  We described the notion of presemisimple $2$-category and of additive $2$-category; we now discuss idempotent completeness for $2$-categories.  Further details about idempotent completeness and idempotent completion, and a number of proofs, are given in Appendix~\ref{app:ic}.

\subsubsection{Categorified idempotents}
\skiptocparagraph{Idempotents and split idempotents}

In a 1-category, a \emph{section-retraction pair} is a pair of 1-morphisms $i: A \to B$ and $r: B \to A$ such that $r \xo i=\Io_A$.  Associated to such a pair there is the 1-morphism $e := i \xo r : B \to B$, which is an \emph{idempotent} (or `projection'), meaning it is a 1-morphism $e: B \to B$ such that $e \xo e = e$.  An arbitrary idempotent $e: B \to B$ is \emph{splitable} (or more informally `split') when there exists a section-retraction pair $(i,r)$ such that $e = i \xo r$; a `splitting' is a choice of such a pair.  As before, a 1-category is idempotent complete if every idempotent splits.  A 2-category $\cC$ is \emph{locally idempotent complete} if for all objects $A, B \in \cC$, the 1-category $\Hom_\cC(A,B)$ is idempotent complete.

\skiptocparagraph{Idempotent monads and reflectively split idempotent monads}

A natural categorification of the notion of section-retraction pair is the following: a \emph{reflective subcategory} is a pair of functors $\iota: \cA \to \cB$ and $\rho: \cB \to \cA$ where $\iota$ is fully faithful and $\rho$ is equipped with the structure of a left adjoint of $\iota$.  More generally, in a 2-category, a \emph{reflective adjunction} is a fully faithful 1-morphism $\iota: A \to B$ together with a left adjoint $\rho: B \to A$.  An adjunction $\iota \vdash \rho$ is reflective exactly when its counit is an isomorphism $\epsilon: \rho \xo \iota \xra{\cong} \Io_A$; this condition on the counit is a strong categorification of the condition $r \xo i = \Io_A$ on a section-retraction pair. 
Associated to a reflective adjunction $\iota \vdash \rho$ in a 2-category, there is the 1-morphism $E := \iota \xo \rho : B \to B$, which is an \emph{idempotent monad} (a kind of `categorified projection'), meaning it is a 1-morphism $E: B \to B$ equipped with a 2-isomorphism $m: E \xo E \to E$ (determined by the counit of the adjunction) and a 2-morphism $u: \Io_B \to E$ (determined by the unit of the adjunction), such that $m \xt (m \xo \It_E) = m \xt (\It_E \xo m)$ and $m \xt (u \xo \It_E) = \It_E = m \xt (\It_E \xo u)$.
An idempotent monad $E: B \to B$ is \emph{reflectively splitable} (or more informally `reflectively split') when there exists a reflective adjunction $\iota \vdash \rho$ and an isomorphism of monads $E \cong \iota \xo \rho$; a `reflective splitting' is a choice of such an adjunction and isomorphism.

\skiptocparagraph{Monads and split monads}

If we drop the fully faithful (equivalently counit isomorphism) condition in a reflective adjunction, we are left simply with 1-morphisms $\iota: A \to B$ and $\rho: B \to A$ forming an \emph{adjunction} $\iota \vdash \rho$.  
The associated 1-morphism $E := \iota \xo \rho : B \to B$ is a \emph{monad}, meaning it is a 1-morphism $E: B \to B$ equipped with a 2-morphism (not necessarily a 2-isomorphism) $m: E \xo E \to E$ and a 2-morphism $u: \Io_B \to E$, satisfying the same equations as an idempotent monad. (Concisely, a monad $E$ in a 2-category $\cC$ is an algebra object in the endomorphism 1-category $\Hom_\cC(B,B)$ of an object $B \in \cC$.)
An arbitrary monad $E: B \to B$ is \emph{splitable} (or more informally `split') when there exists an adjunction $\iota \vdash \rho$ and an isomorphism of monads $E \cong \iota \xo \rho$.

\subsubsection{Categorified idempotent splitting}

\skiptocparagraph{Uniqueness of splitting an idempotent}

Given an idempotent $e: B \to B$ in a 1-category, if it admits a splitting, then there is a unique splitting.  (That is, any two splittings $(i: A \to B, r: B \to A)$ and $(i':A' \to B, r': B \to A')$ are isomorphic by a unique isomorphism, namely the intertwiner $r' \xo i$.)  Indeed, the splitting may be expressed either as a colimit or as a limit, as follows.  Given an idempotent $e: B \to B$, consider the diagram \smash{$B\raisebox{-0.02cm}{
 \begin{tz}[scale=0.3]
\def\l{1.3}
\draw[-{>[scale=0.65]}] (0,0.5) to node[above,yshift=-0.05cm]{$\scriptstyle e$} (\l,0.5);
\draw[-{>[scale=0.65]}] (0,0.1) to node[below,yshift=0.05cm]{$\scriptstyle \Io$} (\l,0.1);
\end{tz}} B$}.  If it exists, the coequalizer $B\smash{\raisebox{-0.02cm}{
 \begin{tz}[scale=0.3]
\def\l{1.3}
\draw[-{>[scale=0.65]}] (0,0.5) to node[above,yshift=-0.05cm]{$\scriptstyle e$} (\l,0.5);
\draw[-{>[scale=0.65]}] (0,0.1) to node[below,yshift=0.05cm]{$\scriptstyle \Io$} (\l,0.1);
\end{tz}}} B \xra{r} A$  provides a splitting of the idempotent (where the morphism $i: A \to B$ is determined by the universal property of the coequalizer).  Similarly, if it exists, the equalizer $A\xra{i} B\smash{\raisebox{-0.02cm}{
 \begin{tz}[scale=0.3]
\def\l{1.3}
\draw[-{>[scale=0.65]}] (0,0.5) to node[above,yshift=-0.05cm]{$\scriptstyle e$} (\l,0.5);
\draw[-{>[scale=0.65]}] (0,0.1) to node[below,yshift=0.05cm]{$\scriptstyle \Io$} (\l,0.1);
\end{tz}}} B$ provides a splitting of the idempotent (where the morphism $r: B \to A$ is then determined by the universal property of the equalizer).  In particular, if either the coequalizer or equalizer exists, then the other does, and the coequalizing object is isomorphic to the equalizing object.

\skiptocparagraph{Uniqueness of reflectively splitting an idempotent monad}

As idempotents in a 1-category have unique splittings (when they are split), so too idempotent monads in a locally idempotent complete 2-category, have unique reflective splittings (when they are reflectively split).  
However, an arbitrary monad in a 2-category, even if it admits a splitting, need not admit a unique splitting.  We would like to restrict attention to a class of monads $E : B \to B$ for which the multiplication 2-morphism $m: E \xo E \to E$ need not be an isomorphism (by contrast with idempotent monads) but which nevertheless have a unique splitting property (as do idempotent monads).

\skiptocparagraph{Separable monads and separably split separable monads}

The data of an idempotent monad in a 2-category $\cC$ can be expressed as follows: it is a triple $(E: B \to B, m: E \xo E \to E, u: \Io_B \to E)$ forming an algebra object in $\Hom_\cC(B,B)$, such that there exists an $E$-$E$-bimodule map $c: E \to E \xo E$ that is a two-sided inverse to the multiplication $m: E \xo E \to E$.  We can marginally weaken this notion of idempotent monad by only requiring there to exist a one-sided rather than two-sided inverse to the multiplication; this provides a version of categorified idempotent that is more lax than idempotent monad but stronger than arbitrary monad.
\begin{definition}[Separable monad]
A monad $(E: B \to B, m: E \xo E \to E, u: \Io_B \to E)$ in a 2-category is \emph{separable} if there exists an $E$-$E$-bimodule map $c: E \to E \xo E$ that is a right inverse for the multiplication $m: E \xo E \to E$, that is such that $m \xt c = \It_E$.\footnote{A separable monad in a 2-category with one object is a separable algebra object in the monoidal endomorphism category of that object.  A classical separable algebra is a separable algebra object in the symmetric monoidal category of vector spaces.}
\end{definition}
\nid Similarly, the data of a reflective splitting of an idempotent monad $E: B \to B$ can be expressed as follows: it is an adjunction $\iota \vdash \rho \origequiv (\iota: A \to B, \rho: B \to A, \eta: \Io_B \to \iota \xo \rho, \epsilon: \rho \xo \iota \to \Io_A)$ such that there exists a 2-morphism $\phi: \Io_A \to \rho \xo \iota$ that is a two-sided inverse to the counit $\epsilon: \rho \xo \iota \to \Io_A$, together with an isomorphism of monads $E \cong \iota \xo \rho$.  We can again marginally weaken the invertibility condition here by only requiring there to exist a one-sided inverse to the counit of the adjunction.
\begin{definition}[Separable adjunction]
An adjunction $\iota \vdash \rho \origequiv (\iota: A \to B, \rho: B \to A, \eta: \Io_B \to \iota \xo \rho, \epsilon: \rho \xo \iota \to \Io_A)$ in a 2-category is \emph{separable} if there exists a 2-morphism $\phi: \Io_A \to \rho \xo \iota$ that is a right inverse for the counit $\epsilon: \rho \xo \iota \to \Io_A$, that is such that $\epsilon \xt \phi = \It_A$.
\end{definition}
\nid The notion of separable adjunction is a categorification of section-retraction pair that is more lax than reflective adjunction but stronger than arbitrary adjunction; there is therefore a corresponding version of categorified idempotent splitting that is more lax than reflective splitting but stronger than arbitrary splitting.
\begin{definition}[Separably split monad]
A separable monad $E : B \to B$ in a 2-category is \emph{separably splitable} (or simply `separably split') when there exists a separable adjunction $\iota \vdash \rho$ and an isomorphism of monads $E \cong \iota \xo \rho$; a \emph{separable splitting} is a choice of such an adjunction and isomorphism. 
\end{definition}
\nid Note that if a monad admits a separable splitting then the monad itself is necessarily separable.

\skiptocparagraph{Uniqueness of separable splittings of separable monads}

Separable monads do indeed have unique separable splittings (when they are separably split), as desired:
\begin{proposition}[Separable splittings are unique]
A separable monad, in a locally idempotent complete 2-category, that admits a separable splitting, admits a unique up-to-equivalence separable splitting.
\end{proposition}
\nid This is an immediate corollary of Theorem~\ref{thm:EMKleisliSeparable} in Appendix~\ref{app:ic}, but we briefly and informally sketch the argument here.  Recall how one sees the uniqueness of splittings for a 1-categorical idempotent $e: B \to B$: the colimit \smash{$\colim(B\smash{\raisebox{-0.02cm}{
 \begin{tz}[scale=0.3]
\def\l{1.3}
\draw[-{>[scale=0.65]}] (0,0.5) to node[above,yshift=-0.05cm]{$\scriptstyle e$} (\l,0.5);
\draw[-{>[scale=0.65]}] (0,0.1) to node[below,yshift=0.05cm]{$\scriptstyle \Io$} (\l,0.1);
\end{tz}}} B)$} (necessarily unique) provides a splitting and the limit$\vphantom{\frac{a}{b}}$ \smash{$\lim(B\smash{\raisebox{-0.02cm}{
 \begin{tz}[scale=0.3]
\def\l{1.3}
\draw[-{>[scale=0.65]}] (0,0.5) to node[above,yshift=-0.05cm]{$\scriptstyle e$} (\l,0.5);
\draw[-{>[scale=0.65]}] (0,0.1) to node[below,yshift=0.05cm]{$\scriptstyle \Io$} (\l,0.1);
\end{tz}}} B)$} (necessarily unique) provides a splitting, and any splitting provides both a colimit and a limit; it follows that the splitting is unique and that the colimit and limit objects are isomorphic.  

Observe that a 1-categorical idempotent $e: B \to B$ in a category $\oc{C}$ may be reexpressed (somewhat contortionistically) as a lax 2-semifunctor of 2-categories $\ast \to[\scriptscriptstyle (B,e)] \oc{C}$, where we have reinterpreted $\oc{C}$ as a discrete 2-category; the 2-colimit, respectively 2-limit, of that lax functor is exactly the ordinary colimit, respectively limit, splitting of the idempotent as above.  Now a monad $E: B \to B$ in a 2-category $\tc{C}$ is simply a lax 2-functor $\ast \to[\scriptscriptstyle (B,E)] \tc{C}$, and we may consider the lax 2-colimit $\colim(\ast \to[\scriptscriptstyle (B,E)] \tc{C})$ or lax 2-limit $\lim(\ast \to[\scriptscriptstyle (B,E)] \tc{C})$; that 2-colimit, when it exists, is usually called a \emph{Kleisli object} for the monad, and that 2-limit, when it exists, is usually called an \emph{Eilenberg--Moore object} for the monad.  Exactly as for an idempotent in the 1-categorical case, for a monad in the 2-categorical case, when the 2-colimit exists, it provides a splitting, and when the 2-limit exists, it provides a splitting.  

The trouble is that a splitting need not provide a 2-colimit or a 2-limit; in particular, the 2-colimit and 2-limit objects need not be the same.  However, provided we restrict attention to a separable monad in a locally idempotent complete $2$-category, then (see Appendix~\ref{app:ic}) either a 2-colimit or a 2-limit provides a separable splitting, and any separable splitting provides both a 2-colimit and a 2-limit; from this it follows that the separable splitting is unique and that the 2-colimit and 2-limit objects agree.

\begin{remark}[Separable monads and separable splittings are preserved by all $2$-functors]\label{rem:absolutemonad}
Recall from Remark~\ref{rem:absolute} that direct sums and zero objects are preserved by all 2-functors and are therefore absolute 2-colimits.  Similarly, both the separability of a monad and the existence of a separable splitting of a separable monad are `equational' conditions, in that they are defined in terms of the existence of certain morphisms satisfying certain equations.  Separable monads and their separable splittings are therefore preserved by all 2-functors. Since separable splittings of separable monads in a locally idempotent complete 2-category are 2-colimits, they are therefore absolute 2-colimits.
\end{remark}

\subsubsection{Categorified idempotent completeness}
We have arrived at our proper context, a notion of `lax' 2-categorical idempotent that admits unique splittings, and therefore a notion of 2-category in which all idempotents and categorical idempotents split.
\begin{definition}[Idempotent complete 2-category]
A 2-category $\cC$ is \emph{idempotent complete} if it is locally idempotent complete and if every separable monad in $\cC$ admits a separable splitting.
\end{definition}
\nid The above discussion shows that we can, as in this definition, sensibly ask that separable monads admit separable splittings.  In practice, we will be interested in locally finite semisimple 2-categories, and in that restricted context, we can furthermore see that separable monads is the largest class of monads we would want to insist have splittings, as follows.  As before, given a monad $E: B \to B$, if the lax 2-colimit $\colim(\ast \to[\scriptscriptstyle (B,E)] \tc{C})$ exists, it provides a splitting of the monad and if the lax 2-limit $\lim(\ast \to[\scriptscriptstyle (B,E)] \tc{C})$ exists, it also provides a splitting; indeed we may think of the 2-colimit as a `universal right splitting' and similarly of the 2-limit as a `universal left splitting'.  We take it for granted that we want to consider splittings of monads that are universal, indeed preferably ones where the universal left and right splittings both exist and agree.  Such a splitting must be a separable splitting of a separable monad.
\begin{proposition}[Universally split monads are separable]
Let $\cC$ be a locally finite semisimple 2-category.  If a monad in $\cC$ admits a universal left splitting (that is, an Eilenberg--Moore object) or a universal right splitting (that is, a Kleisli object), then the monad is separable and admits a separable splitting.
\end{proposition}
\begin{proof}
Suppose the monad $E: B \to B$ admits a universal left splitting, that is the lax 2-limit $A: = \lim(\ast \to[\scriptscriptstyle (B,E)] \tc{C})$ exists.  By definition of a 2-limit and the notion of left module over a monad, the 2-limit $A$ corepresents left $E$-module structures (see Appendix~\ref{app:ic}).  In particular, the category $\Hom_{\tc{C}}(B,A)$ is equivalent to the category $\LMod_E(B)$ of left $E$-module structures on the object $B$.  Observe that this category $\LMod_E(B)$ of module structures is precisely the category $\Mod(E)$ of modules of $E$ considered as an algebra object in the monoidal category $\Hom_{\tc{C}}(B,B)$.  And this category $\Mod(E)$ of modules is of course a module category for $\Hom_{\tc{C}}(B,B)$.  By local finite semisimplicity, the monoidal category $\Hom_{\tc{C}}(B,B)$ is finite semisimple and, because the base field is characteristic zero, it is a separable tensor category~\cite[Cor 2.6.8]{DTC}.  The category $\Hom_{\tc{C}}(B,A)$ is also semisimple, therefore the module category $\Mod(E)$ is semisimple.  By~\cite[Prop 2.5.10]{DTC}, a semisimple module category over a separable tensor category is necessarily separable.  By definition this means that the monad $E$ is separable. By Theorem~\ref{thm:EMKleisliSeparable}, a separable monad with a univeral left splitting admits a separable splitting. The argument for right splittings is the same.
\end{proof}


\subsubsection{Categorified idempotent completion}

\begin{construction}[Idempotent completion of a 1-category]
A 1-category $\oc{C}$ can be completed to an idempotent complete 1-category $\oc{C}^\idm$, whose objects are idempotents in $\oc{C}$ and whose morphisms are bilodules.  Here a `left lodule' for an endomorphism $g: b \to b$ in a 1-category is a 1-morphism $f: a \to b$ such that $g \circ f = f$; similarly a `right lodule' is a 1-morphism $h: b \to c$ such that $h \circ g = h$.  A `bilodule' from an endomorphism $e: b \to b$ to an endomorphism $e' : b' \to b'$ is a morphism $j: b \to b'$ that is a right lodule for $e$ and a left lodule for $e'$.  Note that if $e$ and $e'$ are split idempotents, the data of a bilodule is the same as the data of a morphism from the splitting object of $e$ to the splitting object of $e'$.  
\end{construction}
\nid If the category $\oc{C}$ is already idempotent complete, then the completion $\oc{C}^\idm$ is equivalent to $\oc{C}$.\looseness=-2

\begin{construction}[Idempotent completion of a 2-category] \label{con:iccompletion}
A locally idempotent complete $2$-category $\tc{C}$ can be completed to an idempotent complete $2$-category $\tc{C}^\idm$, whose objects are separable monads in $\tc{C}$, whose $1$-morphisms are bimodules between those monads, and whose $2$-morphisms are bimodule maps.  See Appendix~\ref{sec:appicdef} for a discussion of various properties of this idempotent completion construction.
\end{construction}

\begin{remark}[Idempotent completion is idempotent] \label{rem:iccompletion}
If the locally idempotent complete 2-category $\tc{C}$ is already idempotent complete, then the completion $\tc{C}^\idm$ is equivalent to $\tc{C}$; this is shown in Appendix~\ref{app:ic} as Proposition~\ref{prop:icequiv}.  This is a consequence of the fact that separable splittings of separable monads (in locally idempotent complete $2$-categories) are absolute 2-colimits (see Remark~\ref{rem:absolutemonad}), and the completion $\tc{C}^\idm$ is the free cocompletion under those colimits.
\end{remark}

\begin{remark}[Cauchy completion of a 1-category]
Recall that a linear 1-category is called `Cauchy complete' if it has all absolute colimits.  A linear 1-category is Cauchy complete if and only if it is additive and idempotent complete~\cite[Prop 2.11]{111-IV}.  The Cauchy completion of a linear 1-category $\oc{C}$ is $(\oc{C}^\oplus)^\idm \equiv (\oc{C}^\idm)^\oplus$, where $\oc{C}^\oplus$ denotes the direct sum completion, and as above $\oc{C}^\idm$ denotes the idempotent completion. 
\end{remark}

\begin{remark}[Cauchy completion of a 2-category]
A linear $2$-category is `Cauchy complete' if it has all absolute $2$-colimits. We speculate that a linear locally Cauchy complete 2-category (or at least a locally finite semisimple 2-category) is Cauchy complete if and only if it is additive and idempotent complete, and we imagine that the Cauchy completion of $\tc{C}$ is given by $(\tc{C}^\boxplus)^\idm \equiv (\tc{C}^\idm)^\boxplus$. 
\end{remark}

A prototypical example of a locally idempotent complete 2-category that is not idempotent complete is the delooping $\B \oc{C}$ of a multifusion category $\oc{C}$; that is, $\B \oc{C}$ is the 2-category with one object whose endomorphism category is $\oc{C}$.  We now show that the idempotent completion $(\B \oc{C})^\idm$ of this 2-category is the 2-category $\Mod(\oc{C})$ of finite semisimple (right) module categories for the multifusion category.  Note that the objects of the idempotent completion $(\B \oc{C})^\idm$, separable monads in $\B \oc{C}$, are in this case just separable algebras in $\oc{C}$.

\begin{prop}[The idempotent completion of the delooping of a multifusion category is the 2-category of modules] \label{prop:completionismod}
Let $\oc{C}$ be a multifusion category.  The 2-functor $\mathrm{mod} : \B\oc{C}^\idm \ra \Mod(\oc{C})$, taking a separable algebra in $\oc{C}$ to its category of left modules, is an equivalence.
\end{prop}
\begin{proof}
Every finite semisimple module category of a multifusion category (in characteristic zero) is the category of modules of a separable algebra~\cite[Cor 2.6.9]{DTC}.  Thus the 2-functor $\mathrm{mod}$ is essentially surjective.  Furthermore the category of internal bimodules between algebras is equivalent to the category of functors of module categories~\cite[Prop 7.11.1]{EGNO}, and the 2-functor $\mathrm{mod}$ is therefore an equivalence on 1-morphism categories, as required.
\end{proof}

\nid
The delooping $\B \oc{C}$ of a multifusion category includes into the 2-category $\Mod(\oc{C})$ of modules, by sending the unique object to the module category $\oc{C}_{\oc{C}}$.

\begin{corollary}[Functors from the delooping of a multifusion category extend to modules] \label{cor:extensionmultifusion}
Let $\oc{C}$ be a multifusion category and let $\tc{D}$ be an idempotent complete $2$-category.  Every 2-functor $\B \oc{C} \ra \tc{D}$ extends uniquely (up to equivalence) to a 2-functor $\Mod(\oc{C}) \ra \tc{D}$.
\end{corollary}

\begin{proof}
Observe that the composite $\B\oc{C} \ra (\B\oc{C})^\idm \xra{\mathrm{mod}} \Mod(\oc{C})$ is the inclusion $\B\oc{C} \ra \Mod(\oc{C})$.  In Appendix~\ref{app:ic}, see especially Proposition~\ref{prop:extension}, we show that the idempotent completion is initial among idempotent complete targets; thus the functor $\B\oc{C} \ra \tc{D}$ extends to a functor $(\B \oc{C})^\idm \ra \tc{D}$.  It follows that the composite $\Mod(\oc{C}) \xra{\mathrm{mod}^{-1}} (\B\oc{C})^\idm \ra \tc{D}$ is the desired extension.
\end{proof}

\begin{remark}[The idempotent completion of the delooping of a multifusion category is already additive]
Given the delooped 1-category $\B A$ of a finite-dimensional semisimple algebra $A$, to obtain the (additive) category of finite-dimensional modules $\Mod(A)$, one must both idempotent and additively complete $\B A$.  By contrast, the idempotent completion $\B\oc{C}$ of the deloop of a multifusion category $\oc{C}$ already has direct sums and need not be further additively completed.
\end{remark}

\subsubsection{Direct sum decomposition in idempotent complete 2-categories}

In a locally additive 2-category, a direct sum decomposition $X\equiv \bigboxplus_i X_i$ of an object $X$, with inclusion and projection $1$-morphisms $\iota_i\!:X_i\leftrightarrows X:\! \rho_i$, induces by definition a direct sum decomposition of the identity 1-morphism $\Io_{X}\iso \bigoplus_i \iota_i \xo \rho_i \in \Hom(X,X)$.  A crucial property of idempotent complete 2-categories is that, conversely, a direct sum decomposition of an identity 1-morphism induces a direct sum decomposition of the corresponding object.
\begin{prop}[Identity splitting implies object splitting] \label{prop:idempotentcompletesum}
Let $X$ be an object in an idempotent complete linear $2$-category $\tc{C}$.  If $\Io_X \iso \bigoplus_{i\in I} f_i$ is a finite decomposition of $\Io_X$ into nonzero $1$-morphisms, then there is a finite decomposition $X\equiv \bigboxplus_{i\in I} X_i$ of $X$ into nonzero objects with inclusions and projections $\iota_i\!:X_i \leftrightarrows X:\! \rho_i$ such that $f_i \iso \iota_i \xo \rho_i$.
\end{prop}

\begin{proof} Let $r_i: f_i \To \Io_X$ and $s_i: \Io_X \To f_i$ be the inclusion and projection $2$-morphisms exhibiting the direct sum decomposition $\Io_X \iso \bigoplus_{i \in I} f_i$. For each $i$, the following $2$-morphisms form the multiplication $m_i$ and unit $u_i$ of a separable monad:
\begin{align*}
m_i &:= s_i \xt\left(r_i \xo r_i\right) : f_i \xo f_i \To f_i 
\\
u_i &:= s_i : \Io_X \To f_i
\end{align*}
A separating section of $m_i$ is given by $\Delta_i := \left(s_i \xo s_i\right) \xt r_i$.  Observe that $\Delta_i \xt m_i = (s_i \xo s_i)\xt (r_i \xo r_i) = \It_{f_i\xo f_i}$, and so $\Delta_i$ (and hence $m_i$) is an isomorphism.

Separably splitting each monad gives an object $X_i$ and adjoint $1$-morphisms $\iota_i\!: X_i \leftrightarrows X: \!\rho_i$ such that $\iota_i \xo \rho_i \iso f_i$.  Because the monad is separably split, there is (by an argument given in the second part of the proof of Theorem~\ref{thm:EMKleisliSeparable}) a section $\delta_i: \Io_{X_i} \To \rho_i \xo \iota_i$ of the counit $\rho_i \xo \iota_i \To \Io_{X_i}$ such that $\Delta_i = \iota_i \xo \delta_i\xo \rho_i$.  (Here we have omitted the isomorphism $\iota_i \xo \rho_i\iso f_i$ from the notation). Since $\Delta_i$ is an isomorphism, it follows that $\delta_i$ is also an isomorphism, and so $\rho_i \circ \iota_i \iso \Io_{X_i}$.  

By assumption the composite $s_i \xt r_j = 0$ for $i \neq j$.  Note that $s_i \xt r_j = (\iota_i \xo \rho_i \xo r_j) \xt (s_i \xo \iota_j \xo \rho_j)$.  Precomposing with $(r_i \xo \iota_j \xo \rho_j)$ and postcomposing with $(\iota_i \xo \rho_i \xo s_j)$ gives $\It_{\iota_i \xo \rho_i \xo \iota_j \xo \rho_j} = 0$.  Left and right composing with $\rho_i$ and $\iota_j$ gives $\It_{\rho_i \xo \iota_j} = 0$, thus $\rho_i \xo \iota_j = 0$.  Altogether then, these 1-morphisms exhibit a direct sum decomposition $X\equiv \bigboxplus_{i \in I} X_i$, as desired.
\end{proof}

\subsection{Semisimple $2$-categories}

\subsubsection{The definition of semisimple $2$-categories}

Recall that a presemisimple 1-category is a linear 1-category in which every object decomposes as a finite direct sum of simple objects, and in which the composition of any two nonzero morphisms between simples is nonzero.  A semisimple 1-category is one that is moreover additive (all finite direct sums exist) and idempotent complete (all idempotents split).  A presemisimple 2-category is a locally semisimple linear 2-category in which every 1-morphism has a left and a right adjoint, in which every object decomposes as a finite direct sum of simple objects, and in which the endomorphism category of every simple object is an infusion category.  (As we have seen, this endomorphism condition implies that the composition of nonzero 1-morphisms between simples is nonzero.)  We might expect to define a semisimple 2-category to be a presemisimple 2-category that is moreover additive (all finite direct sums exist) and idempotent complete (all separable monads separably split).  But in fact we can sharpen that definition by dropping both the condition that objects decompose as finite direct sums and the condition that endomorphisms of simples are infusion---local semisimplicity and idempotent completeness will conspire to ensure the object-direct-sum-splitting and endomorphism infusion properties of the 2-category.

\begin{definition}[Semisimple 2-category] \label{def:ss2cat}
A \emph{semisimple 2-category} is a locally semisimple 2-category, admitting adjoints for 1-morphisms, that is additive and idempotent complete.
\end{definition}
\begin{definition}[Finite semisimple 2-category]
A semisimple 2-category is \emph{finite} if it is locally finite semisimple and it has finitely many equivalence classes of simple objects.
\end{definition}

\begin{remark}[All 2-functors preserve sums and idempotent splittings]
As in Remarks~\ref{rem:additivecompletion} and~\ref{rem:iccompletion}, direct sums and idempotent splittings are absolute constructions, that is they are preserved by all linear 2-functors.  In particular, we need not restrict attention to a subclass of functors, but can consider all linear $2$-functors as the natural morphisms of semisimple 2-categories.
\end{remark}

\begin{remark}[Completeness implies the splitting condition for 2-categories]
In the 1-categorical case, cf Remark~\ref{rem:ssconditions}, the conditions in the alternative Footnote~\ref{foot:ss} definition of semisimplicity imply that every object splits as a finite direct sum of simples.  Similarly, in the 2-categorical case, because we have already assumed a splitting condition at the 1-morphism level via local semisimplicity, and because (by Proposition~\ref{prop:idempotentcompletesum}) in an idempotent complete 2-category splittings of identity 1-morphisms provide splittings of objects, it is again not necessary to impose a further object-level splitting condition in the definition of semisimple 2-category.
\end{remark}

\begin{proposition}[Semisimple implies presemisimple]
A semisimple $2$-category is presemisimple.
\end{proposition}
\begin{proof}
By Corollary~\ref{cor:pressdirectsum}, we need only show that any object is a finite direct sum of objects with simple identity.  For any object $X$, by local semisimplicity of the 2-category, there is a decomposition $\Io_X \iso \bigoplus_{i \in I} f_i$ into a finite direct sum of simple 1-morphisms $f_i$.  By Proposition~\ref{prop:idempotentcompletesum}, there is a decomposition $X \equiv \bigboxplus_{i \in I} X_i$ with inclusions and projections $\iota_i\!:X_i \leftrightarrows X:\! \rho_i$ such that $f_i \iso \iota_i \xo \rho_i$ and $\Io_{X_i} \iso \rho_i \xo \iota_i$.  If $\Io_{X_i}$ decomposed into a direct sum, then $f_i \iso \iota_i \xo \rho_i \iso \iota_i \xo \Io_{X_i} \xo \rho_i$ would also decompose, contradicting the simplicity of $f_i$.
\end{proof}
\nid By Proposition~\ref{prop:uniquedecomp}, it follows that objects in a semisimple 2-category decompose uniquely:
\begin{corollary}[Decomposition into simples is unique in semisimple 2-categories]
Every object in a semisimple $2$-category is a finite direct sum of simple objects, and this decomposition is unique up to permutation and equivalence.
\end{corollary}

\subsubsection{Semisimple completion of presemisimple $2$-categories} \label{sec:completion}
Given a presemisimple $2$-category $\tc{C}$, combining the additive completion $(-)^\oplus$ from Construction~\ref{con:additivecompletion} and the idempotent completion $(-)^\idm$ from Construction~\ref{con:iccompletion}, we obtain a semisimple $2$-category $(\tc{C}^\oplus)^\idm$.  The natural inclusion of presemisimple $2$-categories $\tc{C} \to (\tc{C}^\oplus)^\idm$ is of course not an equivalence, but we expect, at least when $\tc{C}$ is finite, that it is a $2$-distributor equivalence in the following sense.  The natural notion of morphism between finite presemisimple $2$-categories is not a $2$-functor but a $2$-distributor (also called a $2$-profunctor): given finite presemisimple $2$-categories $\tc{C}$ and $\tc{D}$, a \emph{finite semisimple $2$-distributor} $M: \tc{C} \hto{} \tc{D}$ is a $k$-bilinear $2$-functor $\tc{C}^{\op} \times \tc{D} \to \tVect$ from the product to the $2$-category $\tVect$ of finite semisimple linear $1$-categories.  Finite presemisimple $2$-categories thus form a 3-category with 1-morphisms the $2$-distributors, 2-morphisms the distributor transformations, and 3-morphisms the distributor modifications.  A \emph{$2$-distributor equivalence} of finite presemisimple 2-categories is an equivalence in that 3-category. Though theoretically straightforward enough, in practice it is difficult to determine when finite presemisimple 2-categories are 2-distributor equivalent---one is better off working with their corresponding (completed) semisimple 2-categories.

By contrast with finite presemisimple 2-categories, the natural notion of morphism between finite semisimple 2-categories is simply a 2-functor.  Indeed, every finite semisimple 2-distributor between finite semisimple 2-categories is in fact a 2-functor---this is established in Section~\ref{sec:representability}.  Altogether, finite semisimple 2-categories, with 2-functors, natural transformations, and modifications, form a 3-category.  We expect the inclusion from finite semisimple to finite presemisimple 2-categories and the completion from finite presemisimple to finite semisimple 2-categories form an equivalence between the corresponding 3-categories.

\begin{remark}[Dimension is invariant under completion] \label{rem:diminvariance}
The dimension of a presemisimple $2$-category is invariant under semisimple completion.  The fact that simple objects in a presemisimple $2$-category can be detected by whether their identities are simple ensures that a simple object of a presemisimple $2$-category remains simple during idempotent completion.  It follows that the set of components (and the dimension of each component) is unchanged by idempotent completion of a presemisimple $2$-category, and therefore the overall dimension is similarly unaffected.  More generally, we expect the dimension of a presemisimple $2$-category is invariant under 2-distributor equivalence.
\end{remark}

\subsubsection{Semisimple $2$-categories are module $2$-categories of multifusion categories}\label{sec:semisimple2mod}

Over a field of characteristic zero, in the 2-category of algebras, bimodules, and intertwiners, an algebra is fully dualizable if and only if it is finite-dimensional semisimple.  The category of finite-dimensional modules of a finite-dimensional semisimple algebra is a finite semisimple 1-category.  And in fact every finite semisimple 1-category is such a category of modules~\cite{111-IV}.

Analogously, over a field of characteristic zero, in the 3-category of finite tensor categories, bimodule categories, bimodule functors, and bimodule intertwiners, a finite tensor category is fully dualizable if and only if it is multifusion~\cite{DTC}.  The modules over a multifusion category is the prototypical semisimple 2-category; moreover, in fact every finite semisimple 2-category has this form.  We now prove this correspondence between semisimple 2-categories and module 2-categories for multifusion categories, under our standing assumption that the base field is algebraically closed of characteristic zero.
\begin{theorem}[The module 2-category of a multifusion category is semisimple] \label{thm:multifusiontosemisimple}
The 2-category of finite semisimple module categories of a multifusion category is a finite semisimple 2-category.
\end{theorem}
\begin{proof} 
Let $\mfus$ be a multifusion category and let $\Mod(\mfus)$ denote the 2-category of finite semisimple right module categories of $\mfus$.
Given finite semisimple module categories $\oc{M}_{\oc{C}}$ and $\oc{N}_{\oc{C}}$ it is proven in~\cite[Cor 2.5.6]{DTC} that the category $\Hom_{\Mod(\oc{C})}(\oc{M}, \oc{N})$ is finite semisimple.  By~\cite{RDTP,EO} 
we know that $\Mod(\mfus)$ is a $2$-category in which every $1$-morphism has a right and left adjoint.   

The existence of a zero object and direct sums is immediate. The fact that separable monads split follows from the fact that a separable monad $p: \oc{N}_{\mfus} \to \oc{N}_{\mfus}$ in $\Mod(\mfus)$ is a separable algebra in\ignore{\footnote{The monoidal category $\Hom_{\Mod(\mfus)}(\oc{N}, \oc{N})$ is often denoted by $\mfus^*_{\oc{N}}$ and called the \emph{dual of $\mfus$ with respect to $\oc{N}$}.}} $\Hom_{\Mod(\mfus)}(\oc{N}, \oc{N})$ and as such gives rise to a finite semisimple right module category $p{-}\Mod$ over $\Hom_{\Mod(\mfus)}(\oc{N}, \oc{N})$.  The $\Hom_{\Mod(\mfus)}(\oc{N}, \oc{N})$-module functor $\Hom_{\Mod(\mfus)}(\oc{N}, \oc{N}) \to\Hom_{\Mod(\mfus)}(\oc{N}, \oc{N})$ corresponding
to the object $p\in \Hom_{\Mod(\mfus)}(\oc{N}, \oc{N})$ can then be split as the following composite of $\Hom_{\Mod(\mfus)}(\oc{N}, \oc{N})$-module functors:
\[\Hom_{\Mod(\mfus)}(\oc{N}, \oc{N})  \to [{}_{p}p \otimes -] {p{-}\Mod} \to[p\otimes_p -]\Hom_{\Mod(\mfus)}(\oc{N}, \oc{N})
\] 
This splitting is separable; under the monoidal equivalence
\[p{-}\Mod{-}p \to\Hom_{\Hom_{\Mod(\oc{C})}(\oc{N}, \oc{N})}(p-\Mod, p-\Mod)\] the counit of the adjunction $p\otimes_p - \vdash _pp \otimes - $ corresponds to the multiplication $m : {}_pp \otimes p_p \To p$ of $p$. Hence, the right inverse $\Delta: {}_pp_p \To {}_pp\otimes p_p$ of $m$ in $p{-}\Mod{-}p$ gives rise to a right inverse of the counit in $\Hom_{\Hom_{\Mod(\oc{C})}(\oc{N}, \oc{N})}(p-\Mod, p-\Mod)$.
Composing with $-\xz_{\Hom_{\Mod(\oc{C})}(N,N)} N$ and noting that $\Hom_{\Mod(\oc{C})}(N,N) \xz_{\Hom_{\Mod(\oc{C})}(N,N)} N \equiv N$,
induces the required splitting of the $\mfus$-module functor \[p: \oc{N} \to p{-}\Mod\xz_{\Hom_{\mfus}(\oc{N}, \oc{N})}\oc{N}\to \oc{N}.\]

Finally, it is proven in~\cite[Cor 9.1.6]{EGNO} that every multifusion category admits only a finite number of equivalence classes of indecomposable module categories, and thus $\Mod(\mfus)$ has only finitely many simple objects. 
\end{proof}

\begin{theorem}[A semisimple 2-category is modules for a multifusion category] \label{thm:semisimplefrommultifusion}
Every finite semisimple $2$-category is equivalent to the $2$-category of finite semisimple module categories of a multifusion category.
\end{theorem}
\begin{proof}
 Let $\{X_i~|~i \in I\}$ denote a set of representatives of the equivalence classes of simple objects of the finite semisimple 2-category $\tc{C}$. Define the object $X:= \boxplus_{i \in I} X_i$ and the multifusion category $\mfus:= \Hom_\tc{C}(X,X)$, and let $\Mod(\mfus)$ denote the $2$-category of finite semisimple module categories of $\mfus$. Observe that for any object $c\in \tc{C}$, the category $\Hom_\tc{C}(X,c)$ has a right $\mfus$-module structure, and any $1$-morphism $f:c\to d$ defines a module functor $\Hom_\tc{C}(X,f) :=f\xo -: \Hom_\tc{C}(X,c) \to \Hom_\tc{C}(X,d)$. We will show that the $2$-functor $\Hom_\tc{C}(X,-): \cC \to \Mod(\mfus)$ is an equivalence.
 
1. \emph{Essential surjectivity on objects.} Let $\oc{M}$ be a finite semisimple $\mfus$-module category. Following~\cite[Cor 2.6.9]{DTC}, there is a separable algebra $m$ in $\mfus= \Hom_\tc{C}(X,X)$ such that $\oc{M}\equiv m{-}\Mod$ as $\mfus$-module categories. In particular, $m:X\to X$ is a separable monad in $\tc{C}$ and therefore admits an Eilenberg--Moore object $X^m$ in $\tc{C}$ (see Appendix~\ref{sec:em}). In particular, there is a left $m$-module $R:X^m \to X$ in $\tc{C}$ such that the induced functor $R\xo -:\Hom_\tc{C}(X, X^m) \to \LMod_m(X) = m{-}\Mod\equiv \oc{M}$ is an equivalence. Since this functor is defined by left composition with a $1$-morphism and the action of $\mfus= \Hom_\tc{C}(X,X)$ is by right composition, it inherits the structure of an equivalence of module categories.

2. \emph{Essential surjectivity on $1$-morphisms.} Let $c$ and $d$ be objects of $\tc{C}$. We show that the functor $\Hom_\tc{C}(X,-): \Hom_\tc{C}(c,d) \to \Hom_{\Mod(\mfus)}\left(\Hom_\tc{C}(X,c), \Hom_\tc{C}(X,d)\right)$ is essentially surjective on objects. Since any linear $2$-functor preserves direct sums, it suffices to prove this under the assumption that $c$ and $d$ are simple objects $X_i$ and $X_j$. Let $\{\iota_i:X_i \leftrightarrows X: \rho_i\}_{i \in I}$ be a direct sum decomposition of $X\equiv \boxplus_{i \in I} X_i$ and fix isomorphisms $\lambda_i: \Io_{X_i} \To \rho_i \xo \iota_i$. Given a module functor $\Psi:\Hom_\tc{C}(X,X_i) \to \Hom_\tc{C}(X,X_j)$ with coherence isomorphism $\psi_{a,b}: \Psi(a\xo b) \To \Psi(a)\xo b$ for $a\in \Hom_\tc{C}(X,X_i)$, $b\in \Hom_\tc{C}(X,X)$, we define $h:=\Psi(\rho_i)\xo \iota_i \in \Hom_\tc{C}(X_i,X_j)$ and claim that $\Hom_\tc{C}(X,h)$ is naturally isomorphic to $\Psi$ as a $\Hom_\tc{C}(X,X)$-module functor. Indeed, the following natural transformation provides such an isomorphism: \[\left\{ \eta_s: \Psi(s) \To[\Psi(\lambda_i\xo s)]\Psi(\rho_i \xo \iota_i \xo s)\To[\psi_{\rho_i,\iota_i s}] \Psi(\rho_i)\xo \iota_i \xo s = h\xo s  \right\}_{s\in \Hom_{\tc{C}}(X,X_i)}
\]

3. \emph{Fully faithful on $2$-morphisms.} Given $1$-morphisms $f,g:X_i \to X_j$ we will now show that the map $\Hom_{\tc{C}}(X,-): \Hom_{\tc{C}}(f,g) \to \Hom_{\Mod(\mfus)}(f\xo-,g \xo -)$ is an isomorphism. Indeed, a transformation between the module functors $f\xo-$ and $g\xo -$ is given by a natural transformation $\{ \eta_s: f\xo s \To g\xo s\}_{s\in \Hom_{\tc{C}}(X,X_i)}$ fulfilling $\eta_{s\xo r} = \eta_s \xo r$ for all $s\in \Hom_{\tc{C}}(X,X_i)$ and $r\in \Hom_{\tc{C}}(X,X)$. The map $\Hom_{\tc{C}}(X,-)$ is injective: if $\alpha, \beta \in \Hom_{\tc{C}}(f,g)$ fulfill that $\alpha\xo s= \beta \xo s$ for all $s\in \Hom_\tc{C}(X,X_i)$, then $\alpha\xo\rho_i = \beta \xo \rho_i$ and thus $\alpha = \beta$ by right invertibility of the $1$-morphism $\rho_i$. It is also surjective: given a natural transformation of module functors $\eta$ one can define the following $2$-morphism $\alpha:f\To g$ which fulfills $\alpha \xo s = \eta_s$: \[\alpha: f\To[f\xo \lambda_i]f\xo \rho_i \xo \iota_i \To[\eta_{\rho_i} \xo \iota_i] g\xo \rho_i \xo \iota_i \To[g\xo \lambda_i^{-1}] g \qedhere\]
\end{proof}

\begin{remark}[Characterizing semisimple 2-categories over non-algebraically-closed fields]
We expect that Theorems~\ref{thm:multifusiontosemisimple} and~\ref{thm:semisimplefrommultifusion} hold as stated over a field of characteristic zero (not necessarily algebraically closed) provided there are only finitely many isomorphism classes of finite-dimensional division algebras over the field.  


\end{remark}

We expect taking the module 2-category gives an equivalence of 3-categories from multifusion categories (and their finite semisimple bimodules, bimodule functors, bimodule transformations) to semisimple 2-categories (and their 2-functors, transformations, modifications).
Though semisimple 2-categories can be faithfully modeled by multifusion categories, semisimple 2-categories have a crucial theoretical advantage over multifusion categories: \emph{while additional structure on multifusion categories must be encoded in a system of bimodule categories and relations between their relative tensor products, additional structure on semisimple 2-categories can be encoded functorially}.  (See the next Section~\ref{sec:representability}, which establishes that any finite semisimple 2-distributor between semisimple categories is a 2-functor.)  In particular, an additional monoidal operation on a multifusion category $\oc{C}$ would take the form of a $(\oc{C} \boxtimes \oc{C})$-$\oc{C}$-bimodule category, whereas a monoidal structure on a semisimple $2$-category $\tc{C}$ is simply a bilinear 2-functor $\tc{C} \times \tc{C} \to \tc{C}$.

\subsubsection{$2$-distributors between semisimple $2$-categories are $2$-functors} \label{sec:representability}

Recall that a (finite semisimple) 2-distributor $\cC \hto{} \cD$ between finite semisimple 2-categories $\tc{C}$ and $\tc{D}$ is a $k$-bilinear 2-functor $\cD^{\op}\times \cC \to \tVect$, where $\tVect$ is the 2-category of finite semisimple 1-categories.  Because every finite semisimple $2$-category is locally finite semisimple, we may consider the `absolute Yoneda embedding' as a 2-functor $y_{\tc{C}}: \tc{C} \to[] \mathrm{Func}\left(\tc{C}^{\mathrm{op}}, \tVect\right)$, where $\mathrm{Func}$ denotes the $2$-category of linear $2$-functors.

\begin{prop}[The absolute Yoneda embedding of finite semisimple $2$-categories is an equivalence] \label{prop:yonedaequivalence}
For $\tc{C}$ a finite semisimple $2$-category, the absolute Yoneda embedding $y_{\tc{C}}: \tc{C} \to[] \mathrm{Func}\left(\tc{C}^{\op}, \tVect\right)$ is an equivalence.
\end{prop}
\begin{proof}
By the Yoneda lemma for 2-categories, the embedding $y_{\tc{C}}$ is an equivalence on 1-morphism categories; it therefore suffices to show $y_{\tc{C}}$ is essentially surjective, i.e.\ that any $\tVect$-valued presheaf on $\tc{C}$ is representable. By Theorem~\ref{thm:semisimplefrommultifusion}, there is a multifusion category $\mfus$ such that $\tc{C} \equiv \Mod(\mfus)$. Given a presheaf $P \in \mathrm{Func}(\Mod(\mfus)^\op,\tVect)$, note that the finite semisimple 1-category $P(\mfus_\mfus)$ is a right $\mfus$-module by precomposition with the left action of $\mfus$.  

We claim that $P$ is represented by this $\mfus$-module $P(\mfus_\mfus)$, that is there is an equivalence $P \equiv \Hom_{\Mod(\mfus)}\left(-, P(\mfus_\mfus)\right)$.  By Corollary~\ref{cor:extensionmultifusion}, two 2-functors $\Mod(\mfus)^\op \to \tVect$ are equivalent if their restrictions to the delooping $(\mathrm{B}\mfus)^{\op}$ are equivalent.  Thus it suffices to show that $P(\mfus_\mfus)$ is equivalent to $\Hom_{\Mod(\mfus)}\left(
\mfus_\mfus, P(\mfus_\mfus)\right)$ as right $\mfus$-module categories.  For any right $\mfus$-module $M_\mfus$, evaluation at the tensor unit $\Iz \in \mfus$ induces an equivalence $\Hom_{\Mod(\mfus)}\left(\mfus_\mfus, M_\mfus\right) \ra M_\mfus$, and so in particular does so for the module $P(\mfus_\mfus)$, as required.
\end{proof}

\begin{corollary}[2-distributors between semisimple 2-categories are 2-functors] \label{cor:2dist}
Every finite semisimple $2$-distributor $P: \cC\hto{} \cD$ between finite semisimple $2$-categories is a $2$-functor; that is, there is a $2$-functor $F:\cC\to \cD$ such that $P(-,-) \equiv \Hom_{\cD}(-,F-)$.
\end{corollary}
\begin{proof}
The 2-distributor $P: \cC \hto{} \cD$ corresponds to a 2-functor $\cC \to[] \mathrm{Func}\left(\cD^{\op}, \tVect\right)$.  Postcomposing with the equivalence $\mathrm{Func}\left(\cD^{\op}, \tVect\right)\equiv \cD$ provided by Proposition~\ref{prop:yonedaequivalence} yields a 2-functor $\cC\to \cD$ representing $P$.
\end{proof}

\subsubsection{Examples of semisimple $2$-categories} \label{sec:egss2cat}

The canonical example of a semisimple 1-category is the 1-category $\Vect$ of finite-dimensional vector spaces, with linear functions as morphisms.  Analogously, the canonical example of a semisimple 2-category is the 2-category of finite semisimple linear 1-categories, with linear functors and natural transformations as 1-morphisms and 2-morphisms; we denote this semisimple 2-category by $\tVect$. Note that $\tVect$ is the 2-category of finite semisimple module categories over the fusion category $\Vect$~\cite{kv,111-IV}.

\skiptocparagraph{Module $2$-categories of multifusion categories}

As discussed in Section~\ref{sec:semisimple2mod}, the finite semisimple module categories of a multifusion 1-category form a finite semisimple 2-category, and any finite semisimple 2-category is of this form.

\begin{example}[Fusion category with a unique indecomposable module category]
Let $\oc{F}$ be a fusion category with a unique indecomposable module category, namely $\oc{F}$ itself, for instance the Fibonacci fusion category or the Ising fusion category.  Then the semisimple 2-category $\Mod(\oc{F})$ of modules has a unique simple object $\oc{F}$, which has ($\oc{F}$-module) endomorphism fusion category again equivalent to $\oc{F}$.  The 2-category may therefore be schematically depicted as:
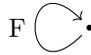
\[
\begin{tz}
\node[dot] at (0,0){};
\draw[->] (-0.1,-0.1) to [out=-135, in=135, looseness=15] node[left] {$\oc{F}$} (-0.1,0.1);
\end{tz}
\] 
Note that in this special situation, the 2-category $\Mod(\oc{F})$ is equivalent to the additive completion of the delooped 2-category $\B\oc{F}$.
\end{example}

\begin{example}[Module categories for $\Vect(\ZZ_2)$] \label{eg:VectZ2}
Over an algebraically closed field of characteristic zero, the fusion category $\Vect(\ZZ_2)$ of $\ZZ_2$-graded vector spaces has, up to equivalence, two indecomposable module categories, namely $\Vect$ and $\Vect(\ZZ_2)$; the 2-category $\Mod(\Vect(\ZZ_2))$ thus has two corresponding simple objects.  Both these simple objects have endomorphism categories $\Vect(\ZZ_2)$ (for the module $\Vect$, though every endomorphism is trivial as a functor, there is a sign choice in the module functor structure), and $\Hom_{\Vect(\ZZ_2)}(\Vect,\Vect(\ZZ_2))$ and $\Hom_{\Vect(\ZZ_2)}(\Vect(\ZZ_2),\Vect)$ are both simply $\Vect$.  This 2-category may therefore be depicted as:
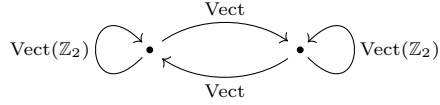
\[\begin{tz}
\node[dot,scale=1.25] at (0,0){};
\node[dot,scale=1.25] at (2,0){};
\draw[->,shorten <=0.2cm, shorten >=0.2cm] (0,0) to [out=35, in=145]node[above, sloped] {$\scriptstyle \Vect$}  (2,0);
\draw[<-,shorten <=0.2cm, shorten >=0.2cm] (0,0) to [out=-35, in=-145]node[below, sloped] {$\scriptstyle \Vect$}  (2,0);
\draw[->] (-0.1,-0.1) to [out=-135, in=135, looseness=15] node[left,xshift=.05cm] {$\scriptstyle \Vect(\ZZ_2)$} (-0.1,0.1);
\draw[->] (2.1,-0.1) to [out=-45, in=45, looseness=15] node[right,xshift=-.05cm] {$\scriptstyle \Vect(\ZZ_2)$} (2.1,0.1);
\end{tz}
\] 
Note well, as visible here, how different the structure of semisimple 2-categories is from that of semisimple 1-categories: in a semisimple 2-category there can be nontrivial 1-morphisms between inequivalent simple objects.
\end{example}

\begin{example}[Module categories for quantum $\mathfrak{sl}_2$ at level 4]
Let $\oc{C}_4$ denote the fusion category of representations of quantum $\mathfrak{sl}_2$ at level 4; this category has five simple objects, denoted $V_0, \ldots, V_4$.  The indecomposable module categories of this fusion category (and more generally of quantum $\mathfrak{sl}_2$ at other levels) have been classified by Ocneanu~\cite{Ocneanu} and Ostrik~\cite{Ostrik}: there is an indecomposable module category called $\oc{A}_5$, coming from the action of $\oc{C}_4$ on itself, and an indecomposable module category called $\oc{D}_4$, whose corresponding algebra object is $V_0 \oplus V_4$.  The module endomorphisms of $\oc{A}_5$ is of course just $\oc{C}_4$, and $\Hom_{\oc{C}_4}(\oc{A}_5,\oc{D}_4)$ and $\Hom_{\oc{C}_4}(\oc{D}_4,\oc{A}_5)$ are both simply $\oc{D}_4$.  The endomorphism fusion category $\End_{\oc{C}_4}(\oc{D}_4)$ of $\oc{D}_4$ is described in Goto~\cite{goto} and has eight simple objects and fusion graph that may be depicted (in the picture of the 2-category $\Mod(\oc{C}_4)$) as follows:
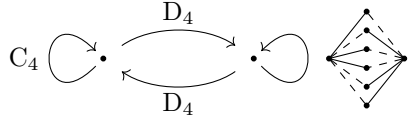
\newsavebox\mybox%
\begin{lrbox}{\mybox}%
\begin{tz}[scale=0.25]
\node[dot] (A1) at (0, 2.5){};
\node[dot] (A2) at (0,1.5){};
\node[dot] (A3) at (0,0.5){};
\node[dot] (A4) at (0,-0.5){};
\node[dot] (A5) at (0,-1.5){};
\node[dot] (A6) at (0,-2.5){};
\node[dot] (L) at (-2,0) {};
\node[dot] (R) at (2,0){};
\draw[-] (L) to (A1);
\draw[-] (L) to (A3);
\draw (L) to (A4);
\draw (R) to (A2);
\draw (R) to (A5);
\draw (R) to (A6);
\draw[dashed] (L) to (A2);
\draw[dashed] (L) to (A5);
\draw[dashed] (L) to (A6);
\draw[dashed] (R) to (A1);
\draw[dashed] (R) to (A3);
\draw[dashed] (R) to (A4);
\end{tz}%
\end{lrbox}%
\[\begin{tz}
\node[dot] at (0,0){};
\node[dot] at (2,0){};
\draw[->,shorten <=0.3cm, shorten >=0.3cm] (0,0) to [out=35, in=145]node[above, sloped] {$\oc{D}_4$}  (2,0);
\draw[<-,shorten <=0.3cm, shorten >=0.3cm] (0,0) to [out=-35, in=-145]node[below, sloped] {$\oc{D}_4$}  (2,0);
\draw[->] (-0.1,-0.1) to [out=-135, in=135, looseness=15] node[left] {$\oc{C}_4$} (-0.1,0.1);
\draw[->] (2.1,-0.1) to [out=-45, in=45, looseness=15] node[right] {$~~~~~\usebox\mybox$} (2.1,0.1);
\end{tz}
\]
Here in $\End_{\oc{C}_4}(\oc{D}_4)$, the top node is the tensor unit, the left and right nodes are generators, and the solid and dashed lines give fusion with these generators, respectively.
\end{example}

\begin{example}[Decomposing the module 2-category of a multi-component multifusion category]
The 2-category of module categories of a multifusion category $\oc{M} = \oc{M}^1\oplus \ldots \oplus \oc{M}^n$ with indecomposable components $\oc{M}^i$ is equivalent to the direct sum of 2-categories $\boxplus_i \Mod(\oc{M}^i)$.  (See Remark~\ref{rem:components}.)  Let $\oc{M}^i = \oplus_{j,k} \oc{M}^i_{jk}$ be the decomposition of $\oc{M}^i$ into fusion categories $\oc{M}^i_{jj}$ and their bimodule categories.  By~\cite[Prop 7.17.5]{EGNO}, for any $j$, there is an invertible bimodule (namely $\oplus_j \oc{M}^i_{jk}$) between the multifusion category $\oc{M}^i$ and the fusion category $\oc{M}^i_{jj}$, and so, picking some index $j_i$ for each $i$, we have $\Mod(\oc{M})$ is equivalent to $\boxplus_i \Mod(\oc{M}^i_{j_i j_i})$.  (This presentation is convenient later when we consider monoidal structures on these semisimple 2-categories.)
\end{example}

\skiptocparagraph{Representations of groupoids}

Many examples of semisimple 1-categories arise by taking a category of functors into $\Vect$, and a natural family of domains for such functors are 1-groupoids. 

\begin{notation}[$n$-groupoids via homotopy groups] \label{notation:groupoids}
Recall that an $n$-groupoid is a homotopy $n$-type, that is a space whose only nontrivial homotopy groups occur in dimensions $0, 1, \ldots, n$.  We will suppress k-invariant information from the notation and denote an $n$-groupoid simply by the tuple $(\pi_0,\pi_1,\ldots,\pi_n)$, where $\pi_0$ is a discrete set (namely the component set of the space), and $\pi_i$ for $i>0$ is a family of discrete groups indexed by $\pi_0$ (namely the family of the $i$-th homotopy groups of the various components of the space).  

Alternatively, an $n$-groupoid may be viewed as an $n$-category all of whose morphisms are invertible.  From that perspective, the k-invariant information is encoded in the weak structure data, that is the higher units, associators, and interchangers of the $n$-category.  In the categorical view, the set $\pi_0$ is the set of equivalence classes of objects of the $n$-category, and the family $\pi_i$ records, for each component of the $n$-category (i.e.\ element of $\pi_0$), the set of equivalence classes of automorphisms of the $(i-1)$-fold identity on an object in that component.

A 0-groupoid is therefore denoted simply $\pi_0$; a 1-groupoid with no nontrivial automorphisms is denoted $(\pi_0,\ast)$---here $\ast$ refers to the $\pi_0$-indexed family of groups that is trivial in every component; a connected 1-groupoid is denoted $(\ast,\pi_1)$; a 2-groupoid with no nontrivial 2-automorphisms is denoted $(\pi_0,\pi_1,\ast)$; a connected 2-groupoid is denoted $(\ast,\pi_1,\pi_2)$.

The notation permits the addition of monoidal structures.  Recall that an $n$-group is by definition a connected $n$-groupoid.  Thought of directly as an $n$-groupoid, this would be denoted $(\ast, \pi_1, \ldots, \pi_n)$, but we may instead consider its loop space as an $(n-1)$-groupoid with a (grouplike) monoidal structure; that grouplike monoidal $(n-1)$-groupoid will be denoted $(\pi_1, \ldots, \pi_n)$.  Similarly, a (grouplike) $k$-fold monoidal $n$-group is, by definition, an $(n+k-1)$-groupoid $(\pi_0, \ldots, \pi_{n+k-1})$ such that $\pi_i$ is trivial for $0 \leq i \leq k-1$.  The $k$-th loop space of that $(n+k-1)$-groupoid is a $k$-fold monoidal $(n-k)$-groupoid denoted $(\pi_k, \ldots, \pi_n)$.  For instance, a grouplike monoidal 0-groupoid (that is, a group) is denoted $\pi_1$; a grouplike 2-fold monoidal 0-groupoid (that is, an abelian group) is denoted $\pi_2$; a grouplike monoidal 1-groupoid (that is, the loop space of a 2-group) is denoted $(\pi_1, \pi_2)$; a grouplike 2-fold monoidal 1-groupoid (that is, the double loop space of a 3-group) is denoted $(\pi_2, \pi_3)$; a grouplike monoidal 2-groupoid (that is, the loop space of a 3-group) is denoted $(\pi_1,\pi_2,\pi_3)$.

An $n$-groupoid is called finite when it has finitely many components and all the homotopy groups of every component are finite.
\end{notation}

\begin{recollection}[1-category of 1-representations of a 1-groupoid]
A representation of a finite 1-groupoid $(\pi_0,\pi_1)$ is a 1-functor $(\pi_0,\pi_1) \to \Vect$.   The linear 1-category $\Rep(\pi_0,\pi_1) := [(\pi_0,\pi_1),\Vect]$ of representations of a finite 1-groupoid $(\pi_0,\pi_1)$ is a finite semisimple 1-category. 
\end{recollection}

\nid%
As a special case, we may think of a group $\pi_1$ as a connected 1-groupoid $(\ast,\pi_1)$, and consider the category of representations 
$\Rep(\ast,\pi_1) \origequiv [(\ast,\pi_1),\Vect]$; the objects of this category are called simply ``$\pi_1$-representations".  As another special case, we may think of a set $\pi_0$ as a discrete 1-groupoid $(\pi_0,\ast)$, and consider the category of representations 
$\Rep(\pi_0,\ast) \origequiv [(\pi_0,\ast),\Vect]$; the objects of this category are called ``$\pi_0$-graded vector spaces".  
We summarize this situation in the following table:

\begin{center}
\begin{tabular}{ r r c l l }
\multicolumn{1}{c}{Input} & \multicolumn{1}{c}{Notation} & & \multicolumn{1}{c}{Definition} & \multicolumn{1}{c}{Name} \\ \hline
groupoid $(\pi_0,\pi_1)$ & $\Rep(\pi_0,\pi_1)$ & $:=$ & $[(\pi_0,\pi_1),\Vect]$ & $(\pi_0,\pi_1)$-representation \\
group $\pi_1$ & $\Rep(\ast,\pi_1)$ & $:=$ & $[(\ast,\pi_1),\Vect]$ & $\pi_1$-representation \\
set $\pi_0$ & $\Rep(\pi_0,\ast)$ & $:=$ & $[(\pi_0,\ast),\Vect]$ & $\pi_0$-graded vector space
\end{tabular}
\end{center}

Given a group $\pi_1$, one may of course think of it as a discrete (grouplike monoidal) 1-groupoid $(\pi_1,\ast)$, and therefore consider the category $\Vect(\pi_1) := \Rep(\pi_1,\ast)$ of $\pi_1$-graded vector spaces.  We defer attention to that case until later, when we are concerned with monoidal structures on these semisimple 1-categories.  We then consider also more generally the category $\Rep(\pi_1,\pi_2)$ of representations of a 2-group $(\pi_1,\pi_2)$.

\skiptocparagraph{2-representations of 2-groupoids}

Many examples of semisimple 2-categories arise by taking a 2-category of 2-functors into $\tVect$.  

\begin{example}[2-category of 2-representations of a 2-groupoid] \label{eg:rep2groupoid}
For a finite 2-groupoid $(\pi_0,\pi_1,\pi_2)$, a representation of $(\pi_0,\pi_1,\pi_2)$ is a (weak) 2-functor $(\pi_0,\pi_1,\pi_2) \to \tVect$.  The linear 2-category $\tRep(\pi_0,\pi_1,\pi_2) := [(\pi_0,\pi_1,\pi_2),\tVect]$ of 2-representations of a finite 2-groupoid $(\pi_0,\pi_1,\pi_2)$ is a finite semsimple 2-category; this can be seen as follows.

Given a finite 2-groupoid $(\pi_0, \pi_1,\pi_2)$, let $k(\pi_0, \pi_1,\pi_2)$ denote its linearization, that is the linear 2-category with the same objects and 1-morphisms as the 2-groupoid, and with 2-morphism vector spaces freely generated by the 2-morphism sets of the 2-groupoid.  Let $\widehat{k}(\pi_0,\pi_1,\pi_2)$ denote the finite presemisimple 2-category obtained by local additive and local idempotent completion of $k(\pi_0, \pi_1,\pi_2)$. Then let $\overline{k}(\pi_0,\pi_1,\pi_2)$ denote the multifusion category obtained by folding $\widehat{k}(\pi_0,\pi_1,\pi_2)$.  Observe that the 2-category of 2-functors $[(\pi_0, \pi_1, \pi_2), \tVect]$ is equivalent to the 2-category of modules $\Mod(\overline{k}(\pi_0,\pi_1,\pi_2))$, and therefore that $\tRep(\pi_0,\pi_1,\pi_2)$ is a finite semisimple 2-category, as desired.
\end{example}

There are a number of important special cases of this example of 2-representations of a 2-groupoid.  We summarize them in the following table:
\[
\hspace*{-0.95cm}
\shrinkalign{.92}{
\begin{tabular}{ r r c l l }
\multicolumn{1}{c}{Input} & \multicolumn{1}{c}{Notation} & & \multicolumn{1}{c}{Definition} & \multicolumn{1}{c}{Name} \\ \hline
2-grpoid $(\pi_0,\pi_1,\pi_2)$ & $\tRep(\pi_0,\pi_1,\pi_2)$ & $:=$ & $[(\pi_0,\pi_1,\pi_2),\tVect]$ & $(\pi_0,\pi_1,\pi_2)$-2-representation \\
2-group $(\pi_1,\pi_2)$ & $\tRep(\ast, \pi_1,\pi_2)$ & $:=$ & $[(\ast,\pi_1,\pi_2),\tVect]$ & $(\pi_1,\pi_2)$-2-representation \\
1-grpoid $(\pi_0,\pi_1)$ & $\tRep(\pi_0,\pi_1,\ast)$ & $:=$ & $[(\pi_0,\pi_1,\ast),\tVect]$ & $(\pi_0,\pi_1)$-graded 2-vector space \\
ab group $\pi_2$ & $\tRep(\ast,\ast,\pi_2)$ & $:=$ & $[(\ast,\ast,\pi_2),\tVect]$ & $\pi_2$-2-representation \\
group $\pi_1$ & $\tRep(\ast,\pi_1,\ast)$ & $:=$ & $[(\ast,\pi_1,\ast),\tVect]$ & $\pi_1$-2-vepresentation \\
set $\pi_0$ & $\tRep(\pi_0,\ast,\ast)$ & $:=$ & $[(\pi_0,\ast,\ast),\tVect]$ & $\pi_0$-graded 2-vector space
\end{tabular}
}
\]
Of course, given a group $\pi_1$, one may treat it as a discrete (grouplike monoidal) 2-groupoid $(\pi_1,\ast,\ast)$, and consider the 2-category $\tVect(\pi_1) := \tRep(\pi_1,\ast,\ast)$ of $\pi_1$-graded 2-vector spaces.  We defer discussion of that case to later attention to monoidal structures, where we also discuss more generally the 2-category $\tRep(\pi_1,\pi_2,\pi_3)$ of 2-representations of a 3-group $(\pi_1,\pi_2,\pi_3)$.

\begin{example}[$\pi_0$-graded 2-vector spaces]
For a finite set $\pi_0$, the 2-category $\tRep(\pi_0,\ast,\ast)$ is the 2-category of $\pi_0$-graded finite semisimple 1-categories $\oc{C} = \bigoplus_{x \in \pi_0} \oc{C}_x$, with 1-morphisms the grading-preserving functors and 2-morphisms the natural transformations.  

The simple objects of this 2-category are of the form $[x] := \bigoplus_{y \in \pi_0} \oc{C}_y$ with $\oc{C}_x = \Vect$ and $\oc{C}_y = 0$ for $y \neq x$.  In this case, the 1-morphism category $\Hom_{\tRep(\pi_0, \ast,\ast)}([x], [y])$ is $\Vect$ if $x = y$ and zero otherwise.
\end{example}

\begin{example}[$\pi_1$-2-vepresentations are modules for graded vector spaces] \label{rem:VepisVect}
For a finite group $\pi_1$, the 2-category $\tRep(\ast,\pi_1,\ast)$ is the 2-category of finite semisimple 1-categories with a (weak) $\pi_1$-action, with 1-morphisms the intertwining functors and 2-morphisms the natural transformations.  The structure of a $\pi_1$-action on a semisimple 1-category is equivalent to the structure of an action of the fusion 1-category $\Vect(\pi_1)$ of $\pi_1$-graded vector spaces (with the monoidal structure induced by the group structure on $\pi_1$).  Thus the 2-category $\tRep(\ast,\pi_1,\ast)$ is equivalent to the 2-category $\Mod(\Vect(\pi_1))$.

For instance, the 2-category $\tRep(\ast,\ZZ_2,\ast)$ is equivalent to the 2-category \linebreak$\Mod(\Vect(\ZZ_2))$ described in Example~\ref{eg:VectZ2}, with two simple objects having nonzero non-equivalence morphisms between them.  More generally, the indecomposable module categories of $\Vect(\pi_1)$, hence the simple objects of $\tRep(\ast,\pi_1,\ast)$, have been classified by Ostrik~\cite{Ostrik}.
\end{example}

\begin{example}[$\pi_2$-2-representations]
For a finite abelian group $\pi_2$, the 2-category \linebreak$\tRep(\ast,\ast,\pi_2)$ is the 2-category of finite semisimple 1-categories $\oc{C}$ with a group homomorphism $\phi: \pi_2 \to \Aut(\Iz_{\oc{C}})$, with 1-morphisms the functors $F:\oc{C} \to \oc{C}'$ such that $F \xo \phi(g) = \phi'(g) \xo F$ for all $g \in \pi_2$ and 2-morphisms the natural transformations.  

An object $(\oc{C}, \phi) \in \tRep(\ast,\ast,\pi_2)$ is simple if and only if $\oc{C} \equiv \Vect$.  Hence the simple objects of $\tRep(\ast,\ast,\pi_2)$ correspond to group homomorphisms $\phi: \pi_2 \to k^*$, and the 1-morphism category $\Hom_{\tRep(\ast, \ast, \pi_2)} (\phi, \phi')$ is $\Vect$ if $\phi = \phi'$ and zero otherwise.  Thus it happens that $\tRep(\ast,\ast,\pi_2)$ is equivalent to $\tRep(\Hom(\pi_2, k^*), \ast,\ast)$.

The simple objects and Hom categories of the common generalization of this example and the previous example, namely $\tRep(\ast,\pi_1,\pi_2)$, have been investigated by Elgueta~\cite{Elgueta}.
\end{example}

\begin{remark}[Completion produces non-invertible simple 1-morphisms]
Even though, by the discussion in Example~\ref{eg:rep2groupoid}, the 2-category $\tRep(\pi_0,\pi_1,\pi_2)$ is the semisimple completion of the linearization $k(\pi_0,\pi_1,\pi_2)$ of the 2-groupoid $(\pi_0,\pi_1,\pi_2)$, it can certainly happen that $\tRep(\pi_0,\pi_1,\pi_2)$ has non-invertible simple 1-morphisms.  Such morphisms are seen, for instance, in $\tRep(\ast,\ZZ_2,\ast) \simeq \Mod(\Vect(\ZZ_2))$ from Remark~\ref{rem:VepisVect} and Example~\ref{eg:VectZ2}.
\end{remark}

\begin{remark}[The dimension of groupoid-graded 2-vector spaces is the groupoid cardinality] \label{rem:groupoidcardinality}
For a finite set $\pi_0$, the semisimple $2$-category $\tRep(\pi_0,\ast,\ast)$ evidently has dimension the order of $\pi_0$.  For a finite group $\pi_1$, by the preceding remark and Remark~\ref{rem:diminvariance}, we see that the dimension of the semisimple $2$-category $\tRep(\ast,\pi_1,\ast)$ is the dimension of the presemisimple $2$-category $\B\Vect(\pi_1)$.  As the global dimension of the fusion category $\Vect(\pi_1)$ is $|\pi_1|$, the dimension of $\B\Vect(\pi_1)$ is $1/|\pi_1|$.  More generally, for a finite $1$-groupoid $(\pi_0,\pi_1)$, the dimension of the $2$-category of $(\pi_0,\pi_1)$-graded $2$-vector spaces $\tRep(\pi_0,\pi_1,\ast)$ is the groupoid cardinality~\cite{Groupoidification} of $(\pi_0,\pi_1)$, that is the sum over components of the reciprocal of the size of the automorphism groups.
\end{remark}

\newpage

\addtocontents{toc}{\protect\vspace{8pt}}

\section{Monoidal $2$-categories}

\subsection{Fusion $2$-categories}

\subsubsection{The definition of fusion $2$-categories}

\skiptocparagraph{Monoidal structures on $2$-categories}

We will work with semistrict monoidal $2$-categories, meaning that the underlying $2$-category is strict and the monoidal structure is strictly unital and associative, though there may be a nontrivial interchange isomorphism between the two distinct ways of taking the monoidal product of two $1$-morphisms.  Such semistrict monoidal $2$-categories first appeared in Gordon--Powers--Street~\cite{GPS} and are often called ``Gray monoids".  A recent presentation of the notion occurs in Barrett--Meusburger--Schaumann (BMS)~\cite{BMS}.  Though structured somewhat differently, our notion of monoidal $2$-category is equivalent to the BMS definition.

\begin{definition}[Monoidal 2-category] \label{def:monoidal2cat}
A \emph{monoidal $2$-category} consists of the following data:
\begin{enumerate}
\item[D1.] a strict $2$-category $\tc{C}$;
\item[D2.] an ``identity" object $\Iz \in \tc{C}$;
\item[D3.] strict ``left and right tensor product" $2$-functors 
\begin{align*}
L_A&\origequiv A\xz -: \tc{C} \to \tc{C}\\
R_A&\origequiv -\xz A: \tc{C} \to \tc{C}, 
\end{align*}
for each object $A\in \tc{C}$;
\item[D4.] an ``interchange" $2$-isomorphism
\[
\phi_{f,g}: \left( f\xz B'\right)\xo \left(A\xz g\right) \To \left(A' \xz g\right)\xo \left(f \xz B \right)
\]
for each pair of $1$-morphisms $f:A\to A'$ and $g:B\to B'$;
\end{enumerate}
subject to the following conditions:
\begin{enumerate}
\item[C1.] left and right multiplication agree: $L_A B = R_B A$ for objects $A,B \in \tc{C}$;
\item[C2.] the tensor product is strictly unital and associative:
\begin{align*} 
&L_\Iz = \mathrm{id}_{\tc{C}} = R_\Iz \\
&L_A L_B = L_{A\xz B} \\
&R_BR_A = R_{A\xz B} \\
&L_A R_B = R_B L_A;
\end{align*}
\item[C3.] the interchanger respects identities: 
\begin{align*}
\phi_{f, \Io_A} &= \It_{f\xz A} \\
\phi_{\Io_A, f} &= \It_{A\xz f}
\end{align*}
for object $A \in \tc{C}$ and 1-morphism $f: C \to D$;
\item[C4.] the interchanger respects composition:
\begin{align*}
\phi_{f'\xo f, g}&= \left(\phi_{f',g} \xo (f\xz B ) \right)\xt \left( (f' \xz B')\xo \phi_{f,g}\right)\\
\phi_{f,g'\xo g} &= \left((A' \xz g')  \xo \phi_{f,g}\right) \xt \left(\phi_{f,g'}\xo ( A \xz g )  \right)
\end{align*}
for $f:A\to A'$, $f':A'\to A''$, $g:B\to B'$ and $g':B'\to B''$;
\item[C5.] the interchanger is natural:
\begin{align*}
\phi_{f',g} \xt \left((\alpha \xz B') \xo (A\xz g ) \right) &= \left((A'\xz g) \xo (\alpha \xz B) \right)\xt \phi_{f,g}\\
\phi_{f,g'}\xt \left( \left(  f\xz B'\right) \xo \left(A \xz  \beta \right) \right) &=  \left( \left( A' \xz \beta \right) \xo \left( f \xz B \right) \right)\xt\phi_{f,g}
\end{align*}
for $1$-morphisms $f,f':A\to A', g,g':B\to B'$ and $2$-morphisms $\alpha: f\To f'$, $\beta:g\To g'$;
\item[C6.] the interchanger respects tensor product:
\begin{align*}
\phi_{A\xz g, h} &= A \xz \phi_{g,h} \\
\phi_{f\xz B, h} &= \phi_{f, B\xz h} \\
\phi_{f,g\xz C} &= \phi_{f,g} \xz C
\end{align*}
for $f:A \to A'$, $g:B\to B'$ and $h:C\to C'$.

\end{enumerate}
\end{definition}

\begin{notation}[Horizontal composition of 1-morphisms] \label{notation:nudging}
Though the tensor product of a monoidal $2$-category does not provide a unique tensor of two $1$-morphisms $f:A\to B$ and $g:C\to D$, as a matter of convenient notation, we use the symbol $f\xz g$ to mean $\left(f\xz D\right)\xo \left(A\xz g\right)$.  This convention is known as ``nudging".  The `tensor' $\alpha \xz \beta$ of two $2$-morphisms $\alpha$ and $\beta$ is similarly used to mean the corresponding nudged composite.
\end{notation}

\begin{definition}[Linear monoidal 2-category] \label{def:linmon2cat}
A \emph{linear monoidal $2$-category} is a linear $2$-category equipped with a monoidal structure such that, for all objects $A$, the functors $A \xz -$ and $- \xz A$ are linear.
\end{definition}

\begin{remark}[Strictification for monoidal 2-categories]
Any weakly monoidal weak 2-category can be strictified to a monoidal 2-category of the flavor given in Definition~\ref{def:monoidal2cat}; similarly any linear weakly monoidal weak 2-category can be strictified to a linear monoidal 2-category a la Definition~\ref{def:linmon2cat}.  The feasibility of those strictifications follow from the usual stricitfication for tricategories~\cite{GPS} and a corresponding $\Vect$-enriched version.  Because of this, we permit ourselves to work with semistrict monoidal 2-categories as described, even though most examples will arise in the first instance in a weaker form.
\end{remark}

\skiptocparagraph{Duality in monoidal $2$-categories}

\begin{definition}[Duals in monoidal 2-categories]  
In a monoidal $2$-category, an object $A^\#$ is a \emph{right dual} of an object $A$, equivalently $A$ is a \emph{left dual} of $A^\#$, if there exist counit and unit $1$-morphisms $e: A \xz A^\# \to \Iz$ and $i: \Iz \to A^\# \xz A$ such that $(e\xz A) \xo  (A\xz i) \iso \Io_{A}$ and $\Io_{A^\#} \iso (A^\# \xz e) \xo (i \xz A^\# )$.
\end{definition}

\begin{definition}[Prefusion and fusion 2-categories] \label{def:fusion2cat} 
A \emph{prefusion $2$-category} is a finite presemisimple monoidal $2$-category that has left and right duals for objects and a simple monoidal unit.  A \emph{fusion $2$-category} is a finite semisimple monoidal $2$-category that has left and right duals for objects and a simple monoidal unit.
\end{definition}

\nid Being a (pre)fusion $2$-category is a property of a linear monoidal $2$-category, and this property is preserved under any linear $2$-equivalence.  Hence, every (pre)fusion linear \emph{weakly} monoidal \emph{weak} $2$-category can be strictified to a (pre)fusion $2$-category in the sense of Definition~\ref{def:fusion2cat}.

\begin{remark}[State sum invariance under completion] \label{rem:invariancecompletion}
The monoidal product in a prefusion or fusion 2-category is given by a 2-functor (as opposed to a 2-distributor).  It is not therefore the case that one can transport a prefusion structure on a presemisimple 2-category across an arbitrary 2-distributor equivalence of presemisimple 2-categories (see Section~\ref{sec:completion}).  However, because every 2-distributor between finite semisimple 2-categories is a 2-functor, it is the case that a prefusion structure on a presemisimple 2-category induces a fusion structure on the completed semisimple 2-category.  (We might say that a prefusion 2-category is `monoidally 2-distributor equivalent' to its completion.)  We expect the state sum will be invariant under this completion, that is the state sum invariant for a prefusion 2-category will be the same as the invariant for the associated (completed) fusion 2-category.
\end{remark}

\begin{remark}[State sum invariance under bimodule equivalence]
A natural notion of equivalence between fusion 2-categories is asking for a monoidal 2-functor that is an equivalence of the underlying 2-categories; this is called a `monoidal 2-functor equivalence'.  (This is a categorification of the notion of monoidal functor equivalence between monoidal categories.)  A coarser notion of equivalence between two prefusion or fusion 2-categories is asking for an invertible bimodule 2-category between them; this is called a `bimodule equivalence'.  (This notion is a categorification of the notion of bimodule or `Morita' equivalence between monoidal categories.)  The state sum of a fusion 2-category is of course invariant under monoidal 2-functor equivalence, but we speculate it is actually invariant under bimodule equivalence as well.
\end{remark}

\subsubsection{The graphical calculus of fusion $2$-categories}\label{sec:graphicalcalculus}

\skiptocparagraph{Calculus of $2$-categories}

A 2-category admits a graphical calculus of labeled 1-manifolds in the plane, so called `string diagrams':
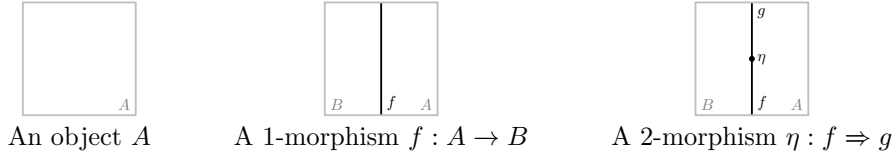
\begin{calign}\nonumber
\begin{tz}[scale=0.5]
\draw[slice] (0,0) rectangle (3,3);
\node[obj, above left] at (3,0){$A$};
\end{tz}
&
\begin{tz}[scale=0.5]
\draw[slice, on layer=front] (0,0) rectangle (3,3);
\draw[wire] (1.5,0) to (1.5,3);
\node[obj, above left] at (3,0){$A$};
\node[obj, above right] at (0,0){$B$};
\node[omor, above right] at (1.5,0) {$f$};
\end{tz}
&
\begin{tz}[scale=0.5]
\draw[slice,on layer=front] (0,0) rectangle (3,3);
\draw[wire] (1.5,0) to (1.5,3);
\node[obj, above left] at (3,0){$A$};
\node[obj, above right] at (0,0){$B$};
\node[omor, above right] at (1.5,0) {$f$};
\node[omor, below right] at (1.5,3) {$g$};
\node[dot] at (1.5,1.5) {};
\node[tmor, right] at (1.5, 1.5) {$\eta$};
\end{tz}
\\\nonumber
\text{An object $A$}
&
\text{A $1$-morphism $f:A\to B$}
&
\text{A $2$-morphism $\eta:f \To g$}
\end{calign}
Here objects are depicted as regions in the plane, a 1-morphism $A \to B$ is depicted as a wire separating the region labeled $A$ from the region labeled $B$, and a 2-morphism $f \To g$ is depicted as a node separating the wire labeled $f$ from the wire labeled $g$.  (The gray bounding box simply indicates the extent of the picture.)  We draw 1-morphism composition from right to left: that is, in a diagram for $g \xo f$, the wire labeled $g$ appears to the left of the wire labeled $f$.  Similarly, we draw 2-morphism composition from bottom to top: that is, in a diagram for $\eta \xt \mu$, the node labeled $\eta$ appears above the node labeled $\mu$.

\skiptocparagraph{Calculus of monoidal structures}

Monoidal 2-categories (and more generally semistrict 3-categories) admit a similar graphical calculus of `surface diagrams' in 3-space~\cite{BMS}:
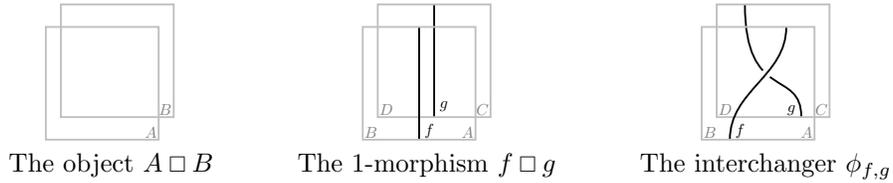
\begin{calign}\nonumber
 \begin{tz}[td]
 \begin{scope}[yzplane=0]
 \draw[slice] (0,0) rectangle (3,3);
 \node[obj, above left] at (-0.1,-0.1){$A$};
 \end{scope}
  \begin{scope}[yzplane=1.]
 \draw[slice] (0,0) rectangle (3,3);
  \node[obj, above left] at (-0.1,-0.1){$B$};
 \end{scope}
 \end{tz}
 &
  \begin{tz}[td]
\begin{scope}[yzplane=0, on layer=front]
 \draw[slice, on layer=superfront] (0,0) rectangle (3,3);
 \draw[wire] (1.5,0) to (1.5,3);
 \node[obj, above right] at (3.1, -0.1) {$B$};
 \node[omor, above right] at (1.5,-0.1) {$f$};
 \node[obj, above left] at (-0.1,-0.1){$A$};
 \end{scope}
  \begin{scope}[yzplane=1.,on layer=back]
 \draw[slice] (0,0) rectangle (3,3);
 \draw[wire, on layer=superback] (1.5,0) to (1.5,3);
 \node[omor, above right] at (1.5,-0.) {$g$};
  \node[obj, above left] at (-0.1,-0.1){$C$};
  \node[obj, above right] at (3.1,-0.1) {$D$};
 \end{scope}
 \end{tz}
 &
 \begin{tz}[td]
 \begin{scope}[yzplane=0, on layer=superfront]
 \draw[slice] (0,0) rectangle (3,3);
 \node[omor, above right] at (2.25,-0.1) {$f$};
 \node[obj, above right] at (3.1, -0.1) {$B$};
 \node[obj, above left] at (-0.1,-0.1){$A$};
 \end{scope}
  \begin{scope}[yzplane=1,on layer=fronta]
 \draw[slice] (0,0) rectangle (3,3);
 \node[omor, above left] at (0.75,-0.1) {$g$};
 \node[obj, above left] at (-0.1,-0.1){$C$};
  \node[obj, above right] at (3.1,-0.1) {$D$};
 \end{scope}
 \draw[wire, on layer=front] (2.25,0,0) to [out=up, in=down] node[mask point, pos=0.6](A){} (0.75,0,3);
 \cliparoundone{A}{\draw[wire] (0.75,1,0) to [out=up, in=down, in looseness=2,] (2.25,1,3);}
 \end{tz}
 \\\nonumber
\text{The object $A\xz B$}&
\text{The $1$-morphism $f\xz g$}&
\text{The interchanger $\phi_{f,g}$}
\end{calign}
The monoidal structure is depicted by layering surfaces behind one another, with the convention that tensor product occurs from back to front: that is, in a diagram for $A \xz B$, the surface labeled $A$ appears in front of the surface labeled $B$.  Note that the tensor of 1-morphisms $f \xz g$ is defined by Notation~\ref{notation:nudging} and indeed in the diagram the morphism $f$ appears slightly to the left of the morphism $g$.  As drawn, the interchanger is depicted as a crossing of wires living in parallel planes. The following is a more complicated example of a surface diagram representing a 2-morphism in a monoidal 2-category:
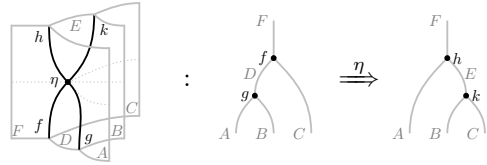
\[
\raisebox{-0.08cm}{%
\begin{tz}[td,scale=1]
\begin{scope}[xyplane=0]
\draw[slice] (0,0) to [out=up, in=\dl] (0.5,1) to [out=up, in=\dl] (1,2) to (1,3);
\draw[slice] (1,0) to [out=up, in=\dr] (0.5,1);
\draw[slice] (2,0) to [out=up, in=\dr] (1,2);
\end{scope}
\begin{scope}[xyplane=\h, on layer=superfront]
\draw[slice] (0,0) to [out=up, in=\dl] (1,2) to (1,3);
\draw[slice] (1,0) to [out=up, in=\dl] (1.5,1) to [out=up, in=\dr] (1,2);
\draw[slice] (2,0) to [out=up, in=\dr] (1.5,1);
\end{scope}
\begin{scope}[xyplane=0.5*\h]
\draw[tinydash] (0,0) to [out=up, in=\dl] (1, 1.5);
\draw[tinydash] (1,0) to (1,3);
\draw[tinydash] (2,0) to [out=up, in=\dr] (1,1.5);
\end{scope}
\begin{scope}[xzplane=0]
\draw[slice,short] (0,0) to (0, \h);
\draw[slice,short] (1,0) to (1,\h);
\draw[slice,short] (2,0) to (2,\h);
\end{scope}
\begin{scope}[xzplane=3]
\draw[slice,short] (1,0) to (1,\h);
\end{scope}
\coordinate (A) at (1.5,1,0.5*\h);
\draw[wire] (1,0.5,0) to [out=up, in=\dr] (A) to [out=\ur, in=down] (1,1.5,\h);
\draw[wire] (2,1,0) to [out=up, in=\dl] (A) to [out=\ul, in=down] (2,1,\h);
\node[dot] at (A){};
\node[obj, above left] at (-0.12,0,-0.08) {$A$};
\node[obj, above left] at (-0.12,1,-0.08) {$B$};
\node[obj, above left] at (-0.12,2,-0.08) {$C$};
\node[obj, above right] at (3.15,1,-0.08){$F$};
\node[obj, above] at (1.4,0.6,-0.08){$D$};
\node[obj, below] at (1.5,1.5,\h){$E$};
\node[omor, above left] at (2,1,0){$f$};
\node[omor, above right] at (1,0.5,0){$g$};
\node[omor, below left] at (2,1,\h){$h$};
\node[omor, below right] at (1,1.5,\h-0.1){$k$};
\node[tmor, left] at ([xshift=-0.1cm]A){$\eta$};
\end{tz}
}
\hspace{0.2cm}:\hspace{0.2cm}
\begin{tz}[scale=0.5]
\draw[slice] (0,0) to [out=up, in=\dl] (0.5,1) to [out=up, in=\dl] (1,2) to (1,3);
\draw[slice] (1,0) to [out=up, in=\dr] (0.5,1);
\draw[slice] (2,0) to [out=up, in=\dr] (1,2);
\node[dot] at (0.5,1){};
\node[dot] at (1,2){};
\node[obj,left] at (0,0) {$A$};
\node[obj,  left] at (1,0) {$B$};
\node[obj, left] at (2,0) {$C$};
\node[obj,  left] at (1, 3) {$F$};
\node[obj,  left] at (0.7, 1.6) {$D$};
\node[omor, left] at (0.5,1) {$g$};
\node[omor, left] at (1,2) {$f$};
\end{tz}
\hspace{0.25cm}\To[~\eta~]\hspace{0cm}
\begin{tz}[scale=0.5,xscale=-1]
\draw[slice] (0,0) to [out=up, in=\dl] (0.5,1) to [out=up, in=\dl] (1,2) to (1,3);
\draw[slice] (1,0) to [out=up, in=\dr] (0.5,1);
\draw[slice] (2,0) to [out=up, in=\dr] (1,2);
\node[dot] at (0.5,1){};
\node[dot] at (1,2){};
\node[obj,left] at (0,0) {$C$};
\node[obj,  left] at (1,0) {$B$};
\node[obj, left] at (2,0) {$A$};
\node[obj,  left] at (1, 3) {$F$};
\node[obj,  right] at (0.7, 1.6) {$E$};
\node[omor, right] at (0.5,1) {$k$};
\node[omor, right] at (1,2) {$h$};
\end{tz}
\]
For clarity, here we have also explicitly depicted the source and target of the 2-morphism; that source and target appear in the surface diagram as the bottom and top horizontal slices, respectively.  Note that we will often omit labels on regions, wires, and nodes, when it is clear from context what those labels should be.

\skiptocparagraph{Calculus of duality}

In a monoidal 2-category in which every object has left and right duals, we can extend the graphical calculus of monoidal 2-categories by introducing the following diagrammatic notation for particular choices of counit and unit 1-morphisms of the object dualities:
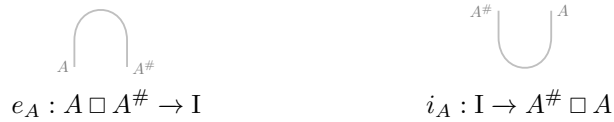
\begin{calign}\nonumber
\begin{tz}[scale=0.7]
\draw[slice] (0,0) to (0,0.5) to [out=up, in=up, looseness=2] (1,0.5) to (1,0);
\node[obj, left]  at (0,0) {$A$};
\node[obj, right]  at (1,0) {$A^\#$};
\end{tz}
&
\begin{tz}[yscale=-1,scale=0.7]
\draw[slice] (0,0) to (0,0.5) to [out=up, in=up, looseness=2] (1,0.5) to (1,0);
\node[obj, right]  at (1,0) {$A$};
\node[obj, left]  at (0,0) {$A^\#$};
\end{tz}
\\\nonumber
e_A:A\xz A^\# \to \Iz
&
i_A: \Iz \to A^\# \xz A
\end{calign}
We may furthermore depict particular choices of the 2-isomorphisms $(e\xz A) \xo  (A\xz i) \iso \Io_{A}$ and $\Io_{A^\#} \iso (A^\# \xz e) \xo (i \xz A^\# )$ by cusps:
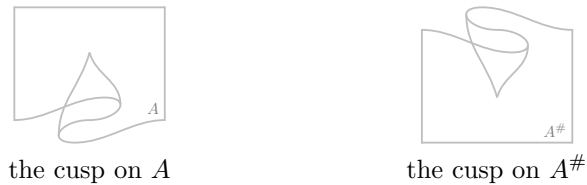
\begin{calign}\nonumber
\begin{tz}[td]
\begin{scope}[xyplane=0]
\draw[slice] (0,0) to[out=up, in=down] (-1,2) to [out=up, in=up, looseness=2] (0,2) to [out=down, in=down, looseness=2] (1,2) to [out=up, in=down](0,4);
\end{scope}
\begin{scope}[xyplane=\h]
\draw[slice] (0,0)to (0,4);
\end{scope}
\begin{scope}[xzplane=0]
\draw[slice,short] (0,0) to (0,\h);
\end{scope}
\begin{scope}[xzplane=4]
\draw[slice,short] (0,0) to (0,\h);
\end{scope}
\coordinate (cusp) at (2,0, 0.6*\h);
\draw[slice] (2.54, -0.7,0) to [out=up, in=\dlcusp] (cusp);
\draw[slice] (1.46, 0.7,0) to [out=up, in=\drcusp](cusp);
\node[obj,above left] at (0,0,0) {$A$};
\end{tz}
&
\begin{tz}[td,yscale=-1]
\begin{scope}[xyplane=0]
\draw[slice] (0,0) to[out=up, in=down] (-1,2) to [out=up, in=up, looseness=2] (0,2) to [out=down, in=down, looseness=2] (1,2) to [out=up, in=down](0,4);
\end{scope}
\begin{scope}[xyplane=\h]
\draw[slice] (0,0)to (0,4);
\end{scope}
\begin{scope}[xzplane=0]
\draw[slice,short] (0,0) to (0,\h);
\end{scope}
\begin{scope}[xzplane=4]
\draw[slice,short] (0,0) to (0,\h);
\end{scope}
\coordinate (cusp) at (2,0, 0.6*\h);
\draw[slice] (2.54, -0.7,0) to [out=up, in=\dlcusp] (cusp);
\draw[slice] (1.46, 0.7,0) to [out=up, in=\drcusp](cusp);
\node[obj,above left] at (0,0,\h) {$A^\#$};
\end{tz}
\\\nonumber
\text{the cusp on $A$}
&
\text{the cusp on $A^\#$}
\end{calign}
These cusp 2-isomorphisms may be chosen to satisfy the swallowtail equations (though that fact will not be of any particular relevance, as later in our definition of pivotal 2-category we explicitly assume choices of cusps satisfying the swallowtail equations):
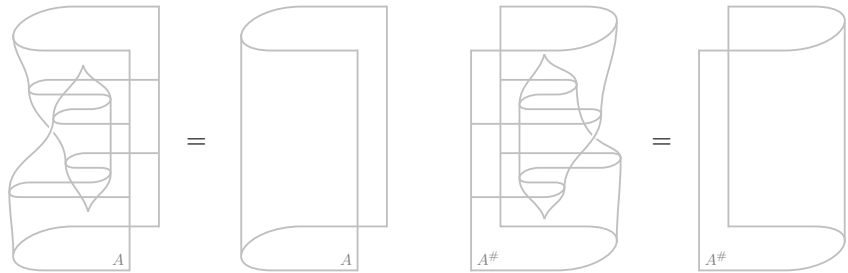
\begin{calign}\nonumber
\begin{tz}[td,scale=0.65]
\coordinate (cuspt) at (2.3, 2, 2.6*\h);
\coordinate (cuspb) at (2.1, 2, 0.6*\h);
\begin{scope}[xyplane=0]
\draw[slice] (1,0) to (1,3.5) to [out=up, in=up, looseness=2] (4, 3.5)to (4,0);
\end{scope}
\begin{scope}[xyplane=\h]
\draw[slice] (1,0) to (1,4.5) to [out=up, in=up, looseness=2] (2,4.5) to (2,2) to [out=down, in=down, looseness=2] (3,2) to (3,3) to [out=up, in=up, looseness=2] (4,3) to (4,0);
\end{scope}
\begin{scope}[xyplane=2*\h]
\draw[slice] (1,0) to (1,2.7) to [out=up, in=up, looseness=2] (2,2.7) to (2,2) to [out=down, in=down, looseness=2] (3,2) to (3,4.5) to [out=up, in=up, looseness=2] (4,4.5) to (4,0);
\end{scope}
\begin{scope}[xyplane=3*\h]
\draw[slice] (1,0) to (1,3.5) to [out=up, in=up, looseness=2] (4, 3.5)to (4,0);
\end{scope}
\begin{scope}[xzplane=0]
\draw[slice,short] (1,0) to (1,3*\h);
\draw[slice,short] (4,0) to (4,3*\h);
\end{scope}
\draw[slice] (5.198,2.1,0) to [out=up, in=down] (5.08,1.4,\h) to [out=up, in=down] node[mask point , pos=0.85](A){} (3.28,1.4,2*\h) to [out=up, in=\dlcusp] (cuspt);
\cliparoundone{A}{\draw[slice] (5.198, 2.1, 3*\h) to [out=down, in=up] (5.08, 3.4, 2*\h) to [out=down, in=up] (3.582, 3.4, \h) to [out=down, in=\ulcusp] (cuspb);}
\draw[slice] (cuspb) to [out=\urcusp, in=down]  (1.415,2.6,\h) to (1.415,2.6, 2*\h) to [out=up, in=\drcusp] (cuspt);
\node[obj, above left] at (0,1,0) {$A$};
\end{tz}
\hspace{0.25cm}=
\begin{tz}[td,scale=0.65]
\begin{scope}[xyplane=0]
\draw[slice] (1,0) to (1,3.5) to [out=up, in=up, looseness=2] (4, 3.5)to (4,0);
\end{scope}
\begin{scope}[xyplane=3*\h]
\draw[slice] (1,0) to (1,3.5) to [out=up, in=up, looseness=2] (4, 3.5)to (4,0);
\end{scope}
\begin{scope}[xzplane=0]
\draw[slice, short] (1,0) to (1,3*\h);
\draw[slice, short] (4,0) to (4,3*\h);
\end{scope}
\draw[slice](5.198,2.1,0) to  (5.198,2.1, 3*\h);
\node[obj, above left] at (0,1,0) {$A$};
\end{tz}
&
\begin{tz}[td,scale=0.65]
\coordinate (cuspt) at (-2.6, 2, -2.5*\h);
\coordinate (cuspb) at (-2., 3.5, -0.55*\h);
\begin{scope}[xyplane=0]
\draw[slice] (1,0) to (1,-3.5) to [out=down, in=down, looseness=2] node[pos=\stdl] (T){} (4, -3.5)to (4,0);
\end{scope}
\begin{scope}[xyplane=-\h]
\draw[slice] (1,0) to (1,-4.5) to [out=down, in=down, looseness=2] node[pos=\stdl] (TR){}(2,-4.5) to (2,-2) to [out=up, in=up, looseness=2] node[pos=\stdr] (TL){} (3,-2) to (3,-2.7) to [out=down, in=down, looseness=2] node[pos=\stdl] (TM){} (4,-2.7) to (4,0);
\end{scope}
\begin{scope}[xyplane=-2*\h]
\draw[slice] (1,0) to (1,-3) to [out=down, in=down, looseness=2] node[pos=\stdl] (BM){} (2,-3) to (2,-2) to [out=up, in=up, looseness=2] node[pos=\stdr] (BL){} (3,-2) to (3,-4.5) to [out=down, in=down, looseness=2] node[pos=\stdl] (BR){} (4,-4.5) to (4,0);
\end{scope}
\begin{scope}[xyplane=-3*\h]
\draw[slice] (1,0) to (1,-3.5) to [out=down, in=down, looseness=2] node[pos=\stdl] (B){}(4,-3.5)to (4,0);
\end{scope}
\begin{scope}[xzplane=0]
\draw[slice,short] (1,0) to (1,-3*\h);
\draw[slice,short] (4,0) to (4,-3*\h);
\end{scope}
\draw[slice] (T.center) to [out=down, in=up] (TR.center)  to [out=down, in=up] node[mask point , pos=0.2](A){}  (BM.center) to [out=down, in=\urcusp] (cuspt);
\cliparoundone{A}{\draw[slice] (B.center) to [out=up, in=down] (BR.center) to [out=up, in=down, out looseness=0.5, in looseness=2] (TM.center) to [out=up, in=\drcusp] (cuspb);}
\draw[slice] (cuspb) to [out=\dlcusp, in=up]  (TL.center) to (BL.center) to [out=down, in=\ulcusp] (cuspt);
\node[obj, above right] at (0,1,-3*\h) {$A^\#$};
\end{tz}
=\hspace{0.25cm}
\begin{tz}[td,scale=0.65]
\begin{scope}[xyplane=0]
\draw[slice] (1,0) to (1,-3.5) to [out=down, in=down, looseness=2] node[pos=\stdl] (B){} (4, -3.5)to (4,0);
\end{scope}
\begin{scope}[xyplane=3*\h]
\draw[slice] (1,0) to (1,-3.5) to [out=down, in=down, looseness=2]  node[pos=\stdl] (T){}(4, -3.5)to (4,0);
\end{scope}
\begin{scope}[xzplane=0]
\draw[slice, short] (1,0) to (1,3*\h);
\draw[slice, short] (4,0) to (4,3*\h);
\end{scope}
\draw[slice](B.center)  to  (T.center);
\node[obj, above right] at (0,1,0) {$A^\#$};
\end{tz}
\end{calign}
In these pictures, and henceforth, we use a tiny gap in a line specifically to indicate the presence of a categorical interchanger; crossings without a gap are merely coincidences of the depicted projection and do not signify a categorical operation.

\subsubsection{Examples of fusion $2$-categories} \label{sec:egfus}

Most naturally occurring monoidal 2-categories are not semistrict, that is are not literally of the form described in Definition~\ref{def:monoidal2cat}.  In the following examples, whenever we say that something is a monoidal, respectively fusion, 2-category, what we mean is that it is a (fully weak) 3-category with one object which after (always possible~\cite{GPS,Gurski}) strictification becomes a fusion or monoidal 2-category in the sense of Definition~\ref{def:monoidal2cat}, respectively Definition~\ref{def:fusion2cat}.

\skiptocparagraph{Representations of groups, and group-graded vector spaces}

We warm up by discussing examples of fusion 1-categories.  Of course, the semisimple 1-category $\Vect$ itself is naturally a fusion 1-category using the tensor product of vector spaces.  Recall the discussion of (monoidal) $n$-groupoids from Notation~\ref{notation:groupoids}.

\begin{recollection}[Group representations]\label{rec:rec1}
For a finite group $\pi_1$, the semisimple 1-category $\Rep(\ast,\pi_1) := [(\ast,\pi_1),\Vect]$ of $\pi_1$-representations has a (symmetric) monoidal structure inherited from the monoidal structure on $\Vect$, and with this structure, $\Rep(\ast,\pi_1)$ is a fusion 1-category.  We will denote this fusion category by $\Rep(\pi_1)$.

More generally, for a finite 1-groupoid $(\pi_0,\pi_1)$, the semisimple 1-category $\Rep(\pi_0,\pi_1) := [(\pi_0,\pi_1),\Vect]$ has a monoidal structure inherited from $\Vect$, but note that this is a multifusion rather than fusion structure, that is the unit need not be simple.  The notation $\Rep(\pi_0,\pi_1)$ will always refer either to the bare semisimple 1-category $[(\pi_0,\pi_1),\Vect]$ or to that category equipped with its symmetric multifusion structure.
\end{recollection}

\begin{recollection}[Group-graded vector spaces]\label{rec:rec2}
Given a finite group $\pi_1$ we may instead form the semisimple 1-category $\Rep(\pi_1,\ast) := [(\pi_1,\ast),\Vect]$ of $\pi_1$-graded vector spaces.  This category has a monoidal structure induced by the group multiplication: the product is the (Day) convolution of functors $(\pi_1,\ast) \to \Vect$, or more concretely, the product of the functor $F(x) = \delta_{x,f} k$ and the functor $G(x) = \delta_{x,g} k$ is the functor $(F \ast G)(x) = \delta_{x,fg} k$.  We will denote this fusion category, as before, by $\Vect(\pi_1)$.

More generally, for a finite 2-group $(\pi_1,\pi_2)$, we could consider the semisimple 1-category $\Rep(\pi_1,\pi_2) := [(\pi_1,\pi_2),\Vect]$ with its convolution product; this is a monoidal semisimple 1-category, which will be denoted $\Vect(\pi_1,\pi_2)$, but in general it is multifusion rather than fusion.
\end{recollection}

Recollections~\ref{rec:rec1} and~\ref{rec:rec2} can be summarized in the following table:
\[\hspace{-0.4cm}
\shrinkalign{.95}{
\begin{tabular}{rrcll}
\multicolumn{1}{c}{{Input}} & \multicolumn{1}{c}{{Notation}} & & \multicolumn{1}{c}{{1-category}} & \multicolumn{1}{c}{{Monoidal structure}} 
\\ \hline
{1-groupoid} $(\pi_0, \pi_1)$ & $\Rep(\pi_0,\pi_1)$ & $:=$ & $[(\pi_0,\pi_1), \Vect]$ & {symmetric from $\Vect$ (multifusion)}
\\
{1-group} $\pi_1$ & $\Rep(\pi_1)$ & $:=$ & $[(\ast,\pi_1), \Vect]$ & {symmetric from $\Vect$ (fusion)}
\\ 
{2-group} $(\pi_1,\pi_2)$ & $\Vect(\pi_1,\pi_2)$ & $:=$ & $[(\pi_1, \pi_2), \Vect]$ & {convolution product (multifusion)}
\\
{1-group} $\pi_1$ & $\Vect(\pi_1)$ & $:=$ & $[(\pi_1, \ast), \Vect]$ & {convolution product (fusion)}
\end{tabular}
}
\]
Since every $2$-group is in particular a $1$-groupoid, the category $[(\pi_1, \pi_2), \Vect]$ has two distinct monoidal structures.  (These structures are compatible in the sense that the symmetric monoidal structure together with a comonoidal structure associated to the convolution product give $[(\pi_1,\pi_2),\Vect]$ the structure of a Hopf $1$-category.)

\begin{recollection}[Twisted group-graded vector spaces]
Given again a finite group $\pi_1$ and a 3-cocycle $w \in Z^3(\pi_1,k^*)$, we may twist the associator of the fusion category $\Vect(\pi_1)$ to obtain a new fusion category denoted $\Vect^w(\pi_1)$.  Note here we may think of the cocycle either as a `group-cohomology-style' cocycle for the ordinary group $\pi_1$ or as a topological cocycle on the space corresponding to the group, namely $\B\pi_1 = (\ast,\pi_1)$.
\end{recollection}

\skiptocparagraph{2-representations of 2-groups, and 2-group-graded 2-vector spaces}

The semisimple 2-category $\tVect$ is of course the canonical fusion 2-category, where the tensor of two finite semisimple linear 1-categories is the Deligne tensor product.  This fusion 2-category has a unique equivalence class of simple objects, represented by the 1-category $\Vect$.

\begin{construction}[2-group 2-representations]\label{con:2group2rep}
Given a finite abelian group $\pi_2$, the semisimple 2-category $\tRep(\ast,\ast,\pi_2) := [(\ast,\ast,\pi_2),\tVect]$ of $\pi_2$-2-representations inherits a (symmetric) monoidal structure from $\tVect$ and with that structure is a fusion 2-category.  We will denote this fusion 2-category by $\tRep(\pi_2)$.

Similarly, given a finite 2-group $(\pi_1,\pi_2)$, the semisimple 2-category $\tRep(\ast,\pi_1,\pi_2) := [(\ast,\pi_1,\pi_2),\tVect]$ of $(\pi_1,\pi_2)$-2-representations inherits a monoidal structure from $\tVect$, and this structure is again fusion.  We will denote this fusion 2-category by $\tRep(\pi_1,\pi_2)$.

More generally, given a finite 2-groupoid $(\pi_0,\pi_1,\pi_2)$, the semisimple 2-category \linebreak $\tRep(\pi_0,\pi_1,\pi_2) := [(\pi_0,\pi_1,\pi_2),\tVect]$ has a monoidal structure, but it is in general multifusion rather than fusion, that is the unit need not be simple.  The notation $\tRep(\pi_0,\pi_1,\pi_2)$ will always refer either to the bare semisimple 2-category $[(\pi_0,\pi_1,\pi_2),\tVect]$ or to that 2-category equipped with its symmetric multifusion structure.
\end{construction}

\begin{construction}[2-group-graded 2-vector spaces]\label{con:2groupgraded2vect}
Given a finite group $\pi_1$, the semisimple 2-category $\tRep(\pi_1,\ast,\ast) :=  [(\pi_1,\ast,\ast),\tVect]$ of $\pi_1$-graded 2-vector spaces has a monoidal structure induced by the group multiplication, namely the convolution product of 2-functors $(\pi_1,\ast,\ast) \to \tVect$, and is thereby a fusion 2-category.  We will denote this fusion 2-category by $\tVect(\pi_1)$.

Similarly, for a finite 2-group $(\pi_1,\pi_2)$, the semisimple 2-category $\tRep(\pi_1,\pi_2,\ast) :=  [(\pi_1,\pi_2,\ast),\tVect]$ of $(\pi_1,\pi_2)$-graded 2-vector spaces has a convolution product, and again is a fusion 2-category.  We will of course denote this fusion 2-category by $\tVect(\pi_1,\pi_2)$.

More generally, given a finite 3-group $(\pi_1,\pi_2,\pi_3)$, the semisimple 2-category \linebreak $\tRep(\pi_1,\pi_2,\pi_3) :=  [(\pi_1,\pi_2,\pi_3),\tVect]$ of $(\pi_1,\pi_2,\pi_3)$-2-representations has a convolution monoidal structure; this monoidal semisimple 2-category will be denoted $\tVect(\pi_1,\pi_2,\pi_3)$, but note it is in general multifusion rather than fusion.
\end{construction}

\begin{remark}[Bimodule equivalence of 2-representations and graded 2-vector spaces]
The fusion 1-category $\Rep(\pi_1)$ of $\pi_1$-representations and the fusion 1-category $\Vect(\pi_1)$ of $\pi_1$-graded vector spaces are bimodule equivalent, and therefore lead to the same 3-manifold invariant.  We suspect the 2-categorical situation is analogous, in that the fusion 2-category $\Rep(\pi_1,\pi_2)$ of $(\pi_1,\pi_2)$-2-representations is bimodule equivalent to the fusion 2-category $\tVect(\pi_1,\pi_2)$ of $(\pi_1,\pi_2)$-graded 2-vector spaces, and that these fusion 2-categories therefore produce the same 4-manifold invariant.
\end{remark}

Constructions~\ref{con:2group2rep} and~\ref{con:2groupgraded2vect} can be summarized in the following table:
\[
\hspace*{-.85cm}
\shrinkalign{.92}{
\begin{tabular}{rrcll}
\multicolumn{1}{c}{{Input}} & \multicolumn{1}{c}{{Notation}} & & \multicolumn{1}{c}{{2-category}} & \multicolumn{1}{c}{{Monoidal structure}} 
\\ \hline
{2-grpoid} $(\pi_0, \pi_1,\pi_2)$ & $\tRep(\pi_0,\pi_1,\pi_2)$ & $:=$ & $[(\pi_0,\pi_1,\pi_2), \tVect]$ & {symm from $\tVect$ (multifus)}
\\
{2-group} $(\pi_1,\pi_2)$ & $\tRep(\pi_1,\pi_2)$ & $:=$ & $[(\ast,\pi_1,\pi_2), \tVect]$ & {symm from $\tVect$ (fusion)}
\\
{ab 1-group} $(\pi_2)$ & $\tRep(\pi_2)$ & $:=$ & $[(\ast,\ast,\pi_2), \tVect]$ & {symmfrom $\tVect$ (fusion)}
\\ 
{3-group} $(\pi_1, \pi_2,\pi_3)$ & $\tVect(\pi_1,\pi_2,\pi_3)$ & $:=$ & $[(\pi_1,\pi_2,\pi_3), \tVect]$ & {conv product (mulitfus)}
\\
{2-group} $(\pi_1, \pi_2)$ & $\tVect(\pi_1,\pi_2)$ & $:=$ & $[(\pi_1,\pi_2,\ast), \tVect]$ & {conv product (fusion)}
\\
{1-group} $(\pi_1)$ & $\tVect(\pi_1)$ & $:=$ & $[(\pi_1,\ast,\ast), \tVect]$ & {conv product (fusion)}
\end{tabular}
}
\]
Since every $3$-group is in particular a $2$-groupoid, the 2-category $[(\pi_1, \pi_2,\pi_3), \tVect]$ has two distinct monoidal structures, one (symmetric) structure coming from the product on $\tVect$ and one not-necessarily symmetric convolution structure coming from the $3$-group itself. (We expect these structure to be compatible, in that the symmetric fusion structure together with a comonoidal structure associated to the convolution will form a Hopf $2$-category.)

\begin{remark}[The convolution product is the completion of group multiplication] Recall from Example~\ref{eg:rep2groupoid} that the finite semisimple $2$-category $[(\pi_0, \pi_1,\pi_2), \tVect]$ is the semisimple completion of the linearization $k(\pi_0, \pi_1,\pi_2)$ of the $2$-groupoid $(\pi_0, \pi_1,\pi_2)$.  If the $2$-groupoid is a $3$-group $(\pi_1,\pi_2,\pi_3)$, then $k(\pi_1,\pi_2,\pi_3)$, and hence its semisimple completion $[(\pi_1,\pi_2,\pi_3),\tVect]$, inherits a monoidal structure from the monoidal structure of $(\pi_1,\pi_2,\pi_3)$ --- the resulting monoidal structure is the convolution product on $\tVect(\pi_1,\pi_2,\pi_3)$.
\end{remark}

\begin{construction}[Twisted 2-group-graded 2-vector spaces] \label{con:2groupgraded2vecttwisted}
Given a finite 2-group $(\pi_1, \pi_2)$, and a 4-cocycle $\omega \in Z^4((\pi_1,\pi_2);k^*)$, one can form a fusion 2-category $\tVect^\omega(\pi_1,\pi_2)$ of `$\omega$-twisted $(\pi_1,\pi_2)$-graded 2-vector spaces', as follows.  (Here $Z^4((\pi_1,\pi_2);k^*)$ is the topological 4-cocycles, with $k^*$ coefficients, of the space $(\ast,\pi_1,\pi_2)$ with only first and second homotopy groups.)

The 4-cocycle $\omega \in Z^4((\pi_1,\pi_2);k^*)$ provides the $k$-invariant for an extension of the 2-group $(\pi_1,\pi_2)$ to a 3-group $(\pi_1,\pi_2,k^*)$, with trivial action of $\pi_1$ on $\pi_3 = k^*$.  Because the $\pi_1$-action on $\pi_3$ is trivial, we may think of this 3-group as a monoidal 2-category enriched in $k^*$-sets.  (Here the enriching category of $k^*$-sets has tensor product $X \times_{k^*} Y := \textrm{coeq}(X \times k^* \times Y \rightrightarrows X \times Y)$.)  Base changing from $k^*$ to $k$ produces a $k$-linear monoidal 2-category denoted $(\pi_1,\pi_2,k^*)_k$.  Precisely, this operation is base change along the functor from $k^*$-sets to (possibly infinite-dimensional) $k$-vector spaces, taking a $k^*$-set $X$ to the $k$-vector space $k \otimes_{k(k^*)} k(X)$, where $k(X)$ is the free $k$-vector space on the set $X$.  Note that the $k$-linear monoidal 2-category $(\pi_1,\pi_2,k^*)_k$ has finitely many equivalence classes of objects, all invertible, and finitely many isomorphism classes of 1-morphisms, all invertible, and that $\Hom_{(\pi_1,\pi_2,k^*)_k}(f,g)$ is $k$ when $f$ and $g$ are isomorphic and $0$ otherwise. In particular, $(\pi_1,\pi_2,k^*)_k$ is a locally presemisimple linear monoidal 2-category.

The fusion 2-category $\tVect^\omega(\pi_1,\pi_2)$ is defined to be the semisimple completion of the local Cauchy completion of the monoidal 2-category $(\pi_1,\pi_2,k^*)_k$.  That is, $\tVect^\omega(\pi_1,\pi_2)$ is a completion of a twisted linearization of $(\pi_1,\pi_2)$.  When the twisting is trivial, the construction recovers $\tVect(\pi_1,\pi_2)$ by the discussion in Example~\ref{eg:rep2groupoid}.
\end{construction}

\begin{remark}[Characterizing twisted 2-group-graded 2-vector spaces]
A fusion 1-category $\oc{C}$ in which every simple object is invertible is necessarily equivalent to the fusion 1-category $\Vect^\omega(\pi_1)$ for some finite group $\pi_1$ and 3-cocycle $\omega \in Z^3(\pi_1; k^*)$; the group $\pi_1$ is determined as the group of isomorphism classes of simple objects of $\oc{C}$.  And of course every simple object of $\Vect^\omega(\pi_1)$ is invertible.

The situation for fusion 2-categories is more complicated.  The fusion 2-category \linebreak $\tVect^\omega(\pi_1,\pi_2)$, for a finite 2-group $(\pi_1,\pi_2)$ and 4-cocycle $\omega \in Z^4((\pi_1,\pi_2);k^*)$, can have non-invertible simple objects and non-invertible simple 1-morphisms.  For instance, the fusion 2-category $\tVect(\ast,\ZZ_2)$ has both non-invertible simple objects and 1-morphisms, as described below in Examples~\ref{eg:fusionVectZ2} and~\ref{eg:2vep} (also see Example~\ref{eg:VectZ2}).

Even if a fusion 2-category has only invertible simple objects (so there is an obvious `group of simple objects' $\pi_1$) and every simple 1-endomorphism of the unit object $\Iz$ is invertible (so there is an obvious `group of simple 1-endomorphisms' $\pi_2$), it is still not necessarily the case that the fusion 2-category is equivalent to $\tVect^\omega(\pi_1,\pi_2)$ for some $(\pi_1,\pi_2)$ and $\omega$.  An example of such a fusion 2-category (that is not twisted 2-group graded 2-vector spaces) is provided by the completion of a $\ZZ_4$-crossed braided structure on the Ising category, see below Example~\ref{eg:Z4example}.

Nevertheless, it is still possible to characterize twisted 2-group-graded 2-vector spaces as follows: a fusion 2-category $\tc{C}$ is monoidally equivalent to $\tVect^\omega(\pi_1,\pi_2)$ for some finite 2-group $(\pi_1,\pi_2)$ and 4-cocycle $\omega \in Z^4((\pi_1,\pi_2);k^*)$ if and only if
\begin{enumerate}
\item every component of $\tc{C}$ contains an invertible object;
\item every simple 1-morphism in $\End_{\tc{C}}(\Iz)$ is invertible;
\item the group homomorphism $\tc{C}^\times \to \pi_0 \tc{C}$, from the group of equivalence classes of invertible objects to the group of components, admits a section.
\end{enumerate}
(Note that the existence of a group structure on the set of components depends on the existence of an invertible object in each component.)  Here the groups $\pi_1$ and $\pi_2$ can be taken to be the group of components $\pi_0 \tc{C}$ and the group of isomorphism classes of simple 1-morphisms in $\End_{\tc{C}}(\Iz)$, respectively.  The aforementioned $\ZZ_4$-crossed braided Ising category fails the third condition as in that case the group of invertible objects $\ZZ_4$ projects onto the group of components $\ZZ_2$.
\end{remark}

\begin{construction}[Abelian-group 2-vepresentations]
Given an abelian group $\pi_2$, the semisimple 2-category of `$\pi_2$-2-vepresentations' $\tRep(\ast,\pi_2,\ast) := [(\ast,\pi_2,\ast),\tVect]$ has both a symmetric and a convolution structure, and both products are fusion, not merely multifusion.  We therefore expect this construction provides in an appropriate sense `fusion Hopf 2-categories'.  We reserve the notation $\tVep(\pi_2)$ for the semisimple 2-category $[(\ast,\pi_2,\ast),\tVect]$ equipped simultaneously with both monoidal structures; by contrast we would use $\tRep(\pi_2,\ast)$ when thinking only of the symmetric fusion structure and $\tVect(\ast,\pi_2)$ when thinking only of the convolution fusion structure.
\end{construction}

\skiptocparagraph{Braided fusion categories and their modules}

\begin{construction}[Fusion 2-categories from braided fusion categories] \label{con:braidedisftc}
The delooping of a braided fusion category $\oc{C}$ is a prefusion 2-category $\B\oc{C}$: the delooping of the underlying fusion category is a presemisimple 2-category, as in Example~\ref{eg:deloopfusion}, and the braiding provides precisely the data of a monoidal structure on the delooping.  By Proposition~\ref{prop:completionismod} the idempotent completion of the delooping $\B\oc{C}$ is the semisimple 2-category of modules $\Mod(\oc{C})$.  The monoidal structure on the delooping $\B\oc{C}$ induces a monoidal structure on the completion, making $\Mod(\oc{C})$ a fusion 2-category.  Module categories for $\oc{C}$ correspond to separable algebras in $\oc{C}$, and the monoidal structure is the tensor product of algebras (which, note well, depends on the braiding of $\oc{C}$).  Directly as $\oc{C}$-module categories, the monoidal product is given by the relative Deligne tensor over $\oc{C}$, cf~\cite{DBTC}.
\end{construction}

\begin{example}[Fusion structures on $\Mod(\Vect(\ZZ_2))$]\label{eg:fusionVectZ2}
Recall from Example~\ref{eg:VectZ2} the structure of the semisimple 2-category $\Mod(\Vect(\ZZ_2))$: there is the simple object $\Vect(\ZZ_2)$---which we will now abbreviate as $\Iz$---and the simple object $\Vect$---which we will now abbreviate as $X$---with morphism categories $\Hom(\Iz,\Iz) \equiv \Hom(X,X) \equiv \Vect(\ZZ_2)$ and $\Hom(X,\Iz) \equiv \Hom(\Iz,X) \equiv \Vect$.

There are two braidings on the fusion category $\Vect(\ZZ_2)$, namely the symmetric braiding and the super braiding, and thus two corresponding fusion structures on the 2-category $\Mod(\Vect(\ZZ_2))$.  We may calculate the fusion rules for each as follows.  An object of $\Mod(\Vect(\ZZ_2))$ corresponds to a separable algebra in $\Vect(\ZZ_2)$.  (This correspondence may either be seen as the classical Ostrik translation or as the fact that $\Mod(\Vect(\ZZ_2))$ is the idempotent---i.e.\ separable algebra---completion of the deloop $\B\Vect(\ZZ_2)$.)  The object $\Iz$ corresponds to the trivial algebra $k \in \Vect(\ZZ_2)$ and the object $X$ corresponds to the `graded group algebra' $k(\ZZ_2) \in \Vect(\ZZ_2)$.  The monoidal product of objects here is, as mentioned in Construction~\ref{con:braidedisftc}, the tensor product of algebras inside the braided fusion category.  Thus $\Iz$ is evidently the identity in the fusion 2-category (for either braiding).  

Now suppose the base field is $\CC$ and the braiding is super.  In this case the tensor product $\CC(\ZZ_2) \otimes \CC(\ZZ_2)$ is, by complex `Bott periodicity', Morita equivalent to the trivial algebra $\CC$.  The corresponding fusion rule in the fusion 2-category is therefore $X \xz X \equiv \Iz$.  We may depict the fusion rules of this 2-category as follows, where the directed edge indicates multiplication by $X$:
\[\begin{tz}
\node[dot,scale=1.25] at (0,0){};
\node[left] at (0,0) {$I$};
\node[dot,scale=1.25] at (1.5,0){};
\node[right] at (1.5,0) {$X$};
\draw[arrow data={0.5}{>}] (0,0) to [out=35, in=145] (1.5,0);
\draw[arrow data={0.5}{<}] (0,0) to [out=-35, in=-145]  (1.5,0);
\end{tz}
\]

For the symmetric braiding, the tensor product $\CC(\ZZ_2) \otimes \CC(\ZZ_2)$ is, by idempotent decomposing that algebra in the symmetric tensor category $\Vect(\ZZ_2)$, isomorphic to the algebra $\CC(\ZZ_2) \oplus \CC(\ZZ_2)$.  The corresponding fusion rule is therefore $X \xz X \equiv X \boxplus X$.  We may depict the fusion rules of this 2-category as follows, where again the directed edge is multiplication by $X$ and the label indicates the multiplicity:
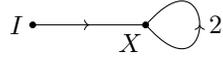
\[\begin{tz}
\node[dot,scale=1.25] at (0,0){};
\node[left] at (0,0) {$I$};
\node[dot,scale=1.25] at (1.5,0){};
\node[below] at (1.3,0) {$X$};
\draw[arrow data={0.5}{>}] (0,0) to (1.5,0);
\draw[arrow data ={0.5}{>}] (1.5,0) to  (1.6,-0.1) to [out=-45, in=45, looseness=15] node[right] {$2$} (1.6,0.1) to (1.5,0);
\end{tz}
\] 
Note well that this sort of fusion graph, where an object has no fusion products containing identity factors, is completely new to fusion 2-categories: in a fusion 1-category, the product of an object and its dual has the identity as a summand, but in a fusion 2-category, because of the existence of nontrivial nonequivalence morphisms between simple objects, this need not be the case.
\end{example}

\begin{example}[Fusion structures on $\Mod(\Vect(\ZZ_2))$ as twisted $2$-vepresentations] \label{eg:2vep}
Recall from Remark~\ref{rem:VepisVect} that the semisimple 2-category $\Mod(\Vect(\ZZ_2))$ of module categories for $\Vect(\ZZ_2)$ is equivalent to the 2-category $\tRep(\ast,\ZZ_2,\ast) := [(\ast,\ZZ_2,\ast),\tVect]$ of 2-vepresentations of $\ZZ_2$.  The monoidal structure on $\Mod(\Vect(\ZZ_2))$ induced by the standard braiding (see Example~\ref{eg:fusionVectZ2}) corresponds to the convolution product on $\tRep(\ast,\ZZ_2,\ast)$, therefore gives the fusion 2-category $\tVect(\ast,\ZZ_2)$ of $(\ast,\ZZ_2)$-graded 2-vector spaces (see Construction~\ref{con:2groupgraded2vect}).  The monoidal structure on $\Mod(\Vect(\ZZ_2))$ induced, by contrast, by the super braiding corresponds to the convolution product on $\tRep(\ast,\ZZ_2,\ast)$ twisted by a nontrivial 4-cocycle $\omega \in Z^4((\ast, \ZZ_2); k^*)$, thus to the twisted 2-group-graded 2-category $\tVect^\omega(\ast,\ZZ_2)$.  Specifically, the super braiding fusion structure is obtained by twisting by the cocycle representing the order 2-element in $H^4((\ast, \ZZ_2); k^*) \cong \ZZ_4$.  (Note more generally that for an abelian group $\pi_2$, the group $Z^4((\ast,\pi_2); k^*)$ twisting the fusion structure on $\tVect(\ast,\pi_2) \simeq \Mod(\Vect(\pi_2))$ is the same as the group of `abelian 3-cocycles' $Z^3_{\mathrm{ab}}(\pi_2, k^*)$~\cite[Sec 8.4]{EGNO} that simultaneously twists the associator and the braiding of the fusion 1-category $\Vect(\pi_2)$.)
\end{example}

\begin{remark}[One-component fusion $2$-categories] \label{rem:onecomponent}
Any fusion 2-category $\tc{C}$ with only one component is (monoidally 2-functor) equivalent to the 2-category $\Mod(\oc{C})$ of modules of a braided fusion category $\oc{C}$; more specifically it is equivalent to the 2-category of modules of the braided fusion category $\Hom_{\tc{C}}(\Iz,\Iz)$ of endomorphisms of the identity of $\tc{C}$.  Indeed, any one-component finite semisimple 2-category is the 2-category of modules of any one of its fusion endocategories, so as a semisimple 2-category, the fusion 2-category $\tc{C}$ is equivalent to $\Mod(\Hom_{\tc{C}}(\Iz,\Iz))$. But using Corollary~\ref{cor:extensionmultifusion} (applied to functors $\Mod(\oc{C} \boxtimes \oc{C}) \to \Mod(\oc{C})$, with $\oc{C} = \Hom_{\tc{C}}(\Iz,\Iz)$), the monoidal structure on $\Mod(\Hom_{\tc{C}}(\Iz,\Iz))$ is completely determined by the monoidal structure on $\B\Hom_{\tc{C}}(\Iz, \Iz)$, that is by the braiding on $\Hom_{\tc{C}}(\Iz, \Iz)$. 
\end{remark}

\skiptocparagraph{Crossed-braided fusion categories}

\begin{construction}[From crossed-braided fusion categories to fusion $2$-categories] \label{con:gradedbraidedfusion}
Recall that a $G$-crossed-braided fusion category is a fusion category $\oc{C}$, together with a $G$-grading $\oc{C} = \bigoplus_{g \in G} \oc{C}_g$, a $G$-action on $\oc{C}$ where $g \in G$ maps $\oc{C}_h$ to $\oc{C}_{ghg^-1}$, and a compatible crossed-braiding isomorphism $X \otimes Y \to g(Y) \otimes X$ whenever $X \in \oc{C}_g$.  To any $G$-crossed-braided fusion category $\oc{C}$, there is an associated monoidal 2-category $\tc{C}$, as follows~\cite[Sec 6]{Cui}.  (To obtain a semistrict monoidal 2-category, for convenience we will start with a strict crossed-braided fusion category~\cite[Def 4.41]{DGNO}.)

The objects of $\tc{C}$ are the elements $g \in G$, and these objects are all simple.  The hom category $\Hom_{\tc{C}}(g,h)$ is the category $C_{h g^{-1}}$.  The composition of 1-morphisms in $\tc{C}$ is given by the tensor product in $\oc{C}$:
\begin{align*}
\Hom_{\tc{C}}(g,h) \times \Hom_{\tc{C}}(f,g) &\xrightarrow{\xo} \Hom_{\tc{C}}(f,h) \\
\oc{C}_{hg^{-1}} \times \oc{C}_{gf^{-1}} &\xrightarrow{\otimes} \oc{C}_{hf^{-1}}
\end{align*}
The monoidal product $g \xz h$ of objects $g,h \in \tc{C}$ is the object $gh$, and the monoidal unit is the identity element $e \in G$.  The monoidal product of an object $g$ on the left on 1- and 2-morphisms in $\tc{C}$ is given by the action of $g$ in $\oc{C}$:
\begin{align*}
\Hom_{\tc{C}}(h,h') &\xrightarrow{g\xz - } \Hom_{\tc{C}}(g\xz h, g \xz h') \\
\oc{C}_{h' h^{-1}} &\xrightarrow{g(-)} \oc{C}_{gh'h^{-1} g^{-1}} = \oc{C}_{gh'(gh)^{-1}}
\end{align*}
By contrast, the monoidal product of an object $g$ on the right in $\tc{C}$ is given by the identity operation in $\oc{C}$:
\begin{align*}
\Hom_{\tc{C}}(h,h') &\xrightarrow{-\xz g} \Hom_{\tc{C}}(h\xz g,  h' \xz g) \\
\oc{C}_{h' h^{-1}} &\xrightarrow{\mathrm{id}} \oc{C}_{h'h^{-1}} = \oc{C}_{h'g(hg)^{-1}}
\end{align*}
Finally, the interchanger 2-isomorphism of $\tc{C}$ is given by the crossed-braiding of $\oc{C}$; for 1-morphisms $X: g_1 \to g_2$, $Y: h_1 \to h_2$:
\begin{align*}
(X\xz h_2) \xo ( g_1 \xz Y) &\xRightarrow{\phi_{X,Y}} (g_2 \xz Y) \xo (X \xz h_1) \\
X\otimes g_1(Y) &\xrightarrow{c_{X, g_1(Y)}} g_2(Y) \otimes X
\end{align*}

This construction defines a prefusion 2-category; the completion is therefore a fusion 2-category, as desired.  If the grading is faithful, that is $\oc{C}_g \neq 0$ for all $g \in G$, then the resulting fusion 2-category has only one component, and is therefore equivalent (see Remark~\ref{rem:onecomponent}) to the completion of the deloop $\B\oc{C}_e$ of the braided fusion category $\oc{C}_e$.
\end{construction}

\begin{remark}[Inequivalent crossed-braided fusion categories giving equivalent fusion 2-categories]
Let $\oc{C}$ be an Ising fusion category, that is one with simple objects $I,f,\sigma$ and fusion rules $f^2 \cong I, f\sigma \cong \sigma f \cong \sigma$, and $\sigma^2 \cong I \oplus f$.  Equip this category with the $\ZZ_2$-grading $\oc{C}_0 = \Vect(\ZZ_2) = \langle I, f \rangle $ and $C_1 = \Vect = \langle \sigma \rangle$, and the trivial $\ZZ_2$-action.  An Ising category admits four inequivalent braidings~\cite[App B]{DGNO}, all of which restrict to the super braiding on $\Vect(\ZZ_2)$.  Any of these braidings makes the given Ising category $\oc{C}$ into a $\ZZ_2$-crossed-braided fusion category, and there is thus an associated fusion 2-category.  However, because the $\ZZ_2$-grading is faithful, each of these fusion 2-categories is equivalent to the completion of the deloop $\B\oc{C}_0 = \B\Vect(\ZZ_2)$ of the super braided category $\Vect(\ZZ_2)$, and thus to the fusion 2-category of modules of the super braided category $\Vect(\ZZ_2)$.
\end{remark}

\begin{remark}[Invertible-object fusion 2-categories are not necessarily crossed-braided]
It is not the case that a fusion 2-category all of whose simple objects are invertible is equivalent to a fusion 2-category associated to a $G$-crossed-braided category.  Given a fusion 2-category $\tc{C}$ with invertible simple objects, the equivalence classes of simple objects form a finite group $G$, and (picking representative simple objects $g$ in each equivalence class $[g] \in G$) the semisimple 1-category $\oc{C}:= \bigoplus_{g\in G} \Hom_{\tc{C}}(e, g)$ appears to want to be a $G$-crossed-braided category---but in general it is not possible to put an appropriate monoidal structure on that category.  Specifically, choose equivalences $\psi(g,h): g\xz h \to gh$ and 2-isomorphisms $\alpha(g_1,g_2,g_3) :\psi(g_1 g_2,g_3)\xo (\psi(g_1,g_2) \xz g_3) \To \psi(g_1,g_2g_3) \xo (g_1\xz \psi(g_2,g_3))$; these isomorphisms $\alpha$ will satisfy the pentagon equations only up to some scalars $\omega(g_1,g_2,g_3,g_4)\in k^*$, and those scalars define a 4-cocycle $\omega \in Z^4(G;k^*)$.  Only if that 4-cocycle is cohomologically trivial, is it possible to (adjust the choices of representing objects, equivalences, and 2-isomorphisms and then) define a multiplication giving $\oc{C}$ the structure of a $G$-crossed-braided fusion category.  (Note that this process does not produce a unique crossed-braided category, but a collection of such categories all lifting the same fusion 2-category.)  If that 4-cocycle is cohomologically nontrivial, then the fusion 2-category encodes a kind of twisting not presentable in the framework of crossed-braided categories.
\end{remark}

\begin{remark}[Endotrivial fusion 2-categories] \label{rem:endotrivialexamples}
A fusion 2-category is called \emph{endotrivial} if the endomorphism fusion category of every indecomposable object is the trivial fusion category $\Vect$.  Endotrivial fusion 2-categories were called simply `fusion 2-categories' in Mackaay~\cite{Mackaay}.  Note that of all the examples of fusion 2-categories described above, including those coming from 2-representations of 2-groups (Construction~\ref{con:2group2rep}), 2-group-graded 2-vector spaces (Construction~\ref{con:2groupgraded2vect}), braided fusion categories (Construction~\ref{con:braidedisftc}), and crossed-braided fusion categories (Construction~\ref{con:gradedbraidedfusion}), the only case that is endotrivial is the special case of 2-group-graded 2-vector spaces where the grading is in fact by a 1-group.
\end{remark}

\begin{example}[$\ZZ_4$-crossed braided Ising categories] \label{eg:Z4example}
For any Ising fusion category $\oc{C}$, the $\ZZ_4$-grading $\oc{C}_0 = \Vect(\ZZ_2) = \langle I, f\rangle$, $\oc{C}_2 = \Vect = \langle \sigma \rangle$, and $\oc{C}_1=\oc{C}_3= 0$, together with the trivial $\ZZ_4$-action and any choice of braiding, gives $\oc{C}$ the structure of a $\ZZ_4$-crossed-braided fusion category.  The associated fusion 2-category may be depicted as follows:
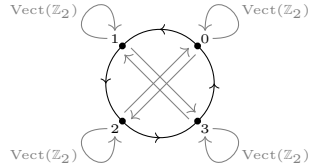
\[
\begin{tz}
\node[dot,scale=1.25] (A) at (.5,.5){};
\node[dot,scale=1.25] (B) at (-.5,.5){};
\node[dot,scale=1.25] (C) at (-.5,-.5){};
\node[dot,scale=1.25] (D) at (.5,-.5){};
\node at (.6,.6) {$\scriptscriptstyle 0$};
\node at (-.6,.6) {$\scriptscriptstyle 1$};
\node at (-.6,-.6) {$\scriptscriptstyle 2$};
\node at (.6,-.6) {$\scriptscriptstyle 3$};
\draw[arrow data={0.5}{>}] (A) to [out=135, in=45] (B);
\draw[arrow data={0.5}{>}] (B) to [out=-135, in=135] (C);
\draw[arrow data={0.5}{>}] (C) to [out=-45, in=-135] (D);
\draw[arrow data={0.5}{>}] (D) to [out=45, in=-45] (A);
\draw[->,shorten <=0.1cm, shorten >=0.1cm, morgray, xshift=1.25, yshift=-1.25] (.5,.5) to node[below] {} (-.5,-.5);
\draw[->,shorten <=0.1cm, shorten >=0.1cm, morgray, xshift=-1.25, yshift=1.25] (-.5,-.5) to node[below] {} (.5,.5);
\draw[->,shorten <=0.1cm, shorten >=0.1cm, morgray, xshift=1.25, yshift=1.25] (-.5,.5) to node[below] {} (.5,-.5);
\draw[->,shorten <=0.1cm, shorten >=0.1cm, morgray, xshift=-1.25, yshift=-1.25] (.5,-.5) to node[below] {} (-.5,.5);
\draw[->,morgray] (.7,.6) to [out=0,in=90,looseness=15] node[right] {$\scriptscriptstyle \Vect(\ZZ_2)$} (.6,.7);
\draw[<-,morgray] (-.6,.7) to [out=90,in=180,looseness=15] node[left] {$\scriptscriptstyle \Vect(\ZZ_2)$} (-.7,.6);
\draw[->,morgray] (-.7,-.6) to [out=180,in=-90,looseness=15] node[left] {$\scriptscriptstyle \Vect(\ZZ_2)$} (-.6,-.7);
\draw[<-,morgray] (.6,-.7) to [out=-90,in=0,looseness=15] node[right] {$\scriptscriptstyle \Vect(\ZZ_2)$} (.7,-.6);
\end{tz}
\]
Here the nodes indicate the four simple objects.  The gray lines show the structure of the underlying semisimple 2-category, where the unlabelled arrows denote the morphism space $\Vect$.  The directed black lines indicate the fusion structure of multiplication by the generating simple object $1$.
\end{example}

\begin{remark}[Fusion 2-categories must allow morphisms between inequivalent simple objects] 
By Theorem~\ref{thm:semisimplefrommultifusion} and Proposition~\ref{prop:completionismod}, every semisimple 2-category is the completion of the delooping of a multifusion category, and is therefore (see Constructions~\ref{con:unfold} and~\ref{con:fold}) 2-distributor equivalent to the presemisimple unfolding of that multifusion category.  That presemisimple 2-category has the attractive feature that the $\Hom$ categories between inequivalent simple objects are all trivial, and so it looks `semisimple' in a more classical sense.  For instance, the underlying semisimple 2-category in Example~\ref{eg:Z4example} is 2-distributor equivalent to the presemisimple 2-category with two simple objects, each with endomorphism fusion category $\Vect(\ZZ_2)$, and no further morphisms.  This is, however, a specious presentation: even if we are prepared to work up to 2-distributor equivalence and even if we are prepared to work with presemisimple 2-categories, it is still \emph{not possible} in general to give a monoidal fusion structure without allowing nontrivial morphisms between inequivalent simple objects.  Example~\ref{eg:Z4example} illustrates this necessity: that fusion 2-category is not monoidally 2-distributor equivalent to any prefusion 2-category with trivial $\Hom$ categories between inequivalent simple objects.  Indeed the order four cyclic fusion group of the simple objects forces there to be at least two distinct simple objects in each connected component, in order to describe the monoidal structure.
\end{remark}

\subsection{Pivotal $2$-categories} \label{sec:piv}

\subsubsection{The definition and graphical calculus of planar pivotal $2$-categories}

\skiptocparagraph{Planar pivotal structures}

A monoidal $1$-category is called pivotal when it is equipped with a monoidal trivialization of the double dual functor; such a category has chosen isomorphisms from the double dual of each object to the object itself, in a way compatible with the tensor product of objects.  This notion generalizes straightforwardly to $2$-categories, by asking for a trivialization of the double adjoint of 1-morphisms, in a way compatible with composition.  Such $2$-categories are usually called `pivotal', but because we will also be concerned with trivializing the double dual of objects in a monoidal 2-category and therefore will have a different need for the modifier `pivotal', we will refer to 2-categories with a 1-morphism-level pivotal structure as `planar pivotal'.  For convenience, we will adopt a somewhat strictified definition as follows.

\begin{definition}[Planar pivotal 2-category] \label{def:planarpivotal}
Let $\tc{C}$ be a strict $2$-category in which every $1$-morphism has a left and a right adjoint.  A \emph{planar pivotal structure} on $\tc{C}$ consists of the following data:
\begin{enumerate}
\item[D1.] a choice of right adjoint $f^*:B\to A$ for every $1$-morphism $f:A\to B$,
\item[D2.] a choice of unit $\eta_f: \Io_A \To f^*\xo f$ and counit $\epsilon_f:f\xo f^* \To \Io_B$, 
\end{enumerate}
subject to the following conditions:
\begin{enumerate}
\item[C1.] the unit and counit satisfy the cusp equations:
\begin{align*}
&\left(\epsilon_f \xo \Io_f \right) \xt\left(\Io_f \xo \eta_f\right)  = \Io_{f} 
\\
&\Io_{f^*} =  \left( \Io_{f^*} \xo \epsilon_f \right)\xt\left(\eta_f \xo  \Io_{f^*} \right)
\end{align*}
\item[C2.] the choice of adjoint is functorial:
\begin{align*}
\left(\Io_A\right)^* &= \Io_A
\\
(f\xo g)^* &= g^* \xo f^*
\end{align*}
\item[C3.] the choice of unit and counit is functorial:
\begin{align*}
\epsilon_{\Io_A} &= \It_{\Io_A}
\\
\eta_{\Io_A} &= \It_{\Io_A}
\\
\epsilon_{f\xo g} &= \epsilon_f \xt \left( f \xo \epsilon_g \xo f^*\right)
\\ 
\eta_{f\xo g} &= \left(g^* \xo \eta_f \xo g \right)\xt  \eta_g
\end{align*}
\item[C4.] the adjoint is involutive: $f^{**} = f$;
\item[C5.] right and left mates agree: for any $2$-morphism $\alpha: f\To g$ we have
\begin{equation*}
\alpha^\ast := \left( \It_{f^*} \xo \epsilon_g\right) \xt\left( \It_{f^*} \xo \alpha \xo \It_{g^*} \right) \xt\left(\eta_f \xo \It_{g^*}\right) 
=
\left( \epsilon_{g^*} \xo \It_{f^*}\right) \xt\left( \It_{g^*} \xo \alpha \xo \It_{f^*} \right) \xt\left(\It_{g^*}\xo \eta_{f^*}\right) =: {}^\ast\alpha 
\end{equation*}
\end{enumerate}
\end{definition}

\nid We refer to a 2-category with a planar pivotal structure simply as a \emph{planar pivotal 2-category}.

\skiptocparagraph{Calculus of planar pivotal structures}

Planar pivotal 2-categories admit a graphical calculus of oriented strings in the plane~\cite{Selinger}.  The unit and counit of an adjunction $f\dashv f^*$ are depicted respectively as follows:
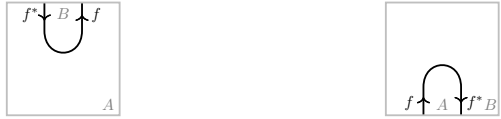
\[
\begin{tz}[scale=0.5,yscale=-1]
\draw[slice, on layer =front] (0,0) rectangle (3,3);
\draw[wire, arrow data={0.15}{>}, arrow data = {0.9}{>}] (1,0) to (1,0.75)  to [out=up, in=up, looseness=2] (2,0.75) to (2,0);
\node[omor,below right] at (2.1,0) {$f$};
\node[omor,below left] at (1,0) {$f^*$};
\node[obj, above left] at (3,3) {$A$};
\node[obj, below] at (1.5,0) {$B$};
\end{tz}
\qquad \qquad \qquad \qquad \qquad
\begin{tz}[scale=0.5]
\draw[slice, on layer =front] (0,0) rectangle (3,3);
\draw[wire, arrow data={0.15}{>}, arrow data = {0.9}{>}] (1,0) to (1,0.75)  to [out=up, in=up, looseness=2] (2,0.75) to (2,0);
\node[omor,above left] at (0.9,0) {$f$};
\node[omor,above right] at (2,0) {$f^*$};
\node[obj, above left] at (3.1,0) {$B$};
\node[obj, above] at (1.5,0) {$A$};
\end{tz}
\]

\nid Here a wire labeled $f$ with an upwards pointing orientation arrow refers to the morphism $f$, and the same wire with a downwards pointing orientation arrow refers to the morphism $f^*$; note that in order to avoid confusion later and contrary to typical convention, we will label this downwards pointing segment by $f^*$, the actual object associated to that wire segment, not by $f$, the object associated to the segment with reversed orientation.

The cusp equations are similarly represented by the following pictures:
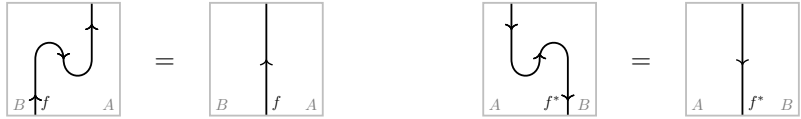
\[
\begin{tz}[scale=0.5]
\draw[slice, on layer=front] (0,0) rectangle (3,3);
\draw[wire, arrow data ={0.1}{>},arrow data={0.5}{>},arrow data ={0.9}{>}] (0.75,0) to (0.75, 1.5) to [out=up, in=up, looseness=2] (1.5, 1.5) to [out=down, in=down, looseness=2] (2.25, 1.5) to (2.25,3);
\node[obj,above left] at (3,0) {$A$};
\node[obj, above right] at (0,0) {$B$};
\node[omor, above right] at (0.75, 0) {$f$};
\end{tz}
\planareqgap
=\planareqgap
\begin{tz}[scale=0.5]
\draw[slice, on layer=front] (0,0) rectangle (3,3);
\draw[wire, arrow data ={0.5}{>}] (1.5, 0) to (1.5, 3);
\node[obj,above left] at (3,0) {$A$};
\node[obj, above right] at (0,0) {$B$};
\node[omor, above right] at (1.5, 0) {$f$};
\end{tz}
\qquad \qquad \qquad
\begin{tz}[scale=0.5,xscale=-1]
\draw[slice, on layer=front] (0,0) rectangle (3,3);
\draw[wire, arrow data ={0.1}{<},arrow data={0.5}{<},arrow data ={0.9}{<}] (0.75,0) to (0.75, 1.5) to [out=up, in=up, looseness=2] (1.5, 1.5) to [out=down, in=down, looseness=2] (2.25, 1.5) to (2.25,3);
\node[obj,above right] at (3,0) {$A$};
\node[obj, above left] at (0,0) {$B$};
\node[omor, above left] at (0.8, 0) {$f^*$};
\end{tz}
\planareqgap
=\planareqgap
\begin{tz}[scale=0.5]
\draw[slice, on layer=front] (0,0) rectangle (3,3);
\draw[wire, arrow data ={0.5}{<}] (1.5, 0) to (1.5, 3);
\node[obj,above left] at (3,0) {$B$};
\node[obj, above right] at (0,0) {$A$};
\node[omor, above right] at (1.5, 0) {$f^*$};
\end{tz}
\]
\nid A 2-morphism $\alpha: f \To g$ is depicted by a dot on a wire, with the lower half of the wire labeled $f$ and the upper half labeled $g$.  In this notation the mate 2-morphism $\alpha^*: g^* \To f^*$ appears as follows:
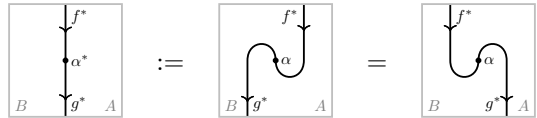
\begin{equation*}
\begin{tz}[scale=0.5]
\draw[slice, on layer =front] (0,0) rectangle (3,3);
\draw[wire, arrow data={0.2}{<}, arrow data={0.8}{<}] (1.5,0) to (1.5, 3);
\node[dot] at (1.5, 1.5) {};
\node[tmor, right] at (1.5, 1.5) {$\alpha^*$};
\node[obj,above left] at (3,0) {$A$};
\node[obj, above right] at (0,0) {$B$};
\node[omor, above right] at (1.5, 0) {$g^*$};
\node[omor, below right] at (1.5, 3) {$f^*$};
\end{tz}
\planareqgap:=\planareqgap
\begin{tz}[scale=0.5]
\draw[slice, on layer=front] (0,0) rectangle (3,3);
\draw[wire, arrow data ={0.1}{<},arrow data ={0.9}{<}] (0.75,0) to (0.75, 1.5) to [out=up, in=up, looseness=2] (1.5, 1.5) to [out=down, in=down, looseness=2] (2.25, 1.5) to (2.25,3);
\node[dot] at (1.5, 1.5) {};
\node[tmor, right] at (1.5, 1.5) {$\alpha$};
\node[obj,above left] at (3,0) {$A$};
\node[obj, above right] at (0,0) {$B$};
\node[omor, above right] at (0.75, 0) {$g^*$};
\node[omor, below left] at (2.25, 3) {$f^*$};
\end{tz}
\planareqgap=\planareqgap
\begin{tz}[scale=0.5,xscale=-1]
\draw[slice, on layer=front] (0,0) rectangle (3,3);
\draw[wire, arrow data ={0.1}{<},arrow data ={0.9}{<}] (0.75,0) to (0.75, 1.5) to [out=up, in=up, looseness=2] (1.5, 1.5) to [out=down, in=down, looseness=2] (2.25, 1.5) to (2.25,3);
\node[dot] at (1.5, 1.5) {};
\node[tmor, right] at (1.5, 1.5) {$\alpha$};
\node[obj,above right] at (3,0) {$B$};
\node[obj, above left] at (0,0) {$A$};
\node[omor, above left] at (0.75, 0) {$g^*$};
\node[omor, below right] at (2.25, 3) {$f^*$};
\end{tz}
\end{equation*}
\nid
Because taking the dual is involutive in a planar pivotal 2-category, we can form well-typed circular strings, known as `traces'.  
\begin{definition}[Planar trace]
In a planar pivotal $2$-category, given a 1-morphism $f: A \to B$ and a 2-endomorphism $\alpha: f \To f$, the \emph{left planar trace} $\tr_L(\alpha): \Io_{A} \To \Io_{A}$ and \emph{right planar trace} $\tr_R(\alpha): \Io_{B} \To \Io_{B}$ are
\begin{calign}\nonumber
\tr_L(\alpha) := \epsilon_{f^*}\xt(\It_{f^*} \xo \alpha)\xt\eta_f = ~~\begin{tz}[scale=0.5] \draw[slice, on layer=front] (-1.5,0) rectangle (1.5,3);
\draw[wire, arrow data = {0.1}{>}, arrow data = {0.5}{>}, arrow data={0.93}{>}] (0.5,1.5) to (0.5, 1.75) to [out=up, in=up, looseness=2] (-0.5, 1.75) to (-0.5, 1.25) to [out=down, in=down, looseness=2] (0.5, 1.25) to (0.5, 1.5);
\node[dot] at (0.5,1.5){};
\node[tmor, right] at (0.5, 1.5) {$\alpha$};
\node[obj, above left] at (1.5,0){$A$};
\end{tz}
&
\tr_R(\alpha) := \epsilon_{f}\xt(\alpha \xo \It_{f^*})\xt\eta_{f^*} = ~~\begin{tz}[scale=0.5,xscale=-1] \draw[slice, on layer=front] (-1.5,0) rectangle (1.5,3);
\draw[wire, arrow data = {0.1}{>}, arrow data = {0.5}{>}, arrow data={0.93}{>}] (0.5,1.5) to (0.5, 1.75) to [out=up, in=up, looseness=2] (-0.5, 1.75) to (-0.5, 1.25) to [out=down, in=down, looseness=2] (0.5, 1.25) to (0.5, 1.5);
\node[dot] at (0.5,1.5){};
\node[tmor, left] at (0.5, 1.5) {$\alpha$};
\node[obj, above left] at (-1.5,0){$B$};
\end{tz}
\end{calign}
\end{definition}

\skiptocparagraph{Monoidal planar pivotal structures}

If we add a monoidal structure to a planar pivotal 2-category, we insist that the monoidal structure respects the planar pivotal structure as follows.

\begin{definition}[Monoidal planar pivotal 2-category] \label{def:monoidalplanarpivotal}
A \emph{monoidal planar pivotal $2$-category} is a monoidal $2$-category equipped with a planar pivotal structure such that
\begin{enumerate}
\item the adjoint of a tensor is the tensor of the adjoints:
\begin{align*}
\left(A \xz f\right)^* &= A\xz f^* 
\\
\left(f \xz A\right)^* &= f^* \xz A
\end{align*}
\item the unit and counit for a tensor are the tensors of the units and counits:
\begin{align*}
\epsilon_{A\xz f} &= A\xz \epsilon_f 
\\
\eta_{A\xz f} &= A \xz \eta_f 
\\
\epsilon_{f\xz A} &= \epsilon_f \xz A 
\\
\eta_{f\xz A} &= \eta_f \xz A
\end{align*}
Here $A$ is an object, $f: A \to B$ is a 1-morphism, and $\epsilon_f$ and $\eta_f$ are the counit and unit of the adjunction between $f$ and $f^*$.
\end{enumerate}
\end{definition}

\nid
The graphical calculus for planar pivotal 2-categories extends to one for monoidal planar pivotal 2-categories: 1-morphisms $f: A \to B$ are again represented by oriented strings, which are constrained to move within a single sheet of the surface diagram.  If the monoidal category has duals, and therefore a graphical calculus where multiple sheets may connect via unit and counit morphisms depicted as `cups' and `caps', the 1-morphism strings may not intersect these cups and caps---such interaction requires further conditions, as described in the next section.

\subsubsection{The definition and graphical calculus of pivotal $2$-categories}

\skiptocparagraph{Pivotal structures}

In a 2-category with adjoints for 1-morphisms (or more simply a monoidal 1-category with duals for objects), a planar pivotal structure (respectively pivotal structure) is a choice of adjoint structure that is coherent (Definition~\ref{def:planarpivotal}-C1), functorial (C2,C3), and involutive (C4), such that left and right mates agree (C5).

Given a monoidal planar pivotal 2-category with duals for objects, we can insist that object duality itself is `pivotal' and that the object-level duality interacts well with the 1-morphism-level duality.  More specifically, below we will define a pivotal structure on a monoidal planar pivotal 2-category to be a choice of duality structure that is coherent (Definition~\ref{def:pivotal}-C1), compatible with tensor (C2,C3,C4), compatible with the existing adjoint structure (C5,C6), and involutive (C7), such that the right-over and left-under twist 2-morphisms between the left and right mates agree (C8).

Given a monoidal 2-category with chosen right dual $A^\#$ for every object $A$ such that $A^{\#\#}= A$, along with chosen unit and counit 1-morphisms $i_A: \Iz \to A^\# \xz A$ and $e_A: A\xz A^\# \to \Iz$, and any 1-morphism $f: A \to B$, the right and left mate of $f$ are defined by
\[
\begin{split}
f^\#&:=(A^\# \xz e_B)\xo (A^\#\xz f \xz B^\#)\xo (i_A\xz B^\#) : B^\# \to A^\# \\
\lix{\#}{f}& :=(e_{B^\#} \xz A^\#)\xo(B^\# \xz f \xz A^\#) \xo (B^\#\xz i_{A^\#}): B^\# \to A^\#\text{.}
\end{split}
\] 
Similarly for any 2-morphism $\alpha:g\To h$ between $1$-morphisms $g,h:A\to B$, the right and left (object-duality) mates of $\alpha$ are defined by
\[ 
\begin{split}
\alpha^\#&:=(A^\# \xz e_B)\xo (A^\#\xz \alpha \xz B^\#)\xo (i_A\xz B^\#) : g^\# \To h^\#
\\
\lix{\#}{\alpha}&:=(e_{B^\#} \xz A^\#)\xo(B^\# \xz \alpha \xz A^\#) \xo (B^\#\xz i_{A^\#}) : \lix{\#}{g} \To \lix{\#}{h}.
\end{split}
\]

\begin{definition}[Pivotal 2-category] \label{def:pivotal} 
Let $\tc{C}$ be a monoidal planar pivotal 2-category in which every object has a left and a right dual.  A \emph{pivotal structure} on $\tc{C}$ consists of the following data:
\begin{enumerate}
\item[D1.] a choice of right dual $A^\#$ for every object $A$ of $\tc{C}$;
\item[D2.] a choice of $1$-morphisms (called \emph{folds}) $i_A: \Iz \to A^\# \xz A$ and $e_A: A\xz A^\# \to \Iz$;
\item[D3.] a choice of invertible $2$-morphisms (called \emph{cusps})
\begin{align*}
&C_A:  \left(e_A \xz A \right) \xo\left(A \xz i_A\right)  \To \Io_{A} 
\\
&D_A:\Io_{A^\#} \To  \left( A^\# \xz e_A \right)\xo\left(i_A \xz  A^\# \right)
\end{align*}
\end{enumerate}
subject to the following conditions:
\begin{enumerate}
\item[C1.] the cusps satisfy the swallowtail equations:
\begin{align*}
\left[e_A\xo\left(C_A\xz A^\#\right)\right]\xt\left[\phi_{e_A,e_A}\xo \left(A\xz i_A \xz A^\#\right)\right]\xt \left[e_A\xo\left(A\xz D_A\right)\right]&=\It_{e_A}
\\
\left[\left(A^\# \xo C_A\right)\xo i_A\right] \xt\left[\left(A^\#\xz e_A \xz A \right)\xo\phi_{i_A,i_A}\right]\xt \left[\left(D_A\xz A\right)\xo i_A\right]&=1_{i_A}
\end{align*} 
\item[C2.] the choice of dual is compatible with tensor:
\begin{align*}
\Iz^\# &= \Iz 
\\
(A\xz B)^\# &= B^\# \xz A^\#
\end{align*}
\item[C3.] the choice of folds is compatible with tensor:
\begin{align*}
i_{I} &=  e_{I} = \Io_I 
\\
 i_{A\xz B} &= \left( B^\# \xz i_A \xz B \right) \xo i_B 
 \\
 e_{A\xz B} &= e_A\xo \left(A\xz e_B \xz A^\# \right)
\end{align*}
\item[C4.] the choice of cusps is compatible with tensor:\\
\hspace*{-.75cm}\shrinkalign{.9}{
\begin{align*}
C_I &=  D_I = \It_{\Io_I}
\\
 C_{A\xz B} &=   \left[\left(C_A\xz B\right)\xo\left( A\xz C_B\right)\right] \xt\left[\left(e_A \xz A \xz B\right) \xo \left(A\xz \phi_{e_B, i_A} \xz B \right)\xo \left(A\xz B \xz i_B \right) \right]
\\
 D_{A\xz B} &= \left[\left(B^\# \xz A^\# \xz e_A\right)\xo\left(B^\# \xz \phi_{i_A, e_B} \xz A^\# \right)\xo\left(i_B\xz B^\#\xz A^\# \right)\right]\xt\left[\left(B^\# \xz D_A\right) \xo \left(D_B\xz A^\#\right)\right]
 \end{align*}
 }
\item[C5.] the folds intertwine duality of objects and adjunction of 1-morphisms: $(e_A)^* = i_{A^\#}$;
\item[C6.] the cusps intertwine duality of objects and mates of 2-morphisms: $(D_A)^* = C_{A^\#}$;
\item[C7.] the dual is involutive: $A^{\#\#} = A$;
\item[C8.] the right-over and left-under twists between the left and right mates agree: for any $1$-morphism $f:A\to B$ we have
\[
 \theta_f = \left(\theta_{f^*}\right)^* : \lix{\#}{f} \To f^\#
 \]
where\pagebreak
\[\begin{split} 
\theta_f := &
\left[\left( A^\# \xz e_B\right) \xo\left(A^\# \xz f \xz  B^\#\right)\xo\left(\epsilon_{\lix{\#}{f}}\xz A \xz B^\#\right) \xo \left(i_A \xz B^\#\right)\right] \\ & \xt
\left[\phi_{\lix{\#}{f}, e_B\xo (f\xz B^\#)} \xo \left((f^*)^\#\xz A \xz B^\#\right) \xo \left(i_A \xz B^\#\right)
\right] \\ & \xt
\left[ \lix{\#}{f}\xo \left( B^\# \xz e_B\right) \xo \left( B^\# \xz f \xz B^\# \right) \xo \left( B^\# \xz e_A \xz A \xz B^\# \right) \xo \left( \phi^{-1}_{(B^\# \xz f^*) \xo i_B, i_A} \xz B^\# \right)  \right] \\& \xt
\left[\lix{\#}{f}  \xo \left(f \xo C_A^{-1} \xo f^*\right)^\# \right] \xt
\left[\lix{\#}{f} \xo \eta_{f^*}^\#  \right] \xt
\left[  \lix{\#}{f}\xo D_B\right]
\end{split}
\]  
\end{enumerate}

\end{definition}
\nid For brevity, we refer to a monoidal planar pivotal 2-category with a pivotal structure as a \emph{pivotal 2-category}.  Though structured somewhat differently, we expect this definition is equivalent to the notion BMS call a `spatial Gray monoid'~\cite{BMS}.

\skiptocparagraph{Calculus of pivotal structures}

The graphical calculus for pivotal 2-categories is a calculus of surfaces with string defects, compare~\cite{BMS}.  By planar pivotality, the fold 1-morphisms $e_A$ and $i_A$ have adjoints $i_A^* = e_{A^\#}$ and $e_A^* = i_{A^\#}$.  The units and counits of those adjunctions are referred to as a \emph{crotch}, \emph{saddle}, \emph{birth of a circle}, and \emph{death of a circle} and (extending the calculus of monoidal 2-categories with duals from Section~\ref{sec:graphicalcalculus}) are depicted as follows:
\begin{calign}\nonumber
\begin{tz}[td,master]
\begin{scope}[xyplane=0]
\draw[slice] (0,-0.5) to (0,0)  to [out=up, in=up, looseness=2] node[pos=0.36](R){} (1,0) to (1,-0.5);
\draw[slice] (0,3.5) to (0,3) to [out=down, in=down, looseness=2] node[pos=0.59](L){} (1,3) to (1,3.5);
\end{scope}
\begin{scope}[xyplane=\h]
\draw[slice] (0,-0.5) to (0,3.5);
\draw[slice] (1,-0.5) to (1,3.5);
\end{scope}
\begin{scope}[xzplane=-0.5]
\draw[slice, short] (0,0) to (0,\h);
\draw[slice, short] (1,0) to (1,\h);
\end{scope}
\begin{scope}[xzplane=3.5]
\draw[slice, short] (0,0) to (0,\h);
\draw[slice, short] (1,0) to (1,\h);
\end{scope}
\coordinate (Rs) at (R.center|-L.center);
\draw[slice] (R.center) to ([yshift=0.2*\h cm]Rs.center) to  [out=up, in=up, looseness=2] ([yshift=0.2*\h cm] L.center) to (L.center);
\node[obj, above left] at (-0.5,0,0){$A$};
\end{tz}
&
\begin{tz}[td,slave]
\begin{scope}[xyplane=\h]
\draw[slice] (0,-0.5) to (0,0)  to [out=up, in=up, looseness=2] node[pos=\stdr](R){} (1,0) to (1,-0.5);
\draw[slice] (0,3.5) to (0,3) to [out=down, in=down, looseness=2] node[pos=\stdl](L){} (1,3) to (1,3.5);
\end{scope}
\begin{scope}[xyplane=0]
\draw[slice] (0,-0.5) to (0,3.5);
\draw[slice] (1,-0.5) to (1,3.5);
\end{scope}
\begin{scope}[xzplane=-0.5]
\draw[slice, short] (0,0) to (0,\h);
\draw[slice, short] (1,0) to (1,\h);
\end{scope}
\begin{scope}[xzplane=3.5]
\draw[slice, short] (0,0) to (0,\h);
\draw[slice, short] (1,0) to (1,\h);
\end{scope}
\coordinate (Rs) at (R.center|-L.center);
\draw[slice] (R.center) to ([yshift=-0.2*\h cm]Rs.center) to  [out=down, in=down, looseness=2] ([yshift=-0.2*\h cm] L.center) to (L.center);
\node[obj, above left] at (-0.5,0,0){$A$};
\end{tz}
&
\begin{tz}[td, slave]
\begin{scope}[xyplane=\h]
\draw[slice] (0,1.5) to [out=up, in=up, looseness=2] node[pos=\stdr] (L){} (1,1.5) to [out=down, in=down, looseness=2]node[pos=1-\stdl](R){} (0,1.5);
\end{scope}
\coordinate (Rd) at (R.center|- L.center);
\draw[slice] (L.center) to  [out=down, in=down, looseness=2] (Rd.center) to (R.center);
\node[obj, below left] at (0.8,0.2,\h){$A$};
\end{tz}
&
\begin{tz}[td,slave]
\begin{scope}[xyplane=0]
\draw[slice] (0,1.5) to [out=up, in=up, looseness=2] node[pos=\stdr] (L){} (1,1.5) to [out=down, in=down, looseness=2]node[pos=1-\stdl](R){} (0,1.5);
\end{scope}
\coordinate (Lu) at (L.center|- R.center);
\draw[slice] (L.center) to  (Lu.center) to [out=up, in=up, looseness=2]  (R.center);
\node[obj, above left] at (0.75,0,0){$A$};
\end{tz}
\\\nonumber 
\shrinker{.9}{\epsilon_{i_{A^\#}}: i_{A^\#} \xo e_{A} \To \Io_{A\xz A^\#}}
&
\shrinker{.9}{\eta_{e_A}: \Io_{A\xz A^\#} \To i_{A^\#}\xo e_A}
& 
\shrinker{.9}{\eta_{i_{A^\#}}: \Io_{\Iz} \To e_A \xo i_{A^\#}}
&
\shrinker{.9}{\epsilon_{e_A}: e_A\xo i_{A^\#} \To \Io_{\Iz}}
\end{calign}
\nid With this pictorial notation, one may draw a surface embedded in 3-dimensional space, with sheets labeled by objects, together with string defects on the surface labeled by 1-morphisms, and point defects on the strings labeled by 2-morphisms.  It is reasonable to conjecture that the resulting 2-morphism of the pivotal 2-category is invariant under isotopy of the picture; BMS go some way towards establishing that result~\cite{BMS}.  Note that we will not need this full graphical calculus and none of our results depend on it---all equations we need will be explicitly established algebraically.

\begin{warning}[Failure of invariance for the pivotal 2-category graphical calculus] \label{warn:pivotal}
Depending on one's perspective, either Definition~\ref{def:pivotal} or its corresponding graphical calculus has a fairly serious and perhaps not evident drawback: the scalar value of a closed surface labeled by an object is not invariant under equivalence of objects.  We do not know how to satisfactorily alter either the definition or the graphical calculus to eliminate this problem.
\end{warning}

\begin{remark}[Cusp flip condition]
The condition C6, in Definition~\ref{def:pivotal}, that cusps intertwine duality of objects and mates of 2-morphisms is equivalent to the `cusp flip equation':
\[
\shrinker{.92}{
\left[\left(A^\# \xz e_{A} \right)\xo\left( \epsilon_{i_{A}} \xz A^\#\right)\right]\xt\left[D_A \xo\left(e_{A^\#} \xz A^\#\right)\right]
= \left[C_{A^\#}\xo\left( A^\#\xz e_A\right) \right]\xt\left[\left( e_{A^\#} \xz A^\#\right)\xo \left( A^\# \xz\eta_{e_A}\right)\right]
}
\]
This equation is depicted graphically as follows:
\begin{calign}\nonumber
\begin{tz}[td,scale=0.85]
\coordinate (cusp) at (2.5,0.5, 0.5*\h);
\begin{scope}[xyplane=0]
\draw[slice](0,-1) to (0,0) to [out=up, in=up, looseness=2] node[pos=\stdr](BR){} (1,0) to (1,-1);
\draw[slice] (2,-1) to (2,1) to [out=up, in=down] (0,4);
\end{scope}
\begin{scope}[xyplane=\h]
\draw[slice](0,-1)  to (0,0)  to [out=up, in=up, looseness=2] node[pos=\stdr](MR){} (1,0) to  (1,-1);
\draw[slice] (2,-1) to (2,3) to [out=up, in=up, looseness=2] node[pos=\stdl] (LL){}(1,3)  to (1,2.5) to  [out=down, in=down, looseness=2] node[pos=\stdr] (ML){} (0,2.5) to (0,4);
\end{scope}
\begin{scope}[xyplane=2*\h]
\draw[slice](0,-1)to (0,4);
\draw[slice] (1,-1)  to (1,3)  to [out=up, in=up, looseness=2] node[pos=\stdr] (TL){} (2,3) to (2,-1);
\end{scope}
\begin{scope}[xzplane=-1]
\draw[slice, short] (0,0) to (0,2*\h);
\draw[slice, short] (1,0) to (1,2*\h);
\draw[slice, short] (2,0) to (2,2*\h);
\end{scope}
\begin{scope}[xzplane=4]
\draw[slice, short] (0,0) to (0,2*\h);
\end{scope}
\coordinate (MRu) at (MR.center|-ML.center);
\draw[slice] (BR.center) to [out=up, in=down] ([yshift=0.2cm]MRu.center) to [out=up, in=up, looseness=3] ([yshift=0.2cm]ML.center) to (ML.center)to [out=down, in=\urcusp] (cusp);
\draw[slice] (cusp) to [out=\ulcusp, in=down] (LL.center) to[out=up, in=down] (TL.center);
\node[obj, above left] at (-1.2,-0.1,0) {$A^\#$};
\end{tz}
\tdeqgap= \tdeqgap
\begin{tz}[td,scale=0.85]
\coordinate (cusp) at (2.75,1.5, 1.6*\h);
\begin{scope}[xyplane=0]
\draw[slice](0,-1) to (0,3) to [out=up, in=up, looseness=2] node[pos=\stdr] (BL){}(1,3) to (1,-1);
\draw[slice] (2,-1) to (2,5);
\end{scope}
\begin{scope}[xyplane=\h]
\draw[slice](1,-1)  to (1,0)  to [out=up, in=up, looseness=2] node[pos=\stdr](MR){} (2,0) to  (2,-1);
\draw[slice] (0,-1) to (0,3) to [out=up, in=up, looseness=2] node[pos=\stdr] (ML){}(1,3)  to (1,2.5) to  [out=down, in=down, looseness=2] node[pos=\stdl] (MM){} (2,2.5) to (2,5);
\end{scope}
\begin{scope}[xyplane=2*\h]
\draw[slice] (0,-1) to [out=up, in=down] (2,5);
\draw[slice] (1,-1) to (1,0) to [out=up, in=up, looseness=2] node[pos=\stdr] (TR){} (2,0) to (2,-1);
\end{scope}
\begin{scope}[xzplane=-1]
\draw[slice, short] (0,0) to (0,2*\h);
\draw[slice, short] (1,0) to (1,2*\h);
\draw[slice, short] (2,0) to (2,2*\h);
\end{scope}
\begin{scope}[xzplane=5]
\draw[slice, short] (2,0) to (2,2*\h);
\end{scope}
\coordinate (MRd) at (MR.center|-MM.center);
\draw[slice] (BL.center) to (ML.center) to [out=up, in=\dlcusp] (cusp);
\draw[slice] (cusp) to [out=\drcusp, in=up] (MM.center) to ([yshift=-0.1cm]MM.center) to [out=down, in=down, looseness=3] ([yshift=-0.1cm]MRd.center) to (TR.center);
\node[obj, above left] at (-1.2,-0.1,0) {$A^\#$};
\end{tz}
\end{calign}
\end{remark}

\begin{remark}[Surface ribbon condition] \label{rem:surfaceribbon}
The final condition C8 of Definition~\ref{def:pivotal}, that the right-over and left-under twists between mates agree, may be depicted graphically as follows:
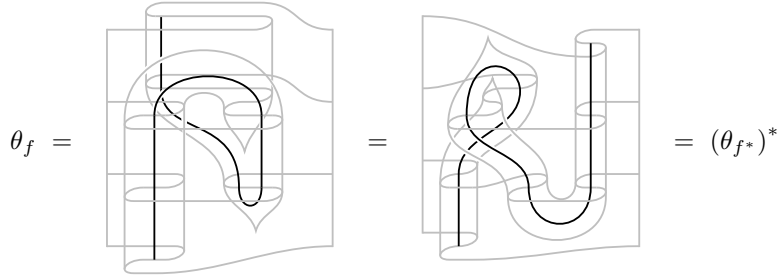
\def\scl{0.6}
\[
\def\hn{3.2}
\theta_f \seqgap = \tdeqgap
\begin{tz}[td,scale=\scl]
\begin{scope}[xyplane=0]
\draw[slice] (2,-3) to [out=up, in=down] (0,4) to  (0,5) to [out=up, in=up, looseness=2] node[pos=\stdr] (L0){} (1,5) to (1,4) to  [out=down, in=down, looseness=2] node[pos=\stdl] (R0){} (2,4) to (2,7);
\end{scope}
\begin{scope}[xyplane=\hn]
\draw[slice] (2,-3) to (2,1) to [out=up, in=up, looseness=2] node[pos=\stdl] (ML1){} (1,1) to (1,-0.75);
\draw[slice, on layer=front] (1,-0.75) to [out=down, in=down,looseness=2] node[pos=\stdr] (MR1){} (0,-0.75) to (0,5) to [out=up, in=up, looseness=2] node[pos=\stdr] (L1){} (1,5);
\draw[slice] (1,5) to (1,4) to  [out=down, in=down, looseness=2] node[pos=\stdl] (R1){} (2,4) to (2,7);
\end{scope}
\begin{scope}[xyplane=2*\hn]
\draw[slice] (2,-3) to [out=up, in=down] (4,1) to (4,5.25) to [out=up, in=up, looseness=2] node[pos=\stdl] (LL2){} (3,5.25) to (3,0.5)to [out=down, in=down, looseness=2]node[pos=\stdr] (MR2){} (2,0.5) to (2,1) to [out=up, in=up, looseness=2] node[pos=\stdl] (ML2){} (1,1) to (1,-0.75);
\draw[slice, on layer=front] (1,-0.75) to [out=down, in=down,looseness=2] node[pos=\stdr] (RR2){} (0,-0.75) to (0,5) to [out=up, in=up, looseness=2] node[pos=\stdr] (L2){} (1,5);
\draw[slice] (1,5) to (1,4);
\draw[slice, on layer=front] (1,4) to  [out=down, in=down, looseness=2] node[pos=\stdl] (R2){} (2,4) to (2,7);
\end{scope}
\begin{scope}
[xyplane=3*\hn]
\draw[slice] (2,-3)  to [out=up, in=down] (4,1) to (4,5.25) to [out=up, in=up, looseness=2] node[pos=\stdl] (LL3){} (3,5.25);
\draw[slice, on layer=front] (3,5.25) to (3,0.5)to [out=down, in=down, looseness=2]node[pos=\stdr] (R3){} (2,0.5) to (2,7);
\end{scope}
\begin{scope}
[xzplane=-3]
\draw[slice,short] (2,0) to (2,3*\hn);
\end{scope}
\begin{scope}
[xzplane=7]
\draw[slice,short] (2,0) to (2,3*\hn);
\end{scope}
\coordinate (cusp0) at (0,1,0.4*\hn);
\coordinate (cusp1) at (0.5,1 ,1.5*\hn);
\draw[slice] (R0.center) to node[pos=0.82, mask point,scale=0.6] (MP8){} node[pos=0.9, mask point] (MP1){} (R2.center) to [out=up, in=up, looseness=1.2] (ML2.center|-R2.center) to (ML2.center) ;
\draw[wire, on layer=front] (4.5,1,2*\hn) to [out=up, in=up, looseness=1.25] node[pos=0.05, mask point] (MP7){} node[pos=0.17, mask point] (MP2){} (-0.25, 1, 2*\hn);
\draw[slice] (MR1.center) to (RR2.center) to [out=up, in=up, looseness=1.5] node[pos=0.15, mask point,scale=0.6] (MP5){} node[pos=0.69, mask point] (MP3){} node[pos=0.77, mask point] (MP4){}  (L2.center|-RR2.center) to  (L2.center);
\cliparoundone{MP4}{
\draw[slice] (LL3.center) to (LL2.center);
}
\cliparoundtwo{MP7}{MP8}{
\draw[slice] (LL2.center) to [out=down, in=up] (ML1.center) to [out=down, in=\ulcusp] (cusp0) to [out=\urcusp, in=down] (MR1.center);
}
\cliparoundtwo{MP1}{MP2}{
\draw[wire] (-0.25,1,\hn) to [out=down, in=down, looseness=2.75] (0.75,1,\hn) to [out=up, in=down, out looseness=1.5] (5.,3, 2*\hn); 
}
\cliparoundone{MP3}{
\draw[wire] (5,3,2*\hn) to (5,3, 3*\hn);
}
\draw[wire] (4.5,1, 0) to (4.5,1, 2*\hn);
\draw[wire] (-0.25, 1, 2*\hn) to node[pos=0.08, mask point, scale=0.8] (MP6){} (-0.25, 1, \hn);
\cliparoundtwo{MP5}{MP6}{
\draw[slice] (ML2.center) to [out=down, in= \ulcusp] (cusp1) to [out=\urcusp, in=down] (MR2.center) to (R3.center);
}
\draw[slice] (L0.center) to (L2.center);
\end{tz}
\def\l{3.5}
\def\r{2.75}
\tdeqgap = \tdeqgap
\begin{tz}[td,scale=\scl]
\begin{scope}[xyplane=0]
\draw[slice] (1,-3) to [out=up, in=down]   (0,4) to (0,5) to [out=up, in=up, looseness=2] node[pos=\stdr] (L0){} (1,5) to  [out=down, in=down, looseness=2] node[pos=\stdl] (R0){} (2,5) to (2,7);
\end{scope}
\begin{scope}[xyplane=\hn]
\draw[slice] (1,-3)  to (1,-1) to [out=up, in=up, looseness=2] node[pos=\stdl] (RL1){} (0,-1) to (0,-1.5);
\draw[slice, on layer=front] (0,-1.5) to [out=down, in=down, looseness=2] node[pos=\stdr] (RR1){} (-1,-1.5) to (-1,1.6) to [out=up, in=up, looseness=2] node[pos=\stdr] (MR1){} (0,1.6);
\draw[slice] (0,1.6)  to (0,1.5) to [out=down, in=down, looseness=2] node[pos=\stdl] (ML1){} (1,1.5)to [out=up, in=down](0,4);
\draw[slice, on layer=front] (0,4)  to  (0,5);
\draw[slice] (0,5) to [out=up, in=up, looseness=2] node[pos=\stdr] (LL1){} (1,5) to [out=down, in=down, looseness=2] node[pos=\stdl] (R1){} (2,5) to (2,7) ;
\end{scope}
\begin{scope}[xyplane=2*\hn]
\draw[slice] (1,-3)  to (1,-1) to [out=up, in=up, looseness=2] node[pos=\stdr] (RL2){} (0,-1) to (0,-1.5);
\draw[slice, on layer=front] (0,-1.5) to [out=down, in=down, looseness=2] node[pos=\stdl] (RR2){} (-1,-1.5) to (-1, 4.25);
\draw[slice] (-1,4.25) to [out=up, in=up, looseness=2] node[pos=\stdr] (LL2){} (0,4.25) to (0,\l) to  [out=down, in=down, looseness=2] node[pos=\stdl] (ML2){} (1,\l) to [out=up, in=up, looseness=2] node[pos=\stdr] (L2){} (2,\l) to (2,\r) to  [out=down, in=down, looseness=2] node[pos=\stdl] (R2){} (3,\r) to (3,4) to [out=up, in=down] (2,7) ;
\end{scope}
\begin{scope}[xyplane=3*\hn]
\draw[slice] (1,-3)  to (1,-1) to [out=up, in=up, looseness=2] node[pos=\stdl] (L3){}  (0,-1) to (0,-1.5) to [out=down, in=down, looseness=2] node[pos=\stdr] (RR3){} (-1,-1.5) to [out=up, in=down] (2,7);
\end{scope}
\begin{scope}[xzplane=-3]
\draw[slice,short] (1,0) to (1,3*\hn);
\end{scope}
\begin{scope}[xzplane=7]
\draw[slice,short] (2,0) to (2,3*\hn);
\end{scope}
\coordinate (cusp0) at (3.5,1,2.33*\hn);
\coordinate (cusp1) at (4.1,2,2.7*\hn);
\draw[slice] (MR1.center) to [out=down, in=down, looseness=1.5] (RR1.center|-MR1.center) to  (RR1.center);
\draw[wire] (-1.25,0, 3*\hn) to  (-1.25,0,\hn) to [out=down, in=down, looseness=2] (1.5,0, \hn) to [out=up, in=down, in looseness=1] node[mask point, pos=0.7] (MP4){}node[mask point, pos=0.55] (MP5){} node[mask point, pos=0.92] (MP6){}(4.25,0, 2*\hn);
\draw[slice] (RL1.center) to [out=down, in=down, looseness=2] (ML1.center|-RL1.center) to  (ML1.center);
\draw[slice] (MR1.center) to [out=up, in=down, out looseness=1.2] node[mask point, pos=0.5] (MP1){} node[mask point, pos=0.65] (MP2){}node[mask point, pos=0.85] (MP3){} (LL2.center);
\cliparoundtwo{MP3}{MP6}{
\draw[slice] (L0.center) to (LL1.center) to [out=up, in=down, in looseness=0.6, out looseness=1.5] (L2.center) to [out=up, in=\dlcusp] (cusp0) to [out=\drcusp, in=up] (ML2.center) to [out=down, in=up] node[mask point,pos=0.32] (MP7){}node[mask point,pos=0.16] (MP8){} (ML1.center);
}
\cliparoundthree{MP1}{MP5}{MP7}{
\draw[slice] (R0.center) to (R1.center) to [out=up, in=down, in looseness=1.5, out looseness=0.5] (R2.center) to [out=up, in=\drcusp] (cusp1) to [out=\dlcusp, in=up] (LL2.center);
}
\draw[slice] (RR1.center) to (RR3.center);
\draw[slice] (RL1.center) to (L3.center);
\cliparoundthree{MP2}{MP4}{MP8}{
\draw[wire] (5,1,0) to (5,1, \hn) to [out=up, in=down, out looseness=1] (2.6,1.8,2*\hn) to [out=up, in=up, looseness=1.3, in looseness=3] (4.25,0, 2*\hn) to [out=down, in=up] (4.25,0, 2*\hn);
}
\end{tz}
\tdeqgap=\seqgap
\left(\theta_{f^*}\right)^*
\]
The two surfaces may be thought of as the two simplest ways to take a surface with a line on it and deform the surface so that the normal vector to the line in the surface undergoes a full counterclockwise rotation.
\end{remark}

\begin{remark}[The loop category of a pivotal 2-category is ribbon]\label{rem:ribbon}
The surface ribbon condition drawn in Remark~\ref{rem:surfaceribbon} is of course reminiscent of the ribbon condition that may be imposed on a braided pivotal category.  Indeed, when the morphism $f$ is an endomorphism of the identity, and the surface is therefore irrelevant, the surface ribbon condition reduces precisely to the ribbon condition:  
\[\theta_f \seqgap=\seqgap \begin{tz}[scale=0.5]
\clip (-0.4,-0.2) rectangle (1.1,2.2);
\draw[wire] (0,0) to [out=up, in=up, looseness=2] node[mask point, pos=0.35](MP){} (1,1);
\cliparoundone{MP}{\draw[wire] (1,1) to [out=down, in=down, looseness=2] (0,2);}
\node[omor, left] at (0,0) {$f$};
\end{tz}
\tdeqgap = \tdeqgap
\begin{tz}[scale=0.5,scale=-1]
\clip (-0.4,-0.2) rectangle (1.1,2.2);
\draw[wire] (0,0) to [out=up, in=up, looseness=2] node[mask point, pos=0.35](MP){} (1,1);
\cliparoundone{MP}{\draw[wire] (1,1) to [out=down, in=down, looseness=2] (0,2);}
\node[omor, right] at (0,2) {$f$};
\end{tz}
\seqgap=\seqgap \left(\theta_{f^*}\right)^*
\]
Thus, for a pivotal $2$-category $\tc{C}$, the braided (planar) pivotal $1$-category $\End_{\tc{C}}(\Iz)$ is a ribbon category; cf~\cite[Cor~4.9]{BMS}.  By the same comparison, the delooping of a ribbon category is a pivotal 2-category.
\end{remark}

\subsubsection{Spherical traces in pivotal $2$-categories}
Recall that a pivotal monoidal 1-category is called `spherical' when the left and right `traces' of any 1-endomorphism agree.  In our terminology, that situation is a planar pivotal 2-category $\tc{C}$ having just one object $\ast$, in which for any 1-morphism $f: \ast \to \ast$ and 2-morphism $\alpha: f \To f$, the left planar trace $\tr_L(\alpha): \Io_{\ast} \To \Io_{\ast}$ is equal to the right planar trace $\tr_R(\alpha): \Io_{\ast} \To \Io_{\ast}$.  

In a general planar pivotal 2-category, it does not make sense to ask that the left planar trace and right planar trace are equal, as when the 1-morphism $f: A \to B$ is not an endomorphism, the left planar trace $\tr_L(\alpha): \Io_{A} \To \Io_{A}$ and right planar trace $\tr_R(\alpha): \Io_{B} \To \Io_{B}$ are 2-endomorphisms of different objects.  However, when the planar pivotal 2-category has a monoidal structure and is in fact pivotal, then we can compare the left and right planar traces by putting each one on a 2-sphere and then comparing the resulting 2-endomorphisms of the identity object.

Let $\cC$ be a pivotal 2-category, and let $f: A \to B$ be a 1-morphism and $\alpha: f \To f$ a 2-morphism in $\cC$.  As before $e_B: B \xz B^\# \to \Iz$ denotes the counit of the object duality; the composite $e_B \xo (\alpha \xz B^\#)$ is an endomorphism of the 1-morphism $e_b \xo (f \xz B^\#): A \xz B^\# \to \Iz$.  The right planar trace of this composite is therefore a 2-endomorphism of the identity object,
\[
\tr_R(e_B \xo (\alpha \xz B^\#)): \Io_\Iz \To \Io_\Iz
\]
and has the graphical representation
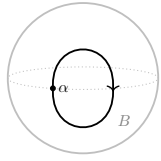
\[
\begin{tz}[scale=0.5]
  \draw[slice] (0,0) circle (2cm);
  \draw[tinydash] (-2,0) arc (180:360:2 and 0.3);
  \draw[tinydash] (2,0) arc (0:180:2 and 0.3);
  \begin{scope}[yshift=-0.27cm]
  \draw[wire, arrow data={0.515}{>}] (-0.8,0) to (-0.8, 0.1) to [out=up, in=up, looseness=2] (0.8,0.1) to (0.8,-0.1) to [out=down, in=down, looseness=2] (-0.8,-0.1) to (-0.8, 0);
  \end{scope}
  \node[dot] at (-0.8,-0.27){};
  \node[ right,tmor] at (-0.8,-0.27){$\alpha$};
     \node[obj,above left] at (\sqt, -\sqt) {$B$};
\end{tz}
\]
Similarly, using the unit $i_{A^\#} : \Iz \to A \xz A^\#$, we can form the composite $(\alpha \xz A^\#) \xo i_{A^\#}$, which is an endomorphism of the 1-morphism $(f \xz A^\#) \xo i_{A^\#}: \Iz \to B \xz A^\#$.  The left planar trace of this composite is therefore a 2-endomorphism of the identity object,
\[
\tr_L((\alpha \xz A^\#) \xo i_{A^\#}): \Io_\Iz \To \Io_\Iz
\]
and has the graphical representation
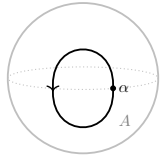
\[
\begin{tz}[scale=0.5,xscale=-1]
  \draw[slice] (0,0) circle (2cm);
  \draw[tinydash] (-2,0) arc (180:360:2 and 0.3);
  \draw[tinydash] (2,0) arc (0:180:2 and 0.3);
  \begin{scope}[yshift=-0.27cm]
  \draw[wire, arrow data={0.515}{>}] (-0.8,0) to (-0.8, 0.1) to [out=up, in=up, looseness=2] (0.8,0.1) to (0.8,-0.1) to [out=down, in=down, looseness=2] (-0.8,-0.1) to (-0.8, 0);
  \end{scope}
  \node[dot] at (-0.8,-0.27){};
  \node[ right,tmor] at (-0.8,-0.27){$\alpha$};
     \node[obj,above left] at (-\sqt, -\sqt) {$A$};
\end{tz}
\]
As is evident from the graphical calculus (by pulling the wire around the back of the sphere), these left and right traces always agree in a pivotal 2-category.
\begin{proposition}[Equality of left and right traces {\cite[Lem 7.6]{BMS}}]\label{prop:leftrighttrace}
For an endomorphism $\alpha: f \To f$ of a 1-morphism $f: A \to B$ in a pivotal 2-category, the following left and right traces agree:
\[
\tr_R(e_B \xo (\alpha \xz B^\#)) = \tr_L((\alpha \xz A^\#) \xo i_{A^\#})
\]
\end{proposition}

\begin{proof}
The cusp is a $2$-isomorphism $C_A: {}^\# (\Io_{A^\#}) \To \Io_{A}$.  Given a 2-endomorphism $\beta: \Io_{A} \To \Io_{A}$, define its \emph{wrinkle dual} $\beta^+ : \Io_{A^\#} \To \Io_{A^\#}$ as follows:
\[
\beta^+ :=C_{A^\#} \xo {}^\#\beta \xo C_{A^\#}^{-1}.
\]
(Graphically the wrinkle dual of a 2-endomorphism is obtained by taking a surface, introducing a `wrinkle', that is a cusp and its inverse, and then drawing the 2-endomorphism on `the back side' of the wrinkle.)  By imitating the usual graphical proof that ribbon categories are spherical, while keeping track of an underlying surface, one can show that the surface ribbon condition implies the following crucial `circular wrinkle equation':
\[
\tr_R(\alpha^\#)= \tr_L(\alpha)^+.
\]
A graphical rendition of that argument (utilizing somewhat different conventions) appears in~\cite[Fig 63]{BMS}.  The swallowtail equation directly implies the following relation between a 2-endomorphism $\beta: \Io_{A} \To \Io_{A}$ and its wrinkle dual:
\[
(A\xz \beta^+) \xo i_{A^\#} = (\beta \xz A^\#) \xo i_{A^\#}
\]

The result now follows from a chain of equalities:
\begin{align*}
\tr_R(e_B \xo (\alpha \xz B^\#)) &= \tr_R(e_A \xo (A \xz \alpha^\#)) \\
&= \tr_R(e_A \xo (A \xz \tr_R(\alpha^\#)) \\
&= \tr_L((A\xz \tr_R(\alpha^\#)) \xo i_{A^\#}) \\
&= \tr_L( (A\xz \tr_L(\alpha)^+)\xo i_{A^\#}) \\
&= \tr_L((\tr_L(\alpha) \xz A^\#) \xo i_{A^\#}) \\
&= \tr_L((\alpha\xz A^\#) \xo i_{A^\#})
\end{align*}
\nid Here the first equality follows from planar pivotality and the cusp-induced equation \linebreak$\tr_R({}^\#(\Io_{A^\#})) = \tr_R(\Io_{A})$.  The third equality holds because $\tr_R(\beta) = \tr_L(\beta^*)$ for any 2-morphism $\beta$, and the $\ast$-mate of a 2-endomorphism of an identity 1-morphism is exactly the original 2-endomorphism.  The fourth equality applies the circular wrinkle equation.  The fifth is the aforementioned consequence of the swallowtail equation.
\end{proof}

\begin{definition}[Back 2-spherical trace]\label{def:spheretrace}
In a pivotal 2-category, the \emph{back 2-spherical trace} of an endomorphism $\alpha: f \To f$ of a 1-morphism $f: A \to B$ is
\[
\Tr_B(\alpha) := \tr_R\left[e_B \xo (\alpha \xz B^\#)\right] = \tr_L\left[(\alpha \xz A^\#) \xo i_{A^\#}\right]: \Io_\Iz \To \Io_\Iz.
\]
\end{definition}

\begin{remark}[Circular trace equality is not sphericality]
One might be tempted to think of the equivalence of the two traces in Definition~\ref{def:spheretrace} as a `sphericality' condition on the pivotal 2-category, but that is not the situation.  Indeed, the equivalence of those traces is a consequence purely of the pivotal structure on the 2-category, that is of the natural object-level involutive duality structure.  There is, as described in the next section, a \emph{further} sphericality condition that can be imposed on a pivotal 2-category, analogous to the sphericality condition that can be imposed on a pivotal 1-category.
\end{remark}

\begin{remark}[Circular trace equality implies surface ribbon condition]\label{rem:traceimpliestwist}
That left and right traces agree, from Proposition~\ref{prop:leftrighttrace}, is one of the main properties of pivotal 2-categories we will use later in analyzing the state sum.  We will not, in particular, directly use the surface ribbon condition C8 from the Definition~\ref{def:pivotal} of pivotal structures.  It may therefore seem like we have not used all the structure available in pivotal 2-categories, but that is not the case: following~\cite[Lem 7.6 \& Lem 7.15]{BMS}, it can be shown, for a presemisimple monoidal planar pivotal 2-category with a pivotal structure possibly not satisfying the surface ribbon condition, that the surface ribbon condition is actually equivalent to the condition that left and right traces agree.  (This result is a generalization of the fact that a semisimple spherical braided pivotal 1-category is always ribbon~\cite[Prop A.4]{CategorifiedTrace}.)
\end{remark}

\subsection{Spherical $2$-categories}

\subsubsection{The definition of sphericality}

Recall that a pivotal 1-category is called `spherical' when the left and right `circular' traces of a 1-endomorphism $f: A \to A$ agree:
\[
\begin{tz}[scale=0.5] 
\draw[wire, arrow data = {0.1}{>}, arrow data = {0.5}{>}, arrow data={0.93}{>}] (0.5,1.5) to (0.5, 1.75) to [out=up, in=up, looseness=2] (-0.5, 1.75) to (-0.5, 1.25) to [out=down, in=down, looseness=2] (0.5, 1.25) to (0.5, 1.5);
\node[dot] at (0.5,1.5){};
\node[tmor, right] at (0.5, 1.5) {$f$};
\node[obj, black, below right] at (0.45,1.15){$A$};
\end{tz}
~~ = ~~
\begin{tz}[scale=0.5,xscale=-1] 
\draw[wire, arrow data = {0.1}{>}, arrow data = {0.5}{>}, arrow data={0.93}{>}] (0.5,1.5) to (0.5, 1.75) to [out=up, in=up, looseness=2] (-0.5, 1.75) to (-0.5, 1.25) to [out=down, in=down, looseness=2] (0.5, 1.25) to (0.5, 1.5);
\node[dot] at (0.5,1.5){};
\node[tmor, left] at (0.5, 1.5) {$f$};
\node[obj, black, below left] at (0.5,1.15){$A$};
\end{tz}
\]
The terminology is motivated by the idea that those traces would agree if the corresponding string diagrams could move freely on a \emph{2-sphere}.  For emphasis then, we might say that the `sphericality' condition on a pivotal 1-category is `2-sphericality'.

The analogous condition on a pivotal 2-category $\cC$ arises as follows.  Given an object $A$ of $\cC$ and a 2-endomorphism $\alpha: \Io_A \To \Io_A$, we may form the back 2-spherical trace $\Tr_B(\alpha): \Io_\Iz \To \Io_\Iz$ from Definition~\ref{def:spheretrace}.  The corresponding graphical picture is
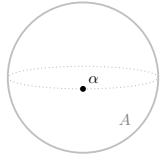
\[
\begin{tz}[scale=0.5]
  \draw[slice] (0,0) circle (2cm);
  \draw[tinydash] (-2,0) arc (180:360:2 and 0.3);
  \draw[tinydash] (2,0) arc (0:180:2 and 0.3);
  \node[dot] at (0,-0.3){};
  \node[above right,tmor] at (0,-0.3){$\alpha$};
   \node[obj,above left] at (\sqt, -\sqt) {$A$};
\end{tz}
\]
As the terminology suggests, there is another possible construction, namely the front 2-spherical trace:
\[
\begin{tz}[scale=0.5]
  \draw[slice] (0,0) circle (2cm);
  \draw[tinydash] (-2,0) arc (180:360:2 and 0.3);
  \draw[tinydash] (2,0) arc (0:180:2 and 0.3);
  \node[dot] at (0,0.3){};
  \node[below right,tmor] at (0,0.3){$\alpha$};
     \node[obj,above left] at (\sqt, -\sqt) {$A^\#$};
\end{tz}
\]
\begin{definition}[Front 2-spherical trace]
In a pivotal 2-category, the \emph{front 2-spherical trace} of an endomorphism $\alpha: f \To f$ of a 1-morphism $f: A \to B$ is
\begin{equation}\nonumber
\Tr_F(\alpha) := \tr_R\left[e_{B^\#} \xo (B^\# \xz \alpha)\right] = \tr_L\left[(A^\# \xz \alpha) \xo i_{A}\right]: \Io_\Iz \To \Io_\Iz.
\end{equation}
\end{definition}
\nid In a pivotal 2-category, it is not necessarily the case that the front and back traces coincide.  For instance, in the corresponding graphical surface calculus, the front trace $\Tr_F(\It_{\Io_A})$ and back trace $\Tr_B(\It_{\Io_A})$ correspond to $A$ labeled spheres of opposite coorientation in $\RR^3$.

Thus, the natural sphericality condition on a pivotal 2-category is that the front and back traces agree; this would be graphically sensible if, hypothetically, the corresponding surface diagrams could move freely in a \emph{3-sphere}.\footnote{Note that this idea of `isotopy of surfaces on a 3-sphere' is merely a heuristic for understanding the definition of `3-sphericality'; even in the classical case of `2-spherical' 1-categories, it is not the case that isotopy on the sphere faithfully represents the categorical structure in question~\cite[Sec 4.3]{Selinger}.}  We might therefore say that the `sphericality' condition on a pivotal 2-category is `3-sphericality'.

\begin{definition}[Spherical 2-category] \label{def:spherical}
A pivotal 2-category $\cC$ is \emph{spherical} if for every object $A$ of $\cC$ and every 2-endomorphism $\alpha: \Io_A \To \Io_A$, the front and back traces of $\alpha$ agree:
\begin{equation}\nonumber
\Tr_F(\alpha) = \Tr_B(\alpha)
\end{equation}
\end{definition}

\nid Note that it is a consequence of the equality of front and back traces for 2-endomorphisms of the form $\alpha: \Io_A \To \Io_A$ that the front and back traces in fact agree for any 2-morphism of the form $\alpha: f \To f$.  We will refer to a pivotal 2-category that is spherical simply as a `spherical 2-category'.  In a spherical 2-category, we will write $\Tr(\alpha)$ for the front or equivalently back 2-spherical trace of a 2-morphism $\alpha: f \To f$ and will refer to it simply as the `trace'.

\begin{example}[2-group 2-representations]\label{eg:tworepspherical}
For a finite 2-group $(\pi_1,\pi_2)$, we expect the (symmetric) fusion 2-category $\tRep(\pi_1,\pi_2) := [(\ast,\pi_1,\pi_2),\tVect]$ (see Construction~\ref{con:2group2rep}) inherits a spherical structure from the spherical structure of $\tVect$.
\end{example}

\begin{example}[2-group-graded 2-vector spaces]\label{eg:twovectspherical}
For a finite 2-group $(\pi_1,\pi_2)$ and a 4-cocycle $\omega \in Z^4((\pi_1,\pi_2); k^*)$, the fusion 2-category $\tVect^\omega(\pi_1,\pi_2)$ of $\omega$-twisted $(\pi_1,\pi_2)$-graded 2-vector spaces (see Construction~\ref{con:2groupgraded2vecttwisted}) should admit a canonical spherical structure as follows.  We expect that every 3-group, seen as a monoidal 2-category, admits a canonical (up to an appropriate notion of equivalence) spherical structure.  In particular, the 3-group $(\pi_1,\pi_2, k^*)$ (determined by the twisting $\omega$) has a spherical structure, which induces a spherical structure on the linearization $(\pi_1,\pi_2, k^*)_k$ and then in turn on its semisimple completion $\tVect^\omega(\pi_1,\pi_2)$.  More concretely, this spherical structure is obtained by taking a semi-strictification of the linear monoidal 2-category $(\pi_1,\pi_2, k^*)_k$ and using the group inverses of $\pi_1$ and $\pi_2$ to define the duals of objects and adjoints of 1-morphisms.
\end{example}

\begin{example}[Ribbon categories] \label{eg:ribbonspherical}
By Remark~\ref{rem:ribbon}, the delooping of a ribbon category is a pivotal 2-category.  As there is only one object, the (self-dual) identity of the monoidal 2-category, this pivotal structure is evidently spherical.
\end{example}

\begin{example}[Crossed-braided spherical categories]\label{eg:gradedbraidedspherical}
Construction~\ref{con:gradedbraidedfusion} produces a fusion 2-category $\tc{C}$ from a crossed-braided fusion 1-category $\oc{C}$.  A spherical structure (in the ordinary 1-categorical sense) on the $G$-crossed-braided category $\oc{C}$ induces a pivotal structure (in the 2-categorical sense of Definition~\ref{def:pivotal}) on the corresponding fusion 2-category $\tc{C}$.  (See Cui~\cite[Sec 6]{Cui} for a discussion of the duals of objects and adjoints of 1-morphisms in this case.)  Note well that the 1-categorical sphericality condition on the braided 1-category $\oc{C}$ provides an equality of left and right traces, which corresponds to the equality of left and right traces (see Proposition~\ref{prop:leftrighttrace}) in a pivotal 2-category; that spherical condition a priori has nothing to do with 2-categorical sphericality.  Nevertheless, in this particular case, because the $G$-action fixes the tensor unit of the 1-category $\oc{C}$, it follows that the pivotal 2-category $\tc{C}$ associated to a $G$-crossed-braided fusion category is in fact always a spherical fusion 2-category.
\end{example}

\begin{remark}[Mackaay's notion of sphericality] \label{rem:mackaayspherical}
Mackaay~\cite[Def 2.8]{Mackaay} defined a notion of `spherical monoidal 2-category'; a presemisimple monoidal 2-category that is Mackaay-spherical is spherical in our sense (Definition~\ref{def:spherical}): conditions C1--C7 of Definition~\ref{def:pivotal} follow from~\cite[Def 2.3]{Mackaay}, condition C8 of Definition~\ref{def:pivotal} follows using Remark~\ref{rem:traceimpliestwist} from~\cite[Def 2.7 \& Lem 2.12]{Mackaay}, and the sphericality condition of Definition~\ref{def:spherical} follows from~\cite[Lem 2.13]{Mackaay}.

However, note that Mackaay's notion of `spherical monoidal 2-category' is much more restrictive than our (Definition~\ref{def:spherical}) notion of sphericality (i.e.\ 3-sphericality) and Mackaay's notion does not correspond to a graphical calculus of surface diagrams moving on a 3-sphere.  Mackaay's definition insists on the existence of a 2-isomorphism between the two `categorical traces' $e_{A^\#}\xo (A^\#\xz f) \xo i_A$ and $e_A\xo (f\xz A^\#)\xo i_{A^\#}$; that 2-isomorphism would be sensible if the corresponding surface diagrams lived in $S^2 \times [0,1]$ rather than in $S^3$, but that is not always the case in relevant examples. For instance, as in Example~\ref{eg:gradedbraidedspherical}, the monoidal 2-category associated to a crossed-braided category is spherical (in our 3-spherical sense) but does not satisfy Mackaay's $S^2 \times [0,1]$-sphericality condition unless the $G$-action on the identity-graded component $\oc{C}_1$ is trivial; see~\cite[Sec 6]{Cui}.
\end{remark}

\subsubsection{Dimension in spherical prefusion $2$-categories}

In a prefusion 2-category $\tc{C}$, the monoidal unit $\Iz$ is simple, and hence the 1-morphism $\Io_\Iz$ is simple.  We therefore may and will identify $\End_{\tc{C}}(\Io_\Iz)$ with the base field $k$.  In a spherical prefusion 2-category $\cC$, then, the trace $\Tr(\alpha)$ of any 2-morphism $\alpha: f \To f$ may be canonically considered as an element of $k$.

\begin{definition}[Dimensions of 1-morphisms and objects] \label{def:quantumdimension}
In a spherical prefusion $2$-category, the \emph{dimension} of a 1-morphism $f: A \to B$ is
\begin{equation} \nonumber
\dim(f) := \Tr(\It_f) \in k.
\end{equation}
The \emph{dimension} of an object $A$ is
\begin{equation} \nonumber
\dim(A) := \dim(\Io_A) = \Tr(\It_{\Io_A}) \in k.
\end{equation}
\end{definition}
\nid%
Graphically, the dimension of an object $A$ is represented by an $A$-labeled sphere; the dimension of a $1$-morphism $f$ is represented by an $A$-labeled sphere with an $f$-labeled loop on it.  By pivotality, the $f$-label may be placed either on the right or left of the loop, and by sphericality, the $A$-label and the loop may be placed on either the front or back of the sphere.  Note that the conditions of planar pivotality ensure that the dimensions of isomorphic 1-morphisms are the same.

Note well that while the dimension of an object or 1-morphism in a prefusion 2-category $\tc{C}$ depends on the monoidal structure, the `dimension of $\tc{C}$' itself refers to the dimension of the underlying presemisimple 2-category, see Definition~\ref{def:dimension2cat}, and therefore does not depend on the monoidal structure; this is rather in contrast with what one might expect from the fact that the dimension of a fusion 1-category certainly does depend on its monoidal structure.

\begin{remark}[Pivotal adjoint equivalence preserves dimension]
As mentioned in Warning~\ref{warn:pivotal}, it is not in general the case that equivalent objects have the same dimension.  Indeed, two objects will have the same dimension only when they are equivalent in a way compatible with the pivotal structure---we might call such an equivalence a `pivotal adjoint equivalence'.  In the constructions that follow, including in the formula for the state sum, we will use dimensions of objects, and so these constructions depend on choices of representative objects in each equivalence class of objects---however, we will show that the overall resulting state sum is independent of those choices.
\end{remark}
\nid%
We now show that in spherical prefusion 2-categories, the dimensions of simple objects and 1-morphisms do not vanish.

\begin{lemma}[Planar trace is nonzero] \label{lem:planartracenonzero}
For $f: A \to B$ a simple 1-morphism into a simple object $B$, in a planar pivotal presemisimple 2-category $\tc{C}$, the right planar trace $\tr_R(\It_f) \in \End_{\tc{C}}(\Io_B) = k$ is nonzero.
\end{lemma}
\begin{proof}
By definition, the trace in question is the composite of the unit $\eta_{f^*}: \Io_B \To f\xo f^*$ and the counit $\epsilon_f: f\xo f^* \To \Io_B$.  Note that by adjunction $\Hom_{\tc{C}}(f \xo f^*,\Io_B) \cong \Hom_{\tc{C}}(f,f) \iso k$ and similarly (or by semisimplicity) $\Hom_{\tc{C}}(\Io_B,f \xo f^*) \iso k$.  Now both $\eta_{f^*}$ and $\epsilon_f$ must be nonzero, as they are a unit and a counit.  By Proposition~\ref{prop:simpleiffidentity}, the identity $\Io_B$ is simple; there must therefore be nonzero maps $x: \Io_B \To f \xo f^*$ and $y: f \xo f^* \To \Io_B$ whose composite is the identity.  Thus $\eta_{f^*}$ and $\epsilon_f$ must be nonzero-proportional to $x$ and $y$ respectively, and by bilinearity of composition, it follows that the trace is nonzero-proportional to the identity, and therefore itself nonzero.
\end{proof}

\begin{prop}[Dimension of simple 1-morphism is nonzero] \label{prop:simplenonzero}
In a spherical prefusion $2$-category $\tc{C}$, the dimension of a simple $1$-morphism $f: A \to B$ is nonzero.
\end{prop}
\begin{proof}
By the duality between $B$ and $B^\#$, we have $\Hom_{\tc{C}}(e_B \xo (f \xz B^\#), e_B \xo (f \xz B^\#)) \cong \Hom_{\tc{C}}(f,f) = k$.  Thus $e_B \xo (f \xz B^\#) : A \xz B^\# \to \Iz$ is a simple 1-morphism to the (simple) identity.  By Lemma~\ref{lem:planartracenonzero}, the right trace $\tr_R(\It_{e_B \xo (f \xz B^\#)})$ is nonzero; but the dimension of $f$ is $\Tr(\It_f) = \tr_R(e_B \xo (\It_f \xz B^\#)) = \tr_R(\It_{e_B \xo (f \xz B^\#)})$.
\end{proof}

\begin{corollary}[Dimension of simple object is nonzero]
In a spherical prefusion $2$-category $\tc{C}$, the dimension of a simple object $A$ is nonzero.
\end{corollary}
\begin{proof}
The dimension of the object is by definition the dimension of its identity, which is a simple 1-morphism.
\end{proof}
\nid%
In addition to nonzero dimensions, the trace also provides a nondegenerate pairing on 2-morphism spaces.
\begin{definition}[Pairing on Hom sets] \label{def:pairing}
For  1-morphisms $f,g: A \to B$ in a spherical prefusion $2$-category, the pairing $\langle\cdot, \cdot\rangle :\Hom_{\tc{C}}(f,g)\otimes \Hom_{\tc{C}}(g,f) \to k$ is given by
\[
\langle \alpha, \beta\rangle := \Tr(\alpha\cdot \beta) = \Tr(\beta\cdot\alpha).
\]
\end{definition}
\begin{proposition}[Pairing on Hom is nondegenerate] \label{prop:pairing}
In a spherical prefusion $2$-category, the pairing $\langle\cdot, \cdot\rangle :\Hom_{\tc{C}}(f,g)\otimes \Hom_{\tc{C}}(g,f) \to k$ is nondegenerate.
\end{proposition}
\begin{proof}
Pick a collection $\{s_i\}$ of simple 1-morphisms $A \to B$, one in each isomorphism class.  Without loss of generality we may assume that $f = \oplus_{i \in I} s_{k_i}$ and $g = \oplus_{j \in J} s_{l_j}$.  Let $p_i: f\leftrightarrows s_{k_i}: \iota_i$ and $p_j' : g \leftrightarrows s_{l_j}: \iota_j'$ be the inclusion and projection $2$-morphisms.

Suppose $\alpha: f \To g$ is a 2-morphism such that $\langle \alpha, \beta \rangle = 0$ for all $\beta: g \To f$.  Every 2-morphism $g \To f$ is a linear combination of the 2-morphisms $\iota_i \xt p_j'$ for which $k_i = l_j$.  Given $i$ and $j$ with $k_i = l_j$, by assumption and planar pivotality $0 = \langle \alpha, \iota_i \xt p_j' \rangle = \Tr(\alpha \xt \iota_i \xt p_j') = \Tr(p_j' \xt \alpha \xt \iota_i)$.  Now $p_j' \xt \alpha \xt \iota_i: s_{k_i} \To s_{l_j}$ is $\lambda_{i,j} \It_{s_k}$ for some scalar $\lambda_{i,j}$, where $k=k_i=l_j$.  By bilinearity of composition, $\Tr(p_j' \xt \alpha \xt \iota_i) = \lambda_{i,j} \Tr(\It_{s_k}) = \lambda_{i,j} \dim(s_k)$.  As $\dim(s_k)$ is nonzero, the scalar $\lambda_{i,j}$ is forced to be zero for all $i$ and $j$.  By local semisimplicity, $p_j' \xt \alpha \xt \iota_i = 0$ for all $i$ and $j$ implies that $\alpha$ itself is zero.
\end{proof}

\newpage

\addtocontents{toc}{\protect\vspace{8pt}}

\section{A state-sum invariant of singular piecewise-linear 4-manifolds} \label{sec:statesum}

Given a spherical prefusion 2-category $\tc{C}$ over an algebraically closed field $k$ of characteristic zero, we define a $k$-valued invariant of closed oriented singular piecewise-linear 4-manifolds.  We expect that a prefusion 2-category is a 4-dualizable object of an appropriate 4-category of linear monoidal 2-categories, and a spherical prefusion 2-category moreover has a canonical $\mathrm{SO}(4)$-fixed point structure.  There would therefore be a corresponding oriented local topological field theory, and we would expect our invariant restricts to the closed 4-manifold invariant of that field theory.

\subsection{Combinatorial, piecewise-linear, and smooth manifolds}
\skiptocparagraph{Simplicial complexes and piecewise-linear maps}
Recall that a finite simplicial complex $K$ is a finite set $K_0$ (of `vertices') together with a collection of subsets of $K_0$ (the `simplices' of $K$), such that a sub-subset of an element of the collection is again in the collection, and such that all singleton sets are in the collection.  Each simplex $\sigma$ of $K$ determines a geometric simplex $|\sigma|$ in $\RR^{K_0}$, namely the convex hull of the basis vectors corresponding to the vertices of that simplex; the standard geometric realization $|K|$ of a simplicial complex $K$ is the union in $\RR^{K_0}$ of the geometric simplices corresponding to the simplices of $K$.  More generally, a geometric realization of the complex $K$ in $\RR^n$ is any subspace constructed as follows: choose an embedding of the vertices $K_0$ in $\RR^n$ such that for any simplex $\sigma$ of $K$, the embedded vertices of $\sigma$ are linearly independent (and thus determine a corresponding geometric simplex as their convex hull), and such that for any two simplices $\sigma$ and $\tau$ of $K$, the intersection of the corresponding geometric simplices in $\RR^n$ is a face of each; the union of all the geometric simplices in $\RR^n$ corresponding to simplices of $K$ is the realization determined by the given vertex embedding.  A subdivision of a simplicial complex $K$ is a simplicial complex $K'$ together with a geometric realization of $K'$ that is, as a subspace of $\RR^{K_0}$, the standard geometric realization $|K|$.  A piecewise-linear map from a simplicial complex $K$ to a simplicial complex $L$ is a map $f: |K| \ra |L|$ such that there exists a subdivision $K'$ of $K$ such that the map $f$ is linear when restricted to each simplex of $K'$.

\skiptocparagraph{Combinatorial manifolds}

A combinatorial $n$-ball is a simplicial complex piecewise-linearly homeomorphic to the standard $n$-simplex, and a combinatorial $n$-sphere is a simplicial complex piecewise-linearly homeomorphic to the boundary of the standard $(n+1)$-simplex.  
\begin{definition}[Combinatorial manifold]
A \emph{combinatorial $n$-manifold} is a finite simplicial complex such that the link of every vertex is a combinatorial $(n-1)$-sphere.
\end{definition}
\nid A `combinatorial $n$-manifold with boundary' is allowed to have vertices with links that are combinatorial $(n-1)$-balls; the boundary is the subcomplex of simplices whose vertices have links that are balls. Note that in a combinatorial $n$-manifold with boundary, the link of every $k$-simplex is necessarily a combinatorial $(n-k-1)$-sphere or a combinatorial $(n-k-1)$-ball.
\begin{definition}[Singular combinatorial manifold]
A \emph{singular combinatorial $n$-manifold} is a finite simplicial complex such that the link of every vertex is a combinatorial $(n-1)$-manifold.
\end{definition}
\nid A `singular combinatorial $n$-manifold with boundary' is allowed to have vertices with links that are combinatorial $(n-1)$-manifolds with boundary; the boundary is the subcomplex of simplices whose vertices have links that have nonempty boundary. Note that in a singular combinatorial $n$-manifold with boundary, the link of every $k$-simplex, for $k \geq 1$, is necessarily a combinatorial $(n-k-1)$-sphere or a combinatorial $(n-k-1)$-ball.

An orientation on a singular combinatorial $n$-manifold with boundary is a choice of orientation of each $n$-simplex such that for every $(n-1)$-simplex not in the boundary, the orientations induced from the two adjacent $n$-simplices are opposite.

\skiptocparagraph{Piecewise-linear manifolds}
A piecewise-linear (PL) manifold is a triple $(M,K,\phi)$ consisting of a topological manifold $M$, a finite simplicial complex $K$, and a homeomorphism $\phi: |K| \ra M$ from the geometric realization of the complex to the manifold; a piecewise-linear map $(M,K,\phi) \ra (M',K',\phi')$ is a map $M \to M'$ such that the induced map $|K| \to |K'|$ is piecewise-linear.  Evidently, the category of PL manifolds and PL maps is equivalent to the category of combinatorial manifolds and PL maps.  For convenience, then, we will work directly with combinatorial manifolds, and more generally with singular combinatorial manifolds.  By a `singular piecewise-linear manifold', or by merely `singular manifold', we will mean a singular combinatorial manifold.

\skiptocparagraph{Piecewise-linear versus smooth 4-manifolds}
In dimension 4 there is no functional difference between smooth and piecewise-linear structures on manifolds; thus our invariant of piecewise-linear 4-manifolds immediately provides a corresponding invariant of smooth 4-manifolds.  More specifically, Whitehead~\cite{whitehead-c1complexes40} proved that in any dimension, given a compact closed smooth manifold $M$, there is a combinatorial manifold $K$ for which there exists a piecewise-differentiable homeomorphism from $|K|$ to $M$, and moreover such a combinatorial manifold is unique up to piecewise-linear homeomorphism.  (A homeomorphism from a combinatorial manifold $|K|$ to a smooth manifold $M$ is piecewise-differentiable if it is a smooth immersion when restricted to each simplex.)  This association provides a canonical map from the diffeomorphism classes of smooth manifolds to the piecewise-linear homeomorphism classes of piecewise-linear manifolds.  By work of Cerf~\cite{cerf-surlesdiffeo68}, Smale~\cite{smale-diffeo2sphere59}, Munkres~\cite{munkres-obstructions60}, and Hirsch--Mazur~\cite{hirsch-obstruction63,hirschmazur-smooths74}, this canonical map from smooth to piecewise-linear manifolds is a bijection in dimension 4.

\subsection{Stellar and bistellar equivalence of singular combinatorial manifolds} \label{sec:stellarbistellar}

\skiptocparagraph{Stellar subdivision and stellar equivalence of simplicial complexes}

The stellar subdivision $S\Delta^k$ of the standard $k$-simplex $\Delta^k$ is obtained from the standard simplex by adding a new vertex in the interior and coning the boundary simplices to it; the resulting simplicial complex has $(k+1)$-many $k$-simplices and is concisely expressed as the join $\partial \Delta^k \star \{\ast\}$.  Given a simplicial complex $X$ and a $k$-simplex $A$ of $X$, recall that the star of $A$ can be expressed as the join $A \star \lk(A)$ of the simplex with its link $\lk(A)$.  The stellar subdivision $S_A X$ of the complex $X$ at $A$ is obtained by replacing the star of $A$ with the complex $SA \star \lk(A)$.  

Two finite simplicial complexes are called stellar equivalent if they are related by a finite zigzag of stellar subdivisions.  One of the first fundamental results of piecewise-linear topology is Alexander's theorem that stellar subdivision generates piecewise-linear equivalence.
\begin{theorem}[Stellar equivalence of simplicial complexes~{\cite{alexander-combinatorialtheory,newman-foundations,lickorish-simplicial}}] \label{thm:alexander}
Two finite simplicial complexes are piecewise-linear homeomorphic if and only if they are stellar equivalent.
\end{theorem}

\nid To produce a piecewise-linear invariant, whether of manifolds or otherwise, it therefore suffices to show that the invariant is unaffected by stellar subdivision.  Unfortunately, even in a fixed dimension and in the context of manifolds, there are infinitely many distinct types of stellar subdivision, depending on the combinatorial structure of the link of the simplex being subdivided.  Thus it is typically impractical to check invariance via stellar subdivision moves.  Much more convenient is the finite collection of bistellar moves.

\skiptocparagraph{Bistellar moves and bistellar equivalence of combinatorial manifolds}

Given a combinatorial $n$-manifold, one may obtain a piecewise-linearly homeomorphic combinatorial manifold by cone-subdividing an $n$-simplex: remove an $n$-simplex $\sigma$ and replace it with the $(n+1)$ $n$-simplices produced by coning the boundary of $\sigma$.  This top-dimensional stellar subdivision is also called a `$(1,n+1)$-bistellar move' and may be thought of as replacing a single simplex by the complement of a simplex in the standard combinatorial $n$-sphere $\partial \Delta^{n+1}$.  More generally, for $(p,q)$ with $p+q = n+2$, let $P$ denote the codimension zero combinatorial submanifold of $\partial \Delta^{n+1}$ with $p$ $n$-simplices, and let $Q$ denote the complementary codimension zero combinatorial submanifold with $q$ $n$-simplices.  Two combinatorial $n$-manifolds $K$ and $K'$ are related by a $(p,q)$-bistellar (or `$(p,q)$-Pachner') move if $K'$ is obtained from $K$ by removing a codimension zero combinatorial submanifold isomorphic to $P$ and replacing it by one isomorphic to $Q$.  

Two combinatorial manifolds are called bistellar equivalent if they are related by a finite series of bistellar moves.  The fundamental result of combinatorial manifold theory is Pachner's theorem that bistellar moves generate piecewise-linear equivalence: 
\begin{theorem}[Bistellar equivalence of combinatorial manifolds~{\cite{pachner,lickorish-simplicial}}]
Two combinatorial $n$-manifolds are piecewise-linearly homeomorphic if and only if they are bistellar equivalent.
\end{theorem}

\nid To produce a piecewise-linear invariant of combinatorial manifolds, it therefore suffices to check that the invariant is unaffected by each of the finitely many bistellar moves.

\skiptocparagraph{Bistellar equivalence of singular combinatorial 4-manifolds}

In a combinatorial 3-manifold, the link of a vertex is a 2-sphere; in a singular combinatorial 3-manifold, the link of a vertex is allowed to be any surface.  Barrett and Westbury proved that Pachner's theorem extends to singular combinatorial 3-manifolds.
\begin{theorem}[Bistellar equivalence of singular 3-manifolds~{\cite{BW}}]
Two singular combinatorial 3-manifolds are piecewise-linearly homeomorphic if and only if they are bistellar equivalent.
\end{theorem}
\nid As the invariance of the Turaev--Viro--Barrett--Westbury state sum is established via invariance under bistellar moves, this result extended the state sum invariant to singular 3-manifolds.

In a singular combinatorial 4-manifold, the link of a vertex is allowed to be any combinatorial 3-manifold.  We prove that Pachner's theorem extends to this case.
\begin{theorem}[Bistellar equivalence of singular 4-manifolds]\label{thm:singular}
Two singular combinatorial 4-manifolds are piecewise-linearly homeomorphic if and only if they are bistellar equvialent.
\end{theorem}
\begin{proof}
Of course bistellar equivalent singular combinatorial 4-manifolds are piecewise-linearly homeomorphic.  Given two piecewise-linearly homeomorphic singular combinatorial 4-mani\-\linebreak folds, they are stellar equivalent by Theorem~\ref{thm:alexander}.  It suffices therefore to show that any stellar subdivision of a singular combinatorial 4-manifold $M$ is a bistellar equivalence.  Stellar subdivision at a 0-simplex is trivial, and stellar subdivision at a 4-simplex is itself a bistellar move.  Let $A$ be a $k$-simplex, with $1 \leq k \leq 3$; note that the link $\lk(A)$ is a combinatorial $i$-sphere, for $0 \leq i \leq 2$, and is therefore shellable.  (The first nonshellable spheres arise in dimension 3~\cite{vince}.)  Because the link $\lk(A)$ is shellable, the stellar subdivision $S_A M$ is bistellar equivalent to $M$ by the proof of~\cite[Cor 5.8]{lickorish-simplicial}.  (Note that Lickorish's proof of bistellar equivalence given shellable links applies in our context of singular manifolds.)
\end{proof}
\nid Our state-sum will be invariant under bistellar moves, and though its invariance depends crucially on the assumption that the links of 1-, 2-, and 3-simplices are spheres, it will not require sphericality of vertex links.  Thus, the state sum will give an invariant not only of piecewise-linear 4-manifolds but also of singular piecewise-linear 4-manifolds.

\begin{convention*}[Manifolds may be singular]
In the remainder of the paper, every time we refer to a `combinatorial 4-manifold' we implicitly mean `singular combinatorial 4-manifold'.  The allowance of singularities will only be made explicit in certain key statements.
\end{convention*}

\subsection{States, associated states, the 10j action, and the partition function}
\label{sec:10j}

Given a 4-dimensional topological field theory $Z$, the numerical invariant $Z(M)$ of a closed 4-manifold $M$ may be thought of as the `partition function' of the theory on that `spacetime'.  In the spirit of the path integral, we might imagine this invariant would be computed as an integral of an appropriate action, integrated over all possible physical histories in that spacetime and weighted by a normalization factor for each history.  In the case of a topological field theory $Z_{\tc{C}}$ associated to a spherical prefusion 2-category $\tc{C}$, a possible `physical history' of a combinatorial 4-manifold $M$ is given by an assignment of an object of $\tc{C}$ to each 1-simplex of $M$, a compatible 1-morphism of $\tc{C}$ to each 2-simplex of $M$, and a compatible 2-morphism to each 3-simplex of $M$.

In practice, instead of considering all such assignments, we may reduce the path integral to a `path sum', also called a `state sum', by insisting that the labels be respectively by simple objects, simple 1-morphisms, and elements of chosen bases of the 2-morphism vector spaces.  The `action' term in the state sum will be given as a product of the 10j symbols that give the pentagonator structure data of the fusion operation of the fusion 2-category.  The normalization factor will be an appropriate ratio of quantum dimensions of the simple objects and simple 1-morphisms in the given labeling.  The state sum, finally, is the average, over all possible labelings of the manifold, weighted by these quantum dimension factors, of a product of the 10j symbols of the fusion 2-category.  (Compare the Barrett--Westbury--Turaev--Viro state sum for combinatorial 3-manifolds based on a fusion 1-category, which is an average, weighted by quantum dimensions of objects, of products of the 6j symbols defining the associator data of the fusion category.)

\skiptocparagraph{States of a 4-manifold}

Let $\tc{C}$ be a spherical prefusion 2-category and let $K$ be an oriented (singular) combinatorial 4-manifold.  As mentioned, roughly speaking a `state' of the manifold would be a labeling of 1-simplices by simple objects, of 2-simplices by simple 1-morphisms, and of 3-simplices by 2-morphism basis elements.  In fact, to streamline later proofs, we will proceed by only labeling 1-simples and 2-simplices, and implicitly sum over the 2-morphism bases as part of an associated state construction.  (One may think of this `partial state', with only 1-simplices and 2-simplices labeled, as corresponding to a state of the neighborhood of the 2-skeleton of the 4-manifold; the associated state construction will account for the effect of filling in the 3-simplices, while the action term will encode the process of filling in the 4-simplices.)  

As the morphisms and 2-morphisms of the fusion 2-category $\tc{C}$ are of course directed, it is convenient to also have a consistent choice of direction on the simplices of the combinatorial 4-manifold $K$; this is achieved by choosing an ordering on the vertices of $K$.
\begin{definition}[Ordered combinatorial manifold]
An \emph{ordered oriented combinatorial 4-manifold} $K^o$ is an oriented combinatorial 4-manifold $K$ with a choice of total order $o$ on its set of 0-simplices $K_0$.
\end{definition}
\nid Note that the `global' order $o$ on the vertices of $K$ induces a `local' order of the vertices of any $n$-simplex $\tau \in K_n$.  For an order-preserving inclusion $[k] := \{0,1,\ldots,k\} \hookrightarrow \{0,1,\ldots,n\} =: [n]$, with image $\{j_0,\ldots,j_k\} \subset [n]$, there is a well-defined $k$-face of $\tau$ determined by the vertices $\{j_0,\ldots,j_k\}$ of $\tau$; that face will be denoted $\partial^o_{[j_0,\ldots,j_k]} \tau \in K_k$.  These face maps give the ordered oriented combinatorial 4-manifold $K^o$ the structure of a semisimplicial set $K^o: \Delta^\op_+ \ra \mathrm{Set}$ with $K^o([n]) = K_n$.  (Here $\Delta_+$ is the category of non-empty ordered finite sets with injective order-preserving maps.)  

The total order $o$ in an ordered oriented combinatorial manifold $K^o$ is not required to respect the orientation in any particular way, and so there are a collection of signs governing how the order and the orientation interact: there is a function $\epsilon_o : K_4 \to \{+1,-1\}$ with $\epsilon_o(\sIV) = +1$ if and only if the orientation of the 4-simplex $\sIV$ agrees with the orientation induced by the vertex order $o$, and for every 4-simplex $\sIV$ there is a function $\epsilon_o^\sIV: \{\sIII \in K_3~|~\sIII \subseteq \sIV\} \to \{+1,-1\}$ with $\epsilon_o^{\sIV}(\sIII) = +1$ if and only if the orientation of the face $\sIII \subseteq \sIV$ induced from the orientation of $\sIV$ agrees with the one induced from the total vertex order. (Recall that an orientation of an $n$-simplex is an even-permutation equivalence class of orderings of its vertices; we will denote the orientation associated to the vertex order $v_0, \ldots, v_n$ by $\lan v_0, \ldots, v_n \ran$ and the opposite orientation by $-\lan v_0, \ldots, v_n \ran$.  The induced orientation of the face of an $n$-simplex opposite to the vertex $v_i$ is $(-1)^i \lan v_0, \ldots, \widehat{v_i}, \ldots v_n\ran$.)

As we will only be labeling the 1-simplices and 2-simplices of $K$, we will only be concerned with the associated 2-truncated semisimplicial set $K^o_{(2)}: \Delta^\op_{+,2} \ra \mathrm{Set}$, which is the restriction of $K^o$ to the full subcategory $\Delta^\op_{+,2}$ on the objects $\{[0],[1],[2]\}$.  A spherical prefusion 2-category $\tc{C}$ also determines a 2-truncated semisimplicial set $\Delta \tc{C}: \Delta^\op_{+,2} \ra \mathrm{Set}$ with $\Delta \tc{C}_0 = \{*\}$, $\Delta \tc{C}_1 = \{\text{simple objects of }\tc{C}\}$, and
\[
\Delta \tc{C}_{2} = \left\{ \left(A,B,C,f\right) ~\middle |~ A,B,C \in \Delta\tc{C}_1, f \text{ a simple $1$-morphism in }\Hom(A\xz B, C)\right\}.
\]
Here the $[01]$, $[12]$, and $[02]$ faces of $(A,B,C,f)$ are respectively $A$, $B$, and $C$.  (Note that this truncated semisimplicial set is not finite---we address that issue later by picking a skeleton of the prefusion 2-category---and that equivalent prefusion 2-categories may produce nonisomorphic associated semisimplicial sets.)
\begin{definition}[State of a combinatorial manifold]
Given a spherical prefusion 2-category $\tc{C}$ and an ordered oriented combinatorial 4-manifold $K^o$, a \emph{$\tc{C}$-state} of $K^o$ is a natural transformation $\Gamma: K^o_{(2)} \To \Delta\tc{C}$.
\end{definition}
\nid Concretely, a $\tc{C}$-state is a function from the 1-simplices of $K$ to simple objects of $\tc{C}$ and a function from the 2-simplices of $K$ to simple 1-morphisms of $\tc{C}$ compatible with the face maps from 2- to 1-simplices.  We will denote the set of $\tc{C}$-states of $K^o$ by $[K^o,\Delta\tc{C}]$.

Given a $\tc{C}$-state $\Gamma: K^o_{(2)} \To \Delta\tc{C}$, the value of the state on a 1-simplex $\sI$ or 2-simplex $\sII$ is written simply as $\Gamma(\sI)$, respectively $\Gamma(\sII)$; it is convenient to have a succinct notation for the value of the state on the boundary simplices of 3- and 4-simplices.
\begin{notation}[State labels of 1- and 2-simplices in 3- and 4-simplices] \label{not:labels}
For a $\tc{C}$-state $\Gamma$, and $\s \in K_p$ a $p$-simplex for $p$ either 3 or 4:
\begin{align*}
[ij]_{\Gamma}^\s &:= \Gamma\left( \partial^o_{[ij]}\s\right)~\text{ for }0\leq i < j \leq p
\\
[ijk]_\Gamma^\s &:= \Gamma\left( \partial^o_{[ijk]} \s\right) \text{ for }0\leq i<j<k\leq p
\end{align*}
\end{notation}
\nid Furthermore, in a 3-simplex the four 2-simplices naturally divide into two composable pairs; following Mackaay~\cite{Mackaay}, we use a compact notation for the composites of those pairs of 2-simplices.
\begin{notation}[State labels of associated 3-simplices in 3- and 4-simplices] \label{not:parenthesisnotation}
For a $\tc{C}$-state $\Gamma$, and $\s \in K_p$ a $p$-simplex for $p$ either 3 or 4:
\begin{align*}
[(ijk)l]_\Gamma^\s &:= [ikl]_\Gamma^{\s} \xo \left(  [ijk]_\Gamma^\s \xz \Io_{[kl]_\Gamma^\s} \right)~\text{ for }0\leq i<j<k<l\leq p
\\
[i(jkl)]_{\Gamma}^\s &:=[ijl]_{\Gamma}^\s \xo \left( \Io_{[ij]_{\Gamma}^\s} \xz [jkl]_\Gamma^\s\right)~\text{ for }0\leq i<j<k<l\leq p
\end{align*}
\end{notation}
\nid Finally we have a shorthand for duals and adjoints of state labels:
\begin{notation}[Duals and adjoints of state labels] \label{not:duals}
For a $\tc{C}$-state $\Gamma$, and $\s \in K_p$ a $p$-simplex for $p$ either 3 or 4:
\begin{align*}
\conj{[ij]}_\Gamma^\s &:= \left( \vp [ij]_\Gamma^\s\right)^\# \text{  for }0\leq i< j \leq p 
\\
\conj{[ijk]}_\Gamma^\s &:= \left( \vp [ijk]_\Gamma^\s\right)^*\text{  for }0\leq i<j<k\leq p
\end{align*}
\end{notation}
\nid In all of these notations, we will often omit the super- and sub-scripts if the intended simplex and $\tc{C}$-labeling are clear.

\skiptocparagraph{The canonical associated state}

As defined, a $\tc{C}$-state $\Gamma: \K_{(2)}^o \To \Delta \tc{C}$ provides labels of the 2-simplices of $K$ by 1-morphisms.  Given a 3-simplex, there are vector spaces of associator-like 2-morphisms compatible with the given labels.
\begin{definition}[Associator state spaces] \label{def:defvectorspaceV}
For a $\tc{C}$-state $\Gamma$ on an ordered oriented combinatorial 4-manifold, the (positive and negative) \emph{associator state spaces} of a 3-simplex $\sIII \in K_3$ are the vector spaces
\begin{align*}
V^+(\Gamma, \sIII) &:= \Hom_{\tc{C}}\left(\vp [(012)3]_\Gamma^\sIII, [0(123)]_\Gamma^\sIII\right)
\\
V^-(\Gamma, \sIII) &:= \Hom_{\tc{C}}\left(\vp [0(123)]_\Gamma^\sIII, [(012)3]_\Gamma^\sIII \right)
\end{align*}
\end{definition}
\nid Omitting the $\sIII$ superscript and $\Gamma$ subscript, vectors in these spaces are denoted graphically as follows:
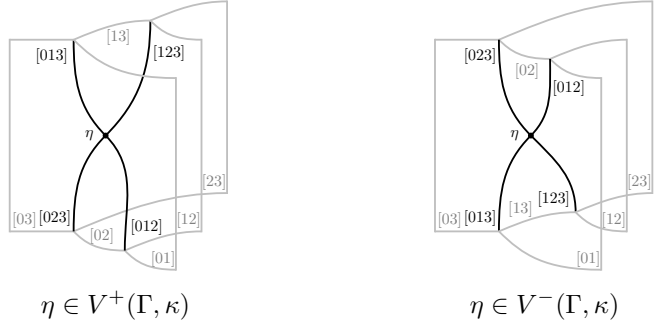
\begin{calign}
\nonumber
\begin{tz}[td,scale=1.7]
\begin{scope}[xyplane=0]
\draw[slice] (0,0) to [out=up, in=\dl] (0.5,1) to [out=up, in=\dl] (1,2) to (1,3);
\draw[slice] (1,0) to [out=up, in=\dr] (0.5,1);
\draw[slice] (2,0) to [out=up, in=\dr] (1,2);
\end{scope}
\begin{scope}[xyplane=\h, on layer=superfront]
\draw[slice] (0,0) to [out=up, in=\dl] (1,2) to (1,3);
\draw[slice] (1,0) to [out=up, in=\dl] (1.5,1) to [out=up, in=\dr] (1,2);
\draw[slice] (2,0) to [out=up, in=\dr] (1.5,1);
\end{scope}
\begin{scope}[xyplane=0.5*\h]
\end{scope}
\begin{scope}[xzplane=0]
\draw[slice,short] (0,0) to (0, \h);
\draw[slice,short] (1,0) to (1,\h);
\draw[slice,short] (2,0) to (2,\h);
\end{scope}
\begin{scope}[xzplane=3]
\draw[slice,short] (1,0) to (1,\h);
\end{scope}
\coordinate (A) at (1.5,1,0.5*\h);
\draw[wire] (1,0.5,0) to [out=up, in=\dr] (A) to [out=\ur, in=down] (1,1.5,\h);
\draw[wire] (2,1,0) to [out=up, in=\dl] (A) to [out=\ul, in=down] (2,1,\h);
\node[dot] at (A){};
\node[obj, above left] at (-0.08,0,-0.04) {$[01]$};
\node[obj, above left] at (-0.08,1,-0.04) {$[12]$};
\node[obj, above left] at (-0.08,2,-0.04) {$[23]$};
\node[obj, above right] at (3.05,1,-0.04){$[03]$};
\node[obj, above] at (1.4,0.6,-0.04){$[02]$};
\node[obj, below] at (1.5,1.5,\h-0.04){$[13]$};
\node[omor, above left] at (1.95,1,0){$[023]$};
\node[omor, above right] at (1,0.5,0.16){$[012]$};
\node[omor, below left] at (2,1,\h-0.04){$[013]$};
\node[omor, below right] at (1.05,1.5,\h-0.3){$[123]$};
\node[tmor, left] at ([xshift=-0.1cm]A){$\eta$};
\end{tz}
&
\begin{tz}[td,scale=1.7]
\begin{scope}[xyplane=\h]
\draw[slice] (0,0) to [out=up, in=\dl] (0.5,1) to [out=up, in=\dl] (1,2) to (1,3);
\draw[slice] (1,0) to [out=up, in=\dr] (0.5,1);
\draw[slice] (2,0) to [out=up, in=\dr] (1,2);
\end{scope}
\begin{scope}[xyplane=0, on layer=superfront]
\draw[slice] (0,0) to [out=up, in=\dl] (1,2) to (1,3);
\draw[slice] (1,0) to [out=up, in=\dl] (1.5,1) to [out=up, in=\dr] (1,2);
\draw[slice] (2,0) to [out=up, in=\dr] (1.5,1);
\end{scope}
\begin{scope}[xyplane=0.5*\h]
\end{scope}
\begin{scope}[xzplane=0]
\draw[slice,short] (0,0) to (0, \h);
\draw[slice,short] (1,0) to (1,\h);
\draw[slice,short] (2,0) to (2,\h);
\end{scope}
\begin{scope}[xzplane=3]
\draw[slice,short] (1,0) to (1,\h);
\end{scope}
\coordinate (A) at (1.5,1,0.5*\h);
\draw[wire] (1,0.5,\h) to [out=down, in=\ur] (A) to [out=\dr, in=up] (1,1.5,0);
\draw[wire] (2,1,\h) to [out=down, in=\ul] (A) to [out=\dl, in=up] (2,1,0);
\node[dot] at (A){};
\node[obj, above left] at (-0.08,0,-0.04) {$[01]$};
\node[obj, above left] at (-0.08,1,-0.04) {$[12]$};
\node[obj, above left] at (-0.08,2,-0.04) {$[23]$};
\node[obj, above right] at (3.05,1,-0.04){$[03]$};
\node[obj, below] at (1.4,0.6,\h-0.04){$[02]$};
\node[obj, above left] at (1.5,1.3,-0.04){$[13]$};
\node[omor, above left] at (1.95,1,0){$[013]$};
\node[omor, above left] at (1,1.5,0.){$[123]$};
\node[omor, below left] at (1.95,1,\h-0.04){$[023]$};
\node[omor, below right] at (1.05,0.5,\h-0.24){$[012]$};
\node[tmor, left] at ([xshift=-0.1cm]A){$\eta$};
\end{tz}
\\[3pt]\nonumber 
\eta \in V^+(\Gamma,\sIII)\hspace{0.1cm} & \eta \in V^-(\Gamma,\sIII)
\end{calign}

If we were considering a state of the 4-manifold to include labelings of 3-simplices, those labels would be by elements of chosen bases for these associator state spaces.  Instead we implicitly rather than explicitly sum over such bases by using a canonical copairing between the associator spaces, as follows.  The pairing between $\Hom(f,g)$ and $\Hom(g,f)$ in a spherical prefusion 2-category, from Definition~\ref{def:pairing}, gives, for any 3-simplex $\sIII$ and labeling $\Gamma$, a pairing
\[
\langle \cdot , \cdot \rangle_{\Gamma,\sIII} :  V^-(\Gamma,\sIII) \otimes V^+(\Gamma,\sIII) \to k
\]
By Proposition~\ref{prop:pairing}, this pairing is nondegenerate; there is therefore a canonically determined copairing
\[
\cup_{\Gamma,\sIII}: k\to V^+(\Gamma,\sIII) \otimes V^-(\Gamma,\sIII)
\]
providing, with the pairing, a duality between $V^+(\Gamma,\sIII)$ and $V^-(\Gamma,\sIII)$.  We can tensor over all the 3-simplices of the manifold to obtain a `global copairing':
\[
\cup_{\Gamma}:=\bigotimes_{\sIII\in \K_3} \cup_{\Gamma,\sIII}: k\to \bigotimes_{\sIII\in \K_3} \left( V^+(\Gamma,\sIII) \otimes V^-(\Gamma,\sIII) \right)
\]
\begin{definition}[Canonical associated state] \label{def:assocstate}
For a $\tc{C}$-state $\Gamma$ on an ordered oriented combinatorial 4-manifold, the \emph{canonical associated state} is
\[
\cup_{\Gamma}(1) \in \bigotimes_{\sIII\in \K_3} \left( V^+(\Gamma,\sIII) \otimes V^-(\Gamma,\sIII) \right)
\]
\end{definition}
\nid Notice that this associated state may be thought of as a sum of `complete states' of the manifold, where a complete state is a $\tc{C}$-state $\Gamma$ labeling 1- and 2-simplices as before, and also a labeling of each 3-simplex $\sIII$ by an element of a chosen basis of $V^+(\Gamma,\sIII)$.

\skiptocparagraph{The 10j symbol and the 10j action}

Given a $\tc{C}$-state $\Gamma$ on a 4-manifold $K$, a chosen 4-simplex $\sIV \in K_4$, and for each 3-simplex $\sIII \subset \sIV$ an `associator vector' $v_\sIII \in V^{\epsilon_o^{\sIV}(\sIII)} (\Gamma, \sIII)$, we can compose these five vectors (which are 2-morphisms in the prefusion 2-category) to obtain a `pentagonator endomorphism'.  The (numerical) 2-spherical trace of this pentagonator endomorphism is the `10j symbol' of that collection of associator vectors---these 10j symbols are the core linear-algebraic structure data of the prefusion 2-category.
\begin{definition}[The 10j symbol]
For a $\tc{C}$-state $\Gamma$ on an ordered oriented combinatorial 4-manifold $K$, and a chosen 4-simplex $\sIV \in K_4$, the \emph{10j symbol} is a map
\[
z(\Gamma,\sIV): \bigotimes_{\sIII\in \K_3, \sIII \subset \sIV} V^{\epsilon_o^{\sIV}(\sIII)} (\Gamma, \sIII)\to k
\]
determined by the 2-spherical trace (Defintion~\ref{def:spheretrace}) of the pentagonator endomorphism depicted in Figure~\ref{fig:definitionz}.  
\end{definition}
\nid (In the figure, the labels $[ijkl]$ denote elements of $V^+(\Gamma, \partial^o_{[ijkl]} \sIV)$ and the labels $\overline{[ijkl]}$ denote elements of $V^-(\Gamma, \partial^o_{[ijkl]}\sIV)$.  As the monoidal unit $\Iz$ of the prefusion 2-category $\tc{C}$ is simple, and therefore the 1-morphism $\Io_\Iz$ is also simple, we may and will identify the target $\End_{\tc{C}}(\Io_\Iz)$ of the trace with the base field $k$.) 

Explicitly, if the orientation of the 4-simplex agrees with the global order, i.e.\ $\epsilon_o(\sIV) =+1$, then the map $z(\Gamma,\sIV)$ has the form
\[
\shrinker{.9}{
z(\Gamma, \sIV):V^+(\Gamma,\partial^o_{[0123]}\sIV)\otimes V^+(\Gamma, \partial^o_{[0134]}\sIV)\otimes V^+(\Gamma,\partial^o_{[1234]}\sIV) \otimes V^-(\Gamma,\partial^o_{[0124]} \sIV )\otimes V^-(\Gamma, \partial^o_{[0234]} \sIV) \to k
}
\]
and is defined by the trace on the left side of the figure; if by contrast the orientation of the 4-simplex does not agree with the global order, i.e.\ $\epsilon_o(\sIV) =-1$, then the map $z(\Gamma,\sIV)$ has the form
\[
\shrinker{.9}{
z(\Gamma, \sIV): V^+(\Gamma,\partial^o_{[0234]}\sIV) \otimes V^+(\Gamma, \partial^o_{[0124]}\sIV) \otimes V^-(\Gamma,\partial^o_{[1234]}\sIV) \otimes V^-(\Gamma, \partial^o_{[0134]}\sIV) \otimes V^-(\Gamma, \partial^o_{[0123]}\sIV)\to k
}
\]
and is defined by the trace on the right side of the figure.

\begin{figure}[h]
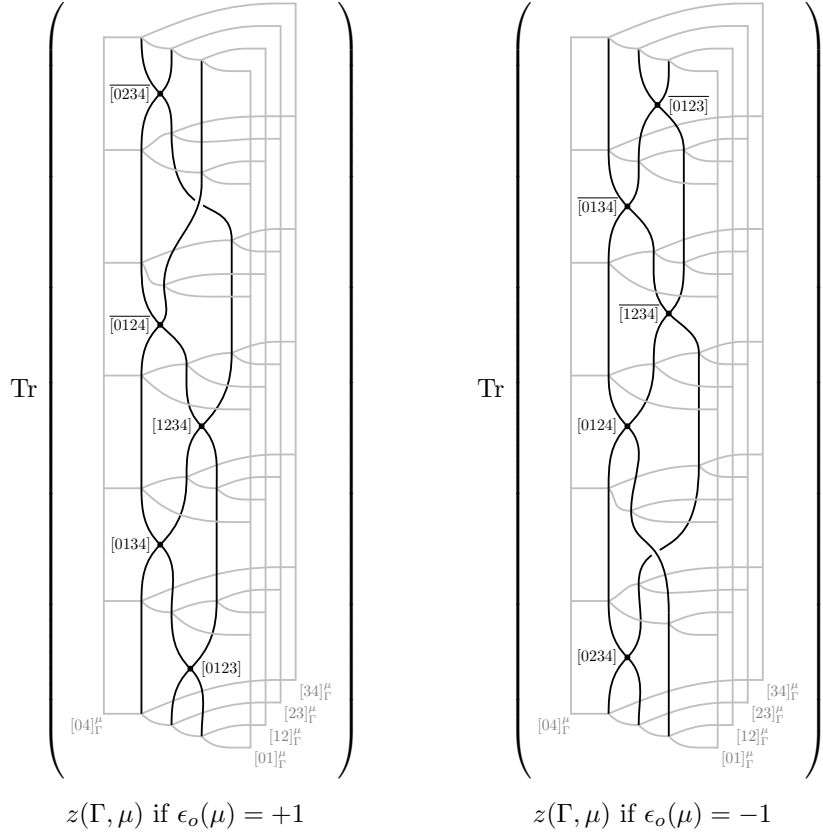

\begin{calign}\nonumber
\def\d{-0.5}
\Tr\left(%
 \begin{tz}[td,scale=1]
 \begin{scope}[xyplane=0]
 		\draw[slice](0,\d) to (0,0) to [out=up, in=\dl] (0.5,1) to [out=up, in=\dl] (1,2) to 					[out=up, in=\dl] (1.5,3) to (1.5,4);
		\draw[slice] (1,\d) to (1,0) to [out=up, in=\dr] (0.5,1);
 		\draw[slice] (2,\d) to (2,0) to [out=up, in=\dr] (1,2);
		\draw[slice] (3,\d) to (3,0) to [out=up, in=\dr] (1.5,3);
 \end{scope}
 \begin{scope}[xyplane=\h]
 	\draw[slice, on layer=front](0,\d) to (0,0) to [out=up, in=\dl]  (1,2);
 	\draw[slice] (1,2) to [out=up, in=\dl] (1.5,3) to (1.5,4);
 	\draw[slice] (1,\d) to (1,0) to [out=up, in=\dl] (1.5,1);
 	\draw[slice] (2,\d) to (2,0) to [out=up, in=\dr] (1.5,1) to [out=up, in=\dr] (1,2);
 	\draw[slice] (3,\d) to (3,0) to [out=up, in=\dr] (1.5,3);
 \end{scope}
 \begin{scope}[xyplane=2*\h]
 	\draw[slice, on layer=front] (0,\d) to (0,0) to [out=up, in=\dl]  (1.5,3) to (1.5,4);
 	\draw[slice] (1,\d) to (1,0) to [out=up, in=\dl] (1.5,1);
 	\draw[slice] (2,\d) to (2,0) to [out=up, in=\dr] (1.5,1) to [out=up, in=\dl] (2,2);
 	\draw[slice] (3,\d) to (3,0) to [out=up, in=\dr] (2,2) to [out=up, in=\dr] (1.5,3);
 \end{scope}
 \begin{scope}[xyplane=3*\h]
 	\draw[slice, on layer=front] (0,\d) to(0,0) to [out=up, in=\dl]  (1.5,3) to (1.5,4);
 	\draw[slice, on layer =frontb] (1,\d) to (1,0) to [out=up, in=\dl] (2,2) to [out=up, in=\dr] (1.5,3);
 	\draw[slice] (2,\d) to (2,0) to [out=up, in=\dl] (2.5,1) to [out=up, in=\dr] (2,2);
 	\draw[slice] (3,\d) to (3,0) to [out=up, in=\dr] (2.5,1) ;
 \end{scope}
 \begin{scope}[xyplane=4*\h]
 	\draw[slice, on layer=fronta] (0,\d) to (0,0) to [out=up, in=\dl]  (0.5,2) to [out=\dr, in=up] (1,0) to (1,\d);
 	\draw[slice]  (0.5,2) to [out=up, in=\dl] (1.5,3) to (1.5,4);
 	\draw[slice] (2,\d) to (2,0) to [out=up, in=\dl] (2.5,1) to [out=up, in=\dr] (1.5,3);
 	\draw[slice] (3,\d) to (3,0) to [out=up, in=\dr] (2.5,1) ;
 \end{scope}
 \begin{scope}[xyplane=5*\h]
 	\draw[slice, on layer=front] (0,\d) to (0,0) to [out=up, in=\dl]  (0.5,1);
 	\draw[slice, on layer=front] (0.5,1) to [out=up, in=\dl] (1.5,3);
 	\draw[slice] (1.5,3) to (1.5,4);
 	\draw[slice] (1,\d) to (1,0) to [out=up, in=\dr] (0.5,1) ;
 	\draw[slice] (2,\d) to (2,0) to [out=up, in=\dl] (2.25,2.5) to [out=up, in=\dr] (1.5,3);
 	\draw[slice] (3,\d) to (3,0) to [out=up, in=\dr] (2.25,2.5) ;
 \end{scope}
 \begin{scope}[xyplane = 6*\h, on layer=superfront]
  	\draw[slice](0,\d) to (0,0) to [out=up, in=\dl] (0.5,1) to [out=up, in=\dl] (1,2) to [out=up, in=\dl] (1.5,3) to (1.5,4);
 	\draw[slice] (1,\d) to (1,0) to [out=up, in=\dr] (0.5,1);
 	\draw[slice]  (2,\d) to (2,0) to [out=up, in=\dr] (1,2);
 	\draw[slice] (3,\d) to (3,0) to [out=up, in=\dr] (1.5,3);
 \end{scope}
 \begin{scope}[xzplane=\d]
 	\draw[slice,short] (0,0) to (0,6*\h);
 	\draw[slice,short] (1,0) to (1,6*\h);
 	\draw[slice,short] (2,0) to (2,6*\h);
 	\draw[slice,short] (3,0) to (3,6*\h);
 \end{scope}
 \begin{scope}[xzplane=4]
 	\draw[slice,short] (1.5,0) to (1.5, 6*\h);
 \end{scope}
 \coordinate (L1) at (2.5, 1.5, 1.5*\h);
 \coordinate (L2) at (2.5,1.5, 3.45*\h);
 \coordinate (L3) at (2.5, 1.5, 5.5*\h);
 \coordinate (R1) at (1.5, 1, 0.5*\h);
 \coordinate (R2) at (1.5, 1.75, 2.5*\h);
 \draw[wire, on layer=front] (3,1.5,0) to (3,1.5, \h) to [out=up, in=\dl] (L1) to [out=\ul, in=down] (3, 1.5, 2*\h) to (3,1.5, 3*\h) to  [out=up, in=\dl] (L2) to [out=\ul, in=down] (3, 1.5, 4*\h);
 \draw[wire,on layer=front] (3,1.5,4*\h) to (3,1.5, 5*\h);
 \draw[wire, on layer=front] (3, 1.5, 5*\h) to [out=up, in=\dl] (L3);
 \draw[wire] (L3) to [out=\ul, in=down] (3, 1.5, 6*\h) ;
 \draw[wire,on layer=front] (2, 1,0) to [out=up, in=\dl] (R1)  to [out=\ul, in=down] (2,1,\h) to [out=up, in=\dr] (L1);
 \draw[wire, on layer=frontb](L1) to [out=\ur, in=down] (2, 2, 2*\h) to [out=up, in=\dl] (R2) to [out=\ul, in=down] (2,2,3*\h) to [out=up, in=\dr] (L2) to [out=\ur, in=down]  (2, 0.5, 4*\h) to [out=up, in=down] node[mask point, pos=0.8](MP){} (1,0.5, 4.9*\h); 
 \draw[wire, on layer=front] (1,0.5,4.9*\h) to (1, 0.5, 6*\h);
\cliparoundone{MP}{ \draw[wire] (1, 0.5,0) to [out=up, in=\dr] (R1) to [out=\ur, in=down] (1, 1.5, \h) to (1,1.5,2*\h) to [out=up, in=\dr] (R2) to [out=\ur, in=down] (1, 2.5, 3*\h) to (1,2.5, 4*\h) to [out=up, in=down, in looseness=2] (2.5, 2.25, 5*\h) to [out=up, in=\dr] (L3) to [out=\ur, in=down] (2, 1, 6*\h);}
\node[dot] at (L1){};
\node[dot] at (L2){};
\node[dot] at (L3){};
\node[dot] at (R1){};
\node[dot] at (R2){};
\node[tmor, right] at (R1) {~$[0123]$};
\node[tmor, left] at ([xshift=-0.1cm]L1) {$[0134]$};
\node[tmor, left] at ([xshift=-0.1cm]R2) {~$[1234]$};
\node[tmor, left] at ([xshift=-0.12cm]L2) {~$\overline{[0124]}$};
\node[tmor, left] at ([xshift=-0.12cm]L3) {~$\overline{[0234]}$};
\node[obj, below right] at (\d+0.1,0.1,0) {$[01]_\Gamma^\sIV$};
\node[obj, below right] at (\d+0.1,1.1,0) {$[12]_\Gamma^\sIV$};
\node[obj, below right] at (\d+0.1,2.1,0) {$[23]_\Gamma^\sIV$};
\node[obj, below right] at (\d+0.1,3.1,0) {$[34]_\Gamma^\sIV$};
\node[obj, below left] at (3.9,1.6,0) {$[04]_\Gamma^\sIV$};
 \end{tz}%
\right)
 &
 \Tr\left(%
 \def\d{-0.5}
 \begin{tz}[td,scale=1]
 \begin{scope}[xyplane=6*\h, on layer=superfront]
 		\draw[slice](0,\d) to (0,0) to [out=up, in=\dl] (0.5,1) to [out=up, in=\dl] (1,2) to 					[out=up, in=\dl] (1.5,3) to (1.5,4);
		\draw[slice] (1,\d) to (1,0) to [out=up, in=\dr] (0.5,1);
 		\draw[slice] (2,\d) to (2,0) to [out=up, in=\dr] (1,2);
		\draw[slice] (3,\d) to (3,0) to [out=up, in=\dr] (1.5,3);
 \end{scope}
 \begin{scope}[xyplane=5*\h]
 	\draw[slice, on layer=front](0,\d) to (0,0) to [out=up, in=\dl]  (1,2);
 	\draw[slice] (1,2) to [out=up, in=\dl] (1.5,3) to (1.5,4);
 	\draw[slice] (1,\d) to (1,0) to [out=up, in=\dl] (1.5,1);
 	\draw[slice] (2,\d) to (2,0) to [out=up, in=\dr] (1.5,1) to [out=up, in=\dr] (1,2);
 	\draw[slice] (3,\d) to (3,0) to [out=up, in=\dr] (1.5,3);
 \end{scope}
 \begin{scope}[xyplane=4*\h]
 	\draw[slice, on layer=front] (0,\d) to (0,0) to [out=up, in=\dl]  (1.5,3) to (1.5,4);
 	\draw[slice] (1,\d) to (1,0) to [out=up, in=\dl] (1.5,1);
 	\draw[slice] (2,\d) to (2,0) to [out=up, in=\dr] (1.5,1) to [out=up, in=\dl] (2,2);
 	\draw[slice] (3,\d) to (3,0) to [out=up, in=\dr] (2,2) to [out=up, in=\dr] (1.5,3);
 \end{scope}
 \begin{scope}[xyplane=3*\h]
 	\draw[slice, on layer=front] (0,\d) to(0,0) to [out=up, in=\dl]  (1.5,3) to (1.5,4);
 	\draw[slice, on layer =frontb] (1,\d) to (1,0) to [out=up, in=\dl] (2,2) to [out=up, in=\dr] (1.5,3);
 	\draw[slice] (2,\d) to (2,0) to [out=up, in=\dl] (2.5,1) to [out=up, in=\dr] (2,2);
 	\draw[slice] (3,\d) to (3,0) to [out=up, in=\dr] (2.5,1) ;
 \end{scope}
 \begin{scope}[xyplane=2*\h]
 	\draw[slice, on layer=fronta] (0,\d) to (0,0) to [out=up, in=\dl]  (0.5,2) to [out=\dr, in=up] (1,0) to (1,\d);
 	\draw[slice]  (0.5,2) to [out=up, in=\dl] (1.5,3) to (1.5,4);
 	\draw[slice] (2,\d) to (2,0) to [out=up, in=\dl] (2.5,1) to [out=up, in=\dr] (1.5,3);
 	\draw[slice] (3,\d) to (3,0) to [out=up, in=\dr] (2.5,1) ;
 \end{scope}
 \begin{scope}[xyplane=\h]
 	\draw[slice, on layer=front] (0,\d) to (0,0) to [out=up, in=\dl]  (0.5,1);
 	\draw[slice, on layer=front] (0.5,1) to [out=up, in=\dl] (1.5,3);
 	\draw[slice] (1.5,3) to (1.5,4);
 	\draw[slice] (1,\d) to (1,0) to [out=up, in=\dr] (0.5,1) ;
 	\draw[slice] (2,\d) to (2,0) to [out=up, in=\dl] (2.25,2.5) to [out=up, in=\dr] (1.5,3);
 	\draw[slice] (3,\d) to (3,0) to [out=up, in=\dr] (2.25,2.5) ;
 \end{scope}
 \begin{scope}[xyplane = 0]
  	\draw[slice](0,\d) to (0,0) to [out=up, in=\dl] (0.5,1) to [out=up, in=\dl] (1,2) to [out=up, in=\dl] (1.5,3) to (1.5,4);
 	\draw[slice] (1,\d) to (1,0) to [out=up, in=\dr] (0.5,1);
 	\draw[slice]  (2,\d) to (2,0) to [out=up, in=\dr] (1,2);
 	\draw[slice] (3,\d) to (3,0) to [out=up, in=\dr] (1.5,3);
 \end{scope}
 \begin{scope}[xzplane=\d]
 	\draw[slice,short] (0,0) to (0,6*\h);
 	\draw[slice,short] (1,0) to (1,6*\h);
 	\draw[slice,short] (2,0) to (2,6*\h);
 	\draw[slice,short] (3,0) to (3,6*\h);
 \end{scope}
 \begin{scope}[xzplane=4]
 	\draw[slice,short] (1.5,0) to (1.5, 6*\h);
 \end{scope}
 \coordinate (L1) at (2.5, 1.5, 4.5*\h);
 \coordinate (L2) at (2.5,1.5, 2.55*\h);
 \coordinate (L3) at (2.5, 1.5, 0.5*\h);
 \coordinate (R1) at (1.5, 1, 5.5*\h);
 \coordinate (R2) at (1.5, 1.75, 3.5*\h);
 \draw[wire, on layer=front] (3,1.5,6*\h) to (3,1.5, 5*\h) to [out=down, in=\ul] (L1) to [out=\dl, in=up] (3, 1.5, 4*\h) to (3,1.5, 3*\h) to  [out=down, in=\ul] (L2) to [out=\dl, in=up] (3, 1.5, 2*\h);
 \draw[wire,on layer=front] (3,1.5,2*\h) to (3,1.5, \h);
 \draw[wire, on layer=front] (3, 1.5, \h) to [out=down, in=\ul] (L3);
 \draw[wire] (L3) to [out=\dl, in=up] (3, 1.5, 0) ;
 \draw[wire,on layer=front] (2, 1,6*\h) to [out=down, in=\ul] (R1)  to [out=\dl, in=up] (2,1,5*\h) to [out=down, in=\ur] (L1);
 \draw[wire, on layer=frontb](L1) to [out=\dr, in=up] (2, 2, 4*\h) to [out=down, in=\ul] (R2) to [out=\dl, in=up] (2,2,3*\h) to [out=down, in=\ur] (L2) to [out=\dr, in=up]  (2, 0.5, 2*\h) to [out=down, in=up, in looseness=2] node[mask point, pos=0.6](MP){} (1,0.5, 1.1*\h); 
 \draw[wire, on layer=front] (1,0.5,1.1*\h) to (1, 0.5, 0);
\cliparoundone{MP}{ \draw[wire] (1, 0.5,6*\h) to [out=down, in=\ur] (R1) to [out=\dr, in=up] (1, 1.5, 5*\h) to (1,1.5,4*\h) to [out=down, in=\ur] (R2) to [out=\dr, in=up] (1, 2.5, 3*\h) to (1,2.5, 2*\h) to [out=down, in=up, out looseness=2] (2.5, 2.25, \h) to [out=down, in=\ur] (L3) to [out=\dr, in=up] (2, 1, 0);}
\node[dot] at (L1){};
\node[dot] at (L2){};
\node[dot] at (L3){};
\node[dot] at (R1){};
\node[dot] at (R2){};
\node[tmor, right] at (R1) {~$\conj{[0123]}$};
\node[tmor, left] at ([xshift=-0.1cm]L1) {$\conj{[0134]}$};
\node[tmor, left] at ([xshift=-0.1cm]R2) {~$\conj{[1234]}$};
\node[tmor, left] at ([xshift=-0.12cm]L2) {~$[0124]$};
\node[tmor, left] at ([xshift=-0.12cm]L3) {~$[0234]$};
\node[obj, below right] at (\d+0.1,0.1,0) {$[01]_\Gamma^\sIV$};
\node[obj, below right] at (\d+0.1,1.1,0) {$[12]_\Gamma^\sIV$};
\node[obj, below right] at (\d+0.1,2.1,0) {$[23]_\Gamma^\sIV$};
\node[obj, below right] at (\d+0.1,3.1,0) {$[34]_\Gamma^\sIV$};
\node[obj, below left] at (3.9,1.6,0) {$[04]_\Gamma^\sIV$};
 \end{tz}
 \right)
 \\[5pt]\nonumber
 z(\Gamma,\sIV)\text{ if } \epsilon_o(\sIV) =+1 
 &
z(\Gamma,\sIV)\text{ if } \epsilon_o(\sIV) =-1 
 \end{calign}
 \caption{The definition of the 10j symbol $z(\Gamma,\sIV)$.}
 \label{fig:definitionz}
\end{figure}

Tensoring together the 10j symbols of all the 4-simplices of $K$, we obtain a global 10j symbol map, which we think of as playing the role of an action, applied to the state specified by $\Gamma$ and a given collection of associator vectors.
\begin{definition}[The 10j action] \label{def:10jaction}
For a $\tc{C}$-state $\Gamma$ on an ordered oriented combinatorial 4-manifold $K$, the \emph{10j action} is the map
\[
z(\Gamma):= \Big(\bigotimes_{\sIV \in \K_4} z(\Gamma,\sIV)\Big) : \bigotimes_{\sIII\in \K_3} \left(V^+(\Gamma, \sIII) \otimes V^-(\Gamma,\sIII)\right) \to k
\]
\end{definition}
\nid In writing the domain of this linear map, we have used the fact that every 3-simplex $\sIII \in K_3$ appears as a face of precisely two 4-simplices $\sIV_1, \sIV_2 \in K_4$, and that $\epsilon_o^{\sIV_1}(\sIII)=-\epsilon_o^{\sIV_2}(\sIII)$.  (We have also suppressed some swap maps reordering the factors in the tensor product in the domain.)  In particular, we can and will consider the following numerical invariant of the state:
\begin{definition}[The 10j action of the canonical associated state] \label{def:10jactioncanonstate}
For a $\tc{C}$-state $\Gamma$ on an ordered oriented combinatorial 4-manifold $K$, the \emph{10j action of the canonical associated state} is the number
\[
Z(\Gamma) := z(\Gamma)\xo \cup_{\Gamma}(1) \in k,
\]
where $z(\Gamma)$ is the 10j action and $\cup_\Gamma(1)$ is the canonical associated state of the state $\Gamma$.
\end{definition}

\skiptocparagraph{The partition function}

We think of a combinatorial 4-manifold $K$ being built up progressively as the sequence of its $i$-skeleta $K_{(i)}$: 
\[
\emptyset = K_{(-1)} \rightsquigarrow K_{(0)} \rightsquigarrow K_{(1)} \rightsquigarrow K_{(2)} \rightsquigarrow K_{(3)} \rightsquigarrow K_{(4)} = K
\]
In keeping with the path-integral inspiration, the state sum invariant can be seen as a result of a corresponding sequence of progressive transformations:
\begin{itemize}
\item[-1.] The initial empty state is represented by the unit of the base field: \[K_{(-1)} \mapsto 1.\]
\item[0.] There are no possible labels of vertices of $K$, and so a state of the 0-skeleton is simply a scalar multiple of the collection of unlabeled vertices.  Really, though, the vertices are labeled by the unique object in the 3-category that deloops the fusion 2-category, and in effect this amounts to labeling these vertices by the prefusion 2-category $\tc{C}$ itself.  For some to-be-determined normalization factor $\phi(\tc{C}) \in k$, the canonical state of the 0-skeleton can therefore be seen as
\[K_{(0)} \mapsto 
\Big(\prod_{K_0} \phi(\tc{C}) \Big) 
[K_0^\tc{C}].\]
Here $[K_0^\tc{C}]$ denotes the collection of vertices, each labeled by the prefusion 2-category $\tc{C}$.
\item[1.] A state of the 1-skeleton is a weighted sum of all possible labelings $\gamma: K^o_{(1)} \To (\Delta \tc{C})_{(1)}$ of the 1-simplices $\sI \in K_1$ by simple objects $\gamma(\sI)$ of $\tc{C}$.  Given a second to-be-determined scalar factor $\phi(\gamma(\sI)) \in k$ associated to each labeling object $\gamma(\sI)$, we may therefore think of the canonical state of the 1-skeleton as
\[K_{(1)} \mapsto 
\sum_{\gamma: K^o_{(1)} \To (\Delta \tc{C})_{(1)}} 
\Big(\prod_{K_0} \phi(\tc{C}) \Big) 
\Big(\prod_{\sI \in K_1} \phi(\gamma(\sI)) \Big) 
[K_1^\gamma].\]
Here $[K_1^\gamma]$ denotes the 1-skeleton with the 1-simplices labeled according to the assignment $\gamma$.
\item[2.] Similarly a state of the 2-skeleton is a weighted sum of all possible labelings $\Gamma: K^o_{(2)} \To \Delta \tc{C}$ of the 1-simplices $\sI \in K_1$ by simple objects $\Gamma(\sI)$ and of the 2-simplices $\sII \in K_2$ by simple 1-morphisms $\Gamma(\sII)$.  Given a third to-be-determined scalar factor $\phi(\Gamma(\sII)) \in k$ associated to each labeling 1-morphism $\Gamma(\sII)$, we imagine the canonical state of the 2-skeleton being
\[K_{(2)} \mapsto 
\sum_{\Gamma: K^o_{(2)} \To \Delta \tc{C}} 
\Big(\prod_{K_0} \phi(\tc{C}) \Big) 
\Big(\prod_{\sI \in K_1} \phi(\Gamma(\sI)) \Big) 
\Big(\prod_{\sII \in K_2} \phi(\Gamma(\sII)) \Big) 
[K_2^\Gamma].\]
Here $[K_2^\Gamma]$ denotes the 2-skeleton itself labeled according to $\Gamma$.
\end{itemize}
We have already discussed, in Definition~\ref{def:assocstate}, the 3-skeleton `canonical associated state' coming from a state of the 2-skeleton, and in turn, in Definition~\ref{def:10jaction}, the 4-skeleton `10j action' coming from a state of the 3-skeleton.  

The three scalar factors for 0-, 1-, and 2-simplex labels are forced by asking the resulting state sum to be piecewise-linear homeomorphism invariant---more specifically by asking for the invariance of the state sum of a ball under the bistellar moves.  The (3,3)-bistellar move (replacing a triangulation of a 4-ball with three 4-simplices by a different triangulation with three 4-simplices) involves summing over the 1-morphism labels $\Gamma(\sII)$ of 2-simplices $\sII$ interior to the triangulations; this relation is satisfied if $\phi(\Gamma(\sII)) = \dim(\Gamma(\sII))$.  The (2,4)-bistellar move (replacing a triangulation of a 4-ball with two 4-simplices by one with four 4-simplices) involves summing over the object labels $\Gamma(\sI)$ of a 1-simplex $\sI$ interior to one of the triangulations; this relations is satisfied if $\phi(\Gamma(\sI)) = \left( \dim(\Gamma(\sI)) \dim(\End_{\tc{C}}(\Gamma(\sI))) n(\Gamma(\sI)) \right)^{-1}$.  Here $\dim(\End_{\tc{C}}(\Gamma(\sI)))$ denotes the global dimension of the fusion category $\End_{\tc{C}}(\Gamma(\sI))$, and $n(\Gamma(\sI))$ denotes the number of equivalence classes of simple objects in the connected component of $\Gamma(\sI)$.  Finally, the (1,5)-bistellar move (replacing a triangulation of a 4-ball with one 4-simplex by one with five 4-simplices) involves `summing over' the unique label $\tc{C}$ of an interior vertex of one of the triangulations; the resulting relation is satisfied if $\phi(\tc{C}) = \dim(\tc{C})^{-1}$.  See Section~\ref{sec:4dbistellar} for a more detailed discussion of the 4-dimensional bistellar moves, and Lemmas~\ref{lem:Pachner33}, \ref{lem:Pachner24}, and \ref{lem:Pachner15} for the statements of respectively the (3,3)-bistellar, (2,4)-bistellar, and (1,5)-bistellar relations that determine the scalar factors in the state sum.

These ingredients would piece together into a state sum expression for the whole 4-manifold, except that there are potentially infinitely many possible labelings $\Gamma: K^o_{(2)} \To \Delta \tc{C}$; we ensure that the sum is finite by restricting the labels to live in a skeletal subsemisimplicial set of $\Delta \tc{C}$, as follows.

\begin{definition}[Simplicial skeleton for a prefusion 2-category]
A \emph{simplicial skeleton} for a prefusion $2$-category $\tc{C}$ is a subsemisimplicial set $\Delta\tc{C}^{\sk}\subseteq \Delta \tc{C}$ such that $(\Delta\tc{C}^{\sk})_1$ contains precisely one object from each equivalence class of simple objects of $\tc{C}$, and the set of elements of $(\Delta\tc{C}^{\sk})_2$ with faces $A,B,C \in (\Delta\tc{C}^{\sk})_1$ contains precisely one 1-morphism from each isomorphism class of simple 1-morphisms of $\tc{C}$ from $A \xz B$ to $C$.
\end{definition}

Altogether, the 10j action of the associated state of the canonical skeletally-labeled state of the 2-skeleton gives our desired 4-manifold invariant.

\begin{definition}[The state sum] \label{def:ss}
Given an oriented singular combinatorial 4-manifold $K$, with a chosen total order $o$ on its vertices, and a spherical prefusion 2-category $\tc{C}$, with a chosen simplicial skeleton $\Delta\tc{C}^{\sk}$, the \emph{state sum} is the number
\[
Z_{\tc{C}}(K)_{o,\sk} := 
\text{\scalebox{.8}{$
\sum_{\Gamma: K^o_{(2)} \To \Delta \tc{C}^\sk} 
\Big(\prod_{K_0} \dim(\tc{C})^{-1}\Big) 
\Big(\prod_{\sI \in K_1} \big(\dim(\Gamma(\sI)) \dim(\End_{\tc{C}}(\Gamma(\sI))) n(\Gamma(\sI))\big)^{-1}\Big) 
\Big(\prod_{\sII \in K_2} \dim(\Gamma(\sII))\Big) 
Z(\Gamma).
$}}
\]
Here the sum is over natural transformations $\Gamma: K^o_{(2)} \To \Delta \tc{C}^\sk$ that label the 1-simplices, respectively 2-simplices, of $K$ by simple objects, respectively 1-morphisms, of the simplicial skeleton for $\tc{C}$.  The dimension $\dim(\tc{C})$ is the dimension of the underlying presemisimple 2-category of $\tc{C}$, from Definition~\ref{def:dimension2cat}, the dimension $\dim(\Gamma(\sI))$ is the dimension of the object of the prefusion 2-category $\tc{C}$, from Definition~\ref{def:quantumdimension}, the dimension $\dim(\Gamma(\sII))$ is the dimension of the 1-morphism of the prefusion 2-category $\tc{C}$, also from Definition~\ref{def:quantumdimension}, the dimension $\dim(\End_{\tc{C}}(\Gamma(\sI)))$ is the dimension of the fusion 1-category $\End_{\tc{C}}(\Gamma(\sI))$, the number $n(\Gamma(\sI)$ is the number of equivalence classes of simple objects in the connected component of the object $\Gamma(\sI)$, and the number $Z(\Gamma)$ is the 10j action of the canonical associated state of $\Gamma$, from Definition~\ref{def:10jactioncanonstate}.
\end{definition}

\begin{theorem}[The state sum is an invariant of the manifold]\label{thm:maintheorem}
Let $M$ be an oriented singular piecewise-linear $4$-manifold and let $\tc{C}$ be a spherical prefusion $2$-category over an algebraically closed field of characteristic zero.  The number $Z_{\tc{C}}(K)_{o,\sk}$ is independent of the choice of triangulation $K$ of $M$, of the choice of order $o$ on $K$, and of the choice of simplicial skeleton $\sk$ for $\tc{C}$, and therefore defines an oriented piecewise-linear invariant $Z_{\tc{C}}(M)$ of the singular $4$-manifold $M$.
\end{theorem}
\nid The proof occupies all of Section~\ref{sec:statesumproof}.\pagebreak

\begin{remark}[The state sum over other base fields] 
We expect Theorem~\ref{thm:maintheorem} could be extended to arbitrary perfect fields; for fields that are not algebraically closed, one must insist that every endomorphism algebra of a simple 1-morphism in the spherical prefusion 2-category is the base field, and for fields of nonzero characteristic, one must insist the spherical prefusion 2-category be nondegenerate in the sense of Remark~\ref{rem:nonzerochar1}.
\end{remark}


\subsection{Special cases of the state sum invariant}

The state sum invariant of Theorem~\ref{thm:maintheorem} applies of course to any spherical prefusion 2-category.  We now observe that this invariant simultaneously generalizes all previously known state-sum invariants of piecewise linear 4-manifolds, except possibly the dichromatic invariant for a pivotal functor with modularizable target.

\subsubsection{Ribbon categories (the Crane--Yetter--Kauffman invariant)}
The first state-sum invariant for 4-manifolds was constructed by Crane--Yetter~\cite{CY} using the data of the modular category of representations of quantum $\mathfrak{sl}_2$; subsequently Crane--Yetter--Kauffman~\cite{CYK} generalized this state sum to use the data of any semisimple ribbon category.  This Crane--Yetter--Kauffman 4-manifold invariant for the semisimple ribbon category $\oc{C}$ agrees with our invariant for the unfolded spherical prefusion 2-category $\tc{C}$ associated to $\oc{C}$, cf Constructions~\ref{con:unfold} and~\ref{con:braidedisftc} and Example~\ref{eg:ribbonspherical}.  This agreement can be seen directly by comparing the state sum constructions in that case, or as a corollary of the fact, discussed below, that our invariant generalizes the Cui invariant, which in turn generalizes the Crane--Yetter--Kauffman invariant.

\subsubsection{Finite 2-groups (the Yetter--Dijkgraaf--Witten invariant)}
Given the data of a finite 2-group, Yetter~\cite{Yetter} defined a `finite 2-gauge theory' state-sum invariant of $n$-manifolds, generalizing the Dijkgraaf--Witten invariant associated to a finite group.  Later, Faria Martins--Porter~\cite{fariamartinsporter} extended Yetter's construction to accommodate a twisting $n$-cocycle (thereby generalizing the twisted Dijkgraaf--Witten invariant).  In the case of 4-manifolds, the (twisted) Yetter--Dijkgraaf--Witten invariant for the finite 2-group $(\pi_1,\pi_2)$ with 4-cocycle $\omega$ agrees with our invariant for the spherical prefusion 2-category $\tVect^\omega(\pi_1,\pi_2)$ of (twisted) 2-group-graded 2-vector spaces, cf Constructions~\ref{con:2groupgraded2vect} and~\ref{con:2groupgraded2vecttwisted} and Example~\ref{eg:twovectspherical}.  This agreement can be seen directly, or, again, by observing below that our invariant generalizes the Cui invariant which in turn generalizes the Yetter--Dijkgraaf--Witten invariant.

\subsubsection{Endotrivial fusion 2-categories (the Mackaay invariant)}

Recall that an endotrivial fusion 2-category is one where the endomorphism fusion category of every indecomposable object is the trivial fusion category $\Vect$.  Given the data of an endotrivial spherical fusion 2-category (cf Remark~\ref{rem:mackaayspherical}), Mackaay~\cite{Mackaay} defined a state-sum invariant of 4-manifolds; in the endotrivial case, our state sum directly simplifies to Mackaay's formula~\cite[Def 3.2]{Mackaay}.  (Note though, as in Remark~\ref{rem:endotrivialexamples}, that we are not aware of any examples of endotrivial fusion 2-categories besides the 2-catgory $\tVect^\omega(\pi_1)$ of twisted 1-group-graded 2-vector spaces, that is the twisted Dijkgraaf-Witten case.  In particular, the examples coming from braided fusion categories, graded braided fusion categories, module tensor categories, 2-representations of 2-groups, and 2-group-graded 2-vector spaces, all have nontrivial endomorphism fusion categories.)

\subsubsection{Crossed-braided fusion categories (the Cui invariant)}

Given the data of a crossed-braided spherical fusion category, Cui~\cite{Cui} defines a state-sum invariant of 4-manifolds, simultaneously generalizing the Crane--Yetter--Kauffman invariant and the Yetter--Dijkgraaf--Witten invariant.  Recall from Construction~\ref{con:gradedbraidedfusion} and Example~\ref{eg:gradedbraidedspherical} that a $G$-crossed-braided spherical fusion category can be interpreted as a spherical prefusion 2-category whose set of objects is the finite group $G$.

Cui's state-sum for a crossed-braided spherical fusion category $\oc{C}$ agrees with our state-sum for the corresponding spherical prefusion $2$-category $\tc{C}$, as follows.  (We will use the notation $\dim_{\oc{C}}$ to refer to dimensions of morphisms thought of in the spherical fusion category $\oc{C}$ and by contrast $\dim_{\tc{C}}$ for the dimensions of objects and morphisms in the spherical fusion 2-category $\tc{C}$.)  The 10j action term $Z(\Gamma)$ in our state sum agrees with the expression associated to the 4-simplices in Cui~\cite[Eq 22]{Cui}.  Observe that in the special case in question, the dimension $\dim_{\tc{C}}(g)$ of every simple object $g \in \tc{C}$ is $1$, and therefore the dimension of any 1-morphism $f: g \to h$ is $\dim_{\tc{C}}(f) = \langle \tr_R(f) \rangle = \dim_{\oc{C}}(f)$.  Hence, the 2-simplex normalization factor $\dim_{\tc{C}}(\Gamma(\sII))$ in our state sum agrees with the corresponding factor in Cui's formula.

Next, recall that the endomorphism category of any object $g \in \tc{C}$ is $\End_{\tc{C}}(g) = \oc{C}_e$, that is the identity-graded piece of the crossed-braided fusion 1-category $\oc{C}$.  More generally, the morphism category between objects $g,h \in \tc{C}$ is $\Hom_{\tc{C}}(g,h) = \oc{C}_{h g^{-1}}$, and so the number of objects in the component of $g$ is $n(g) = |\{h \in G \,|\, \oc{C}_{h g^{-1}} \neq 0\}| = |\{h \in G \,|\, \oc{C}_h \neq 0\}|$.  In particular, in this special case, neither the dimension $\dim_{\tc{C}}(g)$, nor the dimension $\dim(\End_{\tc{C}}(g))$, nor the number $n(g)$, depends on the object $g$.  Hence, our 1-simplex normalization factor $\big(\prod_{\sI \in K_1} (\dim_{\tc{C}}(\Gamma(\sI)) \dim(\End_{\tc{C}}(\Gamma(\sI))) n(\Gamma(\sI)))^{-1}\big)$ drastically simplifies to $\big(\dim(\oc{C}_e) \,\,|\{h \in G \,|\, \oc{C}_h \neq 0\}|\big)^{-|K_1|}$.  M\"uger~\cite{MuegerGcrossed} proves that in a crossed-braided fusion category $\oc{C}$, the dimension of every nonzero graded piece is the same.  Thus $\dim(\oc{C}) = |\{h \in G \,|\, \oc{C}_h \neq 0\}| \dim(\oc{C}_1)$ and the normalization factor in question is simply $(\dim(\oc{C}))^{-|K_1|}$, which indeed agrees with the 1-simplex factor in Cui's state sum.

Finally, observe that because there are $|\{h \in G \,|\, \oc{C}_h \neq 0\}|$ objects in each component of $\tc{C}$, the number of components of $\tc{C}$ is $|G|/|\{h \in G \,|\, \oc{C}_h \neq 0\}|$.  The dimension of the whole prefusion 2-category $\tc{C}$ is therefore \[\dim(\tc{C}) := \sum_{[x] \in \pi_0 \tc{C}} \dim(\End_{\tc{C}}(x))^{-1} = |G| |\{h \in G \,|\, \oc{C}_h \neq 0\}|^{-1} \dim(\oc{C}_e)^{-1} = |G| \dim(\oc{C})^{-1}.\]
In this case, our 0-simplex normalization factor is therefore $\big(|G| \dim(\oc{C})^{-1}\big)^{-|K_0|}$, which agrees with the corresponding factor in Cui's formula.

\newpage

\addtocontents{toc}{\protect\vspace{8pt}}
 
\section{The state sum is a piecewise-linear homeomorphism invariant}\label{sec:statesumproof}

As defined, the state sum expression
$Z_{\tc{C}}(K)_{o,\sk}$
depends on the chosen labeling skeleton $\Delta \tc{C}^\sk$ of the semisimplicial labeling category $\Delta \tc{C}$, on the total order $o$ on the vertices of the combinatorial 4-manifold $K$, and of course on the given combinatorial structure of $K$.  We prove in turn that the numerical value of the state sum does not depend on each of these choices, and so, altogether, the state sum is an invariant of a piecewise-linear 4-manifold.  (Recall from Section~\ref{sec:stellarbistellar} that whenever we refer to a `4-manifold' we implicitly allow vertex singularities.)

\subsection{The state sum is independent of the labeling skeleton}

The state sum expression may be written as the sum over states
\[Z_{\tc{C}}(K)_{o,\sk} := 
\sum_{\Gamma: K^o_{(2)} \To \Delta \tc{C}^\sk} 
N(\Gamma)\]
where
\[
\shrink{
N(\Gamma) :=
\Big(\prod_{K_0} \dim(\tc{C})^{-1}\Big) 
\Big(\prod_{\sI \in K_1} (\dim(\Gamma(\sI)) \dim(\End_{\tc{C}}(\Gamma(\sI))) n(\Gamma(\sI)))^{-1}\Big) 
\Big(\prod_{\sII \in K_2} \dim(\Gamma(\sII))\Big) 
Z(\Gamma)
}
\]
is the normalized 10j action.  

\begin{definition}[Equivalent states] \label{def:equivalentstates}
Given a spherical prefusion 2-category $\tc{C}$ and an ordered oriented combinatorial 4-manifold $K^o$, two $\tc{C}$-states $\Gamma, \Gamma': K^o_{(2)} \To \Delta \tc{C}$ are \emph{equivalent} if for every edge $\sI \in K_1$, there are inverse equivalences 
\[
h_\sI : \Gamma(\sI) \rightleftarrows \Gamma'(\sI) : k_\sI
\] 
and for every 2-simplex $\sII \in K_2$, there are 2-isomorphisms
\[
\Gamma(\sII) \cong k_{\partial^o_{02}\sII}\xo \Gamma'(\sII)\xo \left(h_{\partial^o_{01}\sII} \xz h_{\partial^o_{12}\sII}\right).
\]
\end{definition}

\begin{lemma}[Equivalent states have the same normalized 10j action] \label{lem:10jinvariance}
If the $\tc{C}$-states $\Gamma$ and $\Gamma'$ are equivalent, then their normalized 10j actions are the same: $N(\Gamma) = N(\Gamma')$.
\end{lemma}

\nid This lemma will be established below.  Given two distinct labeling skeleta $\Delta \tc{C}^{\sk_1}$ and $\Delta \tc{C}^{\sk_2}$, for each object $A \in (\Delta \tc{C}^{\sk_1})_1$, choose inverse equivalences $h_A : A \rightleftarrows A' : k_A$ between $A$ and the unique equivalent object $A' \in (\Delta \tc{C}^{\sk_2})_1$; similarly for each 1-morphism $(g: A \xz B \to C) \in (\Delta \tc{C}^{\sk_1})_2$, choose an isomorphism $h_C \xo g \xo (k_A \xz k_B) \cong g'$, where $g'$ is the unique isomorphic 1-morphism $(g': A' \xz B' \to C') \in (\Delta \tc{C}^{\sk_2})_2$.  Composing with these equivalences and isomorphisms provides a bijection $[K^o_{(2)},\Delta \tc{C}^{\sk_1}] \cong [K^o_{(2)},\Delta \tc{C}^{\sk_2}]$ between the set of $\tc{C}$-states with labels in the first skeleton and the set of $\tc{C}$-states with labels in the second skeleton; and of course this bijection takes each state to an equivalent state.

\begin{corollary}[The state sum is invariant under change of labeling skeleton] \label{lem:skeletalinvariance}
Given a spherical prefusion $2$-category $\tc{C}$, an oriented (singular) combinatorial $4$-manifold $K$, a chosen total order $o$ on its vertices, and any two simplicial skeleta $\Delta \tc{C}^{\sk_1}$ and $\Delta \tc{C}^{\sk_2}$ for $\tc{C}$, the corresponding state sums agree:
\[Z_{\tc{C}}(K)_{o,\sk_1} = Z_{\tc{C}}(K)_{o,\sk_2}.\]
\end{corollary}

\nid In light of this independence of the choice of simplicial skeleton, we will henceforth denote the state sum, associated to a spherical prefusion 2-category $\tc{C}$, an oriented combinatorial 4-manifold $K$, and a chosen total order $o$ on vertices, by $Z_{\tc{C}}(K)_o$.

\subsubsection{The 10j action is invariant under 1-morphism state changes}

We begin by considering the situation where $\Gamma$ and $\Gamma'$ are equivalent $\tc{C}$-states, and in fact the object labels of the two states are equal, that is $\Gamma(\sI) = \Gamma'(\sI)$ for all 1-simplices $\sI \in K_1$.

\begin{lemma}[Object-equal equivalent states have the same 10j action] \label{lem:1morphchange}
If the $\tc{C}$-states $\Gamma$ and $\Gamma'$ are equivalent and moreover the equivalences $h_\sI : \Gamma(\sI) \rightleftarrows \Gamma'(\sI) : k_\sI$ are identities, then the corresponding 10j actions are equal: $Z(\Gamma) = Z(\Gamma')$.
\end{lemma}

\begin{proof}
The two states differ only by a collection of 2-isomorphisms $\alpha_\sII: \Gamma(\sII) \To \Gamma'(\sII)$, for $\sII \in K_2$.  These chosen isomorphisms $\alpha_\sII$ induce, by pre- and post-composition, isomorphisms of the corresponding associator state spaces of every 3-simplex $\sIII \in K_3$:
\begin{align*}
l^+_{\sIII}: & V^+(\Gamma,\sIII)\to V^+(\Gamma',\sIII) \\
l^-_{\sIII} : & V^-(\Gamma,\sIII) \to V^-(\Gamma',\sIII)
\end{align*}
The first of these isomorphisms, for instance, may be depicted as follows, where each white dot denotes either an $\alpha_s$ isomorphism or its inverse:
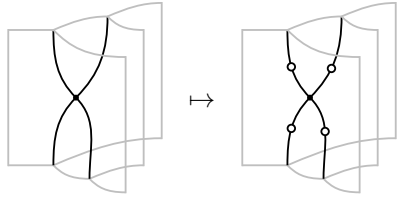
\tikzset{whitedot/.style={circle, scale=0.3, fill=white, draw}}
\begin{align*}
\begin{tz}[td,scale=1.2]
\begin{scope}[xyplane=0]
\draw[slice] (0,0) to [out=up, in=\dl] (0.5,1) to [out=up, in=\dl] (1,2) to (1,3);
\draw[slice] (1,0) to [out=up, in=\dr] (0.5,1);
\draw[slice] (2,0) to [out=up, in=\dr] (1,2);
\end{scope}
\begin{scope}[xyplane=\h, on layer=superfront]
\draw[slice] (0,0) to [out=up, in=\dl] (1,2) to (1,3);
\draw[slice] (1,0) to [out=up, in=\dl] (1.5,1) to [out=up, in=\dr] (1,2);
\draw[slice] (2,0) to [out=up, in=\dr] (1.5,1);
\end{scope}
\begin{scope}[xzplane=0]
\draw[slice,short] (0,0) to (0, \h);
\draw[slice,short] (1,0) to (1,\h);
\draw[slice,short] (2,0) to (2,\h);
\end{scope}
\begin{scope}[xzplane=3]
\draw[slice,short] (1,0) to (1,\h);
\end{scope}
\coordinate (A) at (1.5,1,0.5*\h);
\draw[wire] (1,0.5,0) to [out=up, in=\dr] (A) to [out=\ur, in=down] (1,1.5,\h);
\draw[wire] (2,1,0) to [out=up, in=\dl] (A) to [out=\ul, in=down] (2,1,\h);
\node[dot] at (A){};
\end{tz}
&\hspace{0.25cm}\mapsto\hspace{0.25cm}
\begin{tz}[td,scale=1.2]
\begin{scope}[xyplane=0]
\draw[slice] (0,0) to [out=up, in=\dl] (0.5,1) to [out=up, in=\dl] (1,2) to (1,3);
\draw[slice] (1,0) to [out=up, in=\dr] (0.5,1);
\draw[slice] (2,0) to [out=up, in=\dr] (1,2);
\end{scope}
\begin{scope}[xyplane=\h, on layer=superfront]
\draw[slice] (0,0) to [out=up, in=\dl] (1,2) to (1,3);
\draw[slice] (1,0) to [out=up, in=\dl] (1.5,1) to [out=up, in=\dr] (1,2);
\draw[slice] (2,0) to [out=up, in=\dr] (1.5,1);
\end{scope}
\begin{scope}[xzplane=0]
\draw[slice,short] (0,0) to (0, \h);
\draw[slice,short] (1,0) to (1,\h);
\draw[slice,short] (2,0) to (2,\h);
\end{scope}
\begin{scope}[xzplane=3]
\draw[slice,short] (1,0) to (1,\h);
\end{scope}
\coordinate (A) at (1.5,1,0.5*\h);
\draw[wire] (1,0.5,0) to  [out=up, in=\dr]node[whitedot, pos=0.54] {} (A) to [out=\ur, in=down]node[whitedot,pos=0.4]{} (1,1.5,\h);
\draw[wire] (2,1,0) to [out=up, in=\dl] node[pos=0.5,whitedot]{}(A) to [out=\ul, in=down]node[pos=0.5,whitedot]{} (2,1,\h);
\node[dot] at (A){};
\end{tz}
\end{align*}

Recall that the pairing $\langle \cdot, \cdot, \rangle_{\Gamma,\sIII} : V^-(\Gamma,\sIII) \otimes V^+(\Gamma,\sIII) \to k$ between the negative and positive associator state spaces is defined, see Definitions~\ref{def:pairing} and~\ref{def:spheretrace}, as the 2-spherical trace of the composition.  By the cyclicity of the planar trace (used in the definition of the 2-spherical trace), the various comparison isomorphisms $\alpha_\sII$ cancel out in the trace construction, and the associator state space isomorphisms $l^+$ and $l^-$ therefore intertwine the pairings:
\[
 \langle\cdot, \cdot\rangle_{\Gamma,\sIII} = \langle \cdot,\cdot \rangle_{\Gamma',\sIII}\xo \left( l_\sIII^-\otimes l_\sIII^+\right) : V^-(\Gamma,\sIII) \otimes V^+(\Gamma,\sIII) \to k
\]
It follows that the inverse isomorphisms intertwine the corresponding copairing:
\[
\cup_{\Gamma,\sIII} = \left(l^+_{\sIII} \otimes l^-_{\sIII}\right)^{-1} \xo  \cup_{\Gamma',\sIII} : k \to V^+(\Gamma,\sIII) \otimes V^-(\Gamma,\sIII)
\]
Of course, these isomorphisms then intertwine the global copairings $\cup_\Gamma := \bigotimes_{\sIII \in K_3} \cup_{\Gamma,\sIII}$ and $\cup_{\Gamma'} := \bigotimes_{\sIII \in K_3} \cup_{\Gamma',\sIII}$.

Next, recall that the 10j action is defined, see Definition~\ref{def:10jaction} and Figure~\ref{fig:definitionz}, as a product of 2-spherical traces of pentagonator composites.  Again by the cyclicity of the planar trace, the comparison isomorphisms $\alpha_\sII$ will cancel pairwise in this trace construction, and so the associator state space isomorphisms also intertwine the 10j action:
\[
z(\Gamma) = z(\Gamma') \xo \big(\bigotimes_{\sIII \in \K_3} l^+_{\sIII} \otimes l^-_{\sIII}\big) 
\]
Altogether then we find that the 10j actions for the states $\Gamma$ and $\Gamma'$ agree:
\[
Z(\Gamma) := z(\Gamma) \xo \cup_{\Gamma}(1) = z(\Gamma') \xo \cup_{\Gamma}(1) =: Z(\Gamma')\qedhere
\]
\end{proof}

\nid
Of course, none of the normalization factors in the normalized 10j action are affected by changing 1-morphism labels by isomorphisms, so the normalized 10j action is similarly unaffected by changes of 1-morphism labels.

\subsubsection{The normalized 10j action is invariant under object state changes}

Next we consider the situation where $\Gamma$ and $\Gamma'$ are equivalent $\tc{C}$-states, for which the $2$-isomorphisms may be taken to be equalities and moreover for which the object inverse equivalences are given by a morphism and its chosen (planar pivotal) adjoint; that is, for every edge $\sI \in K_1$, there are inverse equivalences $h_\sI : \Gamma(\sI) \rightleftarrows \Gamma'(\sI) : h_\sI^\ast$ such that $\Gamma(\sII) = h_{\partial^o_{02}\sII}^\ast \xo \Gamma'(\sII) \xo \left(h_{\partial^o_{01}\sII} \xz h_{\partial^o_{12}\sII}\right)$.

\begin{lemma}[Morphism-conjugate equivalent states have the same normalized 10j action] \label{lem:objectchange}
If the $\tc{C}$-states $\Gamma$ and $\Gamma'$ are equivalent, with inverse equivalences $h_\sI : \Gamma(\sI) \rightleftarrows \Gamma'(\sI) : h_\sI^\ast$ between object labels, and equalities $\Gamma(\sII) = h_{\partial^o_{02}\sII}^\ast \xo \Gamma'(\sII) \xo \left(h_{\partial^o_{01}\sII} \xz h_{\partial^o_{12}\sII}\right)$ relate the 1-morphism labels, then the corresponding normalized 10j actions are equal: $N(\Gamma) = N(\Gamma')$.
\end{lemma} 

The normalized 10j action may be considered to have five factors, the 0-simplex factor $\Big(\prod_{K_0} \dim(\tc{C})^{-1}\Big)$, the 1-simplex factor $\Big(\prod_{\sI \in K_1} (\dim(\Gamma(\sI)) \dim(\End_{\tc{C}}(\Gamma(\sI))) n(\Gamma(\sI)))^{-1}\Big)$, the 2-simplex factor $\Big(\prod_{\sII \in K_2} \dim(\Gamma(\sII))\Big)$, the 3-simplex factor $\cup_\Gamma := \bigotimes_{\sIII\in \K_3} \cup_{\Gamma,\sIII}$, and the 4-simplex factor $z(\Gamma):= \Big(\bigotimes_{\sIV \in \K_4} z(\Gamma,\sIV)\Big)$.  The 0-simplex factor is certainly unaffected by changing state labels.  We address the other factors in turn.

\skiptocparagraph{The 1-simplex factor}

The 1-simplex factor has three terms for each 1-simplex $\sI \in K_1$, the number $n(\Gamma(\sI))$ of equivalence classes of simple objects in the component, the dimension $\dim(\End_{\tc{C}}(\Gamma(\sI)))$ of the endomorphism fusion category, and the dimension $\dim(\Gamma(\sI))$ of the object label itself.  Given an equivalence of simple objects $h_\sI: \Gamma'(\sI) \simeq \Gamma(\sI)$, the two objects $\Gamma(\sI)$ and $\Gamma'(\sI)$ are in the same connected component and so the number of equivalence classes of simple objects in that component is evidently unchanged: $n(\Gamma(\sI))=n(\Gamma'(\sI))$.  Similarly, the endomorphism categories of $\Gamma'(\sI)$ and $\Gamma(\sI)$ are equivalent fusion 1-categories and therefore have the same dimension: $\dim(\End_{\tc{C}}(\Gamma'(\sI))) = \dim(\End_{\tc{C}}(\Gamma(\sI)))$.  The dimension of the simple object $\Gamma(\sI)$ itself is not, however, invariant under equivalence, rather it transforms according to the square of the planar trace of the chosen equivalence:
\begin{align*}
\dim(\Gamma(\sI)) &:= \dim ( \Io_{\Gamma(\sI)} ) && \\
&= \dim( h_\sI^*\xo h_\sI ) && [\Io_{\Gamma(\sI)} \cong h_\sI^*\xo h_\sI] \\
&= \langle \tr_R (h_\sI)\rangle \dim(h_\sI^*) && [\Gamma'(\sI)\text{ simple}] \\
&= \langle \tr_R(h_\sI)\rangle \dim(h_\sI) && [\text{Prop.~\ref{prop:leftrighttrace}}] \\
&= \langle \tr_R(h_\sI)\rangle^2 \dim(\Io_{\Gamma'(\sI)}) && [\Gamma'(\sI)\text{ simple}] \\
&= \langle \tr_R(h_\sI)\rangle^2 \dim(\Gamma'(\sI)) &&
\end{align*}
\nid As this planar trace recurs as a scalar transformation factor, it is worth having a compact notation for it:
\[\lambda_\sI := \langle \tr_R(h_\sI)\rangle\]
Recall that this trace of a simple 1-morphism between simple objects is nonzero by Lemma \ref{lem:planartracenonzero}.  Altogether, the 1-simplex factor transforms by the square inverse of that trace factor:
\[
\left( \dim(\Gamma(\sI)) \dim(\End_{\tc{C}}(\Gamma(\sI))) n(\Gamma(\sI))\right)^{-1}
= \lambda_\sI^{-2} \left( \dim(\Gamma'(\sI)) \dim(\End_{\tc{C}}(\Gamma'(\sI))) n(\Gamma'(\sI))\right)^{-1}.
\]

\skiptocparagraph{The 2-simplex factor}

The 2-simplex factor has just one term for each 2-simplex $\sII \in K_2$, namely the dimension $\dim(\Gamma(\sII))$ of the simple 1-morphism label.  Given labels $\Gamma(\sII)$ and $\Gamma'(\sII)$ related by the equality $\Gamma(\sII) = h_{\partial^o_{02}\sII}^\ast \xo \Gamma'(\sII) \xo \left(h_{\partial^o_{01}\sII} \xz h_{\partial^o_{12}\sII}\right)$, the corresponding dimensions are related by a trace factor for each edge of the 2-simplex:
\begin{align*}
\dim(\Gamma(\sII)) &= \dim\left(h_{\partial^o_{[02]}\sII}^* \xo \Gamma'(\sII) \xo ( h_{\partial^o_{[01]}\sII} \xz h_{\partial^o_{[12]}\sII})\right) && \\
&= \dim\left(h^*_{\partial^o_{[02]}\sII} \xo \Gamma'(\sII)\right)~ \lambda_{\partial^o_{[01]}\sII}~ \lambda_{\partial^o_{[12]}\sII}
&& [\Gamma'(\partial^o_{[01]}\sII), \Gamma'(\partial^o_{[12]}\sII) \text{ simple}] \\
&= \lambda_{\partial^o_{[01]}\sII}~\lambda_{\partial^o_{[12]}\sII}~\lambda_{\partial^o_{[02]}\sII}~\dim(\Gamma'(\sII))
&& [\text{Prop~\ref{prop:leftrighttrace}}]
\end{align*}

\skiptocparagraph{The 3-simplex factor}

The 3-simplex factor is a tensor over the 3-simplices $\sIII \in K_3$ of the copairing $\cup_{\Gamma,\sIII}: k \to V^+(\Gamma,\sIII) \otimes V^-(\Gamma,\sIII)$.  To relate the copairing $\cup_{\Gamma,\sIII}$ and the copairing $\cup_{\Gamma',\sIII}$, we need to relate the corresponding associator state spaces $V^+(\Gamma,\sIII)$ and $V^+(\Gamma',\sIII)$, respectively $V^-(\Gamma,\sIII)$ and $V^-(\Gamma',\sIII)$.  To that end, for each 1-simplex $\sI \in K_1$, choose a 2-isomorphism $\beta_\sI : h_\sI^* \xo h_\sI \To \Io_{\Gamma(\sI)}$, and define, for each 3-simplex $\sIII \in K_3$, an isomorphism of associator state spaces as follows:
\begin{align*}
r^+_\sIII:V^+(\Gamma,\sIII) &\hspace{0.25cm}\to\hspace{0.25cm} V^+(\Gamma',\sIII) \\
\begin{tz}[td,scale=1.2]
\begin{scope}[xyplane=0]
\draw[slice] (0,0) to [out=up, in=\dl] (0.5,1) to [out=up, in=\dl] (1,2) to (1,3);
\draw[slice] (1,0) to [out=up, in=\dr] (0.5,1);
\draw[slice] (2,0) to [out=up, in=\dr] (1,2);
\end{scope}
\begin{scope}[xyplane=\h, on layer=superfront]
\draw[slice] (0,0) to [out=up, in=\dl] (1,2) to (1,3);
\draw[slice] (1,0) to [out=up, in=\dl] (1.5,1) to [out=up, in=\dr] (1,2);
\draw[slice] (2,0) to [out=up, in=\dr] (1.5,1);
\end{scope}
\begin{scope}[xzplane=0]
\draw[slice,short] (0,0) to (0, \h);
\draw[slice,short] (1,0) to (1,\h);
\draw[slice,short] (2,0) to (2,\h);
\end{scope}
\begin{scope}[xzplane=3]
\draw[slice,short] (1,0) to (1,\h);
\end{scope}
\coordinate (A) at (1.5,1,0.5*\h);
\draw[wire] (1,0.5,0) to [out=up, in=\dr] (A) to [out=\ur, in=down] (1,1.5,\h);
\draw[wire] (2,1,0) to [out=up, in=\dl] (A) to [out=\ul, in=down] (2,1,\h);
\node[dot] at (A){};
\end{tz}
&\hspace{0.25cm}\mapsto\hspace{0.25cm}
\def\d{-1.2}
\def\w{-0}
\def\l{3.5}
\begin{tz}[td,scale=1.3]
\begin{scope}[xyplane=0]
\draw[slice] (0,\d) to (0,0) to [out=up, in=\dl] (0.5,1) to [out=up, in=\dl] node[pos=0.2](R){} node[pos=0.6](L){} (1,2) to (1,\l);
\draw[slice] (1,\d) to (1,0) to [out=up, in=\dr] (0.5,1);
\draw[slice] (2,\d)to  (2,0) to [out=up, in=\dr] (1,2);
\end{scope}
\begin{scope}[xyplane=\h, on layer=superfront]
\draw[slice](0,\d) to  (0,0) to [out=up, in=\dl] (1,2) to (1,\l);
\draw[slice](1,\d) to  (1,0) to [out=up, in=\dl] (1.5,1) to [out=up, in=\dr] node[pos=0.2] (TR){} node[pos=0.6](TL){}(1,2);
\draw[slice] (2,\d) to (2,0) to [out=up, in=\dr] (1.5,1);
\end{scope}
\begin{scope}[xzplane=\d]
\draw[slice,short] (0,0) to (0, \h);
\draw[slice,short] (1,0) to (1,\h);
\draw[slice,short] (2,0) to (2,\h);
\end{scope}
\begin{scope}[xzplane=\l]
\draw[slice,short] (1,0) to (1,\h);
\end{scope}
\coordinate (A) at (1.5,1,0.5*\h);
\coordinate (B) at (1.5,0.8, 0.25*\h);
\coordinate (C) at (1.5,1.2, 0.7*\h);
\draw[wire] (1,0.5,0) to [out=up, in=\dr] (A) to [out=\ur, in=down] (1,1.5,\h);
\draw[wire] (2,1,0) to [out=up, in=\dl] (A) to [out=\ul, in=down] (2,1,\h);
\node[dot] at (A){};
\draw[wire,dashed, dashed ,dash pattern=on 5pt off 2.3pt, arrow data ={0.59}{<}] (\w,0,0) to (\w,0,\h);
\draw[wire,dashed, dashed ,dash pattern=on 5pt off 2.3pt, arrow data ={0.59}{<}] (\w,1,0) to (\w,1,\h);
\draw[wire,dashed, dashed ,dash pattern=on 5pt off 2.3pt, arrow data ={0.59}{<}] (\w,2,0) to (\w,2,\h);
\draw[wire,dashed, dashed ,dash pattern=on 5pt off 2.3pt, arrow data ={0.59}{>}] (2.7,1,0) to (2.7,1,\h);
\draw[wire, dashed, dash pattern=on 5pt off 1.5pt,arrow data={0.6}{<}] (L.center) to [out=up, in=\dl] (B.center);
\draw[wire, dashed, dash pattern=on 5pt off 2pt,arrow data={0.6}{>}] (R.center) to [out=up, in=\dr] (B.center);
\draw[wire, dashed, dash pattern = on 5pt off 2pt,arrow data={0.3}{>}] (TL.center) to [out=down, in=\ul] (C.center);
\draw[wire, dashed, dash pattern = on 5pt off 2.3pt,arrow data={0.45}{<}] (TR.center) to [out=down, in=\ur] (C.center);
\node[circle, scale=0.3, fill=black, draw] at (B){};
\node[circle, scale=0.3, fill=black, draw] at (C){};
\end{tz}
\end{align*}
\nid Here, the dashed lines denote the equivalence $h_\sI: \Gamma(\sI) \to \Gamma'(\sI)$ or its adjoint, and the two black dots denote the 2-isomorphism $\beta_\sI: h_{\sI}^* \xo h_{\sI}\iso \Io_{\Gamma(\sI)}$ and its inverse.  The isomorphism $r_\sIII^-: V^-(\Gamma,\sIII) \to V^-(\Gamma',\sIII)$ is defined analogously.

Recall that the pairing $\langle \cdot , \cdot \rangle_{\Gamma,\sIII} :  V^-(\Gamma,\sIII) \otimes V^+(\Gamma,\sIII) \to k$ is given by composing and taking a spherical trace.  Precomposing this pairing with the isomorphisms of associator state spaces, we see that the pairings for the labelings $\Gamma$ and $\Gamma'$ are related by a scalar factor for each 1-simplex in the 3-simplex in question:
\[
\langle \cdot ,\cdot  \rangle_{\Gamma,\sIII}  = \left( \lambda_{\partial^o_{[01]}\sIII}\lambda_{\partial^o_{[12]}\sIII}\lambda_{\partial^o_{[23]}\sIII}\lambda_{\partial^o_{[03]}\sIII}\right) \langle \cdot, \cdot \rangle_{\Gamma', \sIII} \xo \left( r_\sIII^- \otimes r_\sIII^+\right).
\]
It follows immediately that the corresponding copairings are related by the inverse scalar factors:
\[
\cup_{\Gamma,\sIII} =\left( \lambda_{\partial^o_{[01]}\sIII}\lambda_{\partial^o_{[12]}\sIII}\lambda_{\partial^o_{[23]}\sIII}\lambda_{\partial^o_{[03]}\sIII}\right)^{-1} \left(r_\sIII^+ \otimes r_{\sIII}^- \right)^{-1}\xo\cup_{\Gamma', \sIII}.
\]

\skiptocparagraph{The 4-simplex factor}

The 4-simplex factor is a tensor over the 4-simplices $\sIV \in K_4$ of 10j symbols such as
\[
\shrinker{.9}{
z(\Gamma,\sIV) : V^+(\Gamma,\partial^o_{[0123]}\sIV)\otimes V^+(\Gamma, \partial^o_{[0134]}\sIV)\otimes V^+(\Gamma,\partial^o_{[1234]}\sIV) \otimes V^-(\Gamma,\partial^o_{[0124]} \sIV )\otimes V^-(\Gamma, \partial^o_{[0234]} \sIV) \to k
}
\]
or
\[
\shrinker{.9}{
z(\Gamma, \sIV): V^+(\Gamma,\partial^o_{[0234]}\sIV) \otimes V^+(\Gamma, \partial^o_{[0124]}\sIV) \otimes V^-(\Gamma,\partial^o_{[1234]}\sIV) \otimes V^-(\Gamma, \partial^o_{[0134]}\sIV) \otimes V^-(\Gamma, \partial^o_{[0123]}\sIV)\to k
}
\]
depending on orientation.  Recall that this 10j symbol $z(\Gamma,\sIV)$ is defined as the spherical trace of one of the pentagonator composites depicted in Figure~\ref{fig:definitionz}.  To compare the 10j symbol $z(\Gamma,\sIV)$ with the 10j symbol $z(\Gamma',\sIV)$ for the alternative labeling $\Gamma'$, we precompose with an appropriate tensor product of the isomorphisms $r^+_\sIII$ and $r^-_\sIII$ of associator state spaces.  In the resulting spherical trace expression, after cancelling various $\beta_\sI$ isomorphisms, exactly five planar trace scalar factors remain, one for each edge of the 4-simplex:
\[
z(\Gamma,\sIV) = \left( \lambda_{\partial^o_{[01]}\sIV}\lambda_{\partial^o_{[12]}\sIV}\lambda_{\partial^o_{[23]}\sIV}\lambda_{\partial^o_{[34]}}\lambda_{\partial^o_{[04]}\sIV}\right)z(\Gamma',\sIV)\xo \bigg(\bigotimes_{\sIII \in K_3,\sIII \subseteq \sIV} r_{\sIII}^{\epsilon^{\sIV}_o(\sIII)} \bigg).
\]

\skiptocparagraph{Spherical links cause factor cancellation}

Combining the scalar factors calculated above, we see that the normalized 10j actions for the states $\Gamma$ and $\Gamma'$ are related by $N(\Gamma) = \gamma N(\Gamma')$ where
\begin{multline*}
\gamma = 
\bigg(\prod_{\sI\in \K_1} \lambda_\sI^{-2}\bigg) \cdot
\bigg( \prod_{\sII\in \K_2} \lambda_{\partial^o_{[01]}\sII}\lambda_{\partial^o_{[12]}\sII} \lambda_{\partial^o_{[02]}\sII}\bigg) \cdot \\
\bigg(\prod_{\sIII \in \K_3} \lambda_{\partial^o_{[01]}\sIII}\lambda_{\partial^o_{[12]}\sIII}\lambda_{\partial^o_{[23]}\sIII}\lambda_{\partial^o_{[03]}\sIII}\bigg)^{-1} \cdot
\bigg(\prod_{\sIV \in \K_4} \lambda_{\partial^o_{[01]}\sIV}\lambda_{\partial^o_{[12]}\sIV}\lambda_{\partial^o_{[23]}\sIV}\lambda_{\partial^o_{[34]}\sIV}\lambda_{\partial^o_{[04]}\sIV}\bigg)
\end{multline*}
We can now observe that the terms of this scalar factor cancel precisely when the link of every $k$-simplex, for $k \geq 1$, is a combinatorial sphere.
\begin{proof}[Proof of Lemma~\ref{lem:objectchange}]
As a convenient compact notation, when $\s \in K_p$ is a $p$-simplex, for $p$ equal to $2$ or $3$, define $\lambda_\s := \lambda_{\partial^o_{[0,p]}\s}$, that is $\lambda_\s$ is the trace factor associated to the maximal edge of the simplex $\s$.  By direct combinatorial rearrangement, we may rewrite the products appearing in the factor $\gamma$ as follows:
\[
\shrink{
\begin{array}{lrl}
\text{For $\sII \in \K_2$:}
&
\lambda_{\partial^o_{[01]}\sII}\lambda_{\partial^o_{[12]}\sII} \lambda_{\partial^o_{[02]}\sII} 
\hspace*{-.7ex}&\hspace*{-.5ex}= \bigg( \prod_{\sI \in \K_1,\sI \subseteq \sII}~ \lambda_\sI \bigg)
\\
\text{For $\sIII \in \K_3$:}
&
\lambda_{\partial^o_{[01]}\sIII}\lambda_{\partial^o_{[12]}\sIII}\lambda_{\partial^o_{[23]}\sIII}\lambda_{\partial^o_{[03]}\sIII}
\hspace*{-.7ex}&\hspace*{-.5ex}=
\bigg(\prod_{\sI\in \K_1, \sI \subseteq \sIII}\lambda_\sI\bigg) \bigg( \prod_{\sII\in \K_2, \sII \subseteq \sIII}\lambda_\sII \bigg)^{-1} (\lambda_{\sIII}^2)
\\
\text{For $\sIV \in \K_4$:}
& \hspace*{-1ex}
\lambda_{\partial^o_{[01]}\sIV}\lambda_{\partial^o_{[12]}\sIV}\lambda_{\partial^o_{[23]}\sIV}\lambda_{\partial^o_{[34]}\sIV}\lambda_{\partial^o_{[04]}\sIV}
\hspace*{-.7ex}&\hspace*{-.5ex}=
\bigg(\prod_{\sI\in \K_1, \sI \subseteq \sIV} \lambda_\sI \bigg)\bigg( \prod_{\sII\in \K_2, \sI \subseteq \sIV} \lambda_\sII \bigg)^{-1} \bigg( \prod_{\sIII \in \K_3, \sIII \subseteq \sIV} \lambda_{\sIII} \bigg)
\end{array}
}
\]
Collecting terms we have
\[
\gamma = \bigg(\prod_{\sI\in \K_1} \lambda_\sI^{\phi_\sI} \bigg)\bigg( \prod_{\sII\in \K_2} \lambda_\sII^{\phi_\sII} \bigg) \bigg( \prod_{\sIII \in \K_3} \lambda_\sIII^{\phi_{\sIII}}\bigg)
\]
where
\[\begin{array}{ll}
\phi_\sI= -2 + \left|\vphantom{\partial^2}\{ \sII\in \K_2~\middle|~ \sI\subseteq \sII\}\right| - \left|\{ \sIII \in \K_3~\middle|~ \sI\subseteq \sIII\}\right| + \left|\{ \sIV \in \K_4~\middle|~ \sI\subseteq \sIV\}\right| 
\hspace*{-.7ex}&\hspace*{-.5ex}= -2 +\chi(\lk(\sI))
\\
\phi_\sII = \left|\vphantom{\partial^2}\{\sIII \in \K_3~\middle|~ \sII\subseteq \sIII\}\right|- \left|\{ \sIV \in \K_4~\middle|~ \sII\subseteq \sIV\}\right|
\hspace*{-.7ex}&\hspace*{-.5ex}= \chi(\lk(\sII))
\\
\phi_{\sIII} = -2+ \left|\vphantom{\partial^2}\{ \sIV \in \K_4~\middle|~ \sIII \subseteq \sIV\}\right| 
\hspace*{-.7ex}&\hspace*{-.5ex}=-2 + \chi(\lk(\sIII))
\end{array}
\]
Here $\lk(\s)$ denotes the link of the simplex $\s$ and $\chi(\lk(\s))$ denotes the Euler characteristic of that link.  Because by assumption $K$ is a closed \emph{singular} combinatorial 4-manifold, the link of every 1-, 2-, and 3-simplex is piecewise-linearly homeomorphic to a combinatorial sphere, and therefore has Euler characteristic 2 or 0 depending on parity; the exponents in the transformation factor $\gamma$ vanish accordingly.
\end{proof}

\nid Note crucially that the preceding proof did not use the Euler characteristic of the link of a vertex of the triangulation, and therefore applies to singular (that is vertex-singular) combinatorial 4-manifolds.

\subsubsection{\for{toc}{Equivalences factor into 1-morphism-only and object-only equivalences}\except{toc}{Equivalences of states factor into 1-morphism-only and object-only equivalences}}

We can wrap up the proof of independence of the choice of simplicial skeleton by factoring any equivalence of states into one that changes only the 1-morphism labels and one that appropriately changes only the object labels.

\begin{proof}[Proof of Lemma~\ref{lem:10jinvariance}]
Observe that if two $\tc{C}$-states $\Gamma$ and $\Gamma'$ are equivalent, then (by composing with the 2-isomorphism between the chosen inverse equivalence $k_\sI$ and the adjoint inverse $h_\sI^\ast$) they are equivalent by a collection of adjoint inverse equivalences $h_\sI : \Gamma(\sI) \rightleftarrows \Gamma'(\sI) : h_\sI^\ast$ and isomorphisms $\Gamma(\sII) \cong h_{\partial^o_{02}\sII}^\ast \xo \Gamma'(\sII) \xo \left(h_{\partial^o_{01}\sII} \xz h_{\partial^o_{12}\sII}\right)$.  There is then a $\tc{C}$-state $\Gamma''$ with $\Gamma''(\sI) = \Gamma(\sI)$ and $\Gamma''(\sII) = h_{\partial^o_{02}\sII}^\ast \xo \Gamma'(\sII) \xo \left(h_{\partial^o_{01}\sII} \xz h_{\partial^o_{12}\sII}\right)$.  As $\Gamma$ and $\Gamma''$ are equivalent with identity $1$-equivalences, and $\Gamma''$ and $\Gamma'$ are equivalent with adjoint inverse equivalences and an equality of appropriate 2-simplex labels, the result follows from Lemma~\ref{lem:1morphchange} and Lemma~\ref{lem:objectchange}.
\end{proof}


\subsection{The state sum is independent of the vertex ordering}

Recall that in the state sum for an ordered oriented combinatorial 4-manifold $K^o$, the sum is over states $\Gamma: K^o_{(2)} \To \Delta \tc{C}^\sk$, that is labelings of the 1-simplices and 2-simplices of the semisimplicial 2-skeleton $K^o_{(2)}$ by elements of the skeletal labeling semisimplicial set $\Delta \tc{C}^\sk$.  We will show that for any two choices of global vertex ordering, there is a bijection of the two sets of (skeletal) states and that corresponding states have the same normalized 10j action.
\begin{lemma}[Reordering yields a 10j-preserving bijection of skeletal states] \label{lem:orderingbijection}
If $o$ and $o'$ are global vertex orderings of the combinatorial 4-manifold $K$, then there is a bijection $\tau: [K^o_{(2)}, \Delta \tc{C}^\sk] \cong [K^{o'}_{(2)}, \Delta \tc{C}^\sk]$ of skeletal states for the two orderings, such that the bijection preserves the normalized 10j action: $N(\tau(\Gamma)) = N(\Gamma)$.
\end{lemma}
\begin{corollary}[The state sum is invariant under vertex reordering] \label{cor:orderingindependence}
Given a spherical prefusion $2$-category, an oriented combinatorial $4$-manifold $K$, and any two orderings $o$ and $o'$ on the vertices of $K$, the corresponding state sums agree:
\[
Z_{\tc{C}}(K)_o = Z_{\tc{C}}(K)_{o'}.
\]
\end{corollary}
\nid Accordingly, we will henceforth denote the state sum simply by $Z_{\tc{C}}(K)$.

\subsubsection{Transposition of states and associator states}\label{sec:transposition}

Of course it suffices to prove Lemma~\ref{lem:orderingbijection} when the two orderings are related by a single transposition of the order of two order-adjacent vertices.  Let $\sigma$ denote the permutation corresponding to such a transposition, and let $\sigma_\s$ denote the permutation of $0, \ldots, p$ induced by the restriction of $\sigma$ to the $p$-simplex $\s$.  The only relevant difference between the orderings $o$ and $o'$ are the face maps of simplices, which control the structure of the corresponding states; these face maps are related by $\partial^{o'}_{[i_1\cdots i_r]} \s = \partial^{o}_{[\sigma_{\s}(i_1)\cdots \sigma_{\s}(i_r)]} \s$, where $0 \leq i_1 < \cdots < i_r \leq p$.  Given a state for the order $o$, we define a corresponding state for a transposed order $o'$, and then define an isomorphism of the corresponding associator state spaces.

\skiptocparagraph{State transposition}

As before, let $o$ and $o'$ be vertex orderings related by a single adjacent transposition $\sigma$.  Define a state transposition map
\begin{align*}
[K^o_{(2)},\Delta \tc{C}] &\ra [K^{o'}_{(2)},\Delta \tc{C}] \\
\Gamma &\mapsto \Gamma^\sigma
\end{align*}
where for $\sI \in K_1$, the transposed state label is
\[
\Gamma^\sigma(\sI) = \left\{\begin{array}{ll} \Gamma(\sI) & \text{if } \sigma_\sI = \id \\ \Gamma(\sI)^\# & \text{if } \sigma_\sI = (01) \end{array}\right.
\]
and for $\sII \in K_2$, the transposed state label $\Gamma^\sigma(\sII)$ is shown in Figure~\ref{fig:translabels}, for each relevant restriction of the transposition $\sigma$ to the 2-simplex $\sII$.  In that figure, as in Notation~\ref{not:labels} and Notation~\ref{not:duals}, for instance $\conj{[01]}_\Gamma$ is shorthand for $\Gamma(\partial^o_{[01]} \sII)^\#$ and $[12]_\sig$ is shorthand for $\sig(\partial^{o'}_{[12]} \sII)$.  Henceforth, when we write $[ij]$ without a subscript, it refers implicitly to $[ij]_\Gamma$.  (Note that the dual of a simple object in a prefusion 2-category is again simple: an object is simple if and only if its identity is a simple 1-morphism, and taking mates induces an isomorphism between the endomorphisms of the identity of an object and the endomorphisms of the identity of its dual.)
\begin{figure}[h]
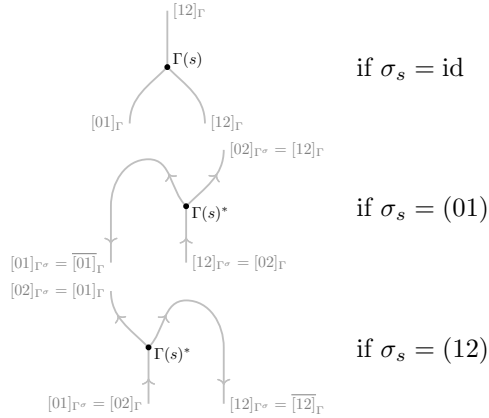

\[
\begin{array}{cl}
\begin{tz}[scale=0.5]
\draw[slice] (0,-0.5) to [out=up, in=\dl] (1,1) to (1,2.5);
\draw[slice] (2,-0.5) to [out=up, in=\dr] (1,1);
\node[dot] at (1,1) {};
\node[omor, right] at (1,1.2) {$\Gamma(\sII)$};
\node[obj, right] at (2,-0.5) {$[12]_\Gamma$};
\node[obj,left] at (0,-0.5) {$[01]_\Gamma$};
\node[obj,right] at (1,2.5) {$ [12]_\Gamma$};
\end{tz}
& \text{if }\sigma_s = \mathrm{id}
\\  
\begin{tz}[scale=0.5]
\draw[slice, arrow data ={0.2}{>},arrow data={0.8}{>}] (0,-0.5) to (0,1) to [out=\ur, in=down] (1,2.5);
\draw[slice, arrow data ={0.2}{>},arrow data ={0.9}{>}] (0,1) to [out=\ul, in=right] (-1, 2.25) to [out=left, in=up]   (-2, 1)  to (-2,-0.5);
\node[dot] at (0,1) {};
\node[omor, right] at (0,0.8) {$\Gamma(\sII)^*$};
\node[obj, right] at (0,-0.5) {$[12]_{\sig} = [02]_\Gamma$};
\node[obj, right] at (1,2.5) {$[02]_{\sig} = [12]_{\Gamma}$};
\node[obj, left] at (-2,-0.5) {$[01]_\sig = \conj{[01]}_\Gamma$};
\end{tz}
& \text{if }\sigma_s = (01)
\\ 
\begin{tz}[scale=0.5,xscale=-1]
\draw[slice, arrow data ={0.2}{>},arrow data={0.8}{>}] (0,-0.5) to (0,1) to [out=\ur, in=down] (1,2.5);
\draw[slice, arrow data ={0.2}{>},arrow data ={0.9}{>}] (0,1) to [out=\ul, in=right] (-1, 2.25) to [out=left, in=up]   (-2, 1)  to (-2,-0.5);
\node[dot] at (0,1) {};
\node[omor, right] at (0,0.8) {$\Gamma(\sII)^*$};
\node[obj, left] at (0,-0.5) {$[01]_\sig = [02]_\Gamma$};
\node[obj, left] at (1,2.5) {$[02]_\sig = [01]_\Gamma$};
\node[obj, right] at (-2,-0.5) {$[12]_\sig = \conj{[12]}_{\Gamma}$};
\end{tz}
& \text{if }\sigma_s=(12)
\end{array}
\]
\caption{The transposed state label $\Gamma^\sigma(\sII)$, for each transposition $\sigma_\sII$ of the 2-simplex $\sII$.
\label{fig:translabels} }
\end{figure}

\skiptocparagraph{Associator state transposition}

Recall from Definition~\ref{def:defvectorspaceV} that for a state $\Gamma$ and a 3-simplex $\sIII \in K_3$, the positive associator state space is $V^+(\Gamma, \sIII) := \Hom_{\tc{C}}\left( [(012)3]_\Gamma^\sIII, [0(123)]_\Gamma^\sIII\right)$, where the 1-morphisms $[(012)3]_\Gamma^\sIII$ and $[0(123)]_\Gamma^\sIII$ are each composites of two 2-simplex labels, as in Notation~\ref{not:parenthesisnotation}.  The negative associator state space is similar, taking $\Hom$ in the opposite direction.  Representative elements of these associator state spaces are depicted just after Definition~\ref{def:defvectorspaceV}.

For a vertex transposition $\sigma$ and the corresponding transposed state $\Gamma^\sigma$, the associator state spaces are Hom-spaces between pairwise composites of the transposed state labels $\Gamma^\sigma(\sII)$ defined above.  For instance, the negative associator state space $V^-(\sig,\sIII)$ is shown in Figure~\ref{fig:negassoc}, for the various restrictions of the transposition $\sigma$ to the particular 3-simplex $\sIII$.
\begin{figure}[h]
\def\scl{0.35}
\def\sh{\!\!}
\def\vpc{\vphantom{\conj{[01]}}}
\vspace{-0.3cm}
\[\arraycolsep=5pt\def\arraystretch{5}
\begin{array}{cl}
\Hom_{\tc{C}}\sh\left(\!
\begin{tz}[scale=\scl,xscale=1.5]
\draw[slice,arrow data={0.21}{<}] (0,0) to (0,3) to [out=up, in=up, looseness=1.5] (1,3) to [out=down, in=\ul] (1.5,2);
\draw[slice] (1,0) to [out=up, in=\dl] (1.5,1) to (1.5,2) to [out=\ur, in=down] (2,3) to (2,4);
\draw[slice] (2,0) to [out=up, in=\dr] (1.5,1);
\node[dot] (B) at (1.5,1){};
\node[dot] (T) at (1.5,2){};
\node[obj, left] at (0,0) {$\conj{[01]}$};
\node[obj, left] at (1,0) {$[02]\vpc$};
\node[obj, right] at (2,0) {$[23]\vpc$};
\node[obj, right] at (2,4) {$[13]$};
\node[omor, right] at ([yshift=4pt]B) {$[023]$};
\node[omor, right] at ([yshift=-6pt]T) {$\conj{[013]}$};
\end{tz}
,\!
\begin{tz}[scale=\scl]
\draw[slice,arrow data ={0.2}{<}] (0,0) to (0,2) to [out=up, in=up, looseness=1.5] (1,2) to [out=down, in=\ul] (1.5,1) to (1.5,0);
\draw[slice] (1.5,1) to [out=\ur, in=down] (2,2) to [out=up, in=\dl] (2.5,3) to (2.5,4);
\draw[slice] (2.5,3) to [out=\dr, in=up] (3,2) to (3,0);
\node[dot](L) at (1.5,1) {};
\node[dot] (R) at (2.5,3){};
\node[obj, left] at (0,0) {$\conj{[01]}$};
\node[obj, left] at (1.5,0) {$[02]\vpc$};
\node[obj, right] at (3,0) {$[23]\vpc$};
\node[obj, right] at (2.5,4) {$[13]$};
\node[omor, right] at ([yshift=-4pt]L) {$\conj{[012]}$};
\node[omor, right] at ([yshift=3pt]R) {$[123]$};
\end{tz}
\!
\right)
& \text{if } \sigma_{\sIII} = (01)
\\
\Hom_{\tc{C}}\sh\left(\!
\begin{tz}[scale=\scl,xscale=-1]
\draw[slice] (0,0) to (0,0.5) to [out=\ul, in=down] (-0.5,1.5) to [out=up, in=\dl] (1.5,3) to (1.5,4);
\draw[slice,arrow data ={0.78}{>}] (0,0.5) to [out=\ur, in=down] (0.5,1.5) to [out=up, in=up, looseness=1.5] (1.5,1.5) to (1.5,0);
\draw[slice] (1.5,3) to [out=\dr, in=up] (3, 1.5) to (3,0);
\node[dot] (B) at (0,0.5){};
\node[dot] (T) at (1.5,3){};
\node[obj, right] at (0,0) {$[13]\vpc$};
\node[obj, left] at (1.5,0) {$\conj{[12]}$};
\node[obj, left] at (3,0) {$[02]\vpc$};
\node[obj, right] at (1.5,4) {$[03]$};
\node[omor, left] at ([yshift=-5pt]B) {$\conj{[123]}$};
\node[omor, left] at ([yshift=5pt]T) {$[023]$};
\end{tz}
,
\begin{tz}[scale=\scl]
\draw[slice] (0,0) to (0,0.5) to [out=\ul, in=down] (-0.5,1.5) to [out=up, in=\dl] (1.5,3) to (1.5,4);
\draw[slice,arrow data ={0.78}{>}] (0,0.5) to [out=\ur, in=down] (0.5,1.5) to [out=up, in=up, looseness=1.5] (1.5,1.5) to (1.5,0);
\draw[slice] (1.5,3) to [out=\dr, in=up] (3, 1.5) to (3,0);
\node[dot] (B) at (0,0.5){};
\node[dot] (T) at (1.5,3){};
\node[obj, left] at (0,0) {$[02]\vpc$};
\node[obj, right] at (1.5,0) {$\conj{[12]}$};
\node[obj, right] at (3,0) {$[13]\vpc$};
\node[obj, right] at (1.5,4) {$[03]$};
\node[omor, right] at ([yshift=-5pt]B) {$\conj{[012]}$};
\node[omor, right] at ([yshift=5pt]T) {$[013]$};
\end{tz}
\!
\right)
& \text{if } \sigma_{\sIII} = (12)
\\
\Hom_{\tc{C}}\sh\left(\!
\begin{tz}[scale=\scl,xscale=-1]
\draw[slice,arrow data ={0.2}{<}] (0,0) to (0,2) to [out=up, in=up, looseness=1.5] (1,2) to [out=down, in=\ul] (1.5,1) to (1.5,0);
\draw[slice] (1.5,1) to [out=\ur, in=down] (2,2) to [out=up, in=\dl] (2.5,3) to (2.5,4);
\draw[slice] (2.5,3) to [out=\dr, in=up] (3,2) to (3,0);
\node[dot](L) at (1.5,1) {};
\node[dot] (R) at (2.5,3){};
\node[obj, left] at (0,0) {$\conj{[01]}$};
\node[obj, left] at (1.5,0) {$[02]\vpc$};
\node[obj, left] at (3,0) {$[23]\vpc$};
\node[obj, right] at (2.5,4) {$[13]$};
\node[omor, right] at ([yshift=-4pt]L) {$\conj{[012]}$};
\node[omor, right] at ([yshift=3pt]R) {$[123]$};
\end{tz}
\,\,\,\,,\!
\begin{tz}[scale=\scl,xscale=-1.5]
\draw[slice,arrow data={0.21}{<}] (0,0) to (0,3) to [out=up, in=up, looseness=1.5] (1,3) to [out=down, in=\ul] (1.5,2);
\draw[slice] (1,0) to [out=up, in=\dl] (1.5,1) to (1.5,2) to [out=\ur, in=down] (2,3) to (2,4);
\draw[slice] (2,0) to [out=up, in=\dr] (1.5,1);
\node[dot] (B) at (1.5,1){};
\node[dot] (T) at (1.5,2){};
\node[obj, left] at (0,0) {$\conj{[01]}$};
\node[obj, left] at (1,0) {$[02]\vpc$};
\node[obj, left] at (2,0) {$[23]\vpc$};
\node[obj, left] at (2,4) {$[13]$};
\node[omor, right] at ([yshift=2pt]B) {$[023]$};
\node[omor, right] at ([yshift=-2pt]T) {$\conj{[013]}$};
\end{tz}
\,
\right)
& \text{if } \sigma_{\sIII} = (23)
\end{array}
\]
\caption{The negative associator state space $V^-(\sig,\sIII)$, for the transpositions $\sigma_\sIII$ of the 3-simplex $\sIII$.
\label{fig:negassoc} }
\end{figure}

We now define isomorphisms of associator state spaces
\[
\Phi_{\sIII, \sigma_{\sIII}}^\pm: V^\pm(\Gamma, \sIII) \to V^{\pm\sgn(\sigma_{\sIII})}\left( \sig, \sIII\right)
\]
from the spaces for labeling $\Gamma$ to the spaces for transposed labeling $\Gamma^\sigma$.  Here $\sgn(\sigma_{\sIII})$ denotes the sign of the permutation $\sigma_\sIII$.  When $\sigma_\sIII = \id$, we set $\Phi^{\pm}_{\sIII, \id} := \id_{V^{\pm}(\Gamma, \sIII)}$.  When $\sigma$ restricts nontrivially to the 3-simplex $\sIII$, the isomorphism $\Phi_{\sIII, \sigma_\sIII}^+$ takes an associator state $\alpha \in V^+(\Gamma, \sIII)$ to one of the three composites $\Phi_{\sIII,\sigma_\sIII}^+(\alpha) \in V^-(\sig,\sIII)$ depicted in Figure~\ref{fig:definitionPhi}, according to the particular transposition $\sigma_\sIII$.  The negative isomorphism $\Phi_{\sIII, \sigma_\sIII}^-$ is defined analogously by the vertical reflections of the diagrams in that figure.

\begin{figure}[h]
\begin{calign}\nonumber 
\hspace{-.6cm}
\begin{tikzpicture}[td,scale=1]
 \begin{scope}[xyplane=0]
 	\draw[slice] (0,0) to (0,1) to [out=up, in=\dl] (0.5,3) to (0.5,4) to[out=\ur, in=down] 		(1.5,6.5);
	\draw[slice] (1.5,0) to (1.5,1.5) to [out=up, in=\dr] (0.5,3);
	\draw[slice] (-1.5,0) to (-1.5,4) to [out=up, in=left] (-0.5,5) to [out=right, in=\ul] 			(0.5,4);
 \end{scope}
 \begin{scope}[xyplane=\h]
 	\draw[slice] (0,0) to (0,1) to [out= \ul, in=\dl, looseness=2] (0,2) to [out=up, in=\dl] 		(0.5,3) to (0.5,4) to[out=\ur, in=down] (1.5,6.5);
	\draw[slice] (0,1) to [out=\ur, in=\dr, looseness=2] (0,2);
	\draw[slice] (1.5,0) to (1.5,1) to [out=up, in=\dr] (0.5,3);
	\draw[slice, on layer=front] (-1.5,0) to (-1.5,4) to [out=up, in=left] (-0.5,5);
	\draw[slice] (-0.5,5) to [out=right, in=\ul] 			(0.5,4);
 \end{scope}
 \begin{scope}[xyplane=2*\h]
 	\draw[slice] (0,0) to (0,1);
 	\draw[slice, on layer =front] (0,1) to [out= \ul, in=\dl, looseness=1.5] (0.5,3);
	\draw[slice] (0.5,3) to (0.5,4) to[out=\ur, in=down] (1.5,6.5);
	\draw[slice] (0,1) to [out=\ur, in=\dl] (1,2);
	\draw[slice] (1.5,0) to [out=up, in=\dr](1,2) to [out=up, in=\dr] (0.5,3);
	\draw[slice, on layer=front] (-1.5,0) to (-1.5,4) to [out=up, in=left] (-0.5,5);
	\draw[slice] (-0.5,5) to [out=right, in=\ul] (0.5,4);
 \end{scope}
 \begin{scope}[xyplane=3*\h]
 	\draw[slice] (0,0) to (0,1);
 	\draw[slice, on layer=front] (0,1) to [out= \ul, in=down] (-0.5,2) to (-0.5, 4) to [out=up, in=right] (-1,5) to [out=left, in=up] (-1.5, 4) to (-1.5, 0);
	\draw[slice] (0,1) to [out=\ur, in=\dl] (1,2);
	\draw[slice] (1.5,0) to [out=up, in=\dr](1,2) to[out=up, in=down] (1.5, 6.5);
 \end{scope}
 \begin{scope}[xyplane=4*\h, on layer=front]
 	\draw[slice] (0,0) to (0,1) to [out= \ul, in=down] (-0.5,2) to [out=up, in=right] 					(-1,2.5);
 	\draw[slice,on layer=front] (-1,2.5) to [out=left, in=up] (-1.5, 2) to (-1.5, 0);
	\draw[slice] (0,1) to [out=\ur, in=\dl] (1,4);
	\draw[slice] (1.5,0) to [out=up, in=\dr](1,4) to[out=up, in=down] (1.5, 6.5);
 \end{scope}
 \begin{scope}[xzplane=0]
 \draw[slice,short] 	(-1.5,0) to (-1.5, 4*\h);
\draw[slice,short] (0,0) to (0,4*\h);
 \draw[slice,short]			(1.5,0) to (1.5,4*\h);
 \end{scope}
  \begin{scope}[xzplane=6.5]
 \draw[slice,short] (1.5,0) to (1.5,4*\h);
 \end{scope}
 \draw[slice, on layer=front] (4.95,-0.81, 0) to (4.95,-0.81 , 2*\h) to [out=up, in=down] (5.01,-1.05, 3*\h) to [out=up, in=down] node[mask point, pos=0.87](MP){} (2.49, -1.1, 4*\h);
  \coordinate (A) at (2.5, 0.5, 1.5*\h);
 \begin{scope}
 \draw[wire,arrow data ={0.11}{<}] (4,0.5,0) tonode[pos=0.15](label1){} (4,0.5, 2.1*\h) to [out=up, in=up, looseness=4] (3, 0.5, 2.1*\h);
  \draw[wire, on layer =front] (3,0.5, 1.9*\h) to (3,0.5, 2.1*\h);
 \draw[wire,arrow data={0.17}{>}] (3,0.5,0) to node[pos=0.3](label2){} (3,0.5, \h) to [out=up, in=\dl] (A) to [out=\ul, in=down] (3,0.5, 1.9*\h);
 \draw[wire] (2,0, \h) to [out=down, in=down, looseness=4] (1,0, \h) to  (1,0,1.9*\h);
 \draw[wire, on layer=front] (1,0,1.9*\h) to (1,0,2.1*\h);
 \draw[wire] (1,0,2.1*\h) to (1,0, 2.8*\h);
 \draw[wire, on layer= front] (1,0,2.8*\h) to (1,0,3.2*\h);
 \draw[wire, arrow data ={0.8}{<}] (1,0,3.2*\h) to node[pos=0.81](label3){} (1,0, 4*\h);
 \cliparoundone{MP}{\draw[wire, arrow data ={0.94}{>}] (2,0,\h) to [out=up, in=\dr] (A) to [out=\ur, in=down] (2,1, 2*\h) to (2,1,3*\h) to [out=up, in=down, in looseness=2] node[pos=0.9](label4){}(4,1, 4*\h);}
 \end{scope}
 \node[dot] at (A) {};
  \node[obj, below right] at (0.1,-1.4,0) {$\overline{[01]}$};
 \node[obj, below right] at (0.1,0.1,0) {$[02]$};
 \node[obj, below right] at (0.1,1.6,0) {$[23]$};
 \node[obj, below left] at (6.4, 1.6, 0){$[13]$};
 \node[tmor, right] at ([xshift=0.1cm]A){$\alpha$};
 \node[omor, left] at(label1) {$\overline{[013]}$}; 
 \node[omor, right] at(label2) {$[023]$}; 
  \node[omor, right] at(label3) {$\overline{[012]}$}; 
  \node[omor, left] at(label4) {$[123]$}; 
 \end{tikzpicture}
&
\def\d{-1.5}
\def\hl{-0.25}
\def\sl{0.5}
\def\fixh{1.75}
\begin{tikzpicture}[td,scale=1]
\begin{scope}[xyplane=0, on layer=superback]
\draw[slice] (0,\d) to (0,0) to [out=up, in=\dl] (1.5,7) to (1.5,8.5);
\draw[slice] (1.5,\d) to  (1.5,1.75) to
[out=up, in=\ul,  out looseness=1, in looseness=2]node[pos=\stdr] (R1){} (3,1.5);
\draw[slice] (3,\d) to  (3,1.5) to [out=\ur, in=\dr] (1.5,7);
\end{scope}
\begin{scope}[xyplane=\h]
\draw[slice] (0,\d) to (0,2.5) to [out=\ul, in=\dl, looseness=1.5] (0,5.25) to [out=up, in=\dl] (1.5, 7) to (1.5, 8.5);
\draw[slice] (0,2.5) to [out=\ur, in=\dr, looseness=1.5] (0,5.25);
\draw[slice, ] (1.5,\d) to (1.5,\fixh);
\draw[slice, on layer=back] (1.5, \fixh) to 
[out =up, in=\ul, in looseness=2] node[pos=\stdr] (R2){} (3,1.5);
\draw[slice, on layer=back] (3,\d) to (3,1.5) to [out=\ur, in=\dr] (1.5,7);
\end{scope}
\begin{scope}[xyplane=2*\h]
\draw[slice, on layer=front] (0,\d) to  (0,\hl) to [out=\ul, in=\dl, looseness=0.5] (0,5.25) to [out=up, in=\dl] (1.5, 7) to (1.5, 8.5);
\draw[slice, on layer=front] (0,\hl) to [out=\ur, in=\dr, looseness=0.5] (0,5.25);
\draw[slice, on layer=front] (1.5,\d) to (1.5,0) to (1.5,2.75) to  [out=up, in=up, looseness=2] node[pos=\stdr] (R3){} (2.25, 2.75);
\draw[slice, on layer=back] (2.25, 2.75) to [out=down, in=\ul] (3,2.5);
\draw[slice, on layer=back] (3,\d) to (3,0) to (3,2.5) to [out=\ur, in=\dr, looseness=0.8] (1.5,7);
\end{scope}
\begin{scope}[xyplane=3*\h]
\draw[slice, on layer=front] (0,\d) to (0,\hl) to [out=\ul, in=\dl, looseness=\sl] (1.5,7)  to (1.5, 8.5);
\draw[slice] (0,\hl) to [out=\ur, in=\dl, out  looseness=0.4, in looseness=2] (3,6);
\draw[slice] (1.5,\d) to (1.5,2.75) to  [out=up, in=up, looseness=2] (2.25, 2.75);
\draw[slice, on layer=back] (2.25, 2.75) to [out=down, in=\ul] (3,2.5);
\draw[slice, on layer=back] (3,\d) to (3,2.5) to [out=\ur, in=\dr, looseness=0.7] (3,6) to [out=up, in=\dr] (1.5,7);
\end{scope}
\begin{scope}[xyplane=4*\h]
\draw[slice, on layer=front] (0,\d) to (0,\hl) to [out=\ul, in=\dl, looseness=\sl] (1.5,7)  to (1.5, 8.5);
\draw[slice] (0,\hl) to [out=\ur, in=down] (0.75, 4.25) to [out=up, in=up, looseness=2] node[pos=\stdr] (L4){} (1.5, 4.25) to [out=down, in=down, looseness=2]node[pos=\stdl](M4){} (2.25, 4.25) to [out=up, in=\dl] (3,6);
\draw[slice] (1.5,\d) to (1.5,2.75) to  [out=up, in=up, looseness=2] node[pos=\stdr] (R4){}(2.25, 2.75);
\draw[slice, on layer=back] (2.25,2.75) to [out=down, in=\ul] (3,2.5);
\draw[slice, on layer=back] (3,\d) to (3,2.5) to [out=\ur, in=\dr, looseness=0.7] (3,6) to [out=up, in=\dr] (1.5,7);
\end{scope}
\begin{scope}[xyplane=5*\h]
\draw[slice, on layer=front] (0,\d) to (0,\hl) to [out=\ul, in=\dl, looseness=\sl] (1.5,7)  to (1.5, 8.5);
\draw[slice] (0,\hl) to [out=\ur, in=down] (0.75, 4.25) to [out=up, in=up, looseness=2]  node[pos=\stdr] (L5){}(1.5, 4.25) to (1.5,0) to (1.5,\d);
\draw[slice, on layer=back] (3,2.5) to [out=\ul, in=\dl, looseness=0.6] (3,6);
\draw[slice, on layer=back] (3,\d) to (3,2.5) to [out=\ur, in=\dr, looseness=0.7] (3,6) to [out=up, in=\dr] (1.5,7);
\end{scope}
\begin{scope}[xyplane=6*\h]
\draw[slice, on layer=front] (0,\d) to (0,\hl) to [out=\ul, in=\dl, looseness=\sl] (1.5,7)  to (1.5, 8.5);
\draw[slice, on layer=front] (0,\hl) to [out=\ur, in=down] (0.75, 4.75) to [out=up, in=up, looseness=2] node[pos=\stdr] (L6){}(1.5, 4.75) to (1.5,\d);
\draw[slice, on layer=superback] (3,2.5) to [out=\ul, in=\dl, looseness=1.5] (3,4);
\draw[slice, on layer=superback] (3,\d) to (3,2.5) to [out=\ur, in=\dr, looseness=1.5] (3,4) to [out=up, in=\dr] (1.5,7);
\end{scope}
\begin{scope}[xyplane=7*\h, on layer=superfront]
\draw[slice] (0,\d) to (0,\hl) to [out=\ul, in=\dl, looseness=\sl] (1.5,7)  to (1.5, 8.5);
\draw[slice] (0,\hl) to [out=\ur, in=down] (0.75, 4.75) to [out=up, in=up, looseness=2] node[pos=\stdr] (L7){}(1.5, 4.75) to (1.5,\d);
\draw[slice] (3,\d) to (3,2.5) to  (3,4.5) to [out=up, in=\dr] (1.5,7);
\end{scope}
\begin{scope}[xzplane=\d] 
\draw[slice, short] (0,0) to (0,7*\h);
\draw[slice, short] (1.5,0) to (1.5,7*\h);
\draw[slice, short] (3,0) to (3,7*\h);
\end{scope}
\begin{scope}[xzplane=8.5] 
\draw[slice, short] (1.5,0) to (1.5,7*\h);
\end{scope}
\coordinate(A) at (6.5,1.5,2.5*\h);
\coordinate (cusp) at (4.2,1.5,3.5*\h);
\draw[slice] (cusp.center) to [out=\ulcusp, in=down] (L4.center) to (L5.center) to [out=up, in=down, out looseness=2] node[mask point, pos=0.51](MPT){} (L6.center) to (L7.center);
\draw[wire, on layer=front,arrow data={0.93}{<}] (5.25,0, \h) to [out=down, in=down, looseness=1.8] (2.5,0,\h) to [out=up, in=down, in looseness=0.5]node[mask point, pos=0.39] (MPB){} node[mask point, pos=0.68](MPB2){}(\hl,0,2*\h) to node[pos=0.89](label3){} (\hl,0, 7*\h);
\draw[wire, on layer=back, arrow data={0.23}{<}](1.5,3,0)to  node[pos=0.2](label4){}(1.5,3,\h);
\cliparoundone{MPB2}{\draw[wire] (1.5, 3, \h) to [out=up, in=down, in looseness=2, out looseness=1] (2.5,3,2*\h);}
\draw[wire, on layer=back] (2.5, 3, 2*\h) to (2.5, 3, 6*\h);
\draw[wire,on layer=superfront] (5.25,0, \h) to (5.25, 0, 2*\h) to [out=up, in=\dr] (A.center);
\draw[wire,on layer=front, arrow data={0.074}{>}, arrow data={0.93}{>}] (7,1.5,0) to  node[pos=0.25](label1){} (7,1.5, 2*\h) to [out=up, in=\dl] (A.center) to [out=\ul, in=down] (7, 1.5, 3*\h) to node[pos=0.87](label2){} (7,1.5,7*\h);
\cliparoundone{MPT}{\draw[wire] (A.center) to [out=\ur, in=down] (6,3,3*\h) to (6,3,5*\h) to [out=up, in=down] (4., 3, 6*\h) to [out=up, in=up, looseness=1.8]  (2.5,3, 6*\h);}
\coordinate (Ru) at (R4.center|-M4.center);
\cliparoundone{MPB}{\draw[slice] (R1.center) to (R2.center)to [out=up, in=down] (R3.center) to (R4.center)to (Ru.center) to[out=up, in=up, looseness=4.5]  (M4.center) to [out=down, in=\urcusp] (cusp.center);}
\node[dot] at (A){};
 \node[tmor, right] at ([xshift=0.1cm]A){$\alpha$};
  \node[obj, below right] at (\d+0.1,0.1,0) {$[02]$};
 \node[obj, below right] at (\d+0.1,1.6,0) {$\overline{[12]}$};
 \node[obj, below right] at (\d+0.1,3.1,0) {$[13]$};
 \node[obj, below left] at (8.4, 1.6, 0){$[03]$};
 \node[omor, left] at (label1.center) {$[023]$};
 \node[omor, left] at (label2.center) {$[013]$};
 \node[omor, right] at (label3.center) {$\overline{[012]}$};
  \node[omor, right] at (label4.center) {$\overline{[123]}$};
\end{tikzpicture}
&
 \begin{tikzpicture}[td]
 \begin{scope}[xyplane=0,xscale=-1]
 	\draw[slice] (0,0) to (0,1) to [out= \ul, in=down] (-0.5,2) to [out=up, in=right] 					(-1,2.5);
 	\draw[slice] (-1,2.5) to [out=left, in=up] (-1.5, 2) to (-1.5, 0);
	\draw[slice] (0,1) to [out=\ur, in=\dl] (1,4);
	\draw[slice] (1.5,0) to [out=up, in=\dr](1,4) to[out=up, in=down] (1.5, 5.5);
 \end{scope}
  \begin{scope}[xyplane=\h,xscale=-1]
 	\draw[slice] (0,0) to (0,1);
 	\draw[slice] (0,1) to [out= \ul, in=down] (-0.5,2) to (-0.5, 4) to [out=up, in=right] (-1,5) to [out=left, in=up] (-1.5, 4) to (-1.5, 0);
	\draw[slice] (0,1) to [out=\ur, in=\dl] (1,2);
	\draw[slice, on layer=front] (1.5,0) to [out=up, in=\dr](1,2) to[out=up, in=down] (1.5, 5.5);
 \end{scope}
  \begin{scope}[xyplane=2*\h,xscale=-1]
 	\draw[slice] (0,0) to (0,1);
 	\draw[slice, on layer =front] (0,1) to [out= \ul, in=\dl, looseness=1.5] (0.5,3);
	\draw[slice] (0.5,3) to (0.5,4) to[out=\ur, in=down] (1.5,5.5);
	\draw[slice] (0,1) to [out=\ur, in=\dl] (1,2);
	\draw[slice, on layer=front] (1.5,0) to [out=up, in=\dr](1,2) to [out=up, in=\dr] (0.5,3);
	\draw[slice, on layer=back] (-1.5,0) to (-1.5,4) to [out=up, in=left] (-1,5.5);
	\draw[slice] (-1,5.5) to [out=right, in=\ul] 			(0.5,4);
 \end{scope}
  \begin{scope}[xyplane=3*\h,xscale=-1]
 	\draw[slice] (0,0) to (0,1) to [out= \ul, in=\dl, looseness=2] (0,2) to [out=up, in=\dl] 		(0.5,3) to (0.5,4) to[out=\ur, in=down] (1.5,5.5);
	\draw[slice] (0,1) to [out=\ur, in=\dr, looseness=2] (0,2);
	\draw[slice, on layer=front] (1.5,0) to (1.5,1) to [out=up, in=\dr] (0.5,3);
	\draw[slice, on layer=back] (-1.5,0) to (-1.5,4) to [out=up, in=left] (-1,5.5);
	\draw[slice] (-1,5.5) to [out=right, in=\ul] 			(0.5,4);
 \end{scope}
\begin{scope}[xyplane=4*\h,xscale=-1, on layer=superfront]
 	\draw[slice] (0,0) to (0,1) to [out=up, in=\dl] (0.5,3) to (0.5,4) to[out=\ur, in=down] 		(1.5,5.5);
	\draw[slice] (1.5,0) to (1.5,1) to [out=up, in=\dr] (0.5,3);
	\draw[slice, on layer=back] (-1.5,0) to (-1.5,4) to  [out=up, in=left]  (-1,5.5) to [out=right, in=\ul] (0.5,4);
 \end{scope}
 \begin{scope}[xzplane=0]
 \draw[slice,short] 	(-1.5,0) to (-1.5, 4*\h);
\draw[slice,short] (0,0) to (0,4*\h);
 \draw[slice,short]			(1.5,0) to (1.5,4*\h);
 \end{scope}
  \begin{scope}[xzplane=5.5,xscale=-1]
 \draw[slice,short] (1.5,0) to (1.5,4*\h);
 \end{scope}
   \coordinate (A) at (2.5, -0.5, 2.5*\h);
 \begin{scope}
 \draw[wire, arrow data= {0.1}{>}] (4,-0.5,4*\h) to node[pos=0.15](label2){} (4,-0.5, 1.9*\h) to [out=down, in=down, looseness=4] (3, -0.5, 1.9*\h);
  \draw[wire, on layer =front] (3,-0.5, 2.1*\h) to (3,-0.5, 1.9*\h);
 \draw[wire, on layer=front,arrow data={0.15}{<}] (3,-0.5,4*\h) to node[pos=0.35](label3){} (3,-0.5, 3*\h) to [out=down, in=\ul] (A) to [out=\dl, in=up] (3,-0.5, 2.1*\h);
 \draw[wire] (2,0, 3*\h) to [out=up, in=up, looseness=4] (1,0, 3*\h) to  (1,0,2.1*\h);
 \draw[wire, on layer=front] (1,0,2.1*\h) to (1,0,1.9*\h);
 \draw[wire,] (1,0,1.9*\h) to (1,0, 1.2*\h);
 \draw[wire, on layer= front] (1,0,1.2*\h) to (1,0,0.8*\h);
 \draw[wire,arrow data ={0.78}{>}] (1,0,0.8*\h) to node[pos=0.4](label4){} (1,0, 0);
\draw[wire,] (2,0,3*\h) to [out=down, in=\ur] (A);
\draw[wire,on layer=superfront, arrow data ={0.93}{<}] (A) to [out=\dr, in=up] (2,-1, 2*\h) to (2,-1,\h) to [out=down, in=up, in looseness=2] node[mask point, pos=0.3](MP){} node[pos=0.93](label1){} (4,-1, 0);
 \cliparoundone{MP}{\draw[slice] (5.452,0.7, 4*\h) to (5.452,0.7 , 2*\h) to [out=down, in=up] (4.973,0.85, \h) to [out=down, in=up]  (2.48, 0.86, 0);}
 \end{scope}
 \node[dot] at (A){};
 \node[obj, below right] at (0.1,-1.4,0) {$[01]$};
 \node[obj, below right] at (0.1,0.1,0) {$[13]$};
 \node[obj, below right] at (0.1,1.6,0) {$\overline{[23]}$};
 \node[obj, below left] at (5.4, -1.4, 0){$[02]$};
 \node[tmor, right] at ([xshift=0.1cm]A){$\alpha$};
 \node[omor, left] at(label1) {$[012]$}; 
 \node[omor, left] at(label2) {$\overline{[023]}$}; 
  \node[omor, right] at(label3) {$[013]$}; 
  \node[omor, left] at(label4) {$\overline{[123]}$}; 
\end{tikzpicture}
\\ \nonumber
\Phi^+_{\sIII,(01)}(\alpha)
&
\Phi^+_{\sIII,(12)}(\alpha)
&
\Phi^+_{\sIII, (23)}(\alpha)
\end{calign}

\caption{The definition of the transposed associator state $\Phi^+_{\sIII, \sigma_{\sIII}}(\alpha) \in V^-(\sig, \sIII)$ for $\alpha \in V^+(\Gamma, \sIII)$.\label{fig:definitionPhi} }
\end{figure}

\subsubsection{State transposition preserves the normalized 10j action}

We now show that the transposition of a state has the same normalized 10j action as the original state.
\begin{lemma}[State transposition preserves the normalized 10j action] \label{lem:transpositionnormalized}
Let $o$ and $o'$ be global vertex orderings of the oriented combinatorial 4-manifold $K$, related by an adjacent vertex transposition $\sigma$.  For any state $\Gamma$ of $K$ with ordering $o$ and transposed state $\Gamma^\sigma$ of $K$ with ordering $o'$, the normalized 10j actions agree: $N(\Gamma^\sigma) = N(\Gamma)$.
\end{lemma}
\nid As the normalized 10j action is a normalization factor times the composite of the 10j symbol with the copairing of associator state spaces, it suffices to check that each of these three elements is appropriately preserved by state transposition.

\skiptocparagraph{The copairing intertwines the associator state transposition}

\begin{lemma}[The copairing intertwines the associator state transposition] \label{lem:transposecopairing}
For $\Gamma$ a state of the 4-manifold $K$ with ordering $o$, a 3-simplex $\sIII \in K_3$, and an adjacent vertex transposition $\sigma$, the isomorphisms of associator state spaces $\Phi_{\sIII, \sigma_{\sIII}}^\pm: V^\pm(\Gamma, \sIII) \to V^{\pm\sgn(\sigma_{\sIII})}\left( \sig, \sIII\right)$ intertwine the copairings $\cup_{\Gamma,\sIII}: k \to V^+(\Gamma, \sIII) \otimes V^-(\Gamma, \sIII)$ and $\cup_{\Gamma^\sigma,\sIII}: k \to V^+(\sig, \sIII) \otimes V^-(\sig, \sIII)$ in the sense that:
\[
\textrm{sw}^{\sigma_\sIII} \xo (\Phi^+_{\sIII, \sigma_{\sIII}} \otimes \Phi^-_{\sIII, \sigma_{\sIII}}) \xo \cup_{\Gamma, \sIII}  = \cup_{\sig, \sIII}.
\]
Here $\textrm{sw}^{\sigma_\sIII}$ is trivial when $\sigma_\sIII$ is trivial and is the swap of the two tensor factors when $\sigma_\sIII$ is nontrivial.
\end{lemma}

To establish this relation we will use a more compact graphical notation where we omit explicitly drawing the horizontal slices (that show the objects), instead recording only the nontrivial 1-morphisms as black wires (when they are elements of the state labeling) and gray wires (when they are fold 1-morphisms).  For an object $[ij]$ (that is $\Gamma(\partial^o_{[ij]} \sIII)$), we graphically distinguish the four fold types $e_{[ij]}$, $i_{\conj{[ij]}} = e_{[ij]}^*$, $e_{\conj{[ij]}} = i_{[ij]}^*$, and $i_{[ij]}$ by an orientation arrow and a thin corona indicating the side of the fold with two sheets, as follows:
\def\hspl{\hspace{0.4cm}}
\def\hspr{\hspace{0.1cm}}
\begin{calign}\nonumber
\begin{tz}
\coordinate (b) at (0,0);
\coordinate (t) at (0,1.7);
\rightpathup{arrow data={0.5}{>}}{(0,0) to (0,1.7)}{b}{t}
\end{tz}
\hspl
\leftrightsquigarrow
\hspr
\begin{tz}[td]
\begin{scope}[xyplane=0]
\draw[slice] (0,-0.5) to (0,0.5) to [out=up, in=up, looseness=2]  node[pos=\stdr] (b){}(1,0.5) to (1,-0.5);
\end{scope}
\begin{scope}[xyplane=\h]
\draw[slice] (0,-0.5) to (0,0.5) to [out=up, in=up, looseness=2] node[pos=\stdr] (t){} (1,0.5) to (1,-0.5);
\end{scope}
\begin{scope}[xzplane=-0.5]
\draw[slice, short] (0,0) to (0,\h);
\draw[slice, short] (1,0) to (1,\h);
\end{scope}
\draw[slice] (b.center) to (t.center);
\node[obj, right] at (-0.5,0,0) {$[ij]$};
\node[obj, right] at (-0.5,1,0) {$\conj{[ij]}$};
\end{tz}
&
\begin{tz}
\coordinate (b) at (0,0);
\coordinate (t) at (0,1.7);
\rightpathup{arrow data={0.5}{<}}{(0,0) to (0,1.7)}{b}{t}
\end{tz}
\hspl
\leftrightsquigarrow
\hspr
\begin{tz}[td]
\begin{scope}[xyplane=0]
\draw[slice] (0,-0.5) to (0,0.5) to [out=up, in=up, looseness=2]  node[pos=\stdr] (b){}(1,0.5) to (1,-0.5);
\end{scope}
\begin{scope}[xyplane=\h]
\draw[slice] (0,-0.5) to (0,0.5) to [out=up, in=up, looseness=2] node[pos=\stdr] (t){} (1,0.5) to (1,-0.5);
\end{scope}
\begin{scope}[xzplane=-0.5]
\draw[slice, short] (0,0) to (0,\h);
\draw[slice, short] (1,0) to (1,\h);
\end{scope}
\draw[slice] (b.center) to (t.center);
\node[obj, right] at (-0.5,0,0) {$\conj{[ij]}$};
\node[obj, right] at (-0.5,1,0) {$[ij]$};
\end{tz}
\\ \nonumber
\begin{tz}
\coordinate (b) at (0,0);
\coordinate (t) at (0,1.7);
\leftpathup{arrow data={0.5}{<}}{(0,0) to (0,1.7)}{b}{t}
\end{tz}
\hspl
\leftrightsquigarrow
\hspr
\begin{tz}[td]
\begin{scope}[xyplane=0]
\draw[slice] (0,0.5) to (0,-0.5) to [out=down, in=down, looseness=2]  node[pos=\stdl] (b){}(1,-0.5) to (1,0.5);
\end{scope}
\begin{scope}[xyplane=\h]
\draw[slice] (0,0.5) to (0,-0.5) to [out=down, in=down, looseness=2] node[pos=\stdl] (t){} (1,-0.5) to (1,0.5);
\end{scope}
\begin{scope}[xzplane=0.5]
\draw[slice, short] (0,0) to (0,\h);
\draw[slice, short] (1,0) to (1,\h);
\end{scope}
\draw[slice] (b.center) to (t.center);
\node[obj, left] at (0.5,0,0) {$[ij]$};
\node[obj, left] at (0.5,1,0) {$\conj{[ij]}$};
\end{tz}
&
\begin{tz}
\coordinate (b) at (0,0);
\coordinate (t) at (0,1.7);
\leftpathup{arrow data={0.5}{>}}{(0,0) to (0,1.7)}{b}{t}
\end{tz}
\hspl
\leftrightsquigarrow
\hspr
\begin{tz}[td]
\begin{scope}[xyplane=0]
\draw[slice] (0,0.5) to (0,-0.5) to [out=down, in=down, looseness=2]  node[pos=\stdl] (b){}(1,-0.5) to (1,0.5);
\end{scope}
\begin{scope}[xyplane=\h]
\draw[slice] (0,0.5) to (0,-0.5) to [out=down, in=down, looseness=2] node[pos=\stdl] (t){} (1,-0.5) to (1,0.5);
\end{scope}
\begin{scope}[xzplane=0.5]
\draw[slice, short] (0,0) to (0,\h);
\draw[slice, short] (1,0) to (1,\h);
\end{scope}
\draw[slice] (b.center) to (t.center);
\node[obj, left] at (0.5,0,0) {$\conj{[ij]}$};
\node[obj, left] at (0.5,1,0) {$[ij]$};
\end{tz}
\end{calign}

\begin{proof}[Proof of Lemma~\ref{lem:transposecopairing}]
This relation of the copairings of course follows from the corresponding relation for the pairing, namely
\[ 
\langle \cdot, \cdot \rangle_{\Gamma, \sIII}\xo \left(\Phi^-_{\sIII, \sigma_{\sIII}}\otimes \Phi^+_{\sIII, \sigma_{\sIII}} \right)^{-1} \xo \textrm{sw}^{\sigma_\sIII} = \langle\cdot, \cdot \rangle_{\sig, \sIII} 
\]
Explicitly, we must show that for each of the three nontrivial transpositions $\sigma_\sIII = (01)$, $(12)$, or $(23)$, and for $\alpha \in V^-(\Gamma, \sIII), \beta \in V^+(\Gamma, \sIII)$, we have $\left\langle \Phi^-_{\sIII, \sigma_{\sIII}}(\beta), \Phi^+_{\sIII, \sigma_{\sIII}}(\alpha) \right\rangle_{\sig, \sIII} = \langle \alpha, \beta \rangle_{\Gamma, \sIII}$.  

For the transposition $\sigma_\sIII = (01)$, the proof appears as the first equation in Figure~\ref{fig:pairingtrans}.  There, `isotopy' refers to moves allowed by Definition~\ref{def:monoidal2cat}.  Note that in this isotopy step the two grey circles interchange which one is on the outside.  In the step using Proposition~\ref{prop:leftrighttrace}, note that the proposition is applied to the endomorphism obtained from the composite of $\alpha$ and $\beta$ with just the right-hand wire traced out; in particular, this is an endomorphism of a 1-morphism from a tensor of two objects to a single object, so the left trace is enclosed by two grey fold circles while the right trace is enclosed by only one grey fold circle.  For the transpositions $\sigma_\sIII = (12)$ and $\sigma_\sIII = (23)$, the proof appears as the second and third equations, respectively, in Figure~\ref{fig:pairingtrans}.
\end{proof}

\begin{figure}[p]
\def\scl{0.3}
\begin{multline*}
\left\langle \Phi_{\sIII, (01)}^-(\beta), \Phi^+_{\sIII, (01)}(\alpha)\right\rangle_{\Gamma^\sigma, \sIII}\!\! =~~
\Tr\left(\,\,
\begin{tz}[yscale=0.8,scale=\scl,scale=0.9]
\coordinate (top) at (2,4.5);
\coordinate (botm) at (-1,-1);
\rightpathup[fronta]{arrow data={0}{<}, on layer=front}{(-1,-1) to (-1,2) to [out=up, in=down] node[mask point, pos=0.7] (MPT){} (2,4.5)}{botm}{top}
\draw[wire] (0,-1) to (0,2) to [out=up, in=up, looseness=1.5] (1,2) to [out=down, in=\ul] (1.5, 1);
\cliparoundone{MPT}{\draw[wire] (1, -1) to (1,0) to [out=up, in=\dl] (1.5,1) to [out=\ur, in=down] (2,2) to [out=up, in=down] (1,4.5);}
\draw[wire] (1.5,1) to [out=\dr, in=up] (2,0) to [out=down, in=down, looseness=1.5] (3,0) to (3,4.5);
\coordinate (topm) at (-1,-1);
\coordinate (bot) at (2,-6.5);
\rightpathup[frontb]{on layer=front}{
(2,-6.5) to[out=up, in=down] node[mask point, pos=0.3](MPB){} (-1,-4) to (-1,-1)}{bot}{topm}
\draw[wire] (0,-1) to (0,-4) to [out=down, in=down, looseness=1.5] (1,-4) to [out=up, in=\dl] (1.5, -3);
\cliparoundone{MPB}{\draw[wire] (1,-1) to (1,-2) to [out=down, in=\ul] (1.5,-3) to [out=\dr, in=up] (2,-4) to [out=down, in=up] (1,-6.5);}
\draw[wire] (1.5,-3) to [out=\ur, in=down] (2,-2) to [out=up, in=up, looseness=1.5] (3,-2) to (3, -6.5);
\node[dot](B) at (1.5,-3){};
\node[dot](T) at (1.5,1){};
\node[tmor, right] at ([xshift=5pt]T) {$\alpha$};
\node[tmor, right] at ([xshift=5pt]B) {$\beta$};
\end{tz}
\,\,
\right)
\labeleqgap \superequals{\raisebox{3pt}{\begin{minipage}{1cm}\centering def+ \\cusp\end{minipage}}} \labeleqgap
\begin{tz}[yscale=0.8,scale=\scl]
\insidepath[frontb]{arrow data={0}{<}, on layer=front}{(-2,0) to (-2,1.5) to [out=up, in=up, looseness=1.5] node[mask point, pos=0.77](MPT){} (1.5,1.5) to (1.5,-1.5) to [out=down, in=down, looseness=1.5]node[mask point, pos=0.23](MPB){}  (-2,-1.5) to (-2,0)}
\draw[wire] (0,0) to [out=up, in=\dl] (0.5,1) to [out=\ul, in=down] (0,2) to [out=up, in=up, looseness=1.5] (-1,2) to (-1,-2) to [out=down, in=down, looseness=1.5] (0,-2) to [out=up ,in=\dl] (0.5,-1) to [out=\ul, in=down] (0,0);
\draw[wire] (0.5,-1) to [out=\ur, in=down] (1,0) to [out=up, in=\dr] (0.5,1);
\cliparoundtwo{MPB}{MPT}{\draw[wire] (0.5,1) to [out=\ur, in=down] (1,2) to (1,3) to [out=up, in=up, looseness=1.5] (2,3) to (2,-3) to [out=down, in=down, looseness=1.5] (1,-3) to (1,-2) to [out=up, in=\dr] (0.5,-1);}
\node[dot](T) at (0.5,1) {};
\node[dot] (B) at (0.5,-1){};
\node[tmor, right] at ([xshift=3pt]T){$\alpha$};
\node[tmor, right] at ([xshift=3pt]B) {$\beta$};
\insidepath{arrow data={0}{>}}{\ignore{\draw[slice,arrow data ={0}{>}]} (-3,0) to (-3,1.5) to [out=up, in=up, looseness=2] (3,1.5) to (3,-1.5) to [out=down, in=down, looseness=2] (-3,-1.5) to (-3,0)}
\end{tz}
~~\labeleqgap\superequals{isotopy} \labeleqgap
\begin{tz}[yscale=0.8,scale=\scl]
\insidepath{arrow data ={0}{>}}{
 (-2,0) to (-2,1.5) to [out=up, in=up, looseness=1.5] node[mask point, pos=0.77](MPT){} (3,1.5) to (3,-1.5) to [out=down, in=down, looseness=1.5]node[mask point, pos=0.23](MPB){}  (-2,-1.5) to (-2,0)}
\draw[wire] (0,0) to [out=up, in=\dl] (0.5,1) to [out=\ul, in=down] (0,2) to [out=up, in=up, looseness=1.5] (-1,2) to (-1,-2) to [out=down, in=down, looseness=1.5] (0,-2) to [out=up ,in=\dl] (0.5,-1) to [out=\ul, in=down] (0,0);
\draw[wire] (0.5,-1) to [out=\ur, in=down] (1,0) to [out=up, in=\dr] (0.5,1);
\draw[wire] (0.5,1) to [out=\ur, in=down] (1,2)  to [out=up, in=up, looseness=1.5] (2,2) to (2,-2) to [out=down, in=down, looseness=1.5] (1,-2)  to [out=up, in=\dr] (0.5,-1);
\node[dot](T) at (0.5,1) {};
\node[dot] (B) at (0.5,-1){};
\node[tmor, right] at ([xshift=3pt]T){$\alpha$};
\node[tmor, right] at ([xshift=3pt]B) {$\beta$};
\insidepath{arrow data={0}{<}}{
 (-3,0) to (-3,1.5) to [out=up, in=up, looseness=1.7] (4,1.5) to (4,-1.5) to [out=down, in=down, looseness=1.7] (-3,-1.5) to (-3,0) 
 }
\end{tz}
\\
\superequals{sphericality} \labeleqgap\eqgap
\begin{tz}[yscale=0.8,scale=\scl]
\insidepath{arrow data={0}{>}}{ 
(-2,0) to (-2,1.5) to [out=up, in=up, looseness=1.5] node[mask point, pos=0.77](MPT){} (3,1.5) to (3,-1.5) to [out=down, in=down, looseness=1.5]node[mask point, pos=0.23](MPB){}  (-2,-1.5) to (-2,0)
}
\draw[wire] (0,0) to [out=up, in=\dl] (0.5,1) to [out=\ul, in=down] (0,2) to [out=up, in=up, looseness=1.5] (-1,2) to (-1,-2) to [out=down, in=down, looseness=1.5] (0,-2) to [out=up ,in=\dl] (0.5,-1) to [out=\ul, in=down] (0,0);
\draw[wire] (0.5,-1) to [out=\ur, in=down] (1,0) to [out=up, in=\dr] (0.5,1);
\draw[wire] (0.5,1) to [out=\ur, in=down] (1,2)  to [out=up, in=up, looseness=1.5] (2,2) to (2,-2) to [out=down, in=down, looseness=1.5] (1,-2)  to [out=up, in=\dr] (0.5,-1);
\node[dot](T) at (0.5,1) {};
\node[dot] (B) at (0.5,-1){};
\node[tmor, right] at ([xshift=3pt]T){$\alpha$};
\node[tmor, right] at ([xshift=3pt]B) {$\beta$};
\insidepath{arrow data={0}{>}}{
 (-3,0) to (-3,1.5) to [out=up, in=up, looseness=1.7] (4,1.5) to (4,-1.5) to [out=down, in=down, looseness=1.7] (-3,-1.5) to (-3,0)
 }
\end{tz}
\labeleqgap \eqgap\hspace{0.05cm}\superequals{Prop~\ref{prop:leftrighttrace}} \labeleqgap\hspace{0.09cm}\hspace{0.05cm}
\begin{tz}[yscale=0.8,scale=\scl]
\draw[wire] (0,0) to [out=up, in=\dl] (0.5,1) to [out=\ul, in=down] (0,2)  to [out=up, in=up, looseness=1.5] (3,2) to (3,-2) to [out=down, in=down, looseness=1.5] (0,-2) to [out=up ,in=\dl] (0.5,-1) to [out=\ul, in=down] (0,0);
\draw[wire] (0.5,-1) to [out=\ur, in=down] (1,0) to [out=up, in=\dr] (0.5,1);
\draw[wire] (0.5,1) to [out=\ur, in=down] (1,2)  to [out=up, in=up, looseness=1.5] (2,2) to (2,-2) to [out=down, in=down, looseness=1.5] (1,-2)  to [out=up, in=\dr] (0.5,-1);
\node[dot](T) at (0.5,1) {};
\node[dot] (B) at (0.5,-1){};
\node[tmor, right] at ([xshift=3pt]T){$\alpha$};
\node[tmor, right] at ([xshift=3pt]B) {$\beta$};
\insidepath{arrow data={0}{>}}{ 
(-1,0) to (-1,1.5) to [out=up, in=up, looseness=1.9] (4,1.5) to (4,-1.5) to [out=down, in=down, looseness=1.9] (-1,-1.5) to (-1,0)
}
\end{tz}
\eqgap \superequals{def} \eqgap 
\Tr\left(\,\,\,
\begin{tz}[yscale=0.8,scale=\scl]
\draw[wire] (0,0) to [out=up, in=\dl] (0.5,1) to [out=\ul, in=down] (0,2);  
\draw[wire] (0,-2) to [out=up ,in=\dl] (0.5,-1) to [out=\ul, in=down] (0,0);
\draw[wire] (0.5,-1) to [out=\ur, in=down] (1,0) to [out=up, in=\dr] (0.5,1);
\draw[wire] (0.5,1) to [out=\ur, in=down] (1,2);
\draw[wire] (1,-2)  to [out=up, in=\dr] (0.5,-1);
\node[dot](T) at (0.5,1) {};
\node[dot] (B) at (0.5,-1){};
\node[tmor, right] at ([xshift=3pt]T){$\alpha$};
\node[tmor, right] at ([xshift=3pt]B) {$\beta$};
\end{tz}\,\,\,
\right)
~~=~~
\left\langle \alpha, \beta \right\rangle_{\Gamma, \sIII}
\end{multline*}
\def\scl{0.3}
\begin{multline*}
\left\langle \Phi^-_{\sIII, (12)} (\beta), \Phi^+_{\sIII, (12)}(\alpha)\right\rangle_{\sig, \sIII}\!\!  =~~
\Tr\left(\,\,%
\begin{tz}[yscale=0.8,scale=\scl,scale=1]
\draw[wire] (0,0) to (0,2) to [out=up, in=\dl] (0.5,3) to[out=\ul, in=down] (0,4) to (0,6);
\draw[wire, on layer=front] (0.5,3) to [out=\dr, in=up] (1,2) to [out=down, in=down, looseness=2] node[mask point, pos=0.71] (TR1){} (3,2) to [out=up, in=down]  node[mask point, pos=0.24] (TR2){}(4,6);
\coordinate(b) at (2.5,0);
\coordinate (m) at (2.5,4);
\coordinate (t) at (0.5,6);
\cliparoundone{TR1}{\rightpathup{}{(2.5,0) to (2.5,4)}{b}{m}}
\rightpathup[frontb]{on layer=front}{ (2.5,4) to [out=up, in=up, looseness=1.5] (2, 4) to [out=down, in=\urcusp] (1.75,3.5) to [out=\ulcusp, in=down] (1.5,4) to [out=up, in=down] node[mask point, pos=0.37] (MPT){} (0.5,6)}{m}{t}
\cliparoundtwo{TR2}{MPT}{\draw[wire] (0.5,3) to [out=\ur, in=down] (1,4) to [out=up, in=up, looseness=2] (3,4) to [out=down, in=up] (4,0);}
\draw[wire] (0,0) to (0,-2) to [out=down, in=\ul] (0.5,-3) to [out=\dl, in=up] (0,-4) to (0,-6);
\draw[wire, on layer=front] (0.5,-3) to [out=\ur, in=down] (1,-2) to [out=up, in=up, looseness=2] node[mask point, pos=0.71] (BR1){} (3,-2) to [out=down, in=up]  node[mask point, pos=0.24] (BR2){} (4,-6) ;
\coordinate(t2) at (0.5,-6);
\coordinate (m2) at (2.5,-4);
\cliparoundone{BR1}{\rightpathdown{arrow data ={0}{>}}{(2.5,0) to (2.5,-4)}{m2}{b}}
\rightpathdown[front]{on layer=superfront}{ (2.5,-4) to [out=down, in=down, looseness=1.5] (2,-4) to [out=up, in=\drcusp] (1.75, -3.5) to [out=\dlcusp, in=up] (1.5,-4) to [out=down, in=up]node[mask point, pos=0.37] (MPB){} (0.5,-6)}{t2}{m2}
\cliparoundtwo{BR2}{MPB}{\draw[wire] (0.5,-3) to [out=\dr, in=up] (1,-4) to [out=down, in=down,looseness=2] (3,-4) to [out=up, in=down] (4,0);}
\node[dot] (T) at (0.5,3){};
\node[dot] (B) at (0.5,-3){};
\node[tmor,right] at ([xshift=5pt]T) {$\alpha$};
\node[tmor,right] at ([xshift=5pt]B) {$\beta$};
\end{tz}
\,\,\right)
\labeleqgap\superequals{\raisebox{3pt}{\begin{minipage}{1cm}\centering def+\\cusp+\\isotopy\end{minipage}}} \labeleqgap
\begin{tz}[scale=\scl]
\draw[wire] (0,0) to [out=up, in=\dl] (0.5,1) to [out=\dr, in=up] (1,0) to [out=down, in=\ur] (0.5,-1) to [out=\ul, in=down] (0,0);
\draw[wire] (0.5,1) to [out=\ur, in=down] (1,2) to [out=up, in=up, looseness=1.5] (3.5,2) to (3.5,-2) to [out=down, in=down, looseness=1.5] (1,-2) to [out=up, in=\dr] (0.5,-1);
\draw[wire] (0.5,1) to [out=\ul, in=down] (0,2) to [out=up, in=up, looseness=1.5] (4.5,2) to (4.5,-2) to [out=down, in=down, looseness=1.5] (0,-2) to [out=up, in=\dl] (0.5,-1);
\insidepath{arrow data={0}{<}}{
 (2.5,0) to (2.5,1) to [out=up, in=up, looseness=1.5] (2, 1) to [out=down, in=\urcusp] (1.75,0.5) to [out=\ulcusp, in=down] (1.5,1) to [out=up, in=up, looseness=2] (3, 1) to (3,-1) to [out=down, in=down, looseness=2] (1.5,-1) to [out=up, in=\dlcusp] (1.75,-0.5) to [out=\drcusp, in=up] (2,-1) to [out=down, in=down, looseness=1.5] (2.5,-1) to (2.5,0)
 }
 \insidepath{arrow data={0}{>}}{
 (-1,0) to (-1,2) to [out=up, in=up, looseness=1.5] (5.5,2) to (5.5,-2) to [out=down, in=down, looseness=1.5] (-1,-2) to cycle
 }
\node[dot] (A) at (0.5,1){};
\node[dot](B) at (0.5,-1){};
\node[tmor,right] at ([xshift=5pt]A) {$\alpha$};
\node[tmor,right] at ([xshift=5pt]B) {$\beta$};
\end{tz}
\\
\labeleqgap\eqgap \superequals{\raisebox{3pt}{\begin{minipage}{1cm}\centering planar\\pivotality\\+cuspinv\end{minipage}}} \labeleqgap\eqgap
\begin{tz}[yscale=0.8,scale=\scl]
\draw[wire] (0,0) to [out=up, in=\dl] (0.5,1) to [out=\ul, in=down] (0,2)  to [out=up, in=up, looseness=1.5] (3,2) to (3,-2) to [out=down, in=down, looseness=1.5] (0,-2) to [out=up ,in=\dl] (0.5,-1) to [out=\ul, in=down] (0,0);
\draw[wire] (0.5,-1) to [out=\ur, in=down] (1,0) to [out=up, in=\dr] (0.5,1);
\draw[wire] (0.5,1) to [out=\ur, in=down] (1,2)  to [out=up, in=up, looseness=1.5] (2,2) to (2,-2) to [out=down, in=down, looseness=1.5] (1,-2)  to [out=up, in=\dr] (0.5,-1);
\node[dot](T) at (0.5,1) {};
\node[dot] (B) at (0.5,-1){};
\node[tmor, right] at ([xshift=3pt]T){$\alpha$};
\node[tmor, right] at ([xshift=3pt]B) {$\beta$};
\insidepath{arrow data={0}{>}}{ 
(-1,0) to (-1,1.5) to [out=up, in=up, looseness=1.9] (4,1.5) to (4,-1.5) to [out=down, in=down, looseness=1.9] (-1,-1.5) to (-1,0)
}
\end{tz}
~~~=~~~
\langle\alpha, \beta \rangle_{\Gamma, \sIII}
\end{multline*}
\begin{multline*}
\def\scl{0.3}
\left\langle \Phi_{\sIII, (23)}^-(\beta), \Phi^+_{\sIII, (23)}(\alpha)\right\rangle_{\Gamma^\sigma, \sIII} \!\! =~~
\Tr\left(\,\,\,\,\,\,\,
\begin{tz}[yscale=0.8,scale=\scl,scale=0.9]
\coordinate (top) at (2,-1);
\coordinate (botm) at (-1,-6.5);
\draw[wire] (1, -6.5) to (1,-5.5) to [out=up, in=\dl] (1.5,-4.5) to [out=\ur, in=down] (2,-3.5) to [out=up, in=down] node[mask point, pos=0.582](MPB){} (1,-1);
\cliparoundone{MPB}{
\rightpathup{}{(-1,-6.5) to (-1,-3.5) to [out=up, in=down] (2,-1)}{botm}{top}
}
\draw[wire] (0,-6.5) to (0,-3.5) to [out=up, in=up, looseness=1.5] (1,-3.5) to [out=down, in=\ul] (1.5, -4.5);
\draw[wire] (1.5,-4.5) to [out=\dr, in=up] (2,-5.5) to [out=down, in=down, looseness=1.5] (3,-5.5) to (3,-1);
\coordinate (topm) at (-1,4.5);
\coordinate (bot) at (2,-1);
\draw[wire] (1,4.5) to (1,3.5) to [out=down, in=\ul] (1.5,2.5) to [out=\dr, in=up] (2,1.5) to [out=down, in=up] node[mask point, pos=0.582] (MPT){} (1,-1);
\cliparoundone{MPT}{
\rightpathup{arrow data={0}{>}}{
(2,-1) to[out=up, in=down] node[mask point, pos=0.3](MPB){} (-1,1.4) to (-1,4.5)}{bot}{topm}
}
\draw[wire] (0,4.5) to (0,1.5) to [out=down, in=down, looseness=1.5] (1,1.5) to [out=up, in=\dl] (1.5, 2.5);
\draw[wire] (1.5,2.5) to [out=\ur, in=down] (2,3.5) to [out=up, in=up, looseness=1.5] (3,3.5) to (3, -1);
\node[dot](B) at (1.5,2.5){};
\node[dot](T) at (1.5,-4.5){};
\node[tmor, right] at ([xshift=5pt]B) {$\alpha$};
\node[tmor, right] at ([xshift=5pt]T) {$\beta$};
\end{tz}
\,\,
\right)
\labeleqgap \superequals{\raisebox{3pt}{\begin{minipage}{1cm} \centering def+\\ isotopy \end{minipage}}}\labeleqgap
\begin{tz}[ yscale=0.8,scale=\scl,scale=0.8]
\draw[wire] (1,1)  to [out=up, in=\dr] (0.5,2) to [out=\ur, in=down] (1,3) to [out=up, in=up, looseness=1.5] (2,3) to (2,-3) to [out=down, in=down, looseness=1.5] (1,-3) to [out=up, in=\dr] (0.5,-2)to [out=\ur, in=down] (1,-1) to(1,1);
\draw[wire] (0.5,2) to [out=\ul, in=down] (0,3) to [out=up, in=up, looseness=1.5] (3,3) to (3,-3) to [out=down, in=down, looseness=1.5] (0,-3) to [out=up, in=\dl] (0.5,-2);
\draw[wire] (0.5,2) to [out=\dl, in=up] (0,1) to [out=down, in=down, looseness=1.5] (-1,1) to (-1,3) to [out=up, in=up, looseness=2] (4,3) to (4,-3) to [out=down, in=down, looseness=2] (-1,-3) to (-1,-1) to [out=up, in=up, looseness=1.5] (0,-1) to [out=down, in=\ul] (0.5,-2);
\insidepath{arrow data={0}{>}}{
(-2,0) to (-2,3) to [out=up, in=up, looseness=2] (5,3) to (5,-3) to [out=down, in=down, looseness=2] (-2,-3) to cycle
}
\insidepath{arrow data={0}{>}}{
(-3,0) to (-3,3) to [out=up, in=up, looseness=2] (6,3) to (6,-3) to [out=down, in=down, looseness=2] (-3,-3) to cycle
}
\node[dot](A) at (0.5,2) {};
\node[dot] (B) at (0.5,-2){};
\node[tmor, right] at ([xshift=5pt]A) {$\alpha$};
\node[tmor, right] at ([xshift=5pt]B) {$\beta$};
\end{tz}
\\
\labeleqgap\eqgap \superequals{\raisebox{3pt}{\begin{minipage}{1,5cm} \centering planar\\ pivotality\end{minipage}}}\labeleqgap\hspace{0.1cm}
\begin{tz}[ yscale=0.8,scale=\scl,scale=0.8]
\draw[wire] (0,0) to [out=up, in=\dl] (0.5,1) to [out=\ul, in=down] (0,2) to [out=up, in=up, looseness=1.5] (-1,2) to (-1,-2) to [out=down, in=down, looseness=1.5] (0,-2) to [out=up, in=\dl] (0.5,-1) to [out=\ul, in=down] (0,0);
\draw[wire] (1,0) to [out=up, in=\dr] (0.5,1) to [out=\ur, in=down] (1,2) to [out=up, in=up, looseness=1.5] (2,2) to (2,-2) to [out=down, in=down, looseness=1.5] (1,-2) to [out=up, in=\dr] (0.5,-1) to [out=\ur, in=down] (1,0);
\insidepath{arrow data={0}{>}}{
(-2,0) to (-2,2) to [out=up, in=up, looseness=1.5] (3,2) to (3,-2) to [out=down, in=down, looseness=1.5] (-2,-2) to cycle
}
\insidepath{arrow data={0}{>}}{
(-3,0) to (-3,2) to [out=up, in=up, looseness=1.5] (4,2) to (4,-2) to [out=down, in=down, looseness=1.5] (-3,-2) to cycle
}
\node[dot] (A) at (0.5,1) {};
\node[dot] (B) at (0.5,-1){};
\node[tmor, right] at ([xshift=5pt]A) {$\beta$};
\node[tmor, right] at([xshift=5pt]B) {$\alpha$};
\end{tz}
\hspace{.19cm}\labeleqgap\superequals{Prop.~\ref{prop:leftrighttrace}} \labeleqgap\hspace{0.19cm}
\begin{tz}[ yscale=0.8,scale=\scl,scale=0.8]
\draw[wire] (0,0) to [out=up, in=\dl] (0.5,1) to [out=\ul, in=down] (0,2) to [out=up, in=up, looseness=1.5] (3,2) to (3,-2) to [out=down, in=down, looseness=1.5] (0,-2) to [out=up, in=\dl] (0.5,-1) to [out=\ul, in=down] (0,0);
\draw[wire] (1,0) to [out=up, in=\dr] (0.5,1) to [out=\ur, in=down] (1,2) to [out=up, in=up, looseness=1.5] (2,2) to (2,-2) to [out=down, in=down, looseness=1.5] (1,-2) to [out=up, in=\dr] (0.5,-1) to [out=\ur, in=down] (1,0);
\insidepath{arrow data={0}{>}}{
(-1,0) to (-1,2) to [out=up, in=up, looseness=1.5] (4,2) to (4,-2) to [out=down, in=down, looseness=1.5] (-1,-2) to cycle
}
\node[dot] (A) at (0.5,1) {};
\node[dot] (B) at (0.5,-1){};
\node[tmor, right] at ([xshift=5pt]A) {$\beta$};
\node[tmor, right] at([xshift=5pt]B) {$\alpha$};
\end{tz}
\eqgap = \eqgap \langle \alpha, \beta \rangle_{\Gamma, \sIII} 
\qedhere
\end{multline*}
\caption{Relating the pairing of associator state spaces and the pairings of their transpositions.}
\label{fig:pairingtrans}
\end{figure}

\skiptocparagraph{The 10j symbol intertwines the associator state transposition}

\begin{lemma}[The 10j symbol intertwines the associator state transposition] \label{lem:transpose10jsymbols}
For $\Gamma$ a state of the 4-manifold $K$ with ordering $o$, a 4-simplex $\sIV \in K_4$, and an adjacent vertex transposition $\sigma$, the isomorphisms $\Phi^\pm_{\sIII,\sigma_\sIII}$ of associator state spaces intertwine the 10j symbols in the sense that
\[
z(\sig, \sIV) \xo \left(\bigotimes_{\sIII \in \K_3, \sIII \subseteq \sIV}  \Phi^{\epsilon^{\sIV}_{o'}(\sIII)~ \sgn(\sigma_{\sIII})}_{\sIII, \sigma_{\sIII}} \right)
=
z(\Gamma, \sIV).
\]
\end{lemma}
\begin{proof}
Note that the two sides of the equation do have the same domain, because \linebreak $\epsilon^{\sIV}_{o'}(\sIII) ~\sgn(\sigma_{\sIII}) = \epsilon^{\sIV}_o(\sIII)$.
Assume $\sigma_\sIV$ is nontrivial and $\epsilon_{o'}(\sIV) = +1$ (thus $\epsilon_{o}(\sIV) = -1$); the case $\epsilon_{o'}(\sIV) = -1$ is analogous.  Abbreviate $\sigma_{ijkl}:= \sigma_{\partial^{o'}_{[ijkl]}\sIV} \in S_4$ for $0\leq i<j<k<l\leq 4$.  We need to show
\begin{multline*}
z(\sig, \sIV)\xo \left(\Phi^{\sgn(\sigma_{0123})}_{\partial^{o'}_{[0123]} \sIV, \sigma_{0123}} \otimes \Phi^{\sgn(\sigma_{0134})}_{\partial^{o'}_{[0134]} \sIV, \sigma_{0134} } \otimes \Phi^{\sgn(\sigma_{1234})}_{ \partial^{o'}_{[1234]}\sIV, \sigma_{1234}} \otimes \Phi^{-\sgn(\sigma_{0124})}_{\partial^{o'}_{[0124]} \sIV,\sigma_{0124}} \otimes \Phi^{-\sgn(\sigma_{0234})}_{\partial^{o'}_{[0234]} \sIV, \sigma_{0234}} \right)\\
\swapeq 
z(\Gamma, \sIV)
\end{multline*}
where by ``$\swapeq$" we mean they are equal up to some unspecified swap of the domain factors, depending on the particular permutation $\sigma_\sIV$ in question.  

Suppose $\sigma_\sIV = (01) \in S_5$.  The permutations $\sigma_{ijkl} \in S_4$ of the boundary faces are
\begin{calign}\nonumber
\sigma_{0123}= (01)& \sigma_{0134}= (01) & \sigma_{1234} = \id & \sigma_{0124} = (01) & \sigma_{0234} = \id
\end{calign}
The desired equation, with the appropriate swap accounted for, is
\begin{multline*}
z(\sig, \sIV) \left( \Phi^-_{\partial^o_{[0123]}\sIV, (01)}\left( \alpha \right),  \Phi^-_{\partial^o_{[0134]}\sIV, (01)}\left(\beta\right), \Phi^+_{\partial^o_{[0234]}\sIV, \id} \left(\gamma\right),  \Phi^+_{\partial^o_{[0124]}\sIV, (01)}\left(\delta\right), \Phi^-_{\partial^o_{[1234]}\sIV, \id}\left(\epsilon\right)\right) \\
= z(\Gamma, \sIV) \left( \gamma, \delta, \epsilon,\beta,\alpha\right)
\end{multline*}
Here $\alpha \in V^-(\Gamma, \partial^o_{[0123]}\sIV)$, $\beta \in V^-(\Gamma, \partial^o_{[0134]}\sIV)$, $\gamma \in V^+(\Gamma,\partial^o_{[0234]}\sIV)$, $\delta \in V^+(\Gamma, \partial^o_{[0124]}\sIV)$, and $\epsilon \in V^-(\Gamma,\partial^o_{[1234]}\sIV)$.  Note that since $\epsilon_{o'}(\sIV) = +1$ and $\epsilon_o(\sIV) = -1$, the expressions $z(\sig,\sIV)$ and $z(\Gamma,\sIV)$ are defined by respectively the left and right traces in Figure~\ref{fig:definitionz}.  The equation is checked as in Figure~\ref{fig:10jtranspose}.
\begin{figure}[h]
\def\scl{0.3}
\def\ls{1.5}
\def\d{-1}
\def\disttmor{4pt}
\begin{align*}
\Tr\left(
\begin{tz}[scale=\scl]
\coordinate (A1) at (2.5,3);
\coordinate(A2) at (-0.5,5);
\coordinate(A3) at (0.5,9);
\coordinate(A4) at (0.5,13);
\coordinate (A5) at (1.5, 17);
\coordinate (b) at (3.5,\d);
\coordinate(t) at (3.5,18);
\rightpathup[frontb]{arrow data={0.5}{<}, on layer=front}{(3.5,\d) to [out=up, in=down] node[mask point, pos=0.175] (MP1){} node[mask point, pos=0.53] (MP2){}(-3,4) to (-3,13) to [out=up, in=down] node[mask point, pos=0.58] (MP3){} node[mask point, pos=0.69] (MP4){} (3.5, 18)}{b}{t}
\cliparoundone{MP1}{
\draw[wire] (3,\d) to (3,2) to [out=up, in=\dr] (A1) to [out=\ur, in=down] (3,4) to [out=up, in=up, looseness=\ls] (4,4) to (4,\d);
}
\cliparoundone{MP2}{
\draw[wire] (A1) to [out=\dl, in=up] (2,2) to [out=down, in=down, looseness=\ls] (1,2) to (1,6) to [out=up, in=up, looseness=\ls] (0,6) to [out=down, in=\ur] (-0.5, 5) to [out=\dr, in=up] (0,4) to (0,\d);
}
\draw[wire] (A2) to [out=\dl, in=up] (-1,4) to [out=down, in=down, looseness=\ls] (-2,4) to (-2,6) to [out=up, in=down] (-1,14) to [out=up, in=up, looseness=\ls] (0,14) to[out=down, in=\ul]  (A4);
\draw[wire] (A1) to [out=\ul, in=down] (2,4) to (2,6) to [out=up, in=down] (1,8) to [out=up, in=\dr]  (A3);
\draw[wire] (A2) to [out=\ul, in=down] (-1, 6) to [out=up, in=down] (0,8) to [out=up, in=\dl] (A3);
\draw[wire] (A3) to [out=\dr, in=down] (0,10) to (0,12) to [out=up, in=\dl] (A4);
\draw[wire] (A4) to [out=\dr, in=up] (1,12) to [out=down, in=down, looseness=\ls] (2,12) to (2,14) to [out=up, in=down] node[mask point, pos=0.28] (MP5){} (4,18);
\cliparoundtwo{MP4}{MP5}{
\draw[wire] (A3) to [out=\ur, in=down] (1,10) to [out=up, in=down] (3,12) to (3,14) to [out=up, in=down] (2,16) to [out=up, in=\dr] (A5);
}
\cliparoundone{MP3}{\draw[wire] (A4) to [out=\ur, in=down] (1,14) to  (1,16) to [out=up, in=\dl] (A5);}
\draw[wire] (A5) to [out=\ul, in=down] (1,18);
\draw[wire] (A5) to [out=\ur, in=down] (2,18);
\node[dot] at (A1){};
\node[dot] at (A2){};
\node[dot] at (A3){};
\node[dot] at (A4){};
\node[dot] at (A5){};
\node[tmor, right] at ([xshift=\disttmor]A1) {$\alpha$};
\node[tmor, right] at ([xshift=\disttmor]A2) {$\beta$};
\node[tmor, right] at ([xshift=\disttmor]A3) {$\gamma$};
\node[tmor, right] at ([xshift=\disttmor]A4) {$\delta$};
\node[tmor, right] at ([xshift=\disttmor]A5) {$\epsilon$};
\end{tz}\,\,\,\,\,
\right)
\labeleqgap & \superequals{\raisebox{3pt}{\begin{minipage}{1cm}\centering def.+\\cusp+\\isotopy\end{minipage}}} \labeleqgap
\begin{tz}[scale=\scl]
\coordinate (A1) at (-0.5,1);
\coordinate (A2) at (0.5,3);
\coordinate (A3) at (-0.5,5);
\coordinate (A4) at (-0.5,8);
\coordinate (A5) at (0.5,10);
\draw[wire] (A1) to [out=\ur, in=down] (0,2) to [out=up, in=\dl] (A2);
\draw[wire] (A1) to [out=\ul, in=down] (-1,2) to (-1,4) to [out=up, in=\dl] (A3) to [out=\dr, in=up] (0,4) to [out=down, in=\ul] (A2);
\draw[wire] (A2) to [out=\ur, in=down] (1,4) to (1,6) to[out=up, in=down]node[mask point, pos=0.5] (MP){} (0,7) to [out=up, in=\dr] (A4);
\draw[wire] (A3) to [out=\ul, in=down] (-1,6)  to (-1,7) to [out=up, in=\dl] (A4);
\cliparoundone{MP}{
\draw[wire] (A3) to [out=\ur, in=down] (0,6) to [out=up, in=down] (1,7) to (1,9) to [out=up, in=\dr] (A5);
}
\draw[wire] (A4) to [out=\ul, in=down] (-1,9) to [out=up, in=up, looseness=\ls] (-2,9) to (-2,0) to [out=down, in=down, looseness=\ls] (-1,0) to [out=up, in=\dl] (A1);
\draw[wire] (A4) to [out=\ur, in=down] (0,9) to [out=up, in=\dl] (A5);
\draw[wire] (A5) to [out=\ur, in=down] (1,11) to [out=up, in=up, looseness=\ls] (2,11) to (2,-1) to [out=down, in=down, looseness=\ls] (1,-1) to (1,2) to [out=up, in=\dr] (A2);
\draw[wire] (A5) to [out=\ul, in=down] (0,11) to [out=up, in=up, looseness=\ls] (3,11) to (3,-1) to [out=down, in=down, looseness=\ls] (0,-1) to (0,0) to [out=up, in=\dr] (A1);
\insidepath{arrow data={0}{<}}{ (-3,5) to (-3,11) to [out=up, in=up, looseness=1.25] (4,11) to (4,-1) to [out=down, in=down, looseness=1.25] (-3,-1) to cycle
}
\insidepath{arrow data={0}{>}}{ (-4,5) to (-4,11) to [out=up, in=up, looseness=1.3] (5,11) to (5,-1) to [out=down, in=down, looseness=1.3] (-4,-1) to cycle
}
\node[dot] at (A1){};
\node[dot] at (A2){};
\node[dot] at (A3){};
\node[dot] at (A4){};
\node[dot] at (A5){};
\node[tmor, right] at ([xshift=\disttmor]A1) {$\beta$};
\node[tmor, right] at ([xshift=\disttmor]A2) {$\alpha$};
\node[tmor, right] at ([xshift=\disttmor]A3) {$\gamma$};
\node[tmor, right] at ([xshift=\disttmor]A4) {$\delta$};
\node[tmor, right] at ([xshift=\disttmor]A5) {$\epsilon$};
\end{tz}
\hspace{0.2cm}\labeleqgap\superequals{\raisebox{3pt}{\begin{minipage}{1.5cm}\centering isotopy+\\Prop.~\ref{prop:leftrighttrace}\end{minipage}}} \labeleqgap\hspace{0.2cm}
\begin{tz}[scale=\scl]
\coordinate (A1) at (-0.5,1);
\coordinate (A2) at (0.5,3);
\coordinate (A3) at (-0.5,5);
\coordinate (A4) at (-0.5,8);
\coordinate (A5) at (0.5,10);
\draw[wire] (A1) to [out=\ur, in=down] (0,2) to [out=up, in=\dl] (A2);
\draw[wire] (A1) to [out=\ul, in=down] (-1,2) to (-1,4) to [out=up, in=\dl] (A3) to [out=\dr, in=up] (0,4) to [out=down, in=\ul] (A2);
\draw[wire] (A2) to [out=\ur, in=down] (1,4) to (1,6) to[out=up, in=down]node[mask point, pos=0.5] (MP){} (0,7) to [out=up, in=\dr] (A4);
\draw[wire] (A3) to [out=\ul, in=down] (-1,6)  to (-1,7) to [out=up, in=\dl] (A4);
\cliparoundone{MP}{
\draw[wire] (A3) to [out=\ur, in=down] (0,6) to [out=up, in=down] (1,7) to (1,9) to [out=up, in=\dr] (A5);
}
\draw[wire] (A4) to [out=\ul, in=down] (-1,9) to [out=up, in=up, looseness=\ls] (-2,9) to (-2,0) to [out=down, in=down, looseness=\ls] (-1,0) to [out=up, in=\dl] (A1);
\draw[wire] (A4) to [out=\ur, in=down] (0,9) to [out=up, in=\dl] (A5);
\draw[wire] (A5) to [out=\ur, in=down] (1,11) to [out=up, in=up, looseness=\ls] (2,11) to (2,-1) to [out=down, in=down, looseness=\ls] (1,-1) to (1,2) to [out=up, in=\dr] (A2);
\draw[wire] (A5) to [out=\ul, in=down] (0,11) to [out=up, in=up, looseness=\ls] (3,11) to (3,-1) to [out=down, in=down, looseness=\ls] (0,-1) to (0,0) to [out=up, in=\dr] (A1);
\insidepath{arrow data={0}{>}}{ (-3,5) to (-3,11) to [out=up, in=up, looseness=1.25] (4,11) to (4,-1) to [out=down, in=down, looseness=1.25] (-3,-1) to cycle
}
\insidepath{arrow data={0}{>}}{ (-4,5) to (-4,11) to [out=up, in=up, looseness=1.3] (5,11) to (5,-1) to [out=down, in=down, looseness=1.3] (-4,-1) to cycle
}
\node[dot] at (A1){};
\node[dot] at (A2){};
\node[dot] at (A3){};
\node[dot] at (A4){};
\node[dot] at (A5){};
\node[tmor, right] at ([xshift=\disttmor]A1) {$\beta$};
\node[tmor, right] at ([xshift=\disttmor]A2) {$\alpha$};
\node[tmor, right] at ([xshift=\disttmor]A3) {$\gamma$};
\node[tmor, right] at ([xshift=\disttmor]A4) {$\delta$};
\node[tmor, right] at ([xshift=\disttmor]A5) {$\epsilon$};
\end{tz}
\\
& \superequals{Prop.~\ref{prop:leftrighttrace}} \labeleqgap\hspace{0.2cm}
\begin{tz}[scale=\scl]
\coordinate (A1) at (-0.5,1);
\coordinate (A2) at (0.5,3);
\coordinate (A3) at (-0.5,5);
\coordinate (A4) at (-0.5,8);
\coordinate (A5) at (0.5,10);
\draw[wire] (A1) to [out=\ur, in=down] (0,2) to [out=up, in=\dl] (A2);
\draw[wire] (A1) to [out=\ul, in=down] (-1,2) to (-1,4) to [out=up, in=\dl] (A3) to [out=\dr, in=up] (0,4) to [out=down, in=\ul] (A2);
\draw[wire] (A2) to [out=\ur, in=down] (1,4) to (1,6) to[out=up, in=down]node[mask point, pos=0.5] (MP){} (0,7) to [out=up, in=\dr] (A4);
\draw[wire] (A3) to [out=\ul, in=down] (-1,6)  to (-1,7) to [out=up, in=\dl] (A4);
\cliparoundone{MP}{
\draw[wire] (A3) to [out=\ur, in=down] (0,6) to [out=up, in=down] (1,7) to (1,9) to [out=up, in=\dr] (A5);
}
\draw[wire] (A4) to [out=\ul, in=down] (-1,9) to (-1,11) to[out=up, in=up, looseness=\ls] (4,11) to (4,0) to [out=down, in=down, looseness=\ls] (-1,0) to [out=up, in=\dl] (A1);
\draw[wire] (A4) to [out=\ur, in=down] (0,9) to [out=up, in=\dl] (A5);
\draw[wire] (A5) to [out=\ur, in=down] (1,11) to [out=up, in=up, looseness=\ls] (2,11) to (2,0) to [out=down, in=down, looseness=\ls] (1,0) to (1,2) to [out=up, in=\dr] (A2);
\draw[wire] (A5) to [out=\ul, in=down] (0,11) to [out=up, in=up, looseness=\ls] (3,11) to (3,0) to [out=down, in=down, looseness=\ls] (0,0) to (0,0) to [out=up, in=\dr] (A1);
\insidepath{arrow data={0}{>}}{ (-2,5.5) to (-2,11) to [out=up, in=up, looseness=\ls] (5,11) to (5,0) to [out=down, in=down, looseness=\ls] (-2,0) to cycle
}
\node[dot] at (A1){};
\node[dot] at (A2){};
\node[dot] at (A3){};
\node[dot] at (A4){};
\node[dot] at (A5){};
\node[tmor, right] at ([xshift=\disttmor]A1) {$\beta$};
\node[tmor, right] at ([xshift=\disttmor]A2) {$\alpha$};
\node[tmor, right] at ([xshift=\disttmor]A3) {$\gamma$};
\node[tmor, right] at ([xshift=\disttmor]A4) {$\delta$};
\node[tmor, right] at ([xshift=\disttmor]A5) {$\epsilon$};
\end{tz}
\hspace{0.2cm}\labeleqgap\superequals{\raisebox{3pt}{\begin{minipage}{1.5cm}\centering planar\\pivotality\end{minipage}}} \labeleqgap\hspace{0.2cm}
\begin{tz}[scale=\scl]
\coordinate (A1) at (0.5,1);
\coordinate (A2) at (0.5,4);
\coordinate (A3) at (1.5,6);
\coordinate (A4) at (0.5,8);
\coordinate (A5) at (1.5,10);
\draw[wire] (A1) to [out=\ul, in=down] (0, 2) to (0,3) to [out=up, in=\dl] (A2);
\draw[wire] (2,0) to (2,2) to [out=up, in=down] node[mask point, pos=0.5] (MP){}(1,3) to [out=up, in=\dr] (A2);
\cliparoundone{MP}{
\draw[wire] (A1) to [out=\ur, in=down] (1,2) to [out=up, in=down] (2,3) to(2,5) to [out=up, in=\dr] (A3) to [out=\ul, in=down] (1,7) to [out=up, in=\dr] (A4);
}
\draw[wire] (A2) to [out=\ur, in=down] (1,5) to [out=up, in=\dl] (A3) to [out=\ur, in=down] (2,7) to (2,9) to [out=up, in=\dr] (A5);
\draw[wire] (A2) to [out=\ul, in=down] (0,5) to (0,7) to [out=up, in=\dl] (A4);
\draw[wire] (A4) to [out=\ur, in=down] (1,9) to [out=up, in=\dl] (A5);
\draw[wire] (A5) to [out=\ur, in=down] (2,11) to [out=up, in=up, looseness=\ls] (3,11) to (3,0) to [out=down, in=down, looseness=\ls] (2,0);
\draw[wire] (A5) to [out=\ul, in=down] (1,11) to [out=up, in=up, looseness=\ls] (4,11) to (4,0) to[out=down, in=down, looseness=\ls] (1,0) to [out=up, in=\dr] (A1);
\draw[wire] (A4) to [out=\ul, in=down] (0,9) to (0,11) to[out=up, in=up, looseness=\ls] (5,11) to (5,0) to [out=down, in=down, looseness=\ls] (0,0) to [out=up, in=\dl] (A1);
\insidepath{arrow data={0}{>}}{ (-1,5.5) to (-1,11) to [out=up, in=up, looseness=\ls] (6,11) to (6,0) to [out=down, in=down, looseness=\ls] (-1,0) to cycle
}
\node[dot] at (A1){};
\node[dot] at (A2){};
\node[dot] at (A3){};
\node[dot] at (A4){};
\node[dot] at (A5){};
\node[tmor, right] at ([xshift=\disttmor]A4) {$\beta$};
\node[tmor, right] at ([xshift=\disttmor]A5) {$\alpha$};
\node[tmor, right] at ([xshift=\disttmor]A1) {$\gamma$};
\node[tmor, right] at ([xshift=\disttmor]A2) {$\delta$};
\node[tmor, right] at ([xshift=\disttmor]A3) {$\epsilon$};
\end{tz}
\end{align*}
\caption{Relating the 10j symbol of a state and its $(01)$ transposition.}
\label{fig:10jtranspose}
\end{figure}

Next, when the permutation is $\sigma_\sIV = (12) \in S_5$, the permutations of the faces are
\begin{calign}\nonumber 
\sigma_{0123}= (12)&\sigma_{0134} = \id & \sigma_{1234} = (01) & \sigma_{0124} = (12) & \sigma_{0234}= \id
\end{calign}
The desired equation is
\[
\begin{split}
z(\sig, \sIV)&\left(\Phi^{-}_{\partial^{o}_{[0123]} \sIV, (12)}(\alpha), \Phi^{+}_{\partial^{o}_{[0234]} \sIV, \id} (\beta), \Phi^{-}_{ \partial^{o}_{[1234]}\sIV, (01)}(\gamma),\Phi^{+}_{\partial^{o}_{[0124]} \sIV,(12)} (\delta), \Phi^{-}_{\partial^{o}_{[0134]} \sIV, \id} (\epsilon)\right) \\
&=z(\Gamma, \sIV)\left(\beta,\delta,\gamma,\epsilon,\alpha\right)
\end{split}
\]
where $\alpha \in V^-(\Gamma, \partial^o_{[0123]}\sIV)$, $\beta \in V^+(\Gamma,\partial^o_{[0234]}\sIV)$, $\gamma \in V^-(\Gamma,\partial^o_{[1234]}\sIV)$, $\delta \in V^+(\Gamma, \partial^o_{[0124]}\sIV)$, and $\epsilon \in V^-(\Gamma, \partial^o_{[0134]}\sIV)$.
This is established as in Figure~\ref{fig:10jtranspose2}.

The cases $\sigma_\sIV = (23)$ and $\sigma_\sIV = (34)$ are analogous.
\end{proof}

\begin{figure}[h]
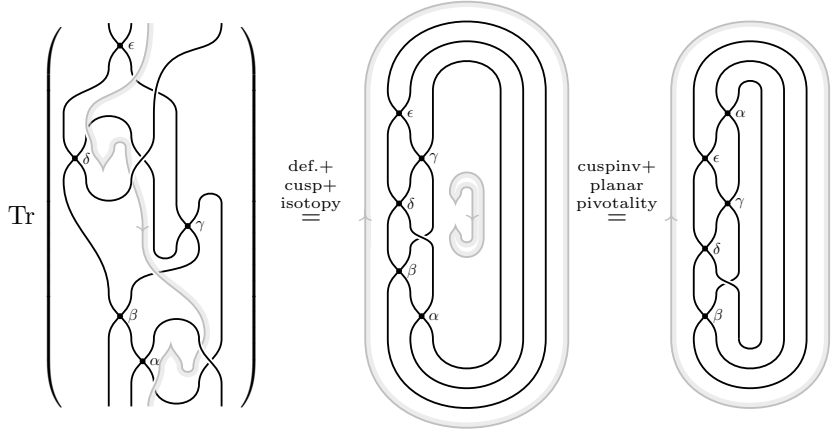

\def\scl{0.3}
\def\ls{1.5}
\def\d{-1}
\def\disttmor{3pt}
\def\upper{1}
\[
\Tr\left(
\begin{tz}[scale=\scl]
\coordinate(A1) at (0.5,2);
\coordinate(A2) at (-0.5,4);
\coordinate (A3) at (1.5+\upper, 8);
\coordinate (A4) at (-3.5+\upper, 11);
\coordinate (A5) at (-1.5+\upper, 16);
\coordinate (b) at (0.75,0);
\coordinate (t) at (0.75, 17);
\draw[wire] (A1) to [out=\ur, in=down] (1,3) to [out=up, in=up, looseness=\ls] (3,3) to [out=down, in=up] node[mask point, pos=0.19] (BR1){} node[mask point, pos=0.55] (BR2){}(4,1)  to (4,0);
\draw[wire] (A4) to [out=\dr, in=up] (-3+\upper,10) to [out=down, in=down, looseness=\ls] (-1+\upper,10) to [out=up, in=down]node[mask point, pos=0.11,minimum width=30pt ] (MPM1){} node[mask point, pos=0.5] (MPM2){} (0+\upper,12)  to (0+\upper, 13) to [out=up, in=down, out looseness=2] node[mask point, pos=0.09] (MPT){} (4,17);
\coordinate(b1) at (1.25, 1.75);
\rightpathup[frontc]{on layer=front}{
(0.75,0) to [out=up, in=down] node[mask point, pos=0.52] (MPG1){} (1.25,1.75)
}{b}{b1}
\coordinate (b2) at (3.25, 3.25);
\cliparoundone{BR1}{
\rightpathup{}{
(1.25,1.75) to [out=up, in=\dlcusp] (1.75, 2.75) to[out=\drcusp, in=up] (2.25, 1.75) to [out=down, in=down, looseness=\ls] (2.75, 1.75) to [out=up, in=down] (3.25, 3.25)
}{b1}{b2}
}
\coordinate (b3) at (-0.5+\upper, 8);
\rightpathup[frontb]{on layer =front, arrow data={1}{<}}{
 (3.25, 3.25) to [out=up, in=down] node[mask point, pos=0.7] (MPG2){} (-0.5+\upper,7.25)to (-0.5+\upper,8)
 }{b2}{b3}
 \coordinate (b4) at (-1.25+\upper, 11.5);
 \cliparoundone{MPM1}{
 \rightpathup{}{
 (-0.5+\upper, 8) to [out=up, in=down] (-1.25+\upper,11.5)
 }{b3}{b4}
 }
 \rightpathup[fronta]{on layer=front}{
 (-1.25+\upper, 11.5) to [out=up, in=up, looseness=\ls] (-1.75+\upper,11.5) to [out=down, in=\urcusp] (-2.25+\upper, 10.5) to [out=\ulcusp, in=down] (-2.75+\upper, 11.5) to [out=up, in=down] node[mask point, pos=0.69] (MPG3){} (-3+\upper, 12.5)  to [out=up, in=down, in looseness=2] node[mask point, pos=0.7] (MPG4){} (0.75, 17)%
}{b4}{t}
\draw[wire] (A2) to [out=\dl, in=up] (-1,3) to (-1,0);
\cliparoundtwo{MPG1}{BR2}{
\draw[wire] (A1) to [out=\dr, in=up] (1,1) to [out=down, in=down, looseness=\ls] (3,1) to [out=up, in=down] (4,3) to[out=up, in=down] (3+\upper,9) to [out=up, in=up, looseness=\ls] (2+\upper,9) to [out=down, in=\ur] (A3);
}
\draw[wire] (0,0) to (0,1) to [out=up, in=\dl] (A1);
\draw[wire] (A1) to [out=\ul, in=down] (0,3) to [out=up, in=\dr] (A2);
\cliparoundone{MPG2}{
\draw[wire] (A2) to [out=\ur, in=down] (0,5) to [out=up, in=down]  (2+\upper,7) to [out=up, in=\dr] (A3);
}
\cliparoundtwo{MPG3}{MPM2}{
\draw[wire] (A3) to [out=\dl, in=up] (1+\upper,7) to [out=down, in=down, looseness=\ls] (\upper,7) to (\upper,10) to [out=up, in=down] (-1+\upper,12) to [out=up, in=up, looseness=\ls] (-3+\upper,12) to [out=down, in=\ur] (A4);
}
\draw[wire] (A2) to [out=\ul, in=down] (-1,5) to [out=up, in=down] (-4+\upper,10)to [out=up, in=\dl] (A4);
\draw[wire] (A4) to [out=\ul, in=down] (-4+\upper,12)to  (-4+\upper, 13) to  [out=up, in=down] (-2+\upper, 15) to [out=up, in=\dl] (A5);
\cliparoundtwo{MPG4}{MPT}{
\draw[wire] (A3) to [out=\ul, in=down] (1+\upper,9) to (1+\upper,13) to  [out=up, in=down] (-1+\upper, 15) to [out=up, in=\dr] (A5);
}
\draw[wire] (A5) to [out=\ul, in=down] (-2+\upper, 17);
\draw[wire] (A5) to[out=\ur, in=down] (-1+\upper, 17);
\node[dot] at (A1){};
\node[dot] at (A2){};
\node[dot] at (A3){};
\node[dot] at (A4){};
\node[dot] at (A5){};
\node[tmor, right] at ([xshift=2pt]A1) {$\alpha$};
\node[tmor, right] at ([xshift=\disttmor]A2) {$\beta$};
\node[tmor, right] at ([xshift=\disttmor]A3) {$\gamma$};
\node[tmor, right] at ([xshift=\disttmor]A4) {$\delta$};
\node[tmor, right] at ([xshift=\disttmor]A5) {$\epsilon$};
\end{tz}
\,\,\,\,
\right)
\hspace{0.2cm}\labeleqgap\superequals{\raisebox{3pt}{\begin{minipage}{1.5cm}\centering def.+\\cusp+\\isotopy\end{minipage}}} \labeleqgap\hspace{0.2cm}
\begin{tz}[scale=\scl]
\coordinate (A1) at (0.5,1);
\coordinate (A2) at (-0.5,3);
\coordinate (A3) at (-0.5,6);
\coordinate (A4) at (0.5,8);
\coordinate (A5) at (-0.5,10);
\draw[wire]  (A1) to [out=\ul, in=down] (0,2) to [out=up, in=\dr] (A2);
\draw[wire] (A1) to [out=\ur, in=down] (1,2) to (1,4) to [out=up, in=down]node[mask point, pos=0.5] (MP){} (0,5) to [out=up, in=\dr] (A3);
\cliparoundone{MP}{
\draw[wire] (A2) to[out=\ur, in=down] (0,4) to [out=up, in=down] (1,5) to (1,7) to [out=up, in=\dr] (A4);
}
\draw[wire] (A2) to [out=\ul, in=down] (-1,4) to (-1,5) to [out=up, in=\dl] (A3);
\draw[wire] (A3) to [out=\ur, in=down] (0,7) to [out=up, in=\dl] (A4);
\draw[wire] (A4) to [out=\ul, in=down] (0,9) to [out=up, in=\dr] (A5);
\draw[wire] (A3) to [out=\ul, in=down] (-1,7) to (-1,9) to [out=up, in=\dl] (A5);
\draw[wire] (A4) to [out=\ur, in=down] (1,9) to (1,11) to [out=up, in=up, looseness=\ls] (4,11) to (4,0) to [out=down, in=down, looseness=\ls] (1,0) to [out=up, in=\dr] (A1);
\draw[wire] (A5) to [out=\ur, in=down] (0,11) to [out=up, in=up, looseness=\ls] (5,11) to (5,0) to [out=down, in=down, looseness=\ls] (0,0) to [out=up, in=\dl] (A1);
\draw[wire] (A5) to [out=\ul, in=down] (-1,11) to [out=up, in=up, looseness=\ls] (6,11) to (6,0) to [out=down, in=down, looseness=\ls] (-1,0) to (-1,2) to [out=up, in=\dl] (A2);
\insidepath{arrow data={0}{>}}{
(-2, 5.5) to (-2,11) to [out=up, in=up, looseness=\ls] (7,11) to (7,0) to [out=down, in=down, looseness=\ls] (-2,0) to cycle
}
\def\hc{5.5}
\def\vc{0.25}
\insidepath{arrow data={0}{<}}{
 (2.5+\vc,\hc) to (2.5+\vc,1+\hc) to [out=up, in=up, looseness=1.5] (2+\vc, 1+\hc) to [out=down, in=\urcusp] (1.75+\vc,0.5+\hc) to [out=\ulcusp, in=down] (1.5+\vc,1+\hc) to [out=up, in=up, looseness=2] (3+\vc, 1+\hc) to (3+\vc,-1+\hc) to [out=down, in=down, looseness=2] (1.5+\vc,-1+\hc) to [out=up, in=\dlcusp] (1.75+\vc,-0.5+\hc) to [out=\drcusp, in=up] (2+\vc,-1+\hc) to [out=down, in=down, looseness=1.5] (2.5+\vc,-1+\hc) to (2.5+\vc,\hc)
 }
\node[dot] at (A1){};
\node[dot] at (A2){};
\node[dot] at (A3){};
\node[dot] at (A4){};
\node[dot] at (A5){};
\node[tmor, right] at ([xshift=2pt]A1) {$\alpha$};
\node[tmor, right] at ([xshift=\disttmor]A2) {$\beta$};
\node[tmor, right] at ([xshift=\disttmor]A3) {$\delta$};
\node[tmor, right] at ([xshift=\disttmor]A4) {$\gamma$};
\node[tmor, right] at ([xshift=\disttmor]A5) {$\epsilon$};
\end{tz}
\hspace{0.2cm}\labeleqgap\superequals{\raisebox{3pt}{\begin{minipage}{1.5cm}\centering cuspinv+\\ planar\\pivotality\end{minipage}}} \labeleqgap\hspace{0.2cm}
\begin{tz}[scale=\scl]
\coordinate (A1) at (0.5,1);
\coordinate (A2) at (0.5,4);
\coordinate (A3) at (1.5,6);
\coordinate (A4) at (0.5,8);
\coordinate (A5) at (1.5,10);
\draw[wire] (A1) to [out=\ul, in=down] (0, 2) to (0,3) to [out=up, in=\dl] (A2);
\draw[wire] (2,0) to (2,2) to [out=up, in=down] node[mask point, pos=0.5] (MP){}(1,3) to [out=up, in=\dr] (A2);
\cliparoundone{MP}{
\draw[wire] (A1) to [out=\ur, in=down] (1,2) to [out=up, in=down] (2,3) to(2,5) to [out=up, in=\dr] (A3) to [out=\ul, in=down] (1,7) to [out=up, in=\dr] (A4);
}
\draw[wire] (A2) to [out=\ur, in=down] (1,5) to [out=up, in=\dl] (A3) to [out=\ur, in=down] (2,7) to (2,9) to [out=up, in=\dr] (A5);
\draw[wire] (A2) to [out=\ul, in=down] (0,5) to (0,7) to [out=up, in=\dl] (A4);
\draw[wire] (A4) to [out=\ur, in=down] (1,9) to [out=up, in=\dl] (A5);
\draw[wire] (A5) to [out=\ur, in=down] (2,11) to [out=up, in=up, looseness=\ls] (3,11) to (3,0) to [out=down, in=down, looseness=\ls] (2,0);
\draw[wire] (A5) to [out=\ul, in=down] (1,11) to [out=up, in=up, looseness=\ls] (4,11) to (4,0) to[out=down, in=down, looseness=\ls] (1,0) to [out=up, in=\dr] (A1);
\draw[wire] (A4) to [out=\ul, in=down] (0,9) to (0,11) to[out=up, in=up, looseness=\ls] (5,11) to (5,0) to [out=down, in=down, looseness=\ls] (0,0) to [out=up, in=\dl] (A1);
\insidepath{arrow data={0}{>}}{ (-1,5.5) to (-1,11) to [out=up, in=up, looseness=\ls] (6,11) to (6,0) to [out=down, in=down, looseness=\ls] (-1,0) to cycle
}
\node[dot] at (A1){};
\node[dot] at (A2){};
\node[dot] at (A3){};
\node[dot] at (A4){};
\node[dot] at (A5){};
\node[tmor, right] at ([xshift=\disttmor]A4) {$\epsilon$};
\node[tmor, right] at ([xshift=\disttmor]A5) {$\alpha$};
\node[tmor, right] at ([xshift=\disttmor]A1) {$\beta$};
\node[tmor, right] at ([xshift=\disttmor]A2) {$\delta$};
\node[tmor, right] at ([xshift=\disttmor]A3) {$\gamma$};
\end{tz}
\]
\caption{Relating the 10j symbol of a state and its $(12)$ transposition.}
\label{fig:10jtranspose2}
\end{figure}

\skiptocparagraph{The 10j normalization factors are preserved by state transposition}

\begin{proof}[Proof of Lemma~\ref{lem:transpositionnormalized}]
Composing Lemma~\ref{lem:transposecopairing} and Lemma~\ref{lem:transpose10jsymbols}, we see that the unnormalized 10j action is invariant under transposition: $Z(\Gamma^\sigma) = Z(\Gamma)$.  To check that the normalized 10j action is similarly invariant, it remains only to see that the normalization factors are unaffected by transposition.

For a 1-simplex $\sI \in K_1$, the transposed labeling $\Gamma^\sigma(\sI)$ is either $\Gamma(\sI)$ or $\Gamma(\sI)^\#$.  By sphericality, therefore, we have $\dim(\Gamma^\sigma(\sI)) = \dim(\Gamma(\sI))$.  Next, observe that for any object $A \in \tc{C}$, dualizing $(-)^\#$ defines an equivalence from the multifusion category $\End_{\tc{C}}(A)$ to $\End_{\tc{C}}(A^\#)^\mpt$, where $(-)^\mpt$ denotes the opposite monoidal product.  As the global dimension of a multifusion category is the same as the global dimension of its monoidal opposite, it follows that $\dim(\End_{\tc{C}}(\Gamma^\sigma(\sI))) = \dim(\End_{\tc{C}}(\Gamma(\sI)))$ for any 1-simplex $\sI$.  Evidently two simple objects are in the same component if and only if their duals are in the same component, so also the number of simple objects is unchanged by transposition: $n(\Gamma^\sigma(\sI)) = n(\Gamma(\sI))$.

Finally, note that for any 2-simplex $\sII \in K_2$, from Proposition~\ref{prop:leftrighttrace} and sphericality, it follows that $\dim(\Gamma^\sigma(\sII)) = \dim(\Gamma(\sII))$.
\end{proof}

\subsubsection{Transposition is a bijection of skeletal states}

We have seen that give a state $\Gamma$ of a given vertex-ordered 4-manifold, there is a corresponding state $\Gamma^\sigma$ for the manifold with a permuted vertex order, and that $\Gamma$ and $\Gamma^\sigma$ have the same normalized 10j action.  Equipped with that fact, we can now establish Lemma~\ref{lem:orderingbijection}, that there is a bijection between skeletal states of a manifold and of the reordered manifold, such that the normalized 10j action is preserved, and therefore complete the proof of Corollary~\ref{cor:orderingindependence}, that the state sum is invariant under vertex reorderings.

\begin{proof}[Proof of Lemma~\ref{lem:orderingbijection}]
Recall that $o$ and $o'$ are global vertex orders related by a permutation $\sigma$, and $\Delta \tc{C}^\sk$ is a simplicial skeleton for $\tc{C}$.  For every simple object $A$ of $\tc{C}$, let $A_0$ denote the unique simple object of $\Delta \tc{C}^\sk$ equivalent to $A$; choose inverse equivalences $h_A : A \leftrightarrows A_0 : k_A$ such that when $A = A_0$, the equivalences $h_A$ and $k_A$ are identities.  There is a natural transformation $X: \Delta \tc{C} \To \Delta \tc{C}^\sk$ taking an object $A$ to $A_0$ and taking a simple 1-morphism $f: A \xz B \ra C$ to the unique 1-morphism isomorphic to $h_C \xo f \xo (k_A \xz k_B)$.  This transformation induces a map $X_*: [K^o,\Delta\tc{C}] \ra [K^o,\Delta\tc{C}^\sk]$, and note that by Lemma~\ref{lem:10jinvariance}, this map preserves the normalized 10j action.  The composite $X_*((-)^\sigma): [K^o,\Delta\tc{C}^\sk] \ra [K^{o'},\Delta\tc{C}^\sk]$ taking a state $\Gamma$ to $X_*(\Gamma^\sigma)$ therefore also preserves the normalized 10j action.  It suffices to see that this composite is a bijection.

Define a map $Y_\sigma: [K^{o'},\Delta\tc{C}^\sk] \ra [K^{o'},\Delta\tc{C}]$ on 1-simplices $\sI \in K_1$ by
\[
Y_\sigma(\Gamma)(\sI) = \left\{\begin{array}{ll} 
\Gamma(\sI) & \text{if } \sigma_\sI = \id \\
((\Gamma(\sI)^\#)_0)^\# & \text{if } \sigma_\sI = (01)
\end{array}
\right.
\]
and on 2-simplices $\sII \in K_2$ by
\[
Y_\sigma(\Gamma)(\sII) = \left\{\begin{array}{ll}
\Gamma(\sII) & \text{if } \sigma_\sII = \id \\
\Gamma(s)\xo ( h_{((\partial_{01}\Gamma(s)^\#)_0 )^\#} \xz \Io_{\partial_{12}\Gamma(s)}) & \text{if } \sigma_\sII = (01) \\
\Gamma(s)\xo ( \Io_{\partial_{01}\Gamma(s)} \xz h_{((\partial_{12}\Gamma(s)^\#)_0 )^\#} ) & \text{if } \sigma_\sII = (12)
\end{array}
\right.
\]
This map is well defined on 2-simplices because $(((A^\#)_0)^\#)_0 = A$ for any object $A$ in the skeleton $\Delta \tc{C}^\sk$.  Finally observe that the composite $X_*((Y_\sigma(-))^\sigma): [K^{o'},\Delta\tc{C}^\sk] \ra [K^o,\Delta\tc{C}^\sk]$ is an inverse to $X_*((-)^\sigma): [K^o,\Delta\tc{C}^\sk] \ra [K^{o'},\Delta\tc{C}^\sk]$, as required.
\end{proof}

\subsection{The state sum is independent of the combinatorial structure}

Lastly, we show that the state sum $Z_{\tc{C}}(K)$ is independent of the combinatorial structure of $K$.

Recall that, by Theorem~\ref{thm:singular}, two singular combinatorial 4-manifolds are piecewise-linearly homeomorphic if and only if they are bistellar equivalent, that is if and only if there is a finite series of bistellar moves transforming one into the other. Hence, Theorem~\ref{thm:maintheorem} is a direct consequence of the following lemma, which we prove in this section.

\begin{lemma}[The state sum is invariant under bistellar equivalence] \label{lem:invariancebistellar} 
Let $\K$ and $\K'$ be bistellar equivalent oriented (singular) combinatorial $4$-manifolds. Then, the corresponding state sums agree: \[Z_{\tc{C}}(\K) = Z_{\tc{C}}(\K')\] 
\end{lemma}


\nid
To simplify notation, we will use the following abbreviations for simple objects $X$ and $1$-morphisms $f$ in the spherical prefusion $2$-category $\tc{C}$:
\begin{align*}
\dm(X)&:=\dim(X)~ \dim\left(\End_{\tc{C}}(X)\right)~ n(X)\\
\dm(f)&:= \dim(f) 
\end{align*}

\subsubsection{The state sum for combinatorial manifolds with boundary}\label{sec:boundary}

Invariance under bistellar moves is most easily established if we extend the definition of $Z_{\tc{C}}$ to combinatorial manifolds with boundary. Let $T$ be a closed oriented combinatorial $3$-manifold, and let $o$ be a total order on the vertices of $T$. Let $\epsilon_o:T_3 \to \{+1,-1\}$ be such that $\epsilon_o(\sIII)=+1$ if and only if the orientation of the $3$-simplex $\sIII\in T_3$ coincides with the one induced from the order $o_{\sIII}$. To such a combinatorial $3$-manifold we assign the following vector space:
\begin{equation}\nonumber W_{\tc{C}}(T, o, \Delta \tc{C}^{\sk}) := \bigoplus_{\Gamma: T_{(2)}^o \To \Delta \tc{C}^{\sk}} \bigotimes_{\sIII \in T_3} V^{\epsilon_o(\sIII)}(\Gamma, \sIII)
\end{equation}
\begin{remark}[The 3-manifold vector space is not an invariant] 
The vector space \linebreak $W_{\tc{C}}(T,o, \Delta\tc{C}^{\sk})$ depends on the combinatorial structure of $T$ and is not invariant under piecewise linear homeomorphisms. In particular, the hypothesized topological field theory extending the state sum $Z_{\tc{C}}$ will assign a certain subspace of $W_{\tc{C}}(T,o, \Delta\tc{C}^{\sk})$ to a 3-manifold with triangulation $T$.
\end{remark}

Denote by $\overline{T}$ the combinatorial manifold $T$ with the opposite orientation. We define a nondegenerate pairing $\langle \cdot, \cdot \rangle_T: W_{\tc{C}}\left(\overline{T}, o, \Delta \tc{C}^{\sk} \right) \otimes W_{\tc{C}}\left(T, o, \Delta \tc{C}^{\sk}\right) \to k$  as follows:
\ignore{ $\langle \cdot, \cdot \rangle_T\!: \!\bigoplus_{\Gamma:T_{(2)}^o \To \Delta \tc{C}^{\sk}} \bigotimes_{\sIII\in T_3} \left(  V^+(\Gamma, \sIII) \otimes V^-(\Gamma, \sIII)\right) \to k$:}%
\begin{equation} \nonumber
\langle \cdot, \cdot \rangle_{T} := \dim(\tc{C})^{-|T_0|} \bigoplus_{\Gamma:T_{(2)}^o \To \Delta \tc{C}^{\sk}}\left(\prod_{\sI \in T_1} \dm(\Gamma(\sI))  \right)^{-1}\left( \prod_{\sII\in T_2} \dm(\Gamma(\sII))\right) \bigotimes_{\sIII\in T_3} \langle \cdot, \cdot \rangle_{\Gamma, \sIII} 
\end{equation}
For an oriented combinatorial $4$-manifold $K$ with boundary, and with a total order $o$ on its vertices $o$, we define a linear map $Z_{\tc{C}}\left(\K, o, \Delta\tc{C}^{\sk}\right): k \to W_\tc{C}\left(\overline{\partial K}, o|_{\partial K}, \Delta \tc{C}^{\sk}\right)$ as follows:
\begin{equation}\nonumber \begin{split}
 \dim\left(\tc{C}\right)^{-|\interior{\K}_0|}\!\!\!\bigoplus_{\Sigma:\partial \K_{(2)}^o \To \Delta \tc{C}^{\sk}}& ~~\sum_{\substack{\Gamma: \K_{(2)}^o \To \Delta \tc{C}^{\sk}\\ \Gamma|_{\partial \K} = \Sigma}} \left(\prod_{\sI \in \interior{\K}_1} \dm(\Gamma(\sI)) \right)^{-1}
 \\
 &\left( \prod_{\sII \in \interior{\K}_2} \dm(\Gamma(\sII))\right)\left( \bigotimes_{\sIV \in K_4} z(\Gamma,\sIV) \right) \xo \left( \bigotimes_{\sIII \in K_3} \cup_{\Gamma, \sIII} \right)
 \end{split}
\end{equation}
Here, $z(\Gamma, \sIV)$ and $\cup_{\Gamma, \sIII}$ are defined as before, and $i(\K)_{r}$ denotes the finite set of $r$-simplices in the interior of $\K$, that is $\interior{\K}_r:=\{\s \in \K_r~|~\s \not \in \partial\K_r\}$. The composition $\xo$ is over all vector spaces appearing both in the domain of $\left( \bigotimes_{\sIV \in K_4} z(\Gamma,\sIV) \right)$ and in the codomain of  $\left( \bigotimes_{\sIII \in K_3} \cup_{\Gamma, \sIII} \right)$. 
Therefore, the codomain of $Z_{\tc{C}}(K, o, \Delta \tc{C}^{\sk})$ is the vector space 
\[\bigoplus_{\Sigma:\partial\K^o_{(2)} \To \Delta \tc{C}^{\sk}}~\bigotimes_{\sIII \in \K_3\text{ s.t.} \exists ! \sIV \in \K_4\text{ with }\sIII \subseteq \sIV} V^{-\epsilon_o^{\sIV}(\sIII)}(\Sigma, \sIII)\]
which indeed agrees with $W_{\tc{C}}(\overline{\partial \K}, o|_{\partial\K}, \Delta \tc{C}^{\sk})$.

\begin{lemma}[The state sum is the pairing of with-boundary state sums] \label{prop:gluinginvariance} 
Let $\K$ and $\K'$ be oriented combinatorial $4$-manifolds with boundary and let $f:\partial K \to \partial K'$ be an orientation reversing simplicial isomorphism. Let $o$ and $o'$ be total orders on $\K_0$ and $\K_0'$ such that the simplicial isomorphism $f:\partial \K\to \partial \K'$ preserves the induced orders. Then, 
\begin{equation}\nonumber Z_{\tc{C}}(\K\cup_{f} \K') =\left\langle Z_{\tc{C}}\left(\K, o, \Delta \tc{C}^{\sk}\right), Z_{\tc{C}}\left(\K', o', \Delta \tc{C}^{\sk}\right)\right\rangle_{\partial K},
\end{equation}
where we used $f$ to identify the vector spaces $W_{\tc{C}}(\overline{\partial \K'}, o'|_{\partial \K'}, \Delta \tc{C}^{\sk})$ and $W_{\tc{C}}(\partial \K, o|_{\partial \K}, \Delta \tc{C}^{\sk})$.
\end{lemma}
\begin{proof}
This is a direct consequence of the definition of $Z_{\tc{C}}\left(\K, o, \Delta \tc{C}^{\sk}\right)$ for combinatorial manifolds with boundary.
\ignore{The right hand side evaluates to 
\[D_{\tc{C}}^{-|K_0| - |K_0'| + |\partial K_0|} \sum_{\Sigma: \partial K_{(2)}^o \To \Delta \tc{C}^{\mathrm{sk}}} 
\]
The right thing evaluates to sum over $\Gamma$ on boundary, sum over $\Gamma$ restricted to the correct value at the boundary. Then, also product on the left and on the right. Remains prod over boundary which was left. 
}
\end{proof}

\subsubsection{The 4-dimensional bistellar moves} \label{sec:4dbistellar}

Recall that a bistellar move replaces a combinatorial $n$-manifold $K$ by a combinatorial manifold obtained from replacing a codimension zero submanifold of $K$ simplicially isomorphic to a subcomplex $I\subseteq \partial \Delta^{n+1}$ with the complementary subcomplex $J\subseteq \partial \Delta^{n+1}$. To show invariance of $Z_{\tc{C}}$ under such a move, it suffices by Lemma~\ref{prop:gluinginvariance} to prove that the following linear maps $k\to W_{\tc{C}}\left(\partial I, o|_{\partial I}, \Delta \tc{C}^{\sk}\right)$ are equal: 
\begin{equation}\nonumber
Z_{\tc{C}}\left(I,o|_I, \Delta \tc{C}^{\sk}\right)  = Z_{\tc{C}}\left(J, o|_J, \Delta \tc{C}^{\sk}\right)
\end{equation}
Here, $o$ is some fixed order on the vertices of $\Delta^{n+1}$, which induces an orientation on $\Delta^{n+1}$, and thus an orientation on the boundary $\partial \Delta^{n+1}$. The subcomplex $I\subseteq \partial \Delta^{n+1}$ carries the orientation induced by this orientation on $\partial \Delta^{n+1}$ and $J\subseteq \partial \Delta^{n+1}$ carries the opposite of the orientation induced by $\partial\Delta^{n+1}$.

For the three relevant $4$-dimensional bistellar moves, we pick an order $o$ on $\Delta^5$ such that $I_4$ and $J_4$ contain the following $4$-simplices, where we have labeled the vertices of $\Delta^5$ according to the order $o$ by $0\ldots 5$ and indicated the orientation of each $4$-simplex relative to the one induced from the order by a sign (these signs agree with the signs $\epsilon_{o|_I}(\sIV)$ or $\epsilon_{o|_J}(\sIV)$, respectively, introduced in Section~\ref{sec:10j}):
\[\begin{array}{c|c|c}
& I_4& J_4 \\ \hline (3,3)\text{-move}&
\lan 01235\ran~\lan 01345\ran ~ \lan 12345\ran
&
\lan 01234\ran ~ \lan 01245 \ran ~ \lan 02345 \ran
\\
\hline
(2,4)\text{-move}
&
\lan 01235 \ran ~ \lan 01345 \ran
&
\lan 02345 \ran ~ \lan 01245 \ran ~ \lan 02345 \ran ~-\!\! \lan 12345 \ran
\\
\hline
(1,5)\text{-move}&
\lan 01235 \ran &
 \lan 02345\ran ~ \lan 01245\ran~\lan 02345\ran~-\!\!\lan 12345\ran~-\!\!\lan 01345\ran
\end{array}
\]
In this notation, the relative sign $\epsilon^{\sIV}(\sIII)$ for a $4$-simplex $\sIV = \lambda \lan v_0\cdots v_4 \ran$ with $\lambda = \pm 1$, and a $3$-simplex $\sIII=\lan v_0\cdots \widehat{v}_i \cdots v_4 \ran$ is $\epsilon^{\sIV}(\sIII)=\lambda (-1)^i$, for both ordered oriented complexes $I$ and $J$.\looseness=-2

The simplices in the interior of $I$ and $J$ are listed in Figures~\ref{fig:interior33},~\ref{fig:interior24} and~\ref{fig:interior15}.

\begin{figure}[h] 
\[\begin{array}{c|c|c}
&\interior{I}_k & \interior{J}_k\\\hline
k=4
&
\lan 01235\ran~\lan 01345\ran~\lan 12345\ran
&
\lan 01234\ran ~\lan 01245\ran ~\lan 02345\ran
\\\hline
k=3
&
\lan 0135 \ran ~\lan 1235\ran ~\lan 1345\ran &
\lan 0124\ran ~\lan 0234\ran ~\lan 0245\ran \\\hline
k=2&\lan 135\ran &\lan 024\ran
\end{array}
\]
\caption{$k$-simplices in the interior of $I$ and $J$ for the $(3,3)$-bistellar move.}
\label{fig:interior33}
\end{figure}
\begin{figure}[h] 
\[\begin{array}{c|c|c}
&\interior{I}_k & \interior{J}_k\\\hline
k=4
&
\lan 01235\ran ~\lan 01345\ran
&
\lan 01234\ran ~\lan 01245\ran ~ \lan 02345\ran ~\lan 12345\ran
\\\hline
k=3
&
\lan 0135\ran&
\lan 0124\ran ~\lan 0234\ran ~\lan 0245\ran ~\lan 2345\ran ~\lan 1245\ran ~\lan 1234\ran \\\hline
k=2&&\lan 024\ran ~\lan 245\ran ~\lan 234\ran ~\lan 124\ran\\ \hline
k=1 & & \lan 24\ran
\end{array}
\]
\caption{$k$-simplices in the interior of $I$ and $J$ for the $(2,4)$-bistellar move.}
\label{fig:interior24}
\end{figure}
\begin{figure}[h] 
\[\begin{array}{c|c|c}
&\interior{I}_k & \interior{J}_k\\\hline
k=4
&
\lan 01235\ran
&
\lan 01234\ran ~\lan 01245\ran ~\lan 02345\ran ~\lan 01345\ran ~\lan 12345\ran
\\\hline
k=3
&&
\lan 0124 \ran ~\lan 0234\ran ~\lan 0245\ran~\lan 2345\ran~\lan 1245\ran~\lan 1234\ran~\lan 0134\ran~\lan 0145\ran~\lan 1345\ran~\lan 0345\ran\\\hline
k=2&&\lan 024\ran~\lan 245\ran~\lan 234\ran~\lan 124\ran~\lan 034\ran~\lan 014\ran~\lan 134\ran~\lan 045\ran~\lan 145\ran~\lan 345\ran\\ \hline
k=1 & & \lan 24\ran~\lan 04\ran~\lan 14\ran~\lan 34\ran~\lan 45\ran\\ \hline
k=0&& \lan 4\ran
\end{array}
\]
\caption{$k$-simplices in the interior of $I$ and $J$ for the $(1,5)$-bistellar move.}
\label{fig:interior15}
\end{figure}

From now on, we omit the choice of order from the notation and explicitly work with the simplices in Figures~~\ref{fig:interior33},~\ref{fig:interior24} and~\ref{fig:interior15} and the sign conventions outlined above. Given a $\tc{C}$-state $\Delta^5_{(2)} \To \Delta\tc{C}$, we denote the simple object assigned to a $1$-simplex $\lan i j \ran $ by $[ij]$ and the simple $1$-morphism assigned to a $2$-simplex $\lan ijk \ran$ by $[ijk]$. Contrary to our previous notation, we henceforth almost always let the state be implicit and omit it from the notation.  For a $3$-simplex $\lan ijkl\ran \in \Delta^5_3$, we recall Notation~\ref{not:parenthesisnotation} and reintroduce the vector spaces from Definition~\ref{def:defvectorspaceV}:
\begin{calign} \nonumber 
V^+(ijkl) := \Hom_{\tc{C}}\left( \fs{(ijk)l}, \fs{i(jkl)}\right) 
&
V^-(ijkl) := \Hom_{\tc{C}}\left( \fs{i(jkl)}, \fs{(ijk)l}\right)
\end{calign}
For a $4$-simplex $\lan ijklm\ran \in \Delta^5_4$, we recall the linear maps defined in Figures~\ref{fig:definitionz}a and~\ref{fig:definitionz}b (with again the state left implicit):
\begin{align}\nonumber z_+(ijklm):&~ V^+(ijkl) \otimes V^+(ijlm) \otimes V^+(jklm) \otimes V^-(ijkm) \otimes V^-(iklm) \to k\\\nonumber
z_-(ijklm):& ~V^+(iklm) \otimes V^+(ijkm) \otimes V^-(jklm) \otimes V^-(ijlm) \otimes V^-(ijkl) \to k
\end{align}
Precomposing $z_{\pm}(ijklm)$ with the appropriate maps $\cup_{abcd}: k \to V^+(abcd) \otimes V^-(abcd)$ (determined by the nondegenerate pairings $\langle \cdot, \cdot \rangle_{abcd}: V^-(abcd) \otimes V^+(abcd)\to k$), leads to linear maps of the following type for every $4$-simplex $\lan ijklm\ran \in \Delta^5_4$:
\begin{align}\nonumber Z_+(ijklm):&~ V^+(ijkl) \otimes V^+(ijlm) \otimes V^+(jklm) \to  V^+(iklm) \otimes V^+(ijkm)\\\nonumber
Z_-(ijklm):&~V^+(iklm) \otimes V^+(ijkm) \to V^+(ijkl) \otimes V^+(ijlm)\otimes V^+(jklm) 
\end{align}
Using these maps, invariance under the various bistellar moves can then explicitly be reexpressed as the following lemmas. 

In the following, all expressions are already postcomposed with appropriate pairings $\langle \cdot , \cdot \rangle$ to undo superfluous $\cup$-maps appearing on either side of the equations. To unclutter notation, the concatenation $A_1\cdots A_n$ of linear maps $A_1, \ldots, A_n$ denotes composition over all vector spaces appearing both in the domain of some $A_i$ and the codomain of some $A_j$ with $j>i$. \emph{Here, and in the following proofs, we will use this concatenation notation, and we will therefore omit all swap maps, all tensor product symbols, and all tensor products with identities.}

\begin{lemma}[Invariance under the $(3,3)$-bistellar move] \label{lem:Pachner33}Let $I$ and $J$ be as in Figure~\ref{fig:interior33}. Then, the following holds for every $\tc{C}$-state $\partial I_{(2)} \To \Delta \tc{C}^{\sk}$ of $\partial I$:%
\begin{equation}\nonumber
\sum_{\fs{135}} \tdim{135} Z_+(01235)Z_+(01345) Z_+(12345) = \sum_{\fs{024}} \tdim{024} Z_+(02345)Z_+(01245)Z_+(01234)
\end{equation}
The sums are over simple $1$-morphisms $\Delta\tc{C}_2^{\sk}\ni \fs{ijk}:\Xs{ij} \xz \Xs{jk} \to \Xs{ik}$.
\end{lemma}

\begin{lemma}[Invariance under the $(2,4)$-bistellar move]\label{lem:Pachner24}Let $I$ and $J$ be as in Figure~\ref{fig:interior24}. Then, the following holds for every $\tc{C}$-state $\partial I_{(2)} \To \Delta \tc{C}^{\sk}$ of $\partial I$:%
\begin{equation}\nonumber
\begin{split}
Z_+&(01235)Z_+(01345)\\
&\hspace*{-4pt}=
\!\!\sum_{\substack{\Xs{24}, \fs{024}, \fs{245},\\ \fs{234}, \fs{124}}}\!\! \frac{\tdim{024}\tdim{245} \tdim{234} \tdim{124}}{\odim{24}} Z_+(02345) Z_+(01245) Z_+(01234) Z_-(12345) 
 \end{split}
\end{equation}
The sums are over simple objects $\Xs{24} \in \Delta \tc{C}_1^{\sk}$, and simple $1$-morphisms $\Delta\tc{C}_2^{\sk}\ni \fs{ijk}: \Xs{ij}\xz \Xs{jk} \to \Xs{ik}$.
\end{lemma}

\begin{lemma}[Invariance under the $(1,5)$-bistellar move] \label{lem:Pachner15}Let $I$ and $J$ be as in Figure~\ref{fig:interior15}. Then, the following holds for every $\tc{C}$-state $\partial I_{(2)} \To \Delta \tc{C}^{\sk}$ of $\partial I$:%
\begin{equation}\nonumber 
\begin{split}\hspace*{-1pt}
Z_+(01235)=\dim\left(\tc{C}\right)^{-1}\!\!&
\!\!\sum_{\substack{\Xs{ij}, 0\leq i<j\leq 5\\i=4\text{ or }j=4}} \,\,\sum_{\substack{\fs{ijk}, 0\leq i<j<k\leq 5\\j=4\text{ or }k=4}} \left(\prod_{\substack{\Xs{ij}, i<j\\i=4\text{ or }j=4}} \odim{ij}\right)^{-1}\!\left(\prod_{\substack{\fs{ijk}, i<j<k\\j=4\text{ or }k=4}} \tdim{ijk} \right)\\[5pt]
&\Tr_{V^+(0345)}\left(\vphantom{\frac{a}{b}}Z_+(02345) Z_+(01245) Z_+(01234) Z_-(12345)Z_-(01345) \right)
 \end{split}
\end{equation}
The sum is over simple objects $\Xs{ij} \in \Delta\tc{C}_1^{\sk}$ and simple $1$-morphisms $\Delta\tc{C}_2^{\sk} \ni \fs{ijk}:\Xs{ij}\xz \Xs{jk} \to \Xs{ik}$, and the (partial) trace is over the vector space $V^+(0345)$.
\end{lemma}
\nid
The range of the sums and the appearance of the trace in the preceeding lemmas follow directly from comparing the explicit expressions for $Z_{\tc{C}}(I, o|_I, \Delta\tc{C}^{\sk})$ and $Z_{\tc{C}}(J, o|_J, \Delta \tc{C}^{\sk})$, where we consider $I$ and $J$ as oriented combinatorial manifolds with common boundary $\partial I$. Note that the expressions for $Z_{\tc{C}}(J, o|_J, \Delta \tc{C}^{\sk})$ in Section~\ref{sec:boundary} are written in terms of the maps $z_{\pm}$ instead of the maps $Z_{\pm}$, which are precomposed with $\cup$-maps. For example, the trace in Lemma~\ref{lem:Pachner15} arises explicitly from the fact that both vector spaces $V^+(0345)$ and $V^-(0345)$ appear in the codomain of $\bigotimes_{\sIII \in J_3} \cup_{\Gamma, \sIII}$ and the domain of $\bigotimes_{\sIV \in J_4} z(\Gamma, \sIV)$ (explicitly, in the domain of $z_-(01345)$ and $z_+(02345)$, respectively) and are hence composed over, but the vector space $V^-(0345)$ in the domain of $z_+(02345)$ is transformed into the vector space $V^+(0345)$ in the codomain of $Z_+(02345)$. The formula in Lemma~\ref{lem:Pachner15} therefore involves a trace over the vector space $V^+(0345)$ in the domain of $Z_-(01345)$ and the codomain of $Z_+(02345)$. 

\subsubsection{Pentagonator compositions in prefusion 2-categories}

The `10j symbol' linear map $z(\Gamma,\sIV)$ defined in Figure~\ref{fig:definitionz} extracts the matrix coefficients of the pentagonator of the monoidal $2$-category $\tc{C}$. The corresponding `partially dualized 10j symbol' maps $Z_{\pm}(ijklm)$ (defined in Section~\ref{sec:4dbistellar}) can be thought of as first `composing associators' along one side of the pentagon, and then `uncomposing' along the other side of the pentagon. More precisely, we will express $Z_+(ijklm)$ as the composite of a linear map $\Lcomp_+(ijklm)$ (that composes associators along the long side of the pentagon), followed by the dual of a linear map $\Scomp_-(ijklm)$ (that composes associators along the short side of the pentagon); similarly $Z_-(ijklm)$ will be a composite of a linear map $\Scomp_+(ijklm)$ and the dual of a linear map $\Lcomp_-(ijklm)$.  We will show that the dual of $\Scomp_-(ijklm)$ is proportional to the inverse of $\Scomp_+(ijklm)$ (this inverse may be though of as `uncomposing along the short side of the pentagon') and that the dual of $\Lcomp_-(ijklm)$ is proportional to a section of $\Lcomp_+(ijklm)$ (this section is thus `uncomposing along the long side of the pentagon').  This detailed decomposition of the 10j symbols will be crucial in proving invariance under the bistellar moves in Section~\ref{sec:Pachnerdetail}.

\skiptocparagraph{Composing along the pentagon}
For a $4$-simplex $\lan ijklm\ran \in \Delta^5_4$, we extend Notation~\ref{not:parenthesisnotation} as follows:
\begin{equation}\nonumber \fs{((ijk)l)m}:= \fs{ilm}\xo\left(\fs{ikl}\xz \Io_{\Xs{lm}} \right)\xo \left(\fs{ijk}\xz \Io_{\Xs{kl}}\xz \Io_{\Xs{lm}}\right)  
\end{equation}
The alternative parenthesizations $\fs{(i(jkl))m},~\fs{i((jkl)m)}, \fs{i(j(klm))}$ are defined analogously. We define the following vector spaces:%
\begin{align*}
V^+(ijklm)&:=\Hom_{\tc{C}}\left(\fs{((ijk)l)m}, \fs{i(j(klm))}\right)\\ 
V^-(ijklm)&:=\Hom_{\tc{C}}\left(\fs{i(j(klm))},\fs{((ijk)l)m}\right)
\end{align*}

As mentioned, the linear maps $Z_{\pm}(ijklm)$ can be considered as first `composing associators' along one side of the pentagon, factoring through the vector spaces $V^+(ijklm)$ or $V^-(ijklm)$ respectively, and then `uncomposing' along the other side of the pentagon. In Figures~\ref{fig:SComp} and~\ref{fig:LComp} we therefore define linear maps 
\begin{align}\nonumber
\Scomp_\pm(ijklm) &: \bigoplus_{\fs{ikm}} V^\pm(iklm) \otimes V^\pm(ijkm)  \to V^\pm(ijklm)\\\nonumber
\Lcomp_\pm(ijklm) &: \bigoplus_{\substack{\Xs{jl}, \fs{ijl}\\\fs{jkl}, \fs{jlm}}} V^\pm(ijkl) \otimes V^\pm(ijlm) \otimes V^\pm(jklm) \to V^\pm(ijklm)
\end{align}
for $4$-simplices $\lan ijklm\ran\in \Delta^5_4$, expressing the composition of associators along the short ($\Scomp$) or long side ($\Lcomp$) of the pentagon.
\begin{figure}[h]
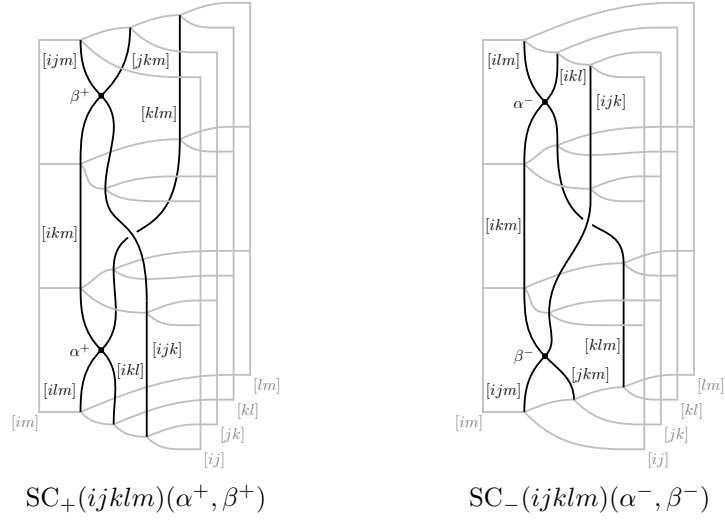

\begin{calign}\nonumber
\def\d{-0.5}
 \begin{tz}[td,scale=1.1]
     \begin{scope}[on layer=front,yzplane=4]
 \node[omor] at (4.5, -1) {$[ilm]$};
 \node[omor] at (2.79, -0.5) {$[ikl]$};
 \node[omor] at (1.9, 0.) {$[ijk]$};
 \node[omor] at (2.1,5.7) {$[klm]$};
 \node[omor] at (4.5, 2.8) {$[ikm]$};
 \node[omor] at (4.5, 7) {$[ijm]$};
 \node[omor] at (2.3,7.0) {$[jkm]$};
 \end{scope}
 \begin{scope}[xyplane=3*\h]
\draw[slice, on layer=front] (0,\d) to(0,0) to [out=up, in=\dl]  (1.5,3) to (1.5,4);
 	\draw[slice, on layer =frontb] (1,\d) to (1,0) to [out=up, in=\dl] (2,2) to [out=up, in=\dr] (1.5,3);
 	\draw[slice] (2,\d) to (2,0) to [out=up, in=\dl] (2.5,1) to [out=up, in=\dr] (2,2);
 	\draw[slice] (3,\d) to (3,0) to [out=up, in=\dr] (2.5,1) ;
 \end{scope}
 \begin{scope}[xyplane=2*\h]
 	\draw[slice, on layer=fronta] (0,\d) to (0,0) to [out=up, in=\dl]  (0.5,2) to [out=\dr, in=up] (1,0) to (1,\d);
 	\draw[slice]  (0.5,2) to [out=up, in=\dl] (1.5,3) to (1.5,4);
 	\draw[slice] (2,\d) to (2,0) to [out=up, in=\dl] (2.5,1) to [out=up, in=\dr] (1.5,3);
 	\draw[slice] (3,\d) to (3,0) to [out=up, in=\dr] (2.5,1) ;
 \end{scope}
 \begin{scope}[xyplane=\h]
 	\draw[slice, on layer=front] (0,\d) to (0,0) to [out=up, in=\dl]  (0.5,1);
 	\draw[slice, on layer=front] (0.5,1) to [out=up, in=\dl] (1.5,3);
 	\draw[slice] (1.5,3) to (1.5,4);
 	\draw[slice] (1,\d) to (1,0) to [out=up, in=\dr] (0.5,1) ;
 	\draw[slice] (2,\d) to (2,0) to [out=up, in=\dl] (2.25,2.5) to [out=up, in=\dr] (1.5,3);
 	\draw[slice] (3,\d) to (3,0) to [out=up, in=\dr] (2.25,2.5) ;
 \end{scope}
 \begin{scope}[xyplane = 0]
  	\draw[slice](0,\d) to (0,0) to [out=up, in=\dl] (0.5,1) to [out=up, in=\dl] (1,2) to [out=up, in=\dl] (1.5,3) to (1.5,4);
 	\draw[slice] (1,\d) to (1,0) to [out=up, in=\dr] (0.5,1);
 	\draw[slice]  (2,\d) to (2,0) to [out=up, in=\dr] (1,2);
 	\draw[slice] (3,\d) to (3,0) to [out=up, in=\dr] (1.5,3);
 \end{scope}
 \begin{scope}[xzplane=\d]
 	\draw[slice,short] (0,0) to (0,3*\h);
 	\draw[slice,short] (1,0) to (1,3*\h);
 	\draw[slice,short] (2,0) to (2,3*\h);
 	\draw[slice,short] (3,0) to (3,3*\h);
 \end{scope}
 \begin{scope}[xzplane=4]
 	\draw[slice,short] (1.5,0) to (1.5, 3*\h);
 \end{scope}
 \coordinate (L2) at (2.5,1.5, 2.55*\h);
 \coordinate (L3) at (2.5, 1.5, 0.5*\h);
 \draw[wire, on layer=front] (3,1.5, 3*\h) to  [out=down, in=\ul] (L2) to [out=\dl, in=up] (3, 1.5, 2*\h);
 \draw[wire,on layer=front] (3,1.5,2*\h) to (3,1.5, \h);
 \draw[wire, on layer=front] (3, 1.5, \h) to [out=down, in=\ul] (L3);
 \draw[wire] (L3) to [out=\dl, in=up] (3, 1.5, 0) ;
 \draw[wire, on layer=frontb] (2,2,3*\h) to [out=down, in=\ur] (L2) to [out=\dr, in=up]  (2, 0.5, 2*\h) to [out=down, in=up, in looseness=2] node[mask point, pos=0.6](MP){} (1,0.5, 1.1*\h); 
 \draw[wire, on layer=front] (1,0.5,1.1*\h) to (1, 0.5, 0);
\cliparoundone{MP}{ \draw[wire](1, 2.5, 3*\h) to (1,2.5, 2*\h) to [out=down, in=up, out looseness=2] (2.5, 2.25, \h) to [out=down, in=\ur] (L3) to [out=\dr, in=up] (2, 1, 0);}
\node[dot] at (L2){};
\node[dot] at (L3){};
\node[tmor, left] at ([xshift=-0.12cm]L2) {~$\beta^+$};
\node[tmor, left] at ([xshift=-0.12cm]L3) {~$\alpha^+$};
\node[obj, below right] at (\d+0.1,0.1,0) {$[ij]$};
\node[obj, below right] at (\d+0.1,1.1,0) {$[jk]$};
\node[obj, below right] at (\d+0.1,2.1,0) {$[kl]$};
\node[obj, below right] at (\d+0.1,3.1,0) {$[lm]$};
\node[obj, below left] at (3.9,1.6,0) {$[im]$};
 \end{tz}
&
\def\d{-0.5}
 \begin{tz}[td,scale=1.1]
   \begin{scope}[on layer=front,yzplane=4]
 \node[omor] at (4.5, 8) {$[ijm]$};
 \node[omor] at (2.4, 8.4) {$[jkm]$};
 \node[omor] at (2.1, 9.1) {$[klm]$};
 \node[omor] at (4.5,12) {$[ikm]$};
 \node[omor] at (4.5, 16) {$[ilm]$};
 \node[omor] at (2.82,15.6) {$[ikl]$};
 \node[omor] at (1.9,15.){$[ijk]$};
 \end{scope}
  \begin{scope}[xyplane=3*\h]
 	\draw[slice, on layer=front] (0,\d) to(0,0) to [out=up, in=\dl]  (1.5,3) to (1.5,4);
 	\draw[slice, on layer =frontb] (1,\d) to (1,0) to [out=up, in=\dl] (2,2) to [out=up, in=\dr] (1.5,3);
 	\draw[slice] (2,\d) to (2,0) to [out=up, in=\dl] (2.5,1) to [out=up, in=\dr] (2,2);
 	\draw[slice] (3,\d) to (3,0) to [out=up, in=\dr] (2.5,1) ;
 \end{scope}
 \begin{scope}[xyplane=4*\h]
 	\draw[slice, on layer=fronta] (0,\d) to (0,0) to [out=up, in=\dl]  (0.5,2) to [out=\dr, in=up] (1,0) to (1,\d);
 	\draw[slice]  (0.5,2) to [out=up, in=\dl] (1.5,3) to (1.5,4);
 	\draw[slice] (2,\d) to (2,0) to [out=up, in=\dl] (2.5,1) to [out=up, in=\dr] (1.5,3);
 	\draw[slice] (3,\d) to (3,0) to [out=up, in=\dr] (2.5,1) ;
 \end{scope}
 \begin{scope}[xyplane=5*\h]
 	\draw[slice, on layer=front] (0,\d) to (0,0) to [out=up, in=\dl]  (0.5,1);
 	\draw[slice, on layer=front] (0.5,1) to [out=up, in=\dl] (1.5,3);
 	\draw[slice] (1.5,3) to (1.5,4);
 	\draw[slice] (1,\d) to (1,0) to [out=up, in=\dr] (0.5,1) ;
 	\draw[slice] (2,\d) to (2,0) to [out=up, in=\dl] (2.25,2.5) to [out=up, in=\dr] (1.5,3);
 	\draw[slice] (3,\d) to (3,0) to [out=up, in=\dr] (2.25,2.5) ;
 \end{scope}
 \begin{scope}[xyplane = 6*\h, on layer=superfront]
  	\draw[slice](0,\d) to (0,0) to [out=up, in=\dl] (0.5,1) to [out=up, in=\dl] (1,2) to [out=up, in=\dl] (1.5,3) to (1.5,4);
 	\draw[slice] (1,\d) to (1,0) to [out=up, in=\dr] (0.5,1);
 	\draw[slice]  (2,\d) to (2,0) to [out=up, in=\dr] (1,2);
 	\draw[slice] (3,\d) to (3,0) to [out=up, in=\dr] (1.5,3);
 \end{scope}
 \begin{scope}[xzplane=\d]
 	\draw[slice,short] (0,3*\h) to (0,6*\h);
 	\draw[slice,short] (1,3*\h) to (1,6*\h);
 	\draw[slice,short] (2,3*\h) to (2,6*\h);
 	\draw[slice,short] (3,3*\h) to (3,6*\h);
 \end{scope}
 \begin{scope}[xzplane=4]
 	\draw[slice,short] (1.5,3*\h) to (1.5, 6*\h);
 \end{scope}
 \coordinate (L2) at (2.5,1.5, 3.45*\h);
 \coordinate (L3) at (2.5, 1.5, 5.5*\h);
 \draw[wire, on layer=front](3,1.5, 3*\h) to  [out=up, in=\dl] (L2) to [out=\ul, in=down] (3, 1.5, 4*\h);
 \draw[wire,on layer=front] (3,1.5,4*\h) to (3,1.5, 5*\h);
 \draw[wire, on layer=front] (3, 1.5, 5*\h) to [out=up, in=\dl] (L3);
 \draw[wire] (L3) to [out=\ul, in=down] (3, 1.5, 6*\h) ;
 \draw[wire, on layer=frontb] (2,2,3*\h) to [out=up, in=\dr] (L2) to [out=\ur, in=down]  (2, 0.5, 4*\h) to [out=up, in=down] node[mask point, pos=0.8](MP){} (1,0.5, 4.9*\h); 
 \draw[wire, on layer=front] (1,0.5,4.9*\h) to (1, 0.5, 6*\h);
\cliparoundone{MP}{ \draw[wire]  (1, 2.5, 3*\h) to (1,2.5, 4*\h) to [out=up, in=down, in looseness=2] (2.5, 2.25, 5*\h) to [out=up, in=\dr] (L3) to [out=\ur, in=down] (2, 1, 6*\h);}
\node[dot] at (L2){};
\node[dot] at (L3){};
\node[tmor, left] at ([xshift=-0.12cm]L2) {~$\beta^-$};
\node[tmor, left] at ([xshift=-0.12cm]L3) {~$\alpha^-$};
\node[obj, below right] at (\d+0.1,0.1,3*\h) {$[ij]$};
\node[obj, below right] at (\d+0.1,1.1,3*\h) {$[jk]$};
\node[obj, below right] at (\d+0.1,2.1,3*\h) {$[kl]$};
\node[obj, below right] at (\d+0.1,3.1,3*\h) {$[lm]$};
\node[obj, below left] at (3.9,1.6,3*\h) {$[im]$};
 \end{tz}
 \\\nonumber
 \Scomp_+(ijklm)(\alpha^+, \beta^+)
 &
 \Scomp_-(ijklm)(\alpha^-, \beta^-)
\end{calign}
\caption{$\Scomp_\pm(ijklm)$ for $\alpha^\pm \in V^\pm(iklm)$ and $\beta^\pm \in V^\pm(ijkm)$.}
\label{fig:SComp}
\end{figure}
\begin{figure}[h]
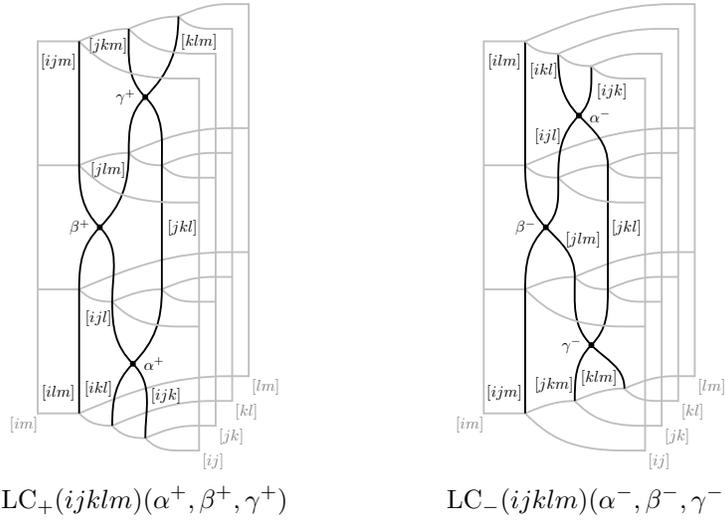

\begin{calign}\nonumber
\def\d{-0.5}
 \begin{tz}[td,scale=1.1]
 \begin{scope}[on layer=front,yzplane=4]
 \node[omor] at (4.5, -1) {$[ilm]$};
 \node[omor] at (3.53, -0.9) {$[ikl]$};
 \node[omor] at (1.9, -1.07) {$[ijk]$};
 \node[omor] at (1.5,3) {$[jkl]$};
 \node[omor] at (3.53, 0.8) {$[ijl]$};
 \node[omor] at (4.5, 7) {$[ijm]$};
 \node[omor] at (3.3,7.4) {$[jkm]$};
 \node[omor] at (1.1,7.5){$[klm]$};
 \node[omor] at (3.3,4.38){$[jlm]$};
 \end{scope}
 \begin{scope}[xyplane=0]
 		\draw[slice](0,\d) to (0,0) to [out=up, in=\dl] (0.5,1) to [out=up, in=\dl] (1,2) to 					[out=up, in=\dl] (1.5,3) to (1.5,4);
		\draw[slice] (1,\d) to (1,0) to [out=up, in=\dr] (0.5,1);
 		\draw[slice] (2,\d) to (2,0) to [out=up, in=\dr] (1,2);
		\draw[slice] (3,\d) to (3,0) to [out=up, in=\dr] (1.5,3);
 \end{scope}
 \begin{scope}[xyplane=\h]
 	\draw[slice, on layer=front](0,\d) to (0,0) to [out=up, in=\dl]  (1,2);
 	\draw[slice] (1,2) to [out=up, in=\dl] (1.5,3) to (1.5,4);
 	\draw[slice] (1,\d) to (1,0) to [out=up, in=\dl] (1.5,1);
 	\draw[slice] (2,\d) to (2,0) to [out=up, in=\dr] (1.5,1) to [out=up, in=\dr] (1,2);
 	\draw[slice] (3,\d) to (3,0) to [out=up, in=\dr] (1.5,3);
 \end{scope}
 \begin{scope}[xyplane=2*\h]
 	\draw[slice, on layer=front] (0,\d) to (0,0) to [out=up, in=\dl]  (1.5,3) to (1.5,4);
 	\draw[slice] (1,\d) to (1,0) to [out=up, in=\dl] (1.5,1);
 	\draw[slice] (2,\d) to (2,0) to [out=up, in=\dr] (1.5,1) to [out=up, in=\dl] (2,2);
 	\draw[slice] (3,\d) to (3,0) to [out=up, in=\dr] (2,2) to [out=up, in=\dr] (1.5,3);
 \end{scope}
 \begin{scope}[xyplane=3*\h]
 	\draw[slice, on layer=front] (0,\d) to(0,0) to [out=up, in=\dl]  (1.5,3) to (1.5,4);
 	\draw[slice, on layer =frontb] (1,\d) to (1,0) to [out=up, in=\dl] (2,2) to [out=up, in=\dr] (1.5,3);
 	\draw[slice] (2,\d) to (2,0) to [out=up, in=\dl] (2.5,1) to [out=up, in=\dr] (2,2);
 	\draw[slice] (3,\d) to (3,0) to [out=up, in=\dr] (2.5,1) ;
 \end{scope}
  \begin{scope}[xzplane=\d]
 	\draw[slice,short] (0,0) to (0,3*\h);
 	\draw[slice,short] (1,0) to (1,3*\h);
 	\draw[slice,short] (2,0) to (2,3*\h);
 	\draw[slice,short] (3,0) to (3,3*\h);
 \end{scope}
 \begin{scope}[xzplane=4]
 	\draw[slice,short] (1.5,0) to (1.5, 3*\h);
 \end{scope}
  \coordinate (L1) at (2.5, 1.5, 1.5*\h);
 \coordinate (R1) at (1.5, 1, 0.5*\h);
 \coordinate (R2) at (1.5, 1.75, 2.5*\h);
 \draw[wire, on layer=front] (3,1.5,0) to (3,1.5, \h) to [out=up, in=\dl] (L1) to [out=\ul, in=down] (3, 1.5, 2*\h) to (3,1.5, 3*\h);
 \draw[wire,on layer=front] (2, 1,0) to [out=up, in=\dl] (R1)  to [out=\ul, in=down] (2,1,\h) to [out=up, in=\dr] (L1);
 \draw[wire, on layer=frontb](L1) to [out=\ur, in=down] (2, 2, 2*\h) to [out=up, in=\dl] (R2) to [out=\ul, in=down] (2,2,3*\h);
 \draw[wire] (1, 0.5,0) to [out=up, in=\dr] (R1) to [out=\ur, in=down] (1, 1.5, \h) to (1,1.5,2*\h) to [out=up, in=\dr] (R2) to [out=\ur, in=down] (1, 2.5, 3*\h);
 \node[dot] at (L1){};
 \node[dot] at (R1){};
\node[dot] at (R2){};
\node[obj, below right] at (\d+0.1,0.1,0) {$[ij]$};
\node[obj, below right] at (\d+0.1,1.1,0) {$[jk]$};
\node[obj, below right] at (\d+0.1,2.1,0) {$[kl]$};
\node[obj, below right] at (\d+0.1,3.1,0) {$[lm]$};
\node[obj, below left] at (3.9,1.6,0) {$[im]$};
 \node[tmor, right] at (R1) {~$\alpha^+$};
\node[tmor, left] at ([xshift=-0.1cm]L1) {$\beta^+$};
\node[tmor, left] at ([xshift=-0.1cm]R2) {~$\gamma^+$};
 \end{tz}
 &
 \def\d{-0.5}%
 \begin{tz}[td,scale=1.1]
  \begin{scope}[on layer=front,yzplane=4]
 \node[omor] at (4.5, 8) {$[ijm]$};
 \node[omor] at (3.3, 8.2) {$[jkm]$};
  \node[omor] at (2.21, 8.35) {$[klm]$};
 \node[omor] at (1.55,12) {$[jkl]$};
  \node[omor] at (2.6, 11.7) {$[jlm]$};
 \node[omor] at (4.5, 16) {$[ilm]$};
 \node[omor] at (3.55,15.8) {$[ikl]$};
 \node[omor] at (1.9,15.3){$[ijk]$};
 \node[omor] at (3.45,14.2){$[ijl]$};
 \end{scope}
 \begin{scope}[xyplane=6*\h, on layer=superfront]
 		\draw[slice](0,\d) to (0,0) to [out=up, in=\dl] (0.5,1) to [out=up, in=\dl] (1,2) to 					[out=up, in=\dl] (1.5,3) to (1.5,4);
		\draw[slice] (1,\d) to (1,0) to [out=up, in=\dr] (0.5,1);
 		\draw[slice] (2,\d) to (2,0) to [out=up, in=\dr] (1,2);
		\draw[slice] (3,\d) to (3,0) to [out=up, in=\dr] (1.5,3);
 \end{scope}
 \begin{scope}[xyplane=5*\h]
 	\draw[slice, on layer=front](0,\d) to (0,0) to [out=up, in=\dl]  (1,2);
 	\draw[slice] (1,2) to [out=up, in=\dl] (1.5,3) to (1.5,4);
 	\draw[slice] (1,\d) to (1,0) to [out=up, in=\dl] (1.5,1);
 	\draw[slice] (2,\d) to (2,0) to [out=up, in=\dr] (1.5,1) to [out=up, in=\dr] (1,2);
 	\draw[slice] (3,\d) to (3,0) to [out=up, in=\dr] (1.5,3);
 \end{scope}
 \begin{scope}[xyplane=4*\h]
 	\draw[slice, on layer=front] (0,\d) to (0,0) to [out=up, in=\dl]  (1.5,3) to (1.5,4);
 	\draw[slice] (1,\d) to (1,0) to [out=up, in=\dl] (1.5,1);
 	\draw[slice] (2,\d) to (2,0) to [out=up, in=\dr] (1.5,1) to [out=up, in=\dl] (2,2);
 	\draw[slice] (3,\d) to (3,0) to [out=up, in=\dr] (2,2) to [out=up, in=\dr] (1.5,3);
 \end{scope}
  \begin{scope}[xyplane=3*\h]
 	\draw[slice, on layer=front] (0,\d) to(0,0) to [out=up, in=\dl]  (1.5,3) to (1.5,4);
 	\draw[slice, on layer =frontb] (1,\d) to (1,0) to [out=up, in=\dl] (2,2) to [out=up, in=\dr] (1.5,3);
 	\draw[slice] (2,\d) to (2,0) to [out=up, in=\dl] (2.5,1) to [out=up, in=\dr] (2,2);
 	\draw[slice] (3,\d) to (3,0) to [out=up, in=\dr] (2.5,1) ;
 \end{scope}
  \begin{scope}[xzplane=\d]
 	\draw[slice,short] (0,3*\h) to (0,6*\h);
 	\draw[slice,short] (1,3*\h) to (1,6*\h);
 	\draw[slice,short] (2,3*\h) to (2,6*\h);
 	\draw[slice,short] (3,3*\h) to (3,6*\h);
 \end{scope}
 \begin{scope}[xzplane=4]
 	\draw[slice,short] (1.5,3*\h) to (1.5, 6*\h);
 \end{scope}
  \coordinate (L1) at (2.5, 1.5, 4.5*\h);
 \coordinate (R1) at (1.5, 1, 5.5*\h);
 \coordinate (R2) at (1.5, 1.75, 3.5*\h);
 \draw[wire, on layer=front] (3,1.5,6*\h) to (3,1.5, 5*\h) to [out=down, in=\ul] (L1) to [out=\dl, in=up] (3, 1.5, 4*\h) to (3,1.5, 3*\h); 
 \draw[wire,on layer=front] (2, 1,6*\h) to [out=down, in=\ul] (R1)  to [out=\dl, in=up] (2,1,5*\h) to [out=down, in=\ur] (L1);
 \draw[wire, on layer=frontb](L1) to [out=\dr, in=up] (2, 2, 4*\h) to [out=down, in=\ul] (R2) to [out=\dl, in=up] (2,2,3*\h); 
 \draw[wire] (1, 0.5,6*\h) to [out=down, in=\ur] (R1) to [out=\dr, in=up] (1, 1.5, 5*\h) to (1,1.5,4*\h) to [out=down, in=\ur] (R2) to [out=\dr, in=up] (1, 2.5, 3*\h); 
\node[dot] at (L1){};
\node[dot] at (R1){};
\node[dot] at (R2){};
\node[tmor, right] at (R1) {~$\alpha^-$};
\node[tmor, left] at ([xshift=-0.1cm]L1) {$\beta^-$};
\node[tmor, left] at ([xshift=-0.1cm]R2) {~$\gamma^-$};
\node[obj, below right] at (\d+0.1,0.1,3*\h) {$[ij]$};
\node[obj, below right] at (\d+0.1,1.1,3*\h) {$[jk]$};
\node[obj, below right] at (\d+0.1,2.1,3*\h) {$[kl]$};
\node[obj, below right] at (\d+0.1,3.1,3*\h) {$[lm]$};
\node[obj, below left] at (3.9,1.6,3*\h) {$[im]$};
 \end{tz}\\\nonumber
 \Lcomp_+(ijklm)(\alpha^+, \beta^+, \gamma^+)
 &
 \Lcomp_-(ijklm)(\alpha^-, \beta^-, \gamma^-)
 \end{calign}
 \caption{$\Lcomp_\pm(ijklm)$ for $\alpha^\pm \in V^\pm(ijkl)$, $\beta^\pm \in V^\pm(ijlm)$, and $\gamma^\pm \in V^\pm(jklm)$.}
\label{fig:LComp}
\end{figure}

We denote the duals of these linear maps with respect to the non-degenerate pairing between $V^+$ and $V^-$ by
\begin{align}\nonumber \Scomp^\vee_\pm(ijklm)&: V^\mp(ijklm) \to \bigoplus_{\fs{ikm}}  V^\mp(iklm)\otimes V^\mp(ijkm)
\\\nonumber
\Lcomp^\vee_\pm(ijklm) &:  V^\mp(ijklm)\to \bigoplus_{\substack{\Xs{jl}, \fs{ijl}\\\fs{jkl}, \fs{jlm}}} V^\mp(ijkl) \otimes V^\mp(ijlm) \otimes V^\mp(jklm).
\end{align}%
By definition, $Z_+(ijklm)$ and $Z_-(ijklm)$ are the direct sum coefficients of the following linear maps, respectively:
\begin{calign}\nonumber \Scomp_-^\vee(ijklm)\circ \Lcomp_+(ijklm)
&
\Lcomp_-^\vee(ijklm)\circ \Scomp_+(ijklm)
\end{calign}

\skiptocparagraph{Invertibility of the shortside pentagon composite}
The aforementioned heuristic, that the map $Z_+$ is given by composing associators along one side of the pentagon and then `uncomposing' along the other side, of course only makes sense if one of the composition maps $\Lcomp$ or $\Scomp$ is indeed invertible. This is the content of the following lemma.

In the following, we will often write expressions such as $\lambda_i F$, where $\lambda_i \in k$, $i \in I$, $I$ is a finite set, $F$ is a linear map with (co)domain $\bigoplus_{i \in I} W_i$ and $\{W_i\}_{i \in I}$ is a family of vector spaces indexed by $I$. This is shorthand notation for the composite of $F$ with the diagonal map $\bigoplus_{i \in I} \lambda_i \mathrm{id}_{W_i} : \bigoplus_{i \in I} W_i \to \bigoplus_{i \in I} W_i$.
\begin{lemma}[The shortside pentagon composite is invertible] \label{lem:scompinv}
The function $\Scomp_+(ijklm)$ is invertible with inverse $\tdim{ikm}\Scomp^\vee_-(ijklm)$.
\end{lemma}
\begin{proof}
We first prove that $\Scomp_+(ijklm)$ is invertible. Consider the factorization of $\Scomp_+(ijklm)$ in Figure~\ref{fig:factoring}.  The top vertical map and the horizontal map are evidently isomorphisms, the bottom vertical map is an isomorphism since $\tc{C}$ is locally semisimple. Hence, $\Scomp_+(ijklm)$ is invertible. 
\def\di{3.}
\begin{figure}[h]
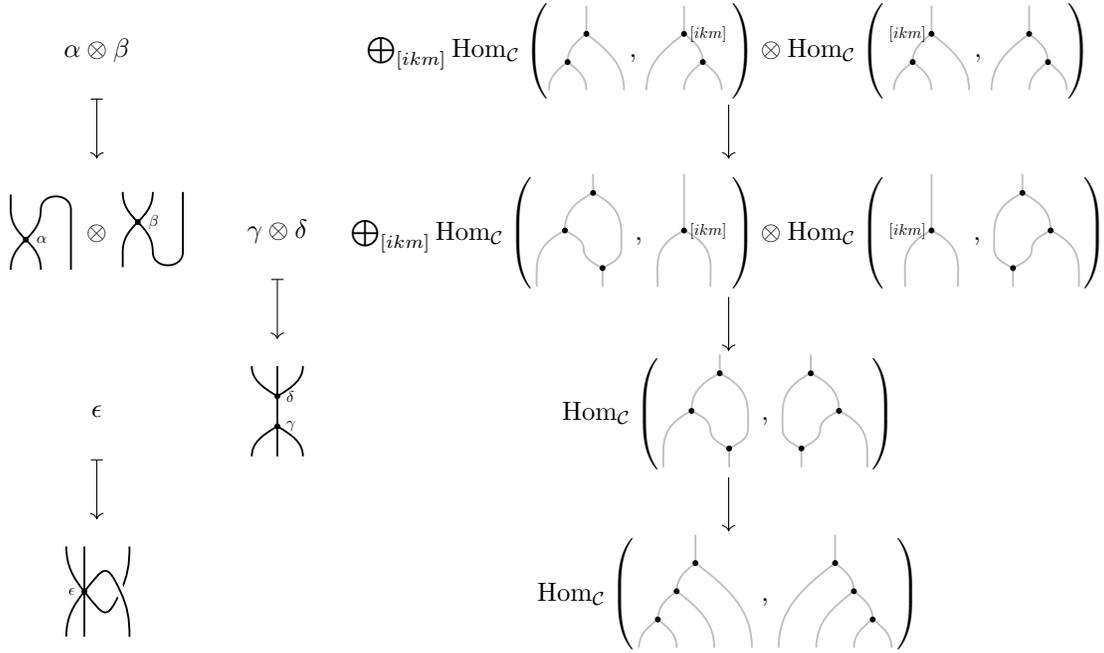

\[\hspace{-0.3cm}
\begin{tz}[scale=0.8]
\node (T) at (0,2*\di) {$\bigoplus_{[ikm]} \Hom_{\tc{C}} \left(
	\begin{tz}[scale=0.5, yscale=0.75]
		\draw[slice] (0,0) to [out=up, in=\dl] (0.5,1) to [out=up, in=\dl] (1,2) to (1,3);
		\draw[slice] (1,0) to [out=up, in=\dr] (0.5,1);
		\draw[slice] (2,0) to [out=up, in=\dr] (1,2);
		\node[dot] at (0.5,1){};
		\node[dot] at (1,2){};
	\end{tz} 
	\,,\,
	\begin{tz}[scale=0.5, yscale=0.75,xscale=-1]
		\draw[slice] (0,0) to [out=up, in=\dl] (0.5,1) to [out=up, in=\dl] (1,2) to (1,3);
		\draw[slice] (1,0) to [out=up, in=\dr] (0.5,1);
		\draw[slice] (2,0) to [out=up, in=\dr] (1,2);
		\node[dot] at (0.5,1){};
		\node[dot](A) at (1,2){};
		\node[right, tmor] at (A) {$[ikm]$};
	\end{tz} 
	\right)\otimes \Hom_{\tc{C}}\left(
	\begin{tz}[scale=0.5, yscale=0.75]
		\draw[slice] (0,0) to [out=up, in=\dl] (0.5,1) to [out=up, in=\dl] (1,2) to (1,3);
		\draw[slice] (1,0) to [out=up, in=\dr] (0.5,1);
		\draw[slice] (2,0) to [out=up, in=\dr] (1,2);
		\node[dot] at (0.5,1){};
		\node[dot] (A) at (1,2){};
		\node[left, tmor] at (A) {$[ikm]$};
	\end{tz} 
	\,,\,
	\begin{tz}[scale=0.5, yscale=0.75,xscale=-1]
		\draw[slice] (0,0) to [out=up, in=\dl] (0.5,1) to [out=up, in=\dl] (1,2) to (1,3);
		\draw[slice] (1,0) to [out=up, in=\dr] (0.5,1);
		\draw[slice] (2,0) to [out=up, in=\dr] (1,2);
		\node[dot] at (0.5,1){};
		\node[dot] at (1,2){};
	\end{tz}
	\right)$};
\node (M) at (0,\di) {$\bigoplus_{[ikm]} \Hom_{\tc{C}}\left(
	\begin{tz}[yscale=0.5,scale=0.5]
		\draw[slice] (1,0) to (1,1) to [out=\ul, in=down] (0.5,2) to [out=up, in=\dr] (0,3) to [out=up, in=\dl] (0.75,5) to (0.75,6);
		\draw[slice] (-0.75,0) to [out=up, in=\dl] (0,3) ;
		\draw[slice] (1,1) to [out=\ur, in=down] (1.5, 2) to [out=up, in=\dr] (0.75,5);
		\node[dot] at (1,1){};
		\node[dot] at (0,3) {};
		\node[dot] at (0.75,5){};
	\end{tz}
	\,\,,\,
	\begin{tz}[yscale=0.5,scale=0.5]
		\draw[slice] (-0.2,0) to [out=up, in=\dl] (0.5,3) to (0.5,6);
		\draw[slice] (1.2,0) to [out=up, in=\dr] (0.5,3);
		\node[dot](A)  at (0.5,3){};
		\node[tmor, right] at (A) {$[ikm]$};
	\end{tz}
	\right)\otimes \Hom_{\tc{C}}\left(
	\begin{tz}[yscale=0.5,scale=0.5]
		\draw[slice] (-0.2,0) to [out=up, in=\dl] (0.5,3) to (0.5,6);
		\draw[slice] (1.2,0) to [out=up, in=\dr] (0.5,3);
		\node[dot] (A) at (0.5,3){};
		\node[tmor, left] at (A) {$[ikm]$};
	\end{tz}
	\,\,,\,
	\begin{tz}[yscale=0.5,scale=0.5,xscale=-1]
		\draw[slice] (1,0) to (1,1) to [out=\ul, in=down] (0.5,2) to [out=up, in=\dr] (0,3) to [out=up, in=\dl] (0.75,5) to (0.75,6);
		\draw[slice] (-0.75,0) to [out=up, in=\dl] (0,3) ;
		\draw[slice] (1,1) to [out=\ur, in=down] (1.5, 2) to [out=up, in=\dr] (0.75,5);
		\node[dot] at (1,1){};
		\node[dot] at (0,3) {};
		\node[dot] at (0.75,5){};
	\end{tz}	
	\right)$};
\node (B) at (0,0){$\Hom_{\tc{C}}\left( 
	\begin{tz}[yscale=0.5,scale=0.5]
		\draw[slice] (1,0) to (1,1) to [out=\ul, in=down] (0.5,2) to [out=up, in=\dr] (0,3) to [out=up, in=\dl] (0.75,5) to (0.75,6);
		\draw[slice] (-0.75,0) to [out=up, in=\dl] (0,3) ;
		\draw[slice] (1,1) to [out=\ur, in=down] (1.5, 2) to [out=up, in=\dr] (0.75,5);
		\node[dot] at (1,1){};
		\node[dot] at (0,3) {};
		\node[dot] at (0.75,5){};
	\end{tz}
	\,\,,\,
	\begin{tz}[yscale=0.5,scale=0.5,xscale=-1]
		\draw[slice] (1,0) to (1,1) to [out=\ul, in=down] (0.5,2) to [out=up, in=\dr] (0,3) to [out=up, in=\dl] (0.75,5) to (0.75,6);
		\draw[slice] (-0.75,0) to [out=up, in=\dl] (0,3) ;
		\draw[slice] (1,1) to [out=\ur, in=down] (1.5, 2) to [out=up, in=\dr] (0.75,5);
		\node[dot] at (1,1){};
		\node[dot] at (0,3) {};
		\node[dot] at (0.75,5){};
	\end{tz}
\right)$};
\node (R) at (0,-\di){$\Hom_{\tc{C}}\left( \,
	\begin{tz}[yscale=0.5*6/4,scale=0.5]
		\draw[slice] (0,0) to [out=up, in=\dl] (0.5,1) to [out=up, in=\dl] (1,2) to [out=up, in=\dl] (1.5,3) to (1.5,4);
		\draw[slice] (1,0) to [out=up, in=\dr] (0.5,1);
		\draw[slice] (2,0) to [out=up, in=\dr] (1,2);
		\draw[slice] (3,0) to  [out=up, in=\dr] (1.5,3);
		\node[dot] at (0.5,1){};
		\node[dot] at (1,2){};
		\node[dot] at (1.5,3){};
	\end{tz}
	\,\,,\,
	\begin{tz}[yscale=0.5*6/4,scale=0.5,xscale=-1]
		\draw[slice] (0,0) to [out=up, in=\dl] (0.5,1) to [out=up, in=\dl] (1,2) to [out=up, in=\dl] (1.5,3) to (1.5,4);
		\draw[slice] (1,0) to [out=up, in=\dr] (0.5,1);
		\draw[slice] (2,0) to [out=up, in=\dr] (1,2);
		\draw[slice] (3,0) to  [out=up, in=\dr] (1.5,3);
		\node[dot] at (0.5,1){};
		\node[dot] at (1,2){};
		\node[dot] at (1.5,3){};
	\end{tz}
	\,\right)$};
\draw[->] (0,5.1) to (0,4.2);
\draw[->] (0,1.9) to (0,1.);
\draw[->] (0,-1.1) to (0,-2);
\node at (-10.5, 6) {$\alpha\otimes \beta$};
\draw[|->] (-10.5,5.2) to (-10.5,4.2);
\node at (-10.5, 3) {$
	\begin{tz}[scale=0.4]
		\draw[wire] (0,0) to [out=up, in=\dl] (0.5,1) to [out=\dr, in=up] (1,0);
		\draw[wire] (0.5,1) to [out=\ul, in=down] (0,2) to (0,2.5);
		\draw[wire] (0.5,1) to [out=\ur, in=down] (1,2) to [out=up, in=up, looseness=1.5] (2,2) to (2,0);
		\node[dot](A) at (0.5,1){};
		\node[right, tmor] at ([xshift=5pt]A) {$\alpha$};
	\end{tz}
	~~\otimes ~~
	\begin{tz}[scale=0.4,yscale=-1]
		\draw[wire] (0,0) to [out=up, in=\dl] (0.5,1) to [out=\dr, in=up] (1,0);
		\draw[wire] (0.5,1) to [out=\ul, in=down] (0,2) to (0,2.5);
		\draw[wire] (0.5,1) to [out=\ur, in=down] (1,2) to [out=up, in=up, looseness=1.5] (2,2) to (2,0);
		\node[dot](A) at (0.5,1){};
		\node[right, tmor] at ([xshift=5pt]A) {$\beta$};
	\end{tz}
	$};
\node at (-7.5, 3) {$\gamma\otimes \delta$};
\draw[|->] (-7.5,2.2) to (-7.5,1.2);
\node at (-7.5, 0) {$
	\begin{tz}[scale=0.4,xscale=1.7]
		\draw[wire] (0,0) to [out=up, in=\dl] (0.5,1) to (0.5,0);
		\draw[wire] (1,0) to [out=up, in=\dr] (0.5,1) to (0.5,3);
		\draw[wire] (0.5,2) to [out=\ul, in=down] (0,3) ;
		\draw[wire] (0.5,2) to [out=\ur, in=down] (1,3);
		\node[dot] (A) at (0.5,1){};
		\node[dot] (B) at (0.5,2) {};
		\node[tmor, right] at ([xshift=2pt]A) {$\gamma$};
		\node[tmor, right] at ([xshift=2pt]B) {$\delta$};
	\end{tz}
	$};
\node at (-10.5,0) {$\epsilon$};
\draw[|->] (-10.5,-0.8) to (-10.5,-1.8);
\node at (-10.5, -3) {$
	\begin{tz}[scale=0.6, xscale=0.8]
		\draw[wire] (0.5,1) to [out=\ur, in=120, looseness=2]  node[mask point, pos=1] (MP){}(1.5,1) to [out=-60, in=up] (1.75,0);
		\cliparoundone{MP}{
		\draw[wire] (0.5,1) to [out=\dr, in=-120, looseness=2] (1.5,1) to [out=60, in=down] (1.75, 2);
		}
		\draw[wire] (0,2) to [out=down, in=\ul] (0.5,1) to (0.5,2);
		\draw[wire] (0,0) to [out=up, in=\dl] (0.5,1) to (0.5,0);
		\node[dot] (A)  at (0.5,1){};
		\node[tmor, left] at ([xshift=-3pt]A) {$\epsilon$};
	\end{tz}
	$};
\end{tz}
\]
\caption{Factoring $\Scomp_+(ijklm)$.}\label{fig:factoring}
\end{figure}

To show that $\tdim{ikm}\Scomp_-^\vee(ijklm)$ is an inverse, it therefore suffices to prove that \[\tdim{ikm}\Scomp_-^\vee(ijklm) \circ \Scomp_+(ijklm) = \id.\] This follows from the calculation in Figure~\ref{fig:pairingsc}, for $\alpha,\beta,\gamma $ and $\delta$ of appropriate type, such that $\alpha\otimes \beta$ and $\gamma\otimes \delta$ are in a direct summand labeled by a simple object $[ikm]$ and $[ikm]'$, respectively.
\end{proof}

\begin{figure}[h]
\begin{align}\nonumber
\tdim{ikm} & \left\langle\vp  \Scomp_+(ijklm)(\alpha,\beta), \Scomp_-(ijklm)(\gamma,\delta) \right\rangle 
\\[-10pt]\nonumber
 &= \eqgap \tdim{ikm}~\begin{tz}[scale=0.5,yscale=0.8]
\draw[wire]  (0.5,1)  to [out=\ul, in=down] (0,2) to (0,4) to  [out=up, in=\dl] (0.5,5);
\draw[wire] (0.5,-1) to [out=\dl, in=up] (0,-2) to (0,-4) to [out=down, in=\ul] (0.5,-5);
\draw[wire] (0.5, -5) to [out=\ur, in=down] (1,-4) to [out=up, in=down] node[mask point, pos=0.5] (MPB){} (2,-2) to (2,2) to [out=up, in=down] node[mask point, pos=0.5] (MPT){} (1,4) to [out=up, in=\dr] (0.5,5);
\cliparoundtwo{MPB}{MPT}{
\draw[wire] (0.5,1) to [out=\ur, in=down] (1,2) to [out=up, in=down] (2,4) to (2,6)  to [out=up, in=up, looseness=1.5] (3,6) to (3,-6)  to [out=down, in=down, looseness=1.5] (2,-6) to (2,-4) to [out=up, in=down] (1,-2) to [out=up, in=\dr] (0.5,-1);
}
\draw[wire] (0.5,-1) to [out=\ul, in=down] (0,0) to [out=up, in=\dl] (0.5,1) to [out=\dr, in=up] (1,0) to [out=down, in=\ur] (0.5,-1);
\draw[wire] (0.5,5) to [out=\ur, in=down] (1,6) to [out=up, in=up, looseness=1.5] (4,6) to (4,-6) to [out=down, in=down, looseness=1.5] (1,-6) to [out=up, in=\dr] (0.5,-5);
\draw[wire] (0.5,5) to [out=\ul, in=down] (0,6) to [out=up, in=up, looseness=1.5] (5,6) to (5,-6) to [out=down, in=down, looseness=1.5] (0,-6) to [out=up, in=\dl] (0.5,-5);
\insidepath{arrow data={0}{>}}{
	(-1.5,0) to (-1.5,6) to [out=up, in=up, looseness=1.5] (6,6) to (6,-6) to [out=down, in=down, looseness=1.5] (-1.5,-6) to cycle}
\node[dot](BB) at (0.5,-5){};
\node[dot](B) at (0.5,-1){};
\node[dot](T) at (0.5,1){};
\node[dot](TT) at (0.5,5){};
\node[right,tmor] at ([xshift=3pt]T) {$\alpha$};
\node[right,tmor] at ([xshift=3pt]TT) {$\beta$};
\node[right,tmor] at ([xshift=3pt]B) {$\gamma$};
\node[right,tmor] at ([xshift=3pt]BB) {$\delta$};
\node[omor, left] at (0.1,3) {$[ikm]'$};
\node[omor, left] at (0,-3) {$[ikm]$};
\end{tz}
\eqgap=\eqgap\tdim{ikm}
\begin{tz}[scale=0.5,yscale=0.8]
\draw[wire]  (0.5,1)  to [out=\ul, in=down] (0,2) to (0,4) to  [out=up, in=\dl] (0.5,5);
\draw[wire] (0.5,-1) to [out=\dl, in=up] (0,-2) to (0,-4) to [out=down, in=\ul] (0.5,-5);
\draw[wire] (0.5, -5) to [out=\ur, in=down] (1,-4) to [out=up, in=down] node[mask point, pos=0.5] (MPB){} (3,-2) to (3,2) to [out=up, in=down] node[mask point, pos=0.5] (MPT){} (1,4) to [out=up, in=\dr] (0.5,5);
\draw[wire] (0.5,1) to [out=\ur, in=down] (1,2) to [out=up, in=up, looseness=1.5] (2,2) to (2,-2) to [out=down, in=down, looseness=1.5] (1,-2) to [out=up, in=\dr] (0.5,-1);
\draw[wire] (0.5,-1) to [out=\ul, in=down] (0,0) to [out=up, in=\dl] (0.5,1) to [out=\dr, in=up] (1,0) to [out=down, in=\ur] (0.5,-1);
\draw[wire] (0.5,5) to [out=\ur, in=down] (1,6) to [out=up, in=up, looseness=1.5] (4,6) to (4,-6) to [out=down, in=down, looseness=1.5] (1,-6) to [out=up, in=\dr] (0.5,-5);
\draw[wire] (0.5,5) to [out=\ul, in=down] (0,6) to [out=up, in=up, looseness=1.5] (5,6) to (5,-6) to [out=down, in=down, looseness=1.5] (0,-6) to [out=up, in=\dl] (0.5,-5);
\insidepath{arrow data={0}{>}}{
	(-1.5,0) to (-1.5,6) to [out=up, in=up, looseness=1.5] (6,6) to (6,-6) to [out=down, in=down, looseness=1.5] (-1.5,-6) to cycle}
\node[dot](BB) at (0.5,-5){};
\node[dot](B) at (0.5,-1){};
\node[dot](T) at (0.5,1){};
\node[dot](TT) at (0.5,5){};
\node[right,tmor] at ([xshift=3pt]T) {$\alpha$};
\node[right,tmor] at ([xshift=3pt]TT) {$\beta$};
\node[right,tmor] at ([xshift=3pt]B) {$\gamma$};
\node[right,tmor] at ([xshift=3pt]BB) {$\delta$};
\node[omor, left] at (0.1,3) {$[ikm]'$};
\node[omor, left] at (0,-3) {$[ikm]$};
\end{tz}
\\[-10pt]\nonumber
&\superequals{\begin{minipage}{1.5cm} \centering simplicity of\\$[ikm], [ikm]'$\end{minipage}}\eqgap \hspace{0.2cm} \eqgap 
\delta_{[ikm], [ikm]'}~ \Tr\left(\alpha\xt \gamma\right) ~\Tr\left(\beta \xt \delta\right) = \left\langle \alpha \otimes \beta, \gamma \otimes \delta \right \rangle
\qedhere
\end{align}
\caption{Pairing the positive and negative shortside pentagon composites.}
\label{fig:pairingsc}
\end{figure}

\skiptocparagraph{Right invertibility of the longside pentagon composite}
While $\Scomp_+$ is invertible, allowing us to express $\ZH_+$ as a `composition around the pentagon', by contrast $\Lcomp_+$ is not in general invertible. This situation is in contrast to the situation with endotrivial fusion categories, in which both $\Scomp_+$ and $\Lcomp_+$ are invertible~\cite{Mackaay}. However, we will show below that $\Lcomp_+$ has a section, or a right inverse, and that $\ZH_-$ can indeed be expressed as a composite of $\Scomp_+$ and this section; hence $\ZH_-$ is itself a section of $\ZH_+$. This suffices later to verify invariance under bistellar moves. 

\begin{lemma}[The longside pentagon composite is right invertible] \label{lem:rightinverse} 
The linear map $\Lcomp_+(ijklm)$ fulfills the following equation:
\begin{equation}\nonumber \Lcomp_+(ijklm) \circ \left(\frac{\tdim{ijl} \tdim{jkl} \tdim{jlm}}{\odim{jl}} \Lcomp^\vee_-(ijklm) \right) = \id
\end{equation}
\end{lemma}
\begin{proof} In this proof, we use the following notation:
\begin{align}\nonumber V^+(i(jkl)m) &:= \Hom_{\tc{C}}(\fs{((ijk)l)m}, \fs{i((jkl)m)}) 
\\ \nonumber
V^-(i(jkl)m) &:=\Hom_{\tc{C}}( \fs{i((jkl)m)},\fs{((ijk)l)m}) 
\end{align}We define linear maps 
\begin{align}\nonumber 
a_\pm:& \hspace{0.57cm}\bigoplus_{\fs{ijl}}\hspace{0.57cm} V^\pm(ijkl) \otimes V^\pm(ijlm) \to V^\pm(i(jkl)m)\\\nonumber
b_\pm:&\bigoplus_{\Xs{jl}, \fs{jkl}, \fs{jlm}} V^\pm(i(jkl)m)\otimes V^\pm(jklm) \to V^\pm(ijklm)\end{align}
as follows, for $\alpha, \alpha', \beta, \beta',\gamma, \gamma', \delta$ and $\delta'$ of appropriate type:

\begin{calign}\nonumber
a_+(\alpha, \beta) := 
\begin{tz}[scale=0.5]
\draw[wire] (0,0) to [out=up, in=\dl] (0.5,1) to [out=\ul, in=down] (0,2) to [out=up, in=\dr] (-0.5,3) to [out=\ur, in=down] (0,4);
\draw[wire] (-1,0) to (-1,2) to [out=up, in=\dl] (-0.5,3) to [out=\ul, in=down] (-1,4);
\draw[wire] (1,0) to [out=up, in=\dr] (0.5,1) to [out=\ur, in=down] (1,2) to (1,4);
\node[dot] (B) at (0.5,1){};
\node[dot] (T) at (-0.5,3){};
\node[right, tmor] at ([xshift=3pt]B) {$\alpha\vphantom{\alpha'_+}$};
\node[left, tmor] at ([xshift=-3pt]T) {$\beta\vphantom{\beta'_+}$};
\end{tz}
&
a_-(\alpha', \beta'):=
\begin{tz}[scale=0.5,xscale=-1]
\draw[wire] (0,0) to [out=up, in=\dl] (0.5,1) to [out=\ul, in=down] (0,2) to [out=up, in=\dr] (-0.5,3) to [out=\ur, in=down] (0,4);
\draw[wire] (-1,0) to (-1,2) to [out=up, in=\dl] (-0.5,3) to [out=\ul, in=down] (-1,4);
\draw[wire] (1,0) to [out=up, in=\dr] (0.5,1) to [out=\ur, in=down] (1,2) to (1,4);
\node[dot] (B) at (0.5,1){};
\node[dot] (T) at (-0.5,3){};
\node[left, tmor] at ([xshift=3pt]B) {$\beta'\vphantom{\beta}$};
\node[right, tmor] at ([xshift=-3pt]T) {$\alpha'\vphantom{\alpha_-}$};
\end{tz}
&
b_+(\gamma, \delta):=
\begin{tz}[scale=0.5]
\draw[wire] (-1,0) to [out=up, in=\dl] (0,1) to [out=\ul, in=down] (-1,2) to (-1,4);
\draw[wire] (0,0) to (0,2) to [out=up, in=\dl] (0.5,3) to [out=\ul, in=down] (0,4);
\draw[wire] (1,0) to [out=up, in=\dr] (0,1) to [out=\ur, in=down] (1,2) to [out=up, in=\dr] (0.5,3) to [out=\ur, in=down] (1,4);
\node[dot] (B)  at (0,1){};
\node[dot] (T) at (0.5,3){};
\node[right,tmor] at ([xshift=3pt]B){$\gamma\vphantom{\gamma'_+}$};
\node[left, tmor] at ([xshift=-3pt]T) {$\delta\vphantom{\delta'_+}$};
\end{tz}
&
b_-(\gamma',\delta'):=
\begin{tz}[scale=0.5,yscale=-1]
\draw[wire] (-1,0) to [out=up, in=\dl] (0,1) to [out=\ul, in=down] (-1,2) to (-1,4);
\draw[wire] (0,0) to (0,2) to [out=up, in=\dl] (0.5,3) to [out=\ul, in=down] (0,4);
\draw[wire] (1,0) to [out=up, in=\dr] (0,1) to [out=\ur, in=down] (1,2) to [out=up, in=\dr] (0.5,3) to [out=\ur, in=down] (1,4);
\node[dot] (B)  at (0,1){};
\node[dot] (T) at (0.5,3){};
\node[right,tmor] at ([xshift=3pt]B){$\gamma'\vphantom{\gamma_-}$};
\node[left, tmor] at ([xshift=-3pt]T) {$\delta'\vphantom{\delta}$};
\end{tz}
\end{calign}
Analogously to the proof of Lemma~\ref{lem:scompinv}, it follows that $a_+$ is invertible with inverse $\tdim{ijl} a^\vee_-$.

The maps $a_\pm$ and $b_\pm$ give rise to a factorization of $\Lcomp_\pm(ijklm)$ as follows: 
\[\Lcomp_\pm(ijklm) = b_\pm\circ \left(\bigoplus_{\Xs{jl}, \fs{jkl}, \fs{jlm}} a_\pm\otimes \id_{V^\pm(jklm)}\right)\] 
In particular, abbreviating the scalar $\lambda =\frac{ \tdim{jkl} \tdim{jlm}}{\odim{jl}}$, and omitting direct sum and tensor products with identities, the following holds:
\[
\shrinker{.9}{
\Lcomp_+(ijklm) \circ \left( \frac{\tdim{ijl} \tdim{jkl} \tdim{jlm}}{\odim{jl})}  \Lcomp^\vee_-(ijklm)\right) = b_+ \circ a_+ \circ \left( \tdim{ijl} a^\vee_-\right) \circ \left( \lambda b^\vee_-\right) = b_+ \circ \left( \lambda b^\vee_-\right)
}
\]
To prove Lemma~\ref{lem:rightinverse}, it therefore suffices to show that $b_+ \circ \left( \lambda b^\vee_-\right) = \id$, or equivalently that the following map is the identity:
\begin{equation}\nonumber
\sum_\alpha \lambda~\left\langle\vp b_- (\widehat{\alpha}), - \right\rangle b_+(\alpha) : V^+(ijklm) \to V^+(ijklm)
\end{equation}
Here, the sum is over a basis $\{\alpha\}$ of $\bigoplus_{\Xs{jl}, \fs{jkl}, \fs{jlm}} V^+(i(jkl)m) \otimes V^+(jklm)$ with corresponding dual basis $\{\widehat{\alpha}\}$ of $\bigoplus_{\Xs{jl}, \fs{jkl}, \fs{jlm}} V^-(i(jkl)m) \otimes V^-(jklm)$ with respect to the pairing $\langle\cdot ,\cdot \rangle$.
To prove this, we introduce the following auxilliary map:
\begin{calign}\nonumber r: V^+(ijklm) \otimes V^-(jklm) \to V^+(i(jkl)m) 
&r(\alpha, \beta):=
\begin{tz}[scale=0.5]
\draw[wire] (-1,0) to [out=up, in=\dl] (0,1) to [out=\ul, in=down] (-1,2) to (-1,4);
\draw[wire] (0,0) to (0,2) to [out=up, in=\dl] (0.5,3) to [out=\ul, in=down] (0,4);
\draw[wire] (1,0) to [out=up, in=\dr] (0,1) to [out=\ur, in=down] (1,2) to [out=up, in=\dr] (0.5,3) to [out=\ur, in=down] (1,4);
\node[dot] (B)  at (0,1){};
\node[dot] (T) at (0.5,3){};
\node[right,tmor] at ([xshift=3pt]B){$\alpha$};
\node[left, tmor] at ([xshift=-3pt]T) {$\beta$};
\end{tz}
\end{calign}
It follows from definition that for $\widehat{F}\in V^-(i(jkl)m)$, $\widehat{c}\in V^-(jklm)$ and $G\in V^+(ijklm)$,
\[
\left\langle b_-(\widehat{F},\widehat{c}), G\right\rangle = \left\langle \widehat{F}, r(G,\widehat{c})\right\rangle .
\]
Choosing bases $\{F\}$ of $V^+(i(jkl)m)$ and $\{c\}$ of $V^+(jklm)$ with corresponding dual bases $\{\widehat{F}\}$ and $\{\widehat{c}\}$ induces a basis $\{\alpha\}$ of $\bigoplus_{\Xs{jl}, \fs{jkl}, \fs{jlm}} V^+(i(jkl)m) \otimes V^+(jklm)$ with which the above expression, evaluated at $G\in V^+(ijklm)$ becomes the following:
\[\sum_{\alpha} \lambda \left \langle\vp  b_-(\widehat{\alpha}), G\right \rangle b_+(\alpha) = \hspace{-0.25cm}
\sum_{\substack{\Xs{jl}, \fs{jkl}, \fs{jlm}\\ F, c}}\hspace{-0.2cm}  \lambda \left\langle b_-(\widehat{F}, \widehat{c}) ,G\right\rangle b_+(F, c) =
\hspace{-0.25cm} \sum_{\substack{\Xs{jl}, \fs{jkl}, \fs{jlm}\\ F, c}}\hspace{-0.2cm} \lambda \left\langle \widehat{F}, r(G, \widehat{c})\right\rangle b_+(F, c) 
\]
\[= \sum_{\Xs{jl}, \fs{jkl}, \fs{jlm}, c} \lambda ~ b_+\left(\vphantom{\frac{a}{b}} r(G, \widehat{c}), c \right)
= \sum_{\Xs{jl}, \fs{jkl}, \fs{jlm}, c} \frac{ \tdim{jkl} \tdim{jlm}}{\odim{jl}}\hspace{0.25cm}
\begin{tz}[scale=0.5,yscale=0.8]
\draw[wire] (0,0) to [out=up, in=\dl] (1,1) to [out=\ul, in=down] (0,2) to (0,6);
\draw[wire] (1,0) to (1,2) to [out=up, in=\dl] (1.5,3) to [out=\ul, in=down] (1,4) to [out=up, in=\dl] (1.5, 5) to [out=\ul, in=down] (1,6);
\draw[wire] (2,0) to [out=up, in=\dr] (1,1) to [out=\ur, in=down] (2,2) to [out=up, in=\dr] (1.5,3) to [out=\ur, in=down] (2,4) to [out=up, in=\dr] (1.5,5) to [out=\ur, in=down] (2,6);
\node[dot] (BB) at (1,1) {};
\node[dot] (B) at (1.5,3) {};
\node[dot] (T) at (1.5,5){};
\node[tmor, right] at ([xshift=3pt]BB){$G$};
\node[tmor,right] at ([xshift=3pt]B){$\widehat{c}\vphantom{c_+}$};
\node[tmor, right] at ([xshift=3pt]T){$c$};
\end{tz}
\]
By Corollary~\ref{cor:formulabasis}, this equals $G$.
\end{proof}

\skiptocparagraph{The summed normalized 10j symbols and their factorization} 

For a $\tc{C}$-state $\Delta^5_{(2)}\To \Delta\tc{C}^{\sk}$ and a $4$-simplex $\lan ijklm\ran \in \Delta^5_4$, we define the following normalized linear maps:%
\begin{calign} \nonumber 
\zh_+(ijklm) := \tdim{ikm} Z_+(ijklm)&
\zh_-(ijklm) :=\frac{\tdim{ijl} \tdim{jkl} \tdim{jlm}}{\odim{jl}} Z_-(ijklm)
\end{calign}
Taking the direct sum over simple objects $\Xs{jl}\in \Delta\tc{C}^{\sk}_1$ and compatible simple $1$-morphisms $\fs{ijl}, \fs{jkl}, \fs{jlm}$ and $\fs{ikm}$ in $\Delta\tc{C}^{\sk}_2$ 
gives rise to the following linear maps:
\begin{align}\nonumber \ZH_+(ijklm) &:= \bigoplus_{\substack{\Xs{jl}, \fs{ijl},\\ \fs{jkl}, \fs{jlm}}} \bigoplus_{\fs{ikm}} \zh_+(ijklm):\\\nonumber
& \bigoplus_{\substack{\Xs{jl}, \fs{ijl},\\ \fs{jkl}, \fs{jlm}}} V^+(ijkl) \otimes V^+(ijlm) \otimes V^+(jklm) \to \bigoplus_{\fs{ikm}} V^+(iklm) \otimes V^+(ijkm)\\\nonumber
\ZH_-(ijklm) &:= \bigoplus_{\fs{ikm}} \bigoplus_{\substack{\Xs{jl}, \fs{ijl},\\ \fs{jkl}, \fs{jlm}}} \zh_-(ijklm) : \\ \nonumber
&\bigoplus_{\fs{ikm}} V^+(iklm) \otimes V^+(ijkm)  \to  \bigoplus_{\substack{\Xs{jl}, \fs{ijl},\\ \fs{jkl}, \fs{jlm}}} V^+(ijkl) \otimes V^+(ijlm)\otimes V^+(jklm) 
\end{align}

\begin{corollary}[The summed normalized 10j symbol factors around the pentagon] \label{cor:Z+decomp}
The direct sum map $\ZH_+(ijklm)$ factors as follows through $V^+(ijklm)$:
\[ 
\ZH_+(ijklm) = \left(\Scomp_+(ijklm)\right)^{-1} \circ\Lcomp_+(ijklm)
\]
\end{corollary}
\begin{proof} By definition, $Z_+(ijklm)$ is a direct sum coefficient of the composite $\Scomp_-^\vee(ijklm) \xo \Lcomp_+(ijklm)$. The corollary therefore follows from Lemma~\ref{lem:scompinv} and the definition of \linebreak $\ZH_+(ijklm)$ as a normalized direct sum.
\end{proof}

\begin{corollary}[The negative and positive summed normalized 10j symbols form a section-retraction pair] \label{cor:Z+Z-}
$\ZH_-(ijklm)$ is a section of $\ZH_+(ijklm)$:
\[ 
\ZH_+(ijklm) \xo \ZH_-(ijklm) = \id
\]
\end{corollary}
\begin{proof}Since $Z_-(ijklm)$ is a direct sum coefficient of $\Lcomp^\vee_-(ijklm) \xo \Scomp_+(ijklm)$, it follows that the normalized direct sum map $\ZH_-$ can be written as follows: 
\[
\ZH_-(ijklm) = \left(\frac{\tdim{ijl} \tdim{jkl} \tdim{jlm}}{\odim{jl}} \Lcomp^\vee_-(ijklm) \right)  \circ  \Scomp_+(ijklm).
\]
The corollary therefore follows from Corollary~\ref{cor:Z+decomp} and Lemma~\ref{lem:rightinverse}.
\end{proof}

\subsubsection{Invariance under the bistellar moves}\label{sec:Pachnerdetail}
In the following, we prove Lemma~\ref{lem:Pachner33}, Lemma~\ref{lem:Pachner24}, and Lemma~\ref{lem:Pachner15}, showing invariance of the state sum under bistellar moves.

\skiptocparagraph{Invariance under the $(3,3)$-bistellar move.}
We start with the $(3,3)$-bistellar move and prove Lemma~\ref{lem:Pachner33}.

\begin{proof}[Proof of Lemma~\ref{lem:Pachner33}] Expressed in terms of the normalized $\zh_+$ and $\zh_-$, the equation in Lemma~\ref{lem:Pachner33} reads as follows:
\[\sum_{\fs{135}} \zh_+(01235) \zh_+(01345) \zh_+(12345) = \sum_{\fs{024}} \zh_+(02345) \zh_+(01245) \zh_+(01234)
\]
Rewritten in terms of the direct sums $\ZH_+$ and $\ZH_-$, and omitting all direct sum symbols, this becomes the following equation:
\begin{equation}\label{eq:pachnersimplified}\tag{B3-3}
\ZH_+(01235)\ZH_+(01345)\ZH_+(12345) = \ZH_+(02345)\ZH_+(01245) \ZH_+(01234)
\end{equation}
between linear maps
\begin{multline}\nonumber
\shrinker{.92}{
		\bigoplus\limits_{\substack{\Xs{13},\fs{123}\\\fs{013}}}	
		\bigoplus\limits_{\substack{\Xs{14},\fs{134}\\ \fs{014},\fs{145}}}
		\bigoplus\limits_{\substack{\Xs{24}, \fs{234}\\\fs{124},\fs{245}}}
	\V{1234}\otimes \V{1245}\otimes \V{2345}\otimes \V{0134}\otimes \V{0145}\otimes\V{0123}
}	
\\
\shrinker{.92}{
		\longrightarrow 
		\bigoplus\limits_{\fs{025},\fs{035}}
		\V{0125} \otimes 
		\V{0235} \otimes \V{0345}.
}
\end{multline}
 The full expression, including direct sum symbols, for the left-hand side and right-hand side of equation~\eqref{eq:pachnersimplified} are given as the composite of the rightmost column of vertical arrows in Figure~\ref{fig:bigcommutativediagram1} and Figure~\ref{fig:bigcommutativediagram2}, respectively. In these figures, the vector spaces $V^+(ijkl)$ are abbreviated as $[ijkl]$ and the parts of the domain (codomain) on which the involved linear maps act non-trivially are underlined (overlined).

\def\RR{11.25}%
\def\RL{8.5}%
\def\LR{3}%
\def\LL{0.5}%
\def\RN{15.5}%
\def\MN{5.5}%
\def\LN{-0.25}%
\def\Rha{12}%
\def\Rhb{8}%
\def\Rhc{4}%
\def\Rhd{0}%
\def\Mha{10}%
\def\Mhb{6}%
\def\Mhc{2}%
\def\scl{0.8}%
\def\distmapsto{2cm}%
\def\shortenmapsto{1.5cm}%
\def\shortenmapsfrom{0cm}%
\def\nodescale{0.7*\scl}%
\def\arrowdist{0.8cm}%
\def\arrowgap{0.2cm}%
\def\arrowshift{0.15cm}%
\def\middlerightshiftb{0.3cm}%
\def\middlerightshiftc{0.3cm}%
\renewcommand\V[1]{[#1]}%
\newcommand\rb[1]{\raisebox{-0.05cm}{$ #1 $}}%
\newcommand\rbz[1]{\raisebox{-0.2cm}{$ #1 $}}%
\tikzset{diagramarrow/.style={draw, ->}}%
\tikzset{arrowlabel/.style={scale=\nodescale, above,sloped}}%
\tikzset{downarrowlabel/.style={scale=\nodescale, above, rotate =180, sloped}}%
\begin{figure}
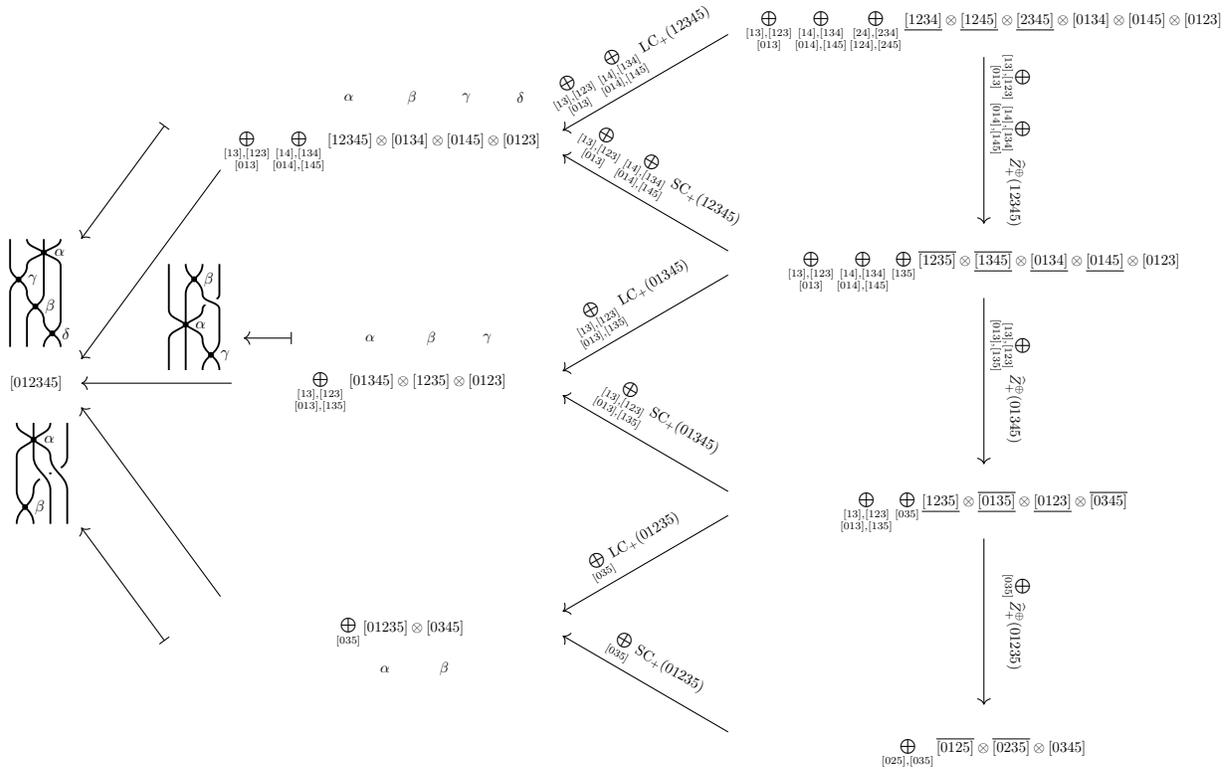

\[\hspace{-1cm}
\begin{tz}[scale=\scl]
\node[scale=\nodescale](rightnodea) at (\RN, \Rha){
		$
		\bigoplus\limits_{\substack{\Xs{13},\fs{123}\\\fs{013}}}	
		\bigoplus\limits_{\substack{\Xs{14},\fs{134}\\ \fs{014},\fs{145}}}
		\bigoplus\limits_{\substack{\Xs{24}, \fs{234}\\\fs{124},\fs{245}}}
		\rb{
		\underline{\V{1234}}\otimes \underline{\V{1245}}\otimes 				
		\underline{\V{2345}}\otimes \V{0134}\otimes \V{0145}\otimes\V{0123}
		}$};
\node[scale=\nodescale](rightnodeb) at (\RN, \Rhb){
		$\bigoplus\limits_{\substack{\Xs{13},\fs{123}\\\fs{013}}}
		\bigoplus\limits_{\substack{\Xs{14},\fs{134}\\ \fs{014},\fs{145}}}
		\bigoplus\limits_{\fs{135}}
		\rb{
		\overline{\V{1235}}\otimes \underline{\overline{\V{1345}}} \otimes \underline{\V{0134}} \otimes \underline{\V{0145}} \otimes \V{0123} 
		}$};
\node[scale=\nodescale] (rightnodec) at (\RN, \Rhc){
		$\bigoplus\limits_{\substack{\Xs{13},\fs{123}\\\fs{013},\fs{135}}}\bigoplus\limits_{\fs{035}} 
		\rb{
		\underline{\V{1235}} \otimes 
			\underline{\overline{\V{0135}}} \otimes \underline{\V{0123}} \otimes 
			\overline{\V{0345}}
			}$};
\node[scale=\nodescale] (rightnoded) at (\RN, \Rhd){
		$\bigoplus\limits_{\fs{025},\fs{035}} 
		\rb{
		\overline{\V{0125}} \otimes 
			\overline{\V{0235}} \otimes \V{0345}
			}$};
\node[scale=\nodescale] (middlenoda) at (\MN, \Mha){
		$\bigoplus\limits_{\substack{\Xs{13},\fs{123}\\\fs{013}}}
		\bigoplus\limits_{\substack{\Xs{14},\fs{134}\\ \fs{014},\fs{145}}}
		\rb{\V{12345} \otimes \V{0134} \otimes \V{0145} \otimes \V{0123}}$};
\node[scale=\nodescale] (middlenodb) at ([xshift=\middlerightshiftb]\MN, \Mhb){
		$\bigoplus\limits_{\substack{\Xs{13},\fs{123}\\\fs{013},\fs{135}}}
		\rb{\V{01345}\otimes \V{1235} \otimes \V{0123} }$};
\node[scale=\nodescale] (middlenodc) at ([xshift=\middlerightshiftc]\MN, \Mhc){
		$\bigoplus\limits_{\fs{035}}\rb{\V{01235}\otimes \V{0345}}$};
\node[scale=\nodescale] (leftnode) at ([yshift=\arrowshift]\LN, \Mhb){
		$\V{012345}$};
\draw[diagramarrow] ([yshift=-\arrowdist]rightnodea) to node[downarrowlabel] {$
		\bigoplus\limits_{\substack{\Xs{13},\fs{123}\\\fs{013}}}
		\bigoplus\limits_{\substack{\Xs{14},\fs{134}\\ \fs{014},\fs{145}}}
		\rbz{\ZH_+(12345)}$} ([yshift=\arrowdist]rightnodeb);
\draw[diagramarrow] ([yshift=-\arrowdist]rightnodeb) to node[downarrowlabel] {$\bigoplus\limits_{\substack{\Xs{13},\fs{123}\\\fs{013},\fs{135}}} \rbz{\ZH_+(01345)}$} ([yshift=\arrowdist]rightnodec);
\draw[diagramarrow] ([yshift=-\arrowdist]rightnodec) to node[downarrowlabel] {$\bigoplus\limits_{\substack{\fs{035}\\\vphantom{\fs{123}}}} 
\rbz{\ZH_+(01235)}$} ([yshift=\arrowdist]rightnoded);
\draw[diagramarrow] ([yshift=-\arrowgap+\arrowshift]\RR, \Rha) to node[arrowlabel]{$
		\bigoplus\limits_{\substack{\Xs{13},\fs{123}\\\fs{013}}}
		\bigoplus\limits_{\substack{\Xs{14},\fs{134}\\ \fs{014},\fs{145}}}
		\rbz{\Lcomp_+(12345)}$}
		 ([yshift=\arrowgap+\arrowshift]\RL, \Mha);
\draw[diagramarrow] ([yshift=\arrowgap+\arrowshift]\RR, \Rhb) to node[arrowlabel] {$
		\bigoplus\limits_{\substack{\Xs{13},\fs{123}\\\fs{013}}}
		\bigoplus\limits_{\substack{\Xs{14},\fs{134}\\ \fs{014},\fs{145}}}
		\rbz{\Scomp_+(12345)}$} 
		([yshift=-\arrowgap+\arrowshift]\RL, \Mha);
\draw[diagramarrow] ([yshift=-\arrowgap+\arrowshift]\RR, \Rhb) to node[arrowlabel] {$
	\bigoplus\limits_{\substack{\Xs{13},\fs{123}\\\fs{013},\fs{135}}}
	\rbz{\Lcomp_+(01345)}$} 
	([yshift=\arrowgap+\arrowshift]\RL, \Mhb);
\draw[diagramarrow] ([yshift=\arrowgap+\arrowshift]\RR, \Rhc) to node[arrowlabel] {$
	\bigoplus\limits_{\substack{\Xs{13},\fs{123}\\\fs{013},\fs{135}}}
	\rbz{\Scomp_+(01345)}$}  ([yshift=-\arrowgap+\arrowshift]\RL, \Mhb);
\draw[diagramarrow] ([yshift=-\arrowgap+\arrowshift]\RR, \Rhc) to node[arrowlabel] {$\bigoplus\limits_{\fs{035}} 
\rb{\Lcomp_+(01235)}$} ([yshift=\arrowgap+\arrowshift]\RL, \Mhc);
\draw[diagramarrow] ([yshift=\arrowgap+\arrowshift]\RR, \Rhd) to node[arrowlabel] {$
	\bigoplus\limits_{\fs{035}} 
	\rb{\Scomp_+(01235)}$} ([yshift=-\arrowgap+\arrowshift]\RL, \Mhc);
\draw[diagramarrow,shorten <=0.25cm] ([yshift=-\arrowgap+\arrowshift]\LR, \Mha) to ([yshift=2*\arrowgap+\arrowshift]\LL, \Mhb);
\draw[diagramarrow] ([yshift=\arrowshift]\LR, \Mhb) to ([yshift=\arrowshift]\LL, \Mhb);
\draw[diagramarrow,shorten <=0.25cm] ([yshift=\arrowgap+\arrowshift]\LR, \Mhc) to ([yshift=-2*\arrowgap+\arrowshift]\LL, \Mhb);
\node (Pict) at ([yshift=\arrowshift]-0.15,7.5) {$
	\begin{tz}[scale=\nodescale,scale=0.4, yscale=0.8]
		\draw[wire] (2,0) to [out=up, in=\dl] (2.5,1) to [out=\ul,in= down] (2,2)  to [out=up, in=\dr] (1.5, 3) to [out=\ur, in=down] (2,4) to (2,8);
		\coordinate (A) at (2.5,1);
		\coordinate (B) at (1.5,3);
		\coordinate (C) at (0.5,5);
		\coordinate (D) at (2, 7);
		\draw[wire] (3,0) to [out=up, in=\dr] (2.5,1) to [out=\ur, in=down] (3,2) to (3,6) to [out=up, in=\dr] (2,7) to [out=\ur, in=down] (3,8);
		\draw[wire] (1,0) to (1,2) to [out=up, in=\dl] (1.5,3) to [out=\ul, in=down] (1,4)  to [out=up, in=\dr]  (0.5,5) to [out=\ur, in=down] (1,6) to [out=up, in=\dl] (2,7) to [out=\ul, in=down] (1,8);
		\draw[wire] (0,0) to (0,4) to [out=up, in=\dl] (0.5,5) to [out=\ul, in=down] (0,6) to (0,8);
		\node[dot] at (A){};
		\node[dot] at (B){};
		\node[dot] at (C){};
		\node[dot] at (D){};
		\node[tmor, right] at ([xshift=0.25cm]A) {$\delta$};
		\node[tmor, right] at ([xshift=0.25cm]B) {$\beta$};
		\node[tmor, right] at ([xshift=0.25cm]C) {$\gamma$};
		\node[tmor, right] at ([xshift=0.3cm]D) {$\alpha$};
	\end{tz}
	$};
\node (Picb) at ([yshift=\arrowshift]-0.15,4.5) {$
	\begin{tz}[scale=\nodescale,scale=0.4]
		\draw[wire] (0,0) to [out=up, in=\dl] (0.5,1) to [out=\ul, in=down] (0,2) to (0,4) to [out=up, in=\dl] (1,5) to [out=\ul, in=down] (0,6);
		\draw[wire] (2,0) to (2,2) to [out=up, in=down] node[scale=2.5,mask point, pos=0.39] (MPL){} (1,4) to (1,6);
		\draw[wire] (3,0) to (3,2) to [out=up, in=down] node[scale=2.5,mask point, pos=0.6] (MPR){} (2,4)to [out=up, in=\dr] (1,5) to [out=\ur, in=down] (2,6);
		\cliparoundtwo{MPL}{MPR}{\draw[wire] (1,0) to [out=up, in=\dr] (0.5,1) to [out=\ur, in=down] (1,2) to [out=up, in=down] (3,4) to (3,6);}
		\coordinate (A) at  (0.5,1) ;
		\coordinate (B) at (1,5) ;
		\node[dot] at (A){};
		\node[dot] at (B){};
		\node[tmor, right] at ([xshift=0.25cm]A) {$\beta$};
		\node[tmor, right] at ([xshift=0.3cm]B) {$\alpha$};
	\end{tz}
	$};
\node (Picc) at ([yshift=1.1cm+\arrowshift]2.5,\Mhb) {$
	\begin{tz}[scale=\nodescale,scale=0.4,yscale=0.9]
		\draw[wire] (3,0) to [out=up, in=\dr] (2.5,1) to [out=\ur, in=down] (3,2)  to (3,4) to [out=up, in=down] node[mask point, scale=2.5,pos=0.5] (MP){} (2,5)  to [out=up, in=\dr] (1.5,6) to [out=\ur, in=down] (2,7);
		\cliparoundone{MP}{
				\draw[wire] (2,0) to [out=up, in=\dl] (2.5,1) to [out=\ul, in=down] (2,2) to [out=up, in=\dr] (1,3) to [out=\ur, in=down] (2,4) to [out=up, in=down] (3,5) to (3,7);
				}
		\draw[wire] (0,0) to (0,2) to [out=up, in=\dl] (1,3) to [out=\ul, in=down] (0,4) to (0,7);
		\draw[wire] (1,0) to (1,5) to [out=up, in=\dl] (1.5,6) to [out=\ul, in=down] (1,7);
		\coordinate (A) at (2.5,1);
		\coordinate (B) at (1,3);
		\coordinate (C) at (1.5, 6);
		\node[dot] at (A){};
		\node[dot] at (B){};
		\node[dot] at (C){};
		\node[tmor, right] at ([xshift=0.25cm]A) {$\gamma$};
		\node[tmor, right] at ([xshift=0.25cm]B) {$\alpha$};
		\node[tmor, right] at ([xshift=0.25cm]C) {$\beta$};
	\end{tz}
	$};
\node[scale=\nodescale] at ([yshift=0.75cm+\arrowshift] 4.95, \Mha) {$\alpha$};
\node[scale=\nodescale] at ([yshift=0.75cm+\arrowshift] 6, \Mha) {$\beta$};
\node[scale=\nodescale] at ([yshift=0.75cm+\arrowshift] 6.9, \Mha) {$\gamma$};
\node[scale=\nodescale] at ([yshift=0.75cm+\arrowshift] 7.8, \Mha) {$\delta$};
\node[scale=\nodescale] at ([yshift=0.75cm+\arrowshift] 5.3, \Mhb) {$\alpha$};
\node[scale=\nodescale] at ([yshift=0.75cm+\arrowshift] 6.32, \Mhb) {$\beta$};
\node[scale=\nodescale] at ([yshift=0.75cm+\arrowshift] 7.25, \Mhb) {$\gamma$};
\node[scale=\nodescale] at ([yshift=-0.75cm+\arrowshift] 5.55, \Mhc) {$\alpha$};
\node[scale=\nodescale] at ([yshift=-0.75cm+\arrowshift] 6.53, \Mhc) {$\beta$};
\draw[diagramarrow, shorten <=\shortenmapsto, shorten >=\shortenmapsfrom, |->] ([yshift=\distmapsto, yshift=-\arrowgap+\arrowshift]\LR, \Mha) to ([yshift=\distmapsto+2*\arrowgap+\arrowshift] \LL, \Mhb);
\draw[diagramarrow,shorten <=\shortenmapsto,shorten >= \shortenmapsfrom, |->] ([yshift=-\distmapsto, yshift=\arrowgap+\arrowshift]\LR, \Mhc) to ([yshift=-\distmapsto-2*\arrowgap+\arrowshift] \LL, \Mhb);
\draw[diagramarrow, |->] ([yshift=0.75cm+\arrowshift] 4,\Mhb) to ([yshift=0.75cm+\arrowshift] 3.2, \Mhb);
\end{tz}
\]
\caption{The left hand side of the equation in the proof of Lemma~\ref{lem:Pachner33}.}
\label{fig:bigcommutativediagram1}
\end{figure}%
\begin{figure}
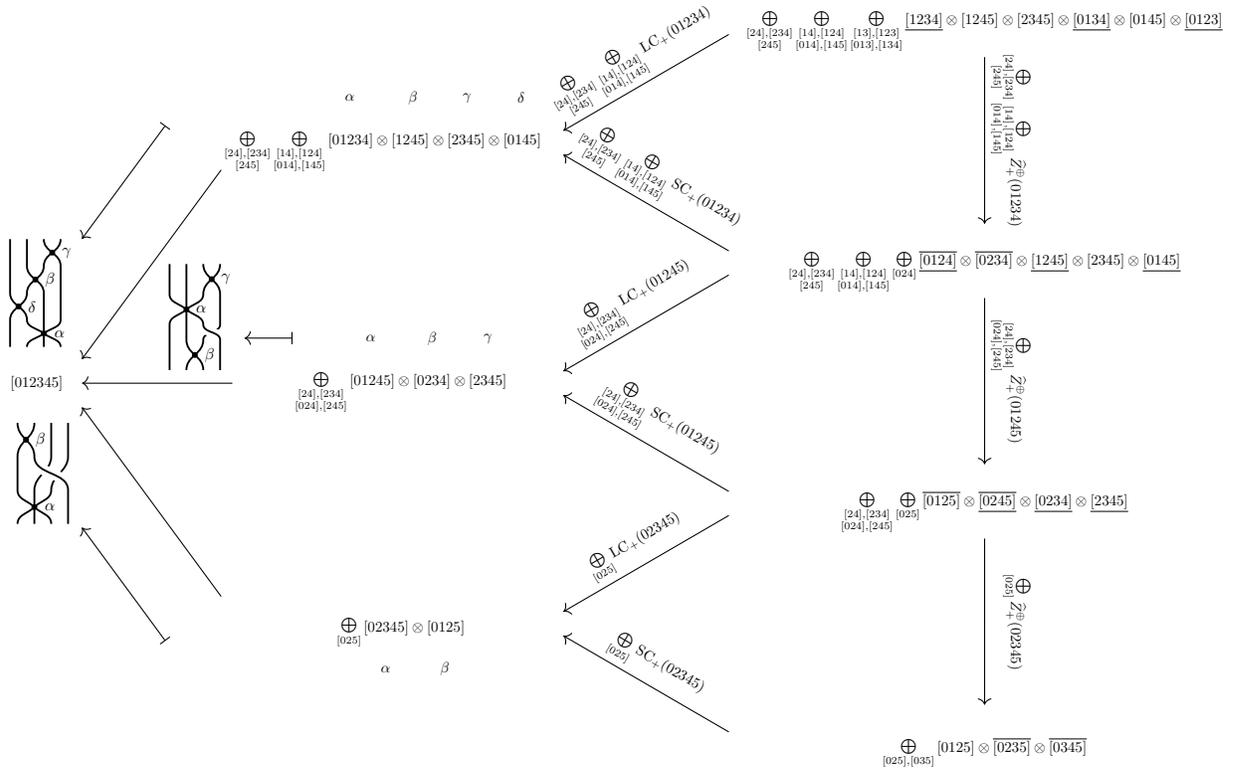

\[\hspace{-1cm}
\begin{tz}[scale=\scl]
\node[scale=\nodescale](rightnodea) at (\RN, \Rha){
		$
		\bigoplus\limits_{\substack{\Xs{24},\fs{234}\\\fs{245}}}	
		\bigoplus\limits_{\substack{\Xs{14},\fs{124}\\ \fs{014},\fs{145}}}
		\bigoplus\limits_{\substack{\Xs{13}, \fs{123}\\\fs{013},\fs{134}}}
		\rb{
		\underline{\V{1234}}\otimes \V{1245}\otimes 				
		\V{2345}\otimes \underline{\V{0134}}\otimes \V{0145}\otimes\underline{\V{0123}}
		}$};
\node[scale=\nodescale](rightnodeb) at (\RN, \Rhb){
		$\bigoplus\limits_{\substack{\Xs{24},\fs{234}\\\fs{245}}}
		\bigoplus\limits_{\substack{\Xs{14},\fs{124}\\ \fs{014},\fs{145}}}
		\bigoplus\limits_{\fs{024}}
		\rb{
		\overline{\underline{\V{0124}}}\otimes \overline{\V{0234}} \otimes \underline{\V{1245}} \otimes \V{2345} \otimes \underline{\V{0145}} 
		}$};
\node[scale=\nodescale] (rightnodec) at (\RN, \Rhc){
		$\bigoplus\limits_{\substack{\Xs{24},\fs{234}\\\fs{024},\fs{245}}}\bigoplus\limits_{\fs{025}} 
		\rb{
		\overline{\V{0125}} \otimes 
			\underline{\overline{\V{0245}}} \otimes \underline{\V{0234}} \otimes 
			\underline{\V{2345}}
			}$};
\node[scale=\nodescale] (rightnoded) at (\RN, \Rhd){
		$\bigoplus\limits_{\fs{025},\fs{035}} 
		\rb{
		\V{0125} \otimes 
			\overline{\V{0235}} \otimes \overline{ \V{0345}}
			}$};
\node[scale=\nodescale] (middlenoda) at (\MN, \Mha){
		$		
		\bigoplus\limits_{\substack{\Xs{24},\fs{234}\\\fs{245}}}	
		\bigoplus\limits_{\substack{\Xs{14},\fs{124}\\ \fs{014},\fs{145}}}
		\rb{\V{01234} \otimes \V{1245} \otimes \V{2345} \otimes \V{0145}}$};
\node[scale=\nodescale] (middlenodb) at ([xshift=\middlerightshiftb]\MN, \Mhb){
		$
		\bigoplus\limits_{\substack{\Xs{24},\fs{234}\\\fs{024},\fs{245}}}
		\rb{\V{01245}\otimes \V{0234} \otimes \V{2345} }$};
\node[scale=\nodescale] (middlenodc) at ([xshift=\middlerightshiftc]\MN, \Mhc){
		$\bigoplus\limits_{\fs{025}}\rb{\V{02345}\otimes \V{0125}}$};
\node[scale=\nodescale] (leftnode) at ([yshift=\arrowshift]\LN, \Mhb){
		$\V{012345}$};
\draw[diagramarrow] ([yshift=-\arrowdist]rightnodea) to node[downarrowlabel] {$
		\bigoplus\limits_{\substack{\Xs{24},\fs{234}\\\fs{245}}}	
		\bigoplus\limits_{\substack{\Xs{14},\fs{124}\\ \fs{014},\fs{145}}}
		\rbz{\ZH_+(01234)}$} ([yshift=\arrowdist]rightnodeb);
\draw[diagramarrow] ([yshift=-\arrowdist]rightnodeb) to node[downarrowlabel] {$\bigoplus\limits_{\substack{\Xs{24},\fs{234}\\\fs{024},\fs{245}}} \rbz{\ZH_+(01245)}$} ([yshift=\arrowdist]rightnodec);
\draw[diagramarrow] ([yshift=-\arrowdist]rightnodec) to node[downarrowlabel] {$\bigoplus\limits_{\substack{\fs{025}\\\vphantom{\fs{123}}}} 
\rbz{\ZH_+(02345)}$} ([yshift=\arrowdist]rightnoded);
\draw[diagramarrow] ([yshift=-\arrowgap+\arrowshift]\RR, \Rha) to node[arrowlabel]{$
		\bigoplus\limits_{\substack{\Xs{24},\fs{234}\\\fs{245}}}	
		\bigoplus\limits_{\substack{\Xs{14},\fs{124}\\ \fs{014},\fs{145}}}
		\rbz{\Lcomp_+(01234)}$}
		 ([yshift=\arrowgap+\arrowshift]\RL, \Mha);
\draw[diagramarrow] ([yshift=\arrowgap+\arrowshift]\RR, \Rhb) to node[arrowlabel] {$
		\bigoplus\limits_{\substack{\Xs{24},\fs{234}\\\fs{245}}}	
		\bigoplus\limits_{\substack{\Xs{14},\fs{124}\\ \fs{014},\fs{145}}}
		\rbz{\Scomp_+(01234)}$} 
		([yshift=-\arrowgap+\arrowshift]\RL, \Mha);
\draw[diagramarrow] ([yshift=-\arrowgap+\arrowshift]\RR, \Rhb) to node[arrowlabel] {$
	\bigoplus\limits_{\substack{\Xs{24},\fs{234}\\\fs{024},\fs{245}}}
	\rbz{\Lcomp_+(01245)}$} 
	([yshift=\arrowgap+\arrowshift]\RL, \Mhb);
\draw[diagramarrow] ([yshift=\arrowgap+\arrowshift]\RR, \Rhc) to node[arrowlabel] {$
	\bigoplus\limits_{\substack{\Xs{24},\fs{234}\\\fs{024},\fs{245}}}
	\rbz{\Scomp_+(01245)}$}  ([yshift=-\arrowgap+\arrowshift]\RL, \Mhb);
\draw[diagramarrow] ([yshift=-\arrowgap+\arrowshift]\RR, \Rhc) to node[arrowlabel] {$\bigoplus\limits_{\fs{025}} 
\rb{\Lcomp_+(02345)}$} ([yshift=\arrowgap+\arrowshift]\RL, \Mhc);
\draw[diagramarrow] ([yshift=\arrowgap+\arrowshift]\RR, \Rhd) to node[arrowlabel] {$
	\bigoplus\limits_{\fs{025}} 
	\rb{\Scomp_+(02345)}$} ([yshift=-\arrowgap+\arrowshift]\RL, \Mhc);
\draw[diagramarrow,shorten <=0.25cm] ([yshift=-\arrowgap+\arrowshift]\LR, \Mha) to ([yshift=2*\arrowgap+\arrowshift]\LL, \Mhb);
\draw[diagramarrow] ([yshift=\arrowshift]\LR, \Mhb) to ([yshift=\arrowshift]\LL, \Mhb);
\draw[diagramarrow,shorten <=0.25cm] ([yshift=\arrowgap+\arrowshift]\LR, \Mhc) to ([yshift=-2*\arrowgap+\arrowshift]\LL, \Mhb);
\node (Pict) at ([yshift=\arrowshift]-0.15,7.5) {$
	\begin{tz}[scale=\nodescale,scale=0.4, yscale=-0.8]
		\draw[wire] (2,0) to [out=up, in=\dl] (2.5,1) to [out=\ul,in= down] (2,2)  to [out=up, in=\dr] (1.5, 3) to [out=\ur, in=down] (2,4) to (2,8);
		\coordinate (A) at (2.5,1);
		\coordinate (B) at (1.5,3);
		\coordinate (C) at (0.5,5);
		\coordinate (D) at (2, 7);
		\draw[wire] (3,0) to [out=up, in=\dr] (2.5,1) to [out=\ur, in=down] (3,2) to (3,6) to [out=up, in=\dr] (2,7) to [out=\ur, in=down] (3,8);
		\draw[wire] (1,0) to (1,2) to [out=up, in=\dl] (1.5,3) to [out=\ul, in=down] (1,4)  to [out=up, in=\dr]  (0.5,5) to [out=\ur, in=down] (1,6) to [out=up, in=\dl] (2,7) to [out=\ul, in=down] (1,8);
		\draw[wire] (0,0) to (0,4) to [out=up, in=\dl] (0.5,5) to [out=\ul, in=down] (0,6) to (0,8);
		\node[dot] at (A){};
		\node[dot] at (B){};
		\node[dot] at (C){};
		\node[dot] at (D){};
		\node[tmor, right] at ([xshift=0.25cm]A) {$\gamma$};
		\node[tmor, right] at ([xshift=0.25cm]B) {$\beta$};
		\node[tmor, right] at ([xshift=0.25cm]C) {$\delta$};
		\node[tmor, right] at ([xshift=0.3cm]D) {$\alpha$};
	\end{tz}
	$};
\node (Picb) at ([yshift=\arrowshift]-0.15,4.5) {$
	\begin{tz}[scale=\nodescale,scale=0.4,yscale=-1]
		\draw[wire] (0,0) to [out=up, in=\dl] (0.5,1) to [out=\ul, in=down] (0,2) to (0,4) to [out=up, in=\dl] (1,5) to [out=\ul, in=down] (0,6);
		\draw[wire] (1,0) to [out=up, in=\dr] (0.5,1) to [out=\ur, in=down] (1,2) to [out=up, in=down]node[mask point, pos=0.5,scale=2.5] (MP){} (3,4) to (3,6);
		\cliparoundone{MP}{
		\draw[wire] (3,0) to (3,2) to [out=up, in=down] (2,4)to [out=up, in=\dr] (1,5) to [out=\ur, in=down] (2,6);
		}
		\cliparoundone{MP}{
		\draw[wire] (2,0) to (2,2) to [out=up, in=down]  (1,4) to (1,6);
		}
		\coordinate (A) at  (0.5,1) ;
		\coordinate (B) at (1,5) ;
		\node[dot] at (A){};
		\node[dot] at (B){};
		\node[tmor, right] at ([xshift=0.25cm]A) {$\beta$};
		\node[tmor, right] at ([xshift=0.3cm]B) {$\alpha$};
	\end{tz}
	$};
\node (Picc) at ([yshift=1.1cm+\arrowshift]2.5,\Mhb) {$
	\begin{tz}[scale=\nodescale,scale=0.4,yscale=-0.9]
		\draw[wire] (2,0) to [out=up, in=\dl] (2.5,1) to [out=\ul, in=down] (2,2) to [out=up, in=\dr] (1,3) to [out=\ur, in=down] (2,4) to [out=up, in=down] node[mask point, pos=0.5, scale=2.5] (MP){} (3,5) to (3,7);
		\cliparoundone{MP}{
		\draw[wire] (3,0) to [out=up, in=\dr] (2.5,1) to [out=\ur, in=down] (3,2)  to (3,4) to [out=up, in=down] (2,5)  to [out=up, in=\dr] (1.5,6) to [out=\ur, in=down] (2,7);
		}
		\draw[wire] (0,0) to (0,2) to [out=up, in=\dl] (1,3) to [out=\ul, in=down] (0,4) to (0,7);
		\draw[wire] (1,0) to (1,5) to [out=up, in=\dl] (1.5,6) to [out=\ul, in=down] (1,7);
		\coordinate (A) at (2.5,1);
		\coordinate (B) at (1,3);
		\coordinate (C) at (1.5, 6);
		\node[dot] at (A){};
		\node[dot] at (B){};
		\node[dot] at (C){};
		\node[tmor, right] at ([xshift=0.25cm]A) {$\gamma$};
		\node[tmor, right] at ([xshift=0.25cm]B) {$\alpha$};
		\node[tmor, right] at ([xshift=0.25cm]C) {$\beta$};
	\end{tz}
	$};
\node[scale=\nodescale] at ([yshift=0.75cm+\arrowshift] 4.95, \Mha) {$\alpha$};
\node[scale=\nodescale] at ([yshift=0.75cm+\arrowshift] 6, \Mha) {$\beta$};
\node[scale=\nodescale] at ([yshift=0.75cm+\arrowshift] 6.9, \Mha) {$\gamma$};
\node[scale=\nodescale] at ([yshift=0.75cm+\arrowshift] 7.8, \Mha) {$\delta$};
\node[scale=\nodescale] at ([yshift=0.75cm+\arrowshift] 5.3, \Mhb) {$\alpha$};
\node[scale=\nodescale] at ([yshift=0.75cm+\arrowshift] 6.32, \Mhb) {$\beta$};
\node[scale=\nodescale] at ([yshift=0.75cm+\arrowshift] 7.25, \Mhb) {$\gamma$};
\node[scale=\nodescale] at ([yshift=-0.75cm+\arrowshift] 5.55, \Mhc) {$\alpha$};
\node[scale=\nodescale] at ([yshift=-0.75cm+\arrowshift] 6.53, \Mhc) {$\beta$};
\draw[diagramarrow, shorten <=\shortenmapsto, shorten >=\shortenmapsfrom, |->] ([yshift=\distmapsto, yshift=-\arrowgap+\arrowshift]\LR, \Mha) to ([yshift=\distmapsto+2*\arrowgap+\arrowshift] \LL, \Mhb);
\draw[diagramarrow,shorten <=\shortenmapsto,shorten >= \shortenmapsfrom, |->] ([yshift=-\distmapsto, yshift=\arrowgap+\arrowshift]\LR, \Mhc) to ([yshift=-\distmapsto-2*\arrowgap+\arrowshift] \LL, \Mhb);
\draw[diagramarrow, |->] ([yshift=0.75cm+\arrowshift] 4,\Mhb) to ([yshift=0.75cm+\arrowshift] 3.2, \Mhb);
\end{tz}
\]
\caption{The right hand side of the equation in the proof of Lemma~\ref{lem:Pachner33}.}
\label{fig:bigcommutativediagram2}
\end{figure}%
\def\RL{6.5}
\def\MN{5.75}
\def\distmapsto{1cm}
\def\shortenmapsto{0.6cm}
\def\arrowtipshift{0.6cm}
\begin{figure}
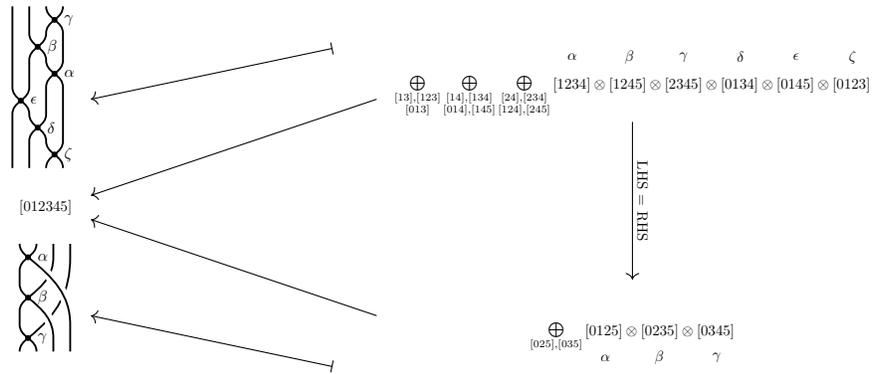

\[
\begin{tz}[scale=\scl]
\node[scale=\nodescale](rightnodea) at (\RN, \Rha){
		$		\bigoplus\limits_{\substack{\Xs{13},\fs{123}\\\fs{013}}}	
		\bigoplus\limits_{\substack{\Xs{14},\fs{134}\\ \fs{014},\fs{145}}}
		\bigoplus\limits_{\substack{\Xs{24}, \fs{234}\\\fs{124},\fs{245}}}
		\rb{
		\V{1234}\otimes \V{1245}\otimes 				
		\V{2345}\otimes \V{0134}\otimes \V{0145}\otimes\V{0123}
		}$
};
\node[scale=\nodescale] (rightnodeb) at (\RN, \Rhb){
		$\bigoplus\limits_{\fs{025},\fs{035}} 
		\rb{
		\V{0125} \otimes 
			\V{0235} \otimes \V{0345}
			}$
};
\node[scale=\nodescale] (middlenodb) at ([yshift=+\arrowshift]\MN, \Mha){
		$\V{012345}$};
\draw[diagramarrow] ([yshift=-\arrowdist+\arrowshift]rightnodea) to node[downarrowlabel] {LHS = RHS} ([yshift=\arrowdist+\arrowshift]rightnodeb);
\draw[diagramarrow] ([yshift=-\arrowgap+\arrowshift]\RR, \Rha) to node[arrowlabel]{} ([yshift=\arrowgap+\arrowshift]\RL, \Mha);
\draw[diagramarrow] ([yshift=\arrowgap+\arrowshift]\RR, \Rhb) to node[arrowlabel] {} ([yshift=-\arrowgap+\arrowshift]\RL, \Mha);
\draw[diagramarrow,shorten <=\shortenmapsto, |->] ([yshift=\distmapsto, yshift=-\arrowgap+\arrowshift]\RR, \Rha) to ([yshift=\distmapsto+\arrowgap+\arrowtipshift+\arrowshift] \RL, \Mha);
\draw[diagramarrow,shorten <=\shortenmapsto, |->] ([yshift=-\distmapsto, yshift=\arrowgap+\arrowshift]\RR, \Rhb) to ([yshift=-\distmapsto-\arrowgap-\arrowtipshift+\arrowshift] \RL, \Mha);
\node[scale=\nodescale] at ([yshift=0.5cm+\arrowshift] 14.5, \Rha) {$\alpha$};
\node[scale=\nodescale] at ([yshift=0.5cm+\arrowshift] 15.45, \Rha) {$\beta$};
\node[scale=\nodescale] at ([yshift=0.5cm+\arrowshift] 16.35, \Rha) {$\gamma$};
\node[scale=\nodescale] at ([yshift=0.5cm+\arrowshift] 17.3, \Rha) {$\delta$};
\node[scale=\nodescale] at ([yshift=0.5cm+\arrowshift] 18.22, \Rha) {$\epsilon$};
\node[scale=\nodescale] at ([yshift=0.5cm+\arrowshift] 19.16, \Rha) {$\zeta$};
\node[scale=\nodescale] at ([yshift=-0.5cm+\arrowshift] 15.05, \Rhb) {$\alpha$};
\node[scale=\nodescale] at ([yshift=-0.5cm+\arrowshift] 15.95, \Rhb) {$\beta$};
\node[scale=\nodescale] at ([yshift=-0.5cm+\arrowshift] 16.9, \Rhb) {$\gamma$};
\node (Pict) at ([yshift=\arrowshift]5.75,8.5) {$
	\begin{tz}[scale=\nodescale,scale=0.4, yscale=0.8]
		\draw[wire] (0,0) to [out=up, in=\dl] (0.5,1) to [out=\ul, in=down] (0,2) to (0,3) to [out=up, in=\dl] (0.5,4) to [out=\ul,in=down] (0,5) to (0,6) to [out=up, in=\dl] (0.5,7) to [out=\ul, in=down] (0,8);
		\draw[wire] (3,0) to (3,2) to [out=up, in=\dr] node[mask point, scale=2.5, pos=0.4] (MP1){} node[mask point,scale=2.5, pos=0.7] (MP2){}(0.5,7) to [out=\ur, in=down] (1,8);
		\cliparoundone{MP2}{\draw[wire] (2,0) to (2,1) to  [out=up, in=\dr] node[mask point, scale=2.5, pos=0.42] (MP3){} (0.5,4) to [out=\ur, in=down] (2,7) to (2,8);
		}
		\cliparoundtwo{MP1}{MP3}{
		\draw[wire] (1,0) to [out=up, in=\dr] (0.5,1) to [out=\ur, in=down]  (3,7) to (3,8);
		}
		\coordinate (A) at (0.5,1);
		\coordinate (B) at (0.5,4);
		\coordinate (C) at (0.5,7);
		\node[dot] at (A){};
		\node[dot] at (B){};
		\node[dot] at (C){};
		\node[tmor, right] at ([xshift=0.25cm]A) {$\gamma$};
		\node[tmor, right] at ([xshift=0.25cm]B) {$\beta$};
		\node[tmor, right] at ([xshift=0.25cm]C) {$\alpha$};
	\end{tz}
	$};
\node (Pict) at ([yshift=\arrowshift]5.75,12) {$
	\begin{tz}[scale=\nodescale,scale=0.4, yscale=0.8]
		\draw[wire] (0,0) to (0,4) to [out=up, in=\dl] (0.5,5) to [out=\ul, in=down] (0,6) to (0,12);
		\draw[wire] (1,0) to (1,2) to [out=up, in=\dl] (1.5,3) to [out=\ul, in=down] (1,4) to [out=up, in=\dr] (0.5,5) to [out=\ur, in=down] (1,6)to (1,8) to [out=up, in=\dl] (1.5,9) to [out=\ul, in=down] (1,10) to (1,12);
		\draw[wire] (2,0) to [out=up, in=\dl] (2.5,1)to [out=\ul, in=down] (2,2) to [out=up, in=\dr] (1.5,3) to [out=\ur, in=down] (2,4) to (2,6) to [out=up, in=\dl] (2.5,7)to [out=\ul, in=down] (2,8) to [out=up, in=\dr] (1.5,9) to [out=\ur, in=down] (2,10) to [out=up, in=\dl] (2.5,11) to [out=\ul, in=down] (2,12);
		\draw[wire] (3,0) to [out=up, in=\dr] (2.5,1) to [out=\ur, in=down] (3,2) to (3,6) to [out=up, in=\dr] (2.5,7) to [out=\ur, in=down] (3,8) to (3,10) to [out=up, in=\dr] (2.5,11) to [out=\ur, in=down] (3,12);
		\coordinate (A) at (2.5,1);
		\coordinate (B) at (1.5,3);
		\coordinate (C) at (0.5,5);
		\coordinate (D) at (2.5,7);
		\coordinate (E) at (1.5,9);
		\coordinate(F) at (2.5,11);
		\node[dot] at (A){};
		\node[dot] at (B){};
		\node[dot] at (C){};
		\node[dot] at (D){};
		\node[dot] at (E){};
		\node[dot] at (F){};
		\node[tmor, right] at ([xshift=0.25cm]A) {$\zeta$};
		\node[tmor, right] at ([xshift=0.25cm]B) {$\delta$};
		\node[tmor, right] at ([xshift=0.25cm]C) {$\epsilon$};
		\node[tmor, right] at ([xshift=0.25cm]D) {$\alpha$};
		\node[tmor, right] at ([xshift=0.25cm]E) {$\beta$};
		\node[tmor, right] at ([xshift=0.25cm]F) {$\gamma$};
	\end{tz}
	$};
\end{tz}
\]
\caption{Comparing the left and right hand side of the equation in the proof of Lemma~\ref{lem:Pachner33}.}
\label{fig:bigcommutativediagram3}
\end{figure}
Consider the diagram in Figure~\ref{fig:bigcommutativediagram1}.  Each of the triangles in the right column commutes by Corollary~\ref{cor:Z+decomp}. The squares in the left column can be shown to commute by explicitly comparing the graphical expressions of the involved maps.  A similar analysis shows that the diagram in Figure~\ref{fig:bigcommutativediagram2} is commutative.  It follows from an explicit comparison of the graphical expressions of the involved maps that the composites from the top-right to the left in Figures~\ref{fig:bigcommutativediagram1} and~\ref{fig:bigcommutativediagram2} coincide. Similarly, the composites from the bottom-right to the left coincide.
Hence, these composites fit into a commutative diagram, depicted in Figure~\ref{fig:bigcommutativediagram3}, in which the vertical map can either be the composite of the rightmost column of maps in Figure~\ref{fig:bigcommutativediagram1} or in Figure~\ref{fig:bigcommutativediagram2}. The bottom map of Figure~\ref{fig:bigcommutativediagram3} is a composite of invertible linear maps (which again follows from local semisimplicity, analogously to the proof of Lemma~\ref{lem:scompinv}) and is therefore invertible. Thus, the composites of the rightmost column of maps in Figure~\ref{fig:bigcommutativediagram1} and~\ref{fig:bigcommutativediagram2} coincide, proving Lemma~\ref{lem:Pachner33}.
\end{proof}
\begin{remark}[$(3,3)$-bistellar move as nonabelian 4-cocycle relation]
The $(3,3)$-bistellar move essentially encodes the higher associativity equation fulfilled by the pentagonator of a monoidal $2$-category. Accordingly, the proof of Lemma~\ref{lem:Pachner33} proceeds by proving that two large graphical expressions are isotopic, making crucial use of the Gray monoid axioms (see Definition~\ref{def:monoidal2cat}) which we use here to model monoidal 2-categories.
\end{remark}
\nid
\skiptocparagraph{Invariance under the $(2,4)$-bistellar move.}
Next, we prove Lemma~\ref{lem:Pachner24}, showing invariance of the state sum under the $(2,4)$-bistellar move.
\begin{proof}[Proof of Lemma~\ref{lem:Pachner24}]
Expressed in terms of the normalized maps $\zh_+$ and $\zh_-$, the equation in Lemma~\ref{lem:Pachner24} becomes
\[
\zh_+(01235)\zh_+(01345)=
\!\!\sum_{\substack{\Xs{24}, \fs{024}, \fs{245},\\ \fs{234}, \fs{124}}}\!\! \zh_+(02345) \zh_+(01245) \zh_+(01234) \zh_-(12345).
\]
Rewritten in terms of the direct sum maps $\ZH_+$ and $\ZH_-$, and again omitting all direct sum symbols, this becomes the following equation
\begin{equation}\label{eq:pachnerdirectsum24}\tag{B2-4}
\ZH_+(01235)\ZH_+(01345)= \ZH_+(02345) \ZH_+(01245) \ZH_+(01234) \ZH_-(12345)
\end{equation}
between linear maps 
\[\begin{split}
\bigoplus_{\substack{\Xs{14},\fs{134}\\\fs{014},\fs{145}}}& \bigoplus_{\substack{\Xs{13}, \fs{123}\\\fs{013}, \fs{135}}} V^+(0134) \otimes \otimes V^+(0145) \otimes V^+(1345) \otimes V^+(0123)  \otimes V^+(1235) \\
&\longrightarrow  \bigoplus_{\fs{035}, \fs{025}} V^+(0345) \otimes V^+(0235) \otimes V^+(0125).
\end{split}
\]
Using equation~\eqref{eq:pachnersimplified}, this is equivalent to the following equation:
\[\ZH_+(01235) \ZH_+(01345) = \ZH_+(01235) \ZH_+(01345) \ZH_+(12345) \ZH_-(12345)
\]
The lemma is therefore a direct consequence of Corollary~\ref{cor:Z+Z-}.
\end{proof}

\skiptocparagraph{Invariance under the $(1,5)$-bistellar move.}
Lastly, we prove Lemma~\ref{lem:Pachner15}, showing invariance of the state sum under the $(1,5)$-bistellar move.
\begin{proof}[Proof of Lemma~\ref{lem:Pachner15}]
Expressed in terms of the normalized $\zh_+$ and $\zh_-$, the equation in Lemma~\ref{lem:Pachner15} becomes the following:
\begin{equation}\nonumber
\begin{split}
\zh_+(01235)=\dim\left(\tc{C}\right)^{-1}\!\!
\!\!&\sum_{\substack{\Xs{ij}, 0\leq i<j\leq 5\\i=4\text{ or }j=4}} \,\,\sum_{\substack{\fs{ijk}, 0\leq i<j<k\leq 5\\j=4\text{ or }k=4}}
\frac{\tdim{034}\tdim{045}\tdim{345}}{\tdim{035}\odim{04}\odim{34}\odim{45}} \\[5pt]
&\Tr_{V^+(0345)}\left(\vphantom{\frac{a}{b}}\zh_+(02345) \zh_+(01245) \zh_+(01234) \zh_-(12345)\zh_-(01345) \right)
 \end{split}
\end{equation}
Rewritten in terms of the direct sum maps $\ZH_+$ and $\ZH_-$, again omitting direct sum symbols, this becomes the equation
\[
\begin{split}
\ZH_+(01235)=\dim\left(\tc{C}\right)^{-1}\!\!
&\sum_{\Xs{04}, \Xs{34}, \Xs{45}} \sum_{\fs{034}, \fs{045}, \fs{345}}
\frac{\tdim{034}\tdim{045}\tdim{345}}{\tdim{035}\odim{04}\odim{34}\odim{45}} \\[5pt]
&\Tr_{V^+(0345)}\left(\vphantom{\frac{a}{b}}\ZH_+(02345) \ZH_+(01245) \ZH_+(01234) \ZH_-(12345)\ZH_-(01345) \right)
 \end{split}
\]
between linear maps
\[
\bigoplus_{\substack{\Xs{13},\fs{123}\\\fs{013},\fs{135}}}\bigoplus_{\fs{035}} V^+(0123) \otimes V^+(0135) \otimes V^+(1235) \longrightarrow 
\bigoplus_{\fs{025},\fs{035}} V^+(0235) \otimes V^+(0125).
\]
Using equation~\eqref{eq:pachnerdirectsum24} and Corollary~\ref{cor:Z+Z-}, this can be simplified as follows:
\begin{equation}\nonumber
\shrinker{.9}{
\ZH_+(01235)=\dim\left(\tc{C}\right)^{-1}\!\!
\sum_{\Xs{04}, \Xs{34}, \Xs{45}}\sum_{\fs{034}, \fs{045}, \fs{345}}
\frac{\tdim{034}\tdim{045}\tdim{345}}{\tdim{035}\odim{04}\odim{34}\odim{45}} \Tr_{V^+(0345)}\left(\vphantom{\frac{a}{b}}\ZH_+(01235)\right)
}
\end{equation}
Since the $3$-simplex $\lan 0345\ran$ is not in the boundary of the $4$-simplex $\lan 01235\ran$, it follows that 
\[\Tr_{V^+(0345)}\left( \ZH_+(01235)\right) = \dim(V^+(0345)) \ZH_+(01235).\]
Hence, to prove Lemma~\ref{lem:Pachner15}, it suffice to prove the following: 
\begin{equation}\nonumber
\dim\left(\tc{C}\right) = \sum_{\Xs{04},\Xs{34},\Xs{45}} \sum_{\fs{034}, \fs{045},\fs{345}} \frac{\tdim{034}\tdim{045} \tdim{345}}{\tdim{035}\odim{04} \odim{34} \odim{45}}\dim(V_+(0345))
\end{equation}
Observing that
\def\scl{0.6}%
\def\ys{0.2cm}
\def\xs{0.0cm}
\[V^+(0345)
~~~=~~~
\Hom_{\tc{C}}\left(~~
\begin{tz}[yscale=5/6,scale=\scl]
\draw[slice] (0,0) to [out=up, in=\dl] (0.5,1) to [out=up, in=\dl] (1,2) to (1,3);
\draw[slice] (1,0) to [out=up, in=\dr] (0.5,1);
\draw[slice] (2,0) to  [out=up, in=\dr] (1,2);
\coordinate (A) at (0.5,1);
\coordinate (B) at (1,2);
\node[dot] at (A){};
\node[dot] at (B) {};
\node[omor, right] at ([xshift=\xs,yshift=\ys]A) {$[034]$};
\node[omor, right] at ([xshift=\xs,yshift=\ys]B) {$[045]$};
\end{tz}
~~~,~~~
\begin{tz}[yscale=5/6,scale=\scl,xscale=-1]
\draw[slice] (0,0) to [out=up, in=\dl] (0.5,1) to [out=up, in=\dl] (1,2) to (1,3);
\draw[slice] (1,0) to [out=up, in=\dr] (0.5,1);
\draw[slice] (2,0) to  [out=up, in=\dr] (1,2);
\coordinate (A) at (0.5,1);
\coordinate (B) at (1,2);
\node[dot] at (A){};
\node[dot] at (B) {};
\node[omor, left] at ([xshift=\xs,yshift=\ys]A) {$[345]$};
\node[omor, left] at ([xshift=\xs,yshift=\ys]B) {$[035]$};
\end{tz}
~~\right)
\eqgap \iso \eqgap
\Hom_{\tc{C}}\left(~~
\begin{tz}[yscale=0.5,scale=\scl]
\draw[slice] (0,-1) to (0,0) to [out=up, in=\dl] (0.5,1) to (0.5,2) to  [out=up, in=\dl] (1,3) to (1,4);
\draw[slice] (1,-1) to (1,0) to [out=up, in=\dr] (0.5,1);
\draw[slice] (1,3) to [out=\dr, in=up] (1.5,2) to [out=down, in=down, looseness=3] (2,2) to (2,4);
\coordinate (A) at (0.5,1);
\coordinate(B) at (1,3);
\node[dot] at (A) {};
\node[dot] at (B) {};
\node[omor, right] at ([xshift=\xs,yshift=\ys]A) {$[034]$};
\node[omor, right] at ([xshift=\xs,yshift=\ys]B) {$[045]$};
\end{tz}
~~~,~~~
\begin{tz}[yscale=0.5,scale=\scl]
\draw[slice] (0,-1) to (0,0) to  [out=up, in=\dl] (0.5,1) to [out=\dr, in=up] (1,0) to [out=down, in=down, looseness=3] (1.5,0) to (1.5,4);
\draw[slice] (0.5,1) to (0.5,2) to [out=up, in=\dr] (0,3) to (0,4);
\draw[slice] (-0.5,-1) to (-0.5,2) to [out=up, in=\dl] (0,3);
\coordinate(A) at (0.5,1);
\coordinate(B) at (0,3);
\node[dot] at (A){};
\node[dot] at (B){};
\node[omor, right] at ([xshift=\xs,yshift=\ys]A) {$[345]$};
\node[omor, right] at ([xshift=\xs,yshift=\ys]B) {$[035]$};
\end{tz}
~~\right),
\]
it therefore follows from Corollary~\ref{cor:formulaHom}, that 
\[\sum_{\Xs{04}, \fs{034}, \fs{045}}\dim(V_+(0345))\frac{\tdim{034} \tdim{045}}{\odim{04}} = \tdim{0(345)} .
\]
Since $\Xs{35}$ is a simple object, it follows from Proposition~\ref{prop:factoringthroughsimple} that
\[\tdim{0(345)} = \frac{\tdim{345} \tdim{035}}{\dim(\Xs{35})}.
\]
The desired equation therefore simplifies to
\[\dim\left(\tc{C}\right) = \sum_{\Xs{34}, \Xs{45}, \fs{345}}\frac{\tdim{345}^2}{\dim(\Xs{35}) \odim{34} \odim{45}}.
\]
This equality is proven in Corollary~\ref{cor:formulaglobaldim}.
%
\ignore{
Using Lemma~\ref{lem:Pachner24}, we can reexpress~\eqref{eq:Pachner15} as follows:
\begin{equation}\nonumber
\begin{split}Z_+(01235) = \dim\left(\tc{C}\right)^{-1} \sum_{\substack{\Xs{04},\Xs{14},\\\Xs{34},\Xs{45}}}\,\, \sum_{\substack{\fs{034}, \fs{014},\fs{134},\\\fs{045},\fs{145},\fs{345}}}\frac{\tdim{034}\tdim{014}\tdim{134} \tdim{045} \tdim{145} \tdim{345}}{\odim{04} \odim{14} \odim{34} \odim{45}}
\\ \Tr_{V^+(0345)}\left(\vphantom{\frac{a}{b}} Z_+(01235)Z_+(01345)Z_-(01345)\right)
\end{split}
\end{equation}
Written in terms of $\zh_+$ and $\zh_-$ (see~\eqref{eq:c+-}), this can be reexpressed as follows:
\[\zh_+(01235) = \dim\left(\tc{C}\right)^{-1}\!\! \sum_{\substack{\Xs{04},\Xs{14},\\\Xs{34},\Xs{45}}}\,\, \sum_{\substack{\fs{034}, \fs{014},\fs{134},\\\fs{045},\fs{145},\fs{345}}}\frac{\tdim{034}\tdim{045} \tdim{345}}{\tdim{035}\odim{04} \odim{34} \odim{45}}\Tr_{V_+(0345)}\left(\zh_+(01235) \zh_+(01345) \zh_-(01345)\right)
\]
Rewritten in terms of the direct sums $\ZH_+$ and $\ZH_-$, and noting that the $4$-simplex $[01235]$ does not contain the $3$-simplex $[0345]$ in its boundary, it suffices to prove the following equation:
\[\dim\left(\tc{C}\right)=  \sum_{\Xs{04},\Xs{34},\Xs{45}} \sum_{\fs{034}, \fs{045},\fs{345}} \frac{\tdim{034}\tdim{045} \tdim{345}}{\tdim{035}\odim{04} \odim{34} \odim{45}}\Tr_{V_+(0345)}\left( \ZH_+(01345) \ZH_-(01345)\right)
\]
Using equation~\eqref{eq:C+C-}, this equation becomes the following:
\begin{equation}\label{eq:proof15}
\dim\left(\tc{C}\right) = \sum_{\Xs{04},\Xs{34},\Xs{45}} \sum_{\fs{034}, \fs{045},\fs{345}} \frac{\tdim{034}\tdim{045} \tdim{345}}{\tdim{035}\odim{04} \odim{34} \odim{45}}\dim(V_+(0345))
\end{equation}
Observing that
\def\scl{0.6}%
\def\ys{0.2cm}
\def\xs{0.0cm}
\[V^+(0345)
~~~=~~~
\Hom_{\tc{C}}\left(~~
\begin{tz}[yscale=5/6,scale=\scl]
\draw[slice] (0,0) to [out=up, in=\dl] (0.5,1) to [out=up, in=\dl] (1,2) to (1,3);
\draw[slice] (1,0) to [out=up, in=\dr] (0.5,1);
\draw[slice] (2,0) to  [out=up, in=\dr] (1,2);
\coordinate (A) at (0.5,1);
\coordinate (B) at (1,2);
\node[dot] at (A){};
\node[dot] at (B) {};
\node[omor, right] at ([xshift=\xs,yshift=\ys]A) {$[034]$};
\node[omor, right] at ([xshift=\xs,yshift=\ys]B) {$[045]$};
\end{tz}
~~~,~~~
\begin{tz}[yscale=5/6,scale=\scl,xscale=-1]
\draw[slice] (0,0) to [out=up, in=\dl] (0.5,1) to [out=up, in=\dl] (1,2) to (1,3);
\draw[slice] (1,0) to [out=up, in=\dr] (0.5,1);
\draw[slice] (2,0) to  [out=up, in=\dr] (1,2);
\coordinate (A) at (0.5,1);
\coordinate (B) at (1,2);
\node[dot] at (A){};
\node[dot] at (B) {};
\node[omor, left] at ([xshift=\xs,yshift=\ys]A) {$[345]$};
\node[omor, left] at ([xshift=\xs,yshift=\ys]B) {$[035]$};
\end{tz}
~~\right)
\eqgap \iso \eqgap
\Hom_{\tc{C}}\left(~~
\begin{tz}[yscale=0.5,scale=\scl]
\draw[slice] (0,-1) to (0,0) to [out=up, in=\dl] (0.5,1) to (0.5,2) to  [out=up, in=\dl] (1,3) to (1,4);
\draw[slice] (1,-1) to (1,0) to [out=up, in=\dr] (0.5,1);
\draw[slice] (1,3) to [out=\dr, in=up] (1.5,2) to [out=down, in=down, looseness=3] (2,2) to (2,4);
\coordinate (A) at (0.5,1);
\coordinate(B) at (1,3);
\node[dot] at (A) {};
\node[dot] at (B) {};
\node[omor, right] at ([xshift=\xs,yshift=\ys]A) {$[034]$};
\node[omor, right] at ([xshift=\xs,yshift=\ys]B) {$[045]$};
\end{tz}
~~~,~~~
\begin{tz}[yscale=0.5,scale=\scl]
\draw[slice] (0,-1) to (0,0) to  [out=up, in=\dl] (0.5,1) to [out=\dr, in=up] (1,0) to [out=down, in=down, looseness=3] (1.5,0) to (1.5,4);
\draw[slice] (0.5,1) to (0.5,2) to [out=up, in=\dr] (0,3) to (0,4);
\draw[slice] (-0.5,-1) to (-0.5,2) to [out=up, in=\dl] (0,3);
\coordinate(A) at (0.5,1);
\coordinate(B) at (0,3);
\node[dot] at (A){};
\node[dot] at (B){};
\node[omor, right] at ([xshift=\xs,yshift=\ys]A) {$[345]$};
\node[omor, right] at ([xshift=\xs,yshift=\ys]B) {$[035]$};
\end{tz}
~~\right)
\]
it therefore follows from Corollary~\ref{cor:formulaHom}, that 
\[\sum_{\Xs{04}, \fs{034}, \fs{045}}\dim(V_+(0345))\frac{\tdim{034} \tdim{045}}{\odim{04}} = \tdim{0(345)} 
\]
Since $\Xs{35}$ is simple, it follows from Proposition~\ref{prop:factoringthroughsimple} that 
\[\tdim{0(345)} = \frac{\tdim{345} \tdim{035}}{\dim(\Xs{35})}
\]
Thus, the right hand side of equation~\eqref{eq:proof15} becomes
\[\sum_{\Xs{34}, \Xs{45}, \fs{345}}\frac{\tdim{345}^2}{\dim(\Xs{35}) \odim{34} \odim{45}}
\]
This expression is proven in Corollary~\ref{cor:formulaglobaldim} to equal $\dim\left(\tc{C}\right)$.
}%
\end{proof}
\nid
These lemmas can now be assembled into a proof of Lemma~\ref{lem:invariancebistellar}, proving that the state sum is invariant under bistellar moves, and hence an invariant of singular piecewise linear $4$-manifolds. 
\begin{proof}[Proof of Lemma~\ref{lem:invariancebistellar}] Combining Lemmas~\ref{lem:Pachner15},~\ref{lem:Pachner24} and~\ref{lem:Pachner33} proves Lemma~\ref{lem:invariancebistellar}.
\end{proof}

\newpage

\appendix

\addtocontents{toc}{\protect\vspace{8pt}}

\section{Separable monads and idempotent completion of $2$-categories} \label{app:ic}

\addtocontents{toc}{\SkipTocEntry}
\subsection{Monads and their bimodules}

Recall that a \emph{monad} in a $2$-category is a $1$-morphism $P:A\to A$, together with $2$-morphisms  $P\xo P \To[m] P$ (the `multiplication') and $\Io_A\To[u] P$ (the `unit') fulfilling the following equations:
\[\left(P\xo P \xo P \To[m\xo P] P \xo P \To[m] P \right) = \left( P\xo P \xo P\To[P \xo m] P\xo P \To[m] P\right)
\]\[
\left(P \To[u\xo P] P \xo P \To[m] P \right) =\left( P \To[\It_P] P\right) =  \left( P\To[P \xo u] P\xo P \To[m] P\right)
\]
Given monads $(A\to[P] A, P\xo P\To[m_P] P, \Io_A\To[u_P] P)$ and $(B\to[Q] B, Q\xo Q\To[m_Q] Q,  \Io_B\To[u_Q] Q)$, recall that a $Q$--$P$-\emph{bimodule} ${}_QM_P:A\to B$ is a $1$-morphism $M:A\to B$ together with a $2$-morphism $\rho: Q\xo M \xo P\To M$ (the `action') fulfilling the following equations:
\shrinkalign{.9}{
\begin{align*}
\left(Q\xo Q \xo M \xo P \xo P \To[m_Q \xo M\xo m_P] Q\xo M \xo P \To[\rho] M \right)&= \left(Q\xo Q\xo M \xo P \xo P\To[Q\xo \rho \xo P] Q\xo M \xo P \To[\rho] M\right)
\\
\left(M\To[u_Q\xo M \xo u_P] Q\xo M \xo P\To[\rho] M\right) &=\left( M\To[\It_M] M\right)
\end{align*}
}
A \emph{bimodule map} ${}_QM_P\To {}_QN_P$ is a $2$-morphism $f: M\To N$ intertwining the action:
\[\left(Q\xo M \xo P \To [\rho] M \To[f]M \right) = \left( Q\xo M \xo P \To[Q\xo f\xo P] Q\xo M \xo P \To[\rho] M\right)
\]

\addtocontents{toc}{\SkipTocEntry}
\subsection{Eilenberg--Moore and Kleisli morphisms and their splittings} \label{sec:em}

Given a monad $P:A\to A$ in a $2$-category $\tc{C}$, we write $\LMod_P(X)$ for the 1-category of left $P$-modules with domain $X$: its objects are pairs $(t:X\to A, \rho: P\xo t \To t)$ of $1$-morphisms $t$ carrying a left module structure $\rho$ of $P$, and its morphisms are $2$-morphisms $t\To t'$ in $\tc{C}$ intertwining the action of $P$. An \emph{Eilenberg--Moore morphism} of $P$ is a $1$-morphism $R:A^P \to A$ together with a left $P$-module structure $\rho: P\xo R\To R$, such that for every object $X$ of $\tc{C}$, the induced functor 
\[
R\xo -: \Hom_{\tc{C}}(X, A^P) \to \LMod_P(X)
\]
is an equivalence. In other words, an Eilenberg--Moore morphism is a universal left $P$-module. Analogously, a \emph{Kleisli morphisms} $L:A\to A_P$ is a universal right $P$-module, that is a right $P$-module such that for every object $X$ of $\tc{C}$, the induced functor
\[
- \xo L: \Hom_{\tc{C}}(A_P,X) \to \RMod_P(X)
\]
is an equivalence. The objects $A^P$ and $A_P$ are often known as Eilenberg--Moore and Kleisli objects, respectively. 

A \emph{splitting} of a monad $P:A \to A$ in a $2$-category $\tc{C}$ is an adjunction $R \vdash L:A \to B $ together with an isomorphism of monads $\psi:R\xo L \To P$. A splitting $(R\vdash L, \psi)$ is an \emph{Eilenberg--Moore splitting} if $R:B\to A$ together with the action of $P$ on $R$ induced by $\psi$ is an Eilenberg--Moore morphism of $P$. Note that any Eilenberg--Moore morphism $R:A^P \to A$ of a monad $P$ admits a left adjoint $L:A\to A^P$ and a canonical isomorphism of monads $R\xo L \iso P$, hence gives rise to an Eilenberg--Moore splitting of $P$. Analogously, a \emph{Kleisli splitting} of $P$ is a splitting $(R\vdash L, \psi)$ such that $L:A\to B$ together with the induced $P$-action on $L$ is a Kleisli morphism of $P$. 

\addtocontents{toc}{\SkipTocEntry}
\subsection{Separable monads and their splittings}

Recall that a monad $P:A \to A$ is \emph{separable} if there exists a $P$-$P$ bimodule section $\Delta:P \To P\xo P$ of the multiplication of $P$. A splitting $(R\vdash L, \psi)$ of a monad is \emph{separable} if the counit $\epsilon: L\xo R \To \Io_{B}$ admits a section.

\begin{theorem}[Separable, Eilenberg--Moore, and Kleisli splittings are equivalent] \label{thm:EMKleisliSeparable}
Let $P:A\to A$ be a separable monad in a locally idempotent complete $2$-category and let $(R\vdash L:A \to B,\psi)$ be a splitting of $P$. Then, the following are equivalent:
\begin{enumerate}
\item The splitting is separable.
\item The splitting is an Eilenberg--Moore splitting.
\item The splitting is a Kleisli splitting.
\end{enumerate}
\end{theorem}
\begin{proof}
We use the isomorphism $\psi$ to identify $P$ with $R\xo L$.

\emph{(1) $\Rightarrow$ (2)} Suppose that $(R\vdash L:A\to B, \psi)$ is a separable splitting; that is suppose that the counit $\epsilon: L \xo R\To \Io_{B}$ is split by the section $\delta:\Io_{B} \To L \xo R$. For any object $X$ of $\tc{C}$, we claim that the functor $R\xo -: \Hom_\tc{C}(X,B) \to \LMod_{R\xo L}(X)$ is an equivalence. To prove that $R\xo -$ is essentially surjective, we let $T:X\to A$ be a left $(R\xo L)$-module with action $\rho: R\xo L\xo T \To T$. Precomposing $L\xo  \rho$ with $\delta$ results in an idempotent $2$-morphism $\rho': L\xo T \To L\xo T$. Splitting this idempotent yields a $1$-morphism $S: X\to B$ and $2$-morphisms $p: L\xo T \To S$ and $i:S \To L\xo T$ such that $p\xt i = \It_{S}$ and $i\xt p = \rho'$. Now observe that the left $(R\xo L)$-module $R\xo S$ is isomorphic to $T$, as follows. Using the adjunction between $R$ and $L$, we turn $p$ into a $2$-morphism $p':T \To R\xo S$ and define the $2$-morphism $i': R\xo S\To[R\xo i] R\xo L\xo T \To[\rho] T$. Direct computation shows that $i'$ and $p'$ are inverse $R\xo L$-module maps. Full faithfulness of $R\xo -$ can be proven from the existence of $\delta: \Io_{B} \To L \xo R$ with $\epsilon\xt  \delta = \It_{\Io_{B}}$.

\emph{(2) $\Rightarrow$ (1)} 
Separability of $P$ implies that there is an $(R\xo L)$--$(R \xo L)$-bimodule $2$-morphism $\Delta: R\xo L \To R\xo L\xo R\xo L$ that splits the multiplication of the monad $R\xo L$. Since $R$ is an Eilenberg--Moore morphism of $R\xo L$ and $\Delta$ is a left $R\xo L$-module map, it follows that there is a $2$-morphism $f:L\To L\xo R\xo L$ such that $\Delta = R\xo f$. Using the adjunction between $R$ and $L$, we obtain a $2$-morphism $f': R\To R\xo L\xo R$. The fact that $\Delta$ is a right $(R\xo L)$-module, combined with the faithfulness of $R\xo -$, implies that $f'$ is a left $(R\xo L)$-module map. Therefore, there is a $2$-morphism $\delta: \Io_{B} \To L\xo R$ such that $f' = R \xo \delta$ and thus $\Delta = R\xo \delta \xo L$.  The fact that $\Delta$ splits the multiplication of $P$ then implies that $\delta$ splits the counit, i.e.\ $\epsilon\xt \delta = \It_{B}$.

\emph{(1) $\Leftrightarrow$ (3)} Note that a splitting $(R\vdash L, \psi)$ is separable in $\tc{C}$ if and only if the splitting $(L\vdash R,\psi)$ is separable in $\tc{C}^{\op}$. (By $\tc{C}^{\op}$ we mean $\tc{C}$ with reversed direction of $1$-morphisms, but not $2$-morphisms.) Applying $(1)\Leftrightarrow (2)$ to $\tc{C}^{\op}$ proves $(1)\Leftrightarrow (3)$.
\end{proof}
\nid In particular, Eilenberg--Moore objects and Kleisli objects of separable monads in locally idempotent complete $2$-categories coincide. 
\begin{corollary}[Eilenberg--Moore and Kleisli morphisms are adjoint] 
Given a separable monad $P$ in a locally idempotent complete $2$-category, a $1$-morphism $R:A^P \to A$ is an Eilenberg--Moore morphism of $P$ if and only if it has a left adjoint $L:A\to A^P$ that is a Kleisli morphism of $P$.
\end{corollary}

\begin{remark}[Monadicity theorem for separable monads] 
Recall that a $1$-morphism $R:A\to B$ in a $2$-category is called monadic if it has a left adjoint $L$ and is an Eilenberg--Moore morphism for the induced monad $R\xo L$. Analogously, a $1$-morphism $R$ is separable monadic if it is monadic and if moreover the monad $R\xo L$ is separable. From this perspective, Theorem~\ref{thm:EMKleisliSeparable} can be understood as a monadicity theorem for separable monads: a $1$-morphism $R$ is separable monadic if and only if it has a left adjoint such that the counit of the adjunction admits a section. A similar result specifically in the $2$-category of categories appears in~\cite{separable}.
\end{remark}

\addtocontents{toc}{\SkipTocEntry}
\subsection{The relative composition of bimodules over a separable monad}

\begin{definition}[Relative composition of modules]
Let $M_P:B\to C$ be a right module and let ${}_PN:A\to B$ be a left module of a separable monad $P:B\to B$ in a locally idempotent complete $2$-category $\tc{C}$. We define their \emph{relative composition} $M\xo_P N:A\to C$ as the $1$-morphism obtained from splitting the idempotent 2-morphism
\[M\xo N \To[M\xo u \xo N] M \xo P \xo N \To[M\xo \Delta \xo N] M \xo P \xo P \xo N \To[\rho_M \xo \rho_N] M \xo N, \]
Here $u: \Io_{B} \To P$ is the unit of $P$, and $\rho_M: M \xo P \To M$ and $\rho_N: P\xo N \To N$ denote the action of $P$, and $\Delta:P \To P\xo P$ is a $P$-$P$ bimodule section of the multiplication of $P$. 
\end{definition}
\nid Note that the resulting $1$-morphism $M\xo_P N$ is, up to isomorphism, independent of the choice of section $\Delta$.
If ${}_QM_P$ and ${}_PN_R$ are bimodules, then the idempotent 2-morphism is a $Q$--$R$-bimodule map, inducing a $Q$--$R$-bimodule structure on the relative composition ${}_QM\xo_PN_R$.

We can reexpress the condition that a separable monad admits a separable splitting in terms of bimodule triviality as follows.  We will say that a monad $P:A \to A$ is bimodule trivial (that is, trivial up to bimodule equivalence) if there are modules ${}_PM :B\to A$ and $N_P:A\to B$ such that ${}_PM \xo N_P\iso {}_PP_P$ as bimodules and $N\xo_P M\iso \Io_B$.
\begin{prop}[Separably split is bimodule trivial] \label{prop:bimodule} 
In a locally idempotent complete $2$-category, a separable monad $P:A\to A$ admits a separable splitting if and only if it is bimodule trivial. 
\end{prop}
\begin{proof} Let ${}_P M:B \to A$ and $N_P: A \to B$ be a bimodule trivialization of $P$.  For any object $X$ of $\tc{C}$, the functor ${}_PM \xo -: \Hom_{\tc{C}}(X,B)\to \LMod_P(X)$ is an equivalence with inverse $N\xo_P - : \LMod_P(X) \to \Hom_{\tc{C}}(X,B)$. Hence, $M:B\to A$ is an Eilenberg--Moore morphism of $P$. The proposition follows from Theorem~\ref{thm:EMKleisliSeparable} and the fact that every Eilenberg--Moore morphism gives rise to an Eilenberg--Moore splitting. 
\end{proof}

\addtocontents{toc}{\SkipTocEntry}
\subsection{Idempotent completion of a $2$-category} \label{sec:appicdef}

A $2$-category is \emph{idempotent complete} if it is locally idempotent complete and if every separable monad admits a separable splitting. By Theorem~\ref{thm:EMKleisliSeparable} this is equivalent to requiring that every separable monad admits an Eilenberg--Moore or Kleisli object.

\begin{definition}[Idempotent completion of a 2-category] \label{def:idempotentcompletion} 
Let $\tc{C}$ be a locally idempotent complete $2$-category. Its \emph{idempotent completion} $\tc{C}^\idm$ is the $2$-category whose objects are separable monads in $\tc{C}$, whose $1$-morphisms are bimodules, and whose $2$-morphisms are bimodule maps. The composition of $1$-morphisms is the relative composition of bimodules. The identity $1$-morphism on a separable monad $P$ is the trivial bimodule ${}_PP_P$.
\end{definition}
\begin{remark}[Well-definition of composition in the idempotent completion] \label{rem:chosensplitting} 
As a splitting of an idempotent, the relative composite of bimodules is only defined up to isomorphism. As stated, the idempotent completion $\tc{C}^\idm$ is therefore only defined for locally idempotent complete $2$-categories $\tc{C}$ with chosen splittings of their idempotent $2$-morphisms. However, since splittings of an idempotent are unique up to a unique isomorphism, different choices of splittings give rise to equivalent completions $\tc{C}^\idm$, and (assuming the axiom of choice) we can always make such choices of splittings.
\end{remark}
\begin{prop}[The idempotent completion is idempotent complete] \label{prop:completioniscomplete} 
The idemepotent completion $\tc{C}^\idm$ of a locally idempotent complete $2$-category $\tc{C}$ is idempotent complete.
\end{prop}
\begin{proof} 
The $2$-category $\tc{C}^\idm$ is locally idempotent complete since any idempotent bimodule map $p: {}_AM_B\To {}_AM_B$ splits in $\tc{C}$, resulting in a $1$-morphism $N$ and $2$-morphisms $N\To[r] M \To[s] N$ such that $r\xt s= p$ and $s\xt r = \It_{N}$. The splitting $2$-morphisms are themselves bimodule maps for the bimodule structure $A\xo N \xo B \To [A\xo r\xo B] A\xo M \xo B \To[\rho] M \To[s] M$ on $N$ (here $\rho:A\xo M\xo B \To M$ denotes the $A$-$B$ action on $M$) and hence split the idempotent $p$ in $\tc{C}^\idm$.

We show that every separable monad in $\tc{C}^\idm$ admits a separable splitting. A separable monad in $\tc{C}^\idm$ is a monad $P:A\to A$ in $\tc{C}$ together with a bimodule ${}_PM_P$ and bimodule 2-morphisms $m:{}_PM \xo_P M_P \To[] {}_PM_P$ and  $u: {}_PP_P\To {}_PM_P$ fulfilling the defining equations for a monad in $\tc{C}^\idm$, and such that $m$ has an $M$--$M$-bimodule section $\Delta:M\xo_PM \To M$. 
The $2$-morphisms 
\begin{calign}\nonumber \widehat{m}:=M\xo M \To M\xo_PM \To[m] M 
&
\widehat{u}:=\Io_{A}\To P \To[u] M
\end{calign}
make the $1$-morphism $M:A\to A$ into a separable monad in $\tc{C}$. Here, the 2-morphism $\Io_{A}\To P$ is the unit of the monad $P$ and $M\xo M \To M \xo_P M$ is the projection onto the image of the idempotent defining the composite $M\xo_P M$. By definition, the multiplication $\widehat{m}$ intertwines the left and right action of $P$ on $M$, leading to bimodules ${}_PM_M$ and ${}_MM_P$.
The bimodule morphisms 
\begin{calign}\nonumber \eta:={}_PP_P \To[u] {}_PM_P \To[\cong] {}_PM\xo_M M_P
&
\epsilon:={}_MM\xo_P M_M \To[m] {}_MM_M
\end{calign}
constitute the unit and counit of an adjunction ${}_PM_M \vdash {}_MM_P$ in $\tc{C}^\idm$. By definition, the $M$-$M$ bimodule section $\Delta: M\xo_P M\To M$ is a right inverse of $\epsilon$. The multiplication $\widehat{m}:M\xo M\To M$ induces an isomorphism of bimodules $\psi:{}_PM\xo_MM_P\To {}_PM_P$. This isomorphism is compatible with the monad structure on ${}_PM\xo_M M_P$ induced from the adjunction ${}_PM_M \vdash {}_MM_P$ and defines an isomorphism of monads in $\tc{C}^\idm$. Hence, $\left({}_PM_M\vdash {}_MM_P, \psi\right)$ is a separable splitting of the separable monad ${}_PM_P$ in $\tc{C}^\Delta$.
\end{proof}

There is a fully faithful $2$-functor $i_{\tc{C}}: \tc{C} \to \tc{C}^\idm$ that sends an object $A$ of $\tc{C}$ to the trivial separable monad $\Io_A:A\to A$ and sends $1$- and $2$-morphisms of $\tc{C}$ to themselves seen as bimodules, or bimodule maps, for the respective trivial monad. 

\begin{prop}[A 2-category is idempotent complete when idempotent completion is an equivalence] \label{prop:icequiv}
A locally idempotent complete $2$-category $\tc{C}$ is idempotent complete if and only if the $2$-functor $i_{\tc{C}}:\tc{C} \to \tc{C}^\idm$ is an equivalence. 
\end{prop}
\begin{proof} Since $i_{\tc{C}}$ is fully faithful, it is an equivalence if and only if it is essentially surjective. Essential surjectivity of $i_{\tc{C}}$ is in turn equivalent to the requirement that any separable monad in $\tc{C}$ is split as a bimodule. By Proposition~\ref{prop:bimodule}, this is equivalent to every separable monad admitting a separable splitting.
\end{proof}

\begin{corollary}[Idempotent completion is an idempotent operation]
For any locally idempotent complete $2$-category, the $2$-functor $i_{\tc{C}^\idm}: \tc{C}^\idm \to \left(\tc{C}^\idm\right)^\idm$ is an equivalence. 
\end{corollary}

\nid This corollary is analogous to a result of Carqueville--Runkel~\cite{Carqueville:2016}, who show that replacing a 2-category by the 2-category of internal separable Frobenius algebras is an idempotent operation.

\addtocontents{toc}{\SkipTocEntry}
\subsection{Idempotent completion of $2$-functors}

Any $2$-functor $F: \tc{C} \to \tc{D}$ between locally idempotent complete $2$-categories maps separable monads to separable monads, bimodules to bimodules, and bimodule maps to bimodule maps, hence gives rise to a $2$-functor $F^\idm : \tc{C}^\idm \to \tc{D}^\idm$ between the completions.
\begin{remark}[Well-defiinition of idempotent completion of 2-functors]  
As in Remark~\ref{rem:chosensplitting}, $F^\idm$ is really defined for $2$-functors $F:\tc{C} \to \tc{D}$ between $2$-categories with chosen splittings of their idempotent $2$-morphisms. Choosing different splittings leads to a $2$-functor $F^{\idm'} : \tc{C}^{\idm'} \to \tc{D}^{\idm'}$ equivalent to $\tc{C}^{\idm'} \equiv \tc{C}^\idm \to[F^\idm] \tc{D}^\idm \equiv \tc{D}^{\idm'}$, where $\tc{C}^{\idm'} \equiv \tc{C}^\idm$ and $\tc{D}^{\idm'} \equiv \tc{D}^\idm$ are the canonical `splitting change' equivalences from Remark~\ref{rem:chosensplitting}. 

More generally, we expect $\left(-\right)^\idm$ to be a $3$-functor on the $3$-category of locally idempotent $2$-categories with chosen splittings, with $2$-functors, natural transformations, and modifications between them. A choice of splittings for every locally idempotent complete $2$-category provides an inverse to the forgetful $3$-functor from the preceding $3$-category to the $3$-category of locally idempotent complete $2$-categories, hence induces a $3$-functor $(-)^\idm$ on the latter $3$-category. Different choices of splittings lead to distinct, but equivalent, inverses to the forgetful $3$-functor, and therefore to equivalent $3$-functors $(-)^\idm$.

Henceforth, whenever we refer to $F^\idm: \tc{C}^\idm \to \tc{D}^\idm$, we have implicitly chosen a splitting of every idempotent $2$-morphism in $\tc{C}$ and $\tc{D}$.
\end{remark}

\begin{remark}[Idempotent completion is an idempotent 3-monad] 
Together with the transformation $i_{C} : \tc{C} \to \tc{C}^\idm$ and the equivalence $\left(\tc{C}^\idm\right)^\idm\to \tc{C}^\idm$, we expect $(-)^\idm$ to be an idempotent $3$-monad on the $3$-category of locally idempotent complete $2$-categories.
\end{remark}
We establish several properties of $(-)^\idm$, all of them consequences of what it would mean to be a $3$-monad on a $2$-category.
\begin{prop}[Invariance and functoriality of idempotent completion] \label{prop:assortedproperties} 
Let $F, G:\tc{C} \to \tc{D}$ and $H:\tc{D} \to \tc{E}$ be $2$-functors between locally idempotent complete $2$-categories. Then the following hold, where $\equiv$ denotes equivalence of $2$-functors:
\begin{enumerate}
\item If $F\equiv G$, then $F^\idm \equiv G^\idm$.
\item $(H\xo G)^\idm \equiv H^\idm \xo G^\idm$.
\item $F^\idm \xo i_{\tc{C}} \equiv i_{\tc{D}}\xo F$.
\item $(i_{\tc{C}})^\idm \equiv i_{\tc{C}^\idm}$.
\end{enumerate}
\end{prop}
\begin{proof}The equivalences (2), (3) and (4) are direct consequences of the definition of $F^\idm$ and $i_{\tc{C}}$. We prove property (1). Let $\eta: F\To G$ be a natural equivalence with component equivalences $\eta_A:F(A) \to G(A)$ for objects $A$ in $\tc{C}$, and isomorphisms $\eta_f: \eta_B \xo F(f) \To  G(f) \xo \eta_A$ for $1$-morphisms $f:A\to B$ in $\tc{C}$. For a separable monad $(A\to[P] A, P \xo P \To[m] P, \Io_A\To[u] P)$ in $\tc{C}$, we define --- omitting all coherence isomorphisms of $F$ and $G$ --- the $G(P)$--$F(P)$-bimodule $\eta^\idm_{(P,m,u)}:= G(P)\xo\eta_A$ with action
\[ 
\shrinker{.85}{
G(P) \xo G(P) \xo  \eta_A \xo F(P) \To[G(m)\xo \eta_A\xo F(P)] G(P)  \xo \eta_A\xo F(P) \To[G(P) \xo \eta_P]   G(P) \xo G(P) \xo \eta_A \To[G(m) \xo \eta_A] G(P) \xo \eta_A.
}
\] 
This bimodule is invertible; an inverse is the bimodule $\eta_A^{-1} \xo G(P)$, with action as above, where $\eta^{-1}:G\To F$ denotes an inverse of the natural equivalence $\eta$. 

For a bimodule $(A\to[M]B, Q\xo M \xo P\To[\rho] M)$ in $\tc{C}$, the $G(P)$--$F(P)$-bimodule morphisms 
\[G(P) \xo  \eta_B \xo F(M) \To[\eta_{P}^{-1}\xo F(M)] \eta_B \xo F(P) \xo F(M) \To[\eta_B \xo F(\rho) ] \eta_B \xo F(M)
\]
\[G(M) \xo G(P) \xo \eta_A \To[G(\rho) \xo \eta_A] G(M) \xo \eta_A
\]
induce bimodule isomorphisms 
\begin{align*}
\eta^\idm_{(Q,m_Q,u_Q)} \xo_{F^\idm(Q,m_Q,u_Q)} F^\idm(M,\rho) &\iso \eta_B \xo F(M)\\
G^\idm(M,\rho) \xo_{G^\idm(P, m_P, u_P)} \eta^\idm_{(P, m_P,u_P)} &\iso G(M) \xo \eta_A.
\end{align*}
Using these isomorphisms, we define the bimodule isomorphism $\eta^\idm_{(M,\rho)}$ as follows:
\[
\shrinker{.9}{
\eta^\idm_{(Q,m_Q,u_Q)} \xo_{F^\idm(Q,m_Q,u_Q)} F^\idm(M,\rho) \iso \eta_B \xo F(M) \To[\eta_M]  G(M) \xo \eta_A \iso G^\idm(M,\rho) \xo_{G^\idm(P, m_P, u_P)} \eta^\idm_{(P, m_P,u_P)}
} 
\]
It can be verified that the equivalences $\eta^\idm_{(P, m, u)}$ and the bimodule isomorphisms $\eta^\idm_{(M, \rho)}$ form a natural equivalence between $F^\idm$ and $G^\idm$. 
\end{proof}

We now show that the $2$-category $\tc{C}^\idm$ deserves the name `idempotent completion': every $2$-functor from the locally idempotent complete $2$-category $\tc{C}$ into an idempotent complete $2$-category $\tc{D}$ factors uniquely through $\tc{C}^\idm$.
\begin{prop}[The idempotent completion is initial among idempotent complete targets] \label{prop:extension} 
Let $\tc{C}$ be a locally idempotent complete $2$-category and let $\tc{D}$ be an idempotent complete $2$-category. Then, any $2$-functor $F:\tc{C} \to \tc{D}$ uniquely extends to a $2$-functor $\widehat{F}: \tc{C}^\idm \to \tc{D}$. That is, there exists a $2$-functor $\widehat{F}: \tc{C}^\idm \to \tc{D}$ such that $\widehat{F} \xo i_{\tc{C}}$ is equivalent to $F$, and if $\widehat{F}, \widehat{F}': \tc{C}^\idm \to \tc{D}$ are $2$-functors such that $\widehat{F} \xo i_{\tc{C}} \equiv  \widehat{F}' \xo i_{\tc{C}}$, then $\widehat{F}$ and $\widehat{F}'$ are equivalent.
\end{prop}
\begin{proof} Since $\tc{D}$ is idempotent complete, the $2$-functor $i_\tc{D}: \tc{D} \to \tc{D}^\idm$ is an equivalence by Proposition~\ref{prop:icequiv}; we fix an inverse $i_{\tc{D}}^{-1}: \tc{D}^\idm \to \tc{D}$. Given a $2$-functor $F:\tc{C} \to \tc{D}$, we define its extension
\[\widehat{F}:= \tc{C}^\idm \to[F^\idm] \tc{D}^\idm \to[i_{\tc{D}}^{-1}] \tc{D}.
\] 
It follows from Proposition~\ref{prop:assortedproperties}(3) that $ \widehat{F} \xo i_{\tc{C}} \equiv i_{\tc{D}}^{-1} \xo i_{\tc{D}} \xo F \equiv F $.

Given two extensions $\widehat{F}, \widehat{F}': \tc{C}^\idm \to \tc{D}$ such that $\widehat{F} \xo i_{\tc{C}} \equiv \widehat{F}' \xo i_{\tc{C}}$, by Proposition~\ref{prop:assortedproperties}(2--4), we have
\[  i_{\tc{D}}^{-1} \xo \left(\widehat{F}\xo i_{\tc{C}}\right)^\idm   \equiv 
i_{\tc{D}}^{-1} \xo \widehat{F}^\idm  \xo i_{\tc{C}}^\idm
\equiv  i_{\tc{D}}^{-1} \xo \widehat{F}^\idm  \xo i_{\tc{C}^\idm}\equiv 
i_{\tc{D}}^{-1} \xo i_{\tc{D}}\xo  \widehat{F} 
\equiv \widehat{F}.
\]
It follows from Proposition~\ref{prop:assortedproperties}(1) that $\widehat{F} \equiv i_{\tc{D}}^{-1} \xo \left( \widehat{F} \xo i_{\tc{C}}\right)^\idm \equiv i_{\tc{D}}^{-1} \xo \left( \widehat{F}' \xo i_{\tc{C}}\right)^\idm \equiv \widehat{F}'$. 
\end{proof}

\newpage

\addtocontents{toc}{\protect\vspace{8pt}}

\section{Dimension formulas in spherical prefusion $2$-categories}

We prove various formulas relating the dimensions of objects and $1$-morphisms in spherical prefusion $2$-categories.  

Recall from Definition~\ref{def:dimhom} that for simple objects $A$ and $B$ of a finite presemisimple 2-category $\tc{C}$, the dimension $\dim(\Hom_{\tc{C}}(A,B))$ is the sum over simple 1-morphisms from $A$ to $B$ of the squared norm of the 1-morphism.  Similarly recall from Definition~\ref{def:dimension2cat} that for a finite presemisimple 2-category $\tc{C}$, the dimension $\dim(\tc{C})$ is the sum over components of the reciprocal of the dimensions of the endomorphism categories.  Finally recall from Definition~\ref{def:quantumdimension} that for a 1-morphism $f$ in a spherical prefusion 2-category, the dimension $\dim(f)$ is the (2-spherical) trace of the identity of $f$, and for an object $A$, the dimension $\dim(A)$ is the dimension of the identity of $A$.

For a simple object $A$ of a spherical prefusion 2-category $\tc{C}$, we will denote by $n(A)$ the number of equivalence classes of simple objects in the component of $A$.  We will also use the abbreviation $\dmo(A) := \dim(A) \dim(\End_{\tc{C}}(A)) n(A)$.

\begin{appprop}[Dimension of 1-morphisms is relatively multiplicative] \label{prop:factoringthroughsimple} 
Let $g:A\to B$ and $f:B\to C$ be $1$-morphisms in a spherical prefusion $2$-category and assume that $B$ is simple. Then
\[
\dim(f\xo g) =\langle \tr_R(\It_g)\rangle  \dim(f) = \langle \tr_L(\It_f) \rangle \dim(g) = \frac{\dim(f) \dim(g)}{\dim(B)}
\]
\end{appprop}
\begin{proof}Simplicity of $B$ implies that $\tr_R(\It_g) = \langle \tr_R(\It_g)\rangle \It_{\Io_B}$. Hence, $\dim(f\xo g) = \Tr(\It_f\xo\tr_R(\It_g)) = \langle\tr_R(\It_g) \rangle \dim(f)$. The second equation follows analogously after an application of Proposition~\ref{prop:leftrighttrace}. The last equation follows since $\dim(f) = \langle \tr_L(\It_f) \rangle \dim(B)$, $\dim(g) = \langle \tr_R(\It_g)\rangle \dim(B)$, and $\dim(B)$ is nonzero.
\end{proof}

\begin{applemma}[Dimension of 1-morphisms is additive]\label{prop:dimensionhom}
Let $f:A\to B$ be a $1$-morphism in a spherical prefusion $2$-category. Then
\[ \sum_{h:A\to B} \dim\left(\Hom_{\tc{C}}(h,f)\right) \dim(h) = \dim(f),
\]
where the sum is over representatives of the simple $1$-morphisms $h:A\to B$.
\end{applemma}
\begin{proof}
Let $\{i_j:f_j \leftrightarrows f: p_j\}_{j \in J}$ be a direct sum decomposition of $f$ into simple $1$-morphisms. For $h:A\to B$ a simple $1$-morphism, let $J_h \subseteq J$ be the subset of the $1$-morphisms $f_j$ isomorphic to $h$. Since $|J_h| = \dim(\Hom_{\tc{C}}(h,f))$, it follows that 
\[
\begin{split}\dim(f) &= \Tr(\It_f) =  \sum_{j \in J} \Tr(i_j\xt p_j) = \sum_{j \in J} \Tr(p_j \xt i_j) = \sum_{j \in J} \dim(f_j) \\ &= \sum_{h:A\to B} |J_h| \dim(h)= \sum_{h:A\to B} \dim\left( \Hom_{\tc{C}}(h,f) \right) \dim(h).
\end{split}
\]
In the second step we used that the trace is additive for $2$-morphisms, that is $\Tr(\alpha + \beta) = \Tr(\alpha) + \Tr(\beta)$ for $\alpha, \beta: g\To h$.
\end{proof}
\begin{appprop}[Dimension of a 1-morphism from its precompositions] \label{prop:maindimensionformula}
Let $f:A\to B$ be a $1$-morphism in a spherical prefusion $2$-category $\tc{C}$. Then
\[
\sum_{C, C\to[g]A} \frac{\dim(g) \dim(f\xo g)}{\dmo(C)} = \dim(f),
\]
where the sum is over representatives of the simple objects $C$ and simple $1$-morphisms $g:C\to A$.
\end{appprop}
\begin{proof}
First, suppose that $A$ is simple and fix a simple $C$. If $A$ and $C$ are in different components, then $\sum_{C\to[g]A} \dim(g)\dim(f\xo g)=0$. If $A$ and $C$ are in the same component, then
\[
\begin{split}
\sum_{C\to[g] A} \dim(g)&\dim(f\xo g)=\!\!\sum_{C\to[g]A} \langle \tr_L(\It_g)\rangle \langle \tr_R(\It_g) \rangle~~\dim(C) \dim(f)\\&
=\dim\left(\Hom_{\tc{C}}(C,A)\right) \dim(C) \dim(f)\hspace{0.6cm}
\superequals{Prop.~\ref{prop:sameblock}}\hspace{0.6cm} \dim\left(\End_{\tc{C}}(C)\right)\dim(C) \dim(f).
\end{split}
\]
Summing over simples $C$ gives the desired formula.

Now suppose that $A$ is an arbitrary object and again fix a simple $C$. Let $\{\iota_i:A_i \leftrightarrows A: \rho_i\}_{i \in I}$ be a direct sum decomposition of $A$ into simple objects $\{A_i\}_{i\in I}$. Observe that if $g:C\to A$ is a simple $1$-morphism, then there is a unique $\alpha(g)\in I$ such that $\rho_{\alpha(g)} \xo g$ is nonzero --- otherwise $g\iso \bigoplus_{i \in I}\iota_i \xo \rho_i \xo g$ would be a non-trivial direct sum decomposition of $g$. Similarly $\rho_{\alpha(g)}\xo g:C\to A_{\alpha(g)}$ itself must be a simple $1$-morphism. Observe that the functions 
\begin{calign}\nonumber 
\{\text{iso classes of simples $C\to A$}\} &\leftrightarrow &\sqcup_{i \in I}\{\text{iso classes of simples $C\to A_i$}\}
\\\nonumber
g &\mapsto& \rho_{\alpha(g)} \xo g: C \to A_{\alpha(g)}
\\\nonumber
\iota_i \xo h&\mapsfrom& h:C \to A_i
\end{calign}
are inverse to each other. It therefore follows that
\[\begin{split}
\sum_{C\to[g] A} \dim(g) \dim(f\xo g) &= \sum_{i\in I} \sum_{C\to[h] A_i} \dim(\iota_i \xo h) \dim(f\xo \iota_i \xo h)\\& \superequals{Prop~\ref{prop:factoringthroughsimple}}\hspace{0.5cm}\sum_{i\in I} \sum_{C\to[h] A_i}\langle\tr_L(\It_{\iota_i})\rangle\dim(h)\dim(f\xo \iota_i \xo h).
\end{split}
\]

For a component $x\in \pi_0\tc{C}$, define the subset $I_x:=\{i \in I~|~ A_i\in x\}$. For a simple object $C$, let $[C]\in \pi_0\tc{C}$ denote the component of $C$. It follows from the previous calculation for simple $A$ that
\[\sum_{i \in I}\sum_{C\to[h] A_i}\langle\tr_L(\It_{\iota_i})\rangle \dim(h)\dim(f\xo \iota_i \xo h) = \sum_{i\in I_{[C]}} \langle \tr_L(\It_{\iota_i}) \rangle \dim\left(\End_{\tc{C}}(C)\right) \dim(C) \dim(f\xo \iota_i).
\]
By Proposition~\ref{prop:coherencedirectsums}, $\iota_i$ is adjoint to $\rho_i$ and hence $\tr_L(\iota_i) = \tr_R(\rho_i)$. Therefore, 
\[\langle \tr_L(\It_{\iota_i}) \rangle \dim(f\xo \iota_i) = \dim(f\xo \iota_i \xo \rho_i).
\]
We conclude that 
\[\sum_{C\to[g] A} \dim(g) \dim(f\xo g) = \dim\left(\End_{\tc{C}}(C)\right) \dim(C) \sum_{i \in I_{[C]}} \dim(f\xo \iota_i \xo \rho_i). 
\]
It follows that 
\[\begin{split}
\sum_{C, C\to[g]A} &\frac{\dim(g) \dim(f\xo g)}{\dim(C) \dim\left(\End_{\tc{C}}(C)\right)n(C)} = \sum_{C}\frac{1}{n(C)} \sum_{i \in I_{[C]}} \dim(f\xo\iota_i \xo \rho_i) \\&= \sum_{x\in \pi_0\tc{C}} \sum_{i \in I_x} \dim(f\xo\iota_i \xo \rho_i) = \sum_{i \in I} \dim(f\xo\iota_i \xo \rho_i) = \dim(f).
\end{split}
\]
In the last step, we have used that for $1$-morphisms $h$ and $k$, it holds that $\dim(h\oplus k) = \dim(h) + \dim(k)$. 
\end{proof}

\begin{appcorollary}[Dimension of an object from incoming morphisms] \label{cor:formuladimobject} 
Let $A$ be an object in a spherical prefusion $2$-category $\tc{C}$. Then
\begin{equation}\nonumber \sum_{B, B\to[f] A} \frac{\dim(f)^2}{\dmo(B)} = \dim(A),
\end{equation}
where the sum is over representatives of the simple objects $B$ and simple $1$-morphisms $f:B \to A$. 
\end{appcorollary}
\begin{proof}
This follows from Proposition~\ref{prop:maindimensionformula} for $f=\Io_A$. 
\end{proof}

\begin{appcorollary}[Dimension of a prefusion 2-category from its fusion rule morphisms] \label{cor:formulaglobaldim}
Let $A$ be a simple object in a spherical prefusion $2$-category $\tc{C}$. Then 
\[
\sum_{B,C, f:B\xz C \to A } \frac{\dim(f)^2}{\dim(A)\dmo(B)\dmo(C) } = \dim(\tc{C}),
\]
where the sum is over representatives of the simple objects $B$ and $C$ and simple $1$-morphisms $f$. 
\end{appcorollary}
\begin{proof} 
By sphericality and Corollary~\ref{cor:formuladimobject}, the left-hand side of this equation can be expressed as follows:
\[\sum_{B,C, B\to[f] A\xz C^{\#}} \frac{\dim(f)^2}{\dim(A) \dmo(B) \dmo(C)}
= 
\sum_{C} \frac{1}{\dim(A)\dmo(C)} \dim(A\xz C^\#)
=
\sum_C \frac{\dim(C)}{\dmo(C)}
\]
Note that in the last equality we have used that the dimension of objects is multiplicative, that is $\dim(A \xz B) = \dim(A) \dim(B)$, and that the dimension of an object and its dual agree, that is $\dim(A^\#) = \dim(A)$.  We furthermore have 
\[  \sum_C \frac{\dim(C)}{\dmo(C)} = \sum_C \frac{1}{\dim(\End_{\tc{C}}(C)) n(C)} =\sum_{[x]\in \pi_0\tc{C}}\frac{1}{\dim(\End_{\tc{C}}(x))} = \dim(\tc{C}). \qedhere 
\]
\end{proof}

\begin{appcorollary}[Dimension of a 1-morphism from its factorizations] \label{cor:formulaHom} 
Let $f:A\to B$ be a $1$-morphism in a spherical prefusion $2$-category $\tc{C}$. Then
\[
\sum_{C, A\to[h]C, C\to[g]B} \dim\left(\Hom_{\tc{C}}(g\xo h, f) \right)\frac{\dim(g) \dim(h)}{\dmo(C)} = \dim(f),
\]
where the sum is over representatives of the simple objects $C$ and simple $1$-morphisms $g$ and $h$.
\end{appcorollary}
\begin{proof}
By pivotality, $\Hom_{\tc{C}}(g\xo h, f) \iso \Hom_{\tc{C}}(h, g^*\xo f)$. Lemma~\ref{prop:dimensionhom} implies that, for fixed simple $C$, 
\[\sum_{A\to[h]C} \dim\left( \Hom_{\tc{C}}(h, g^*\xo f)\right)~\dim(h) = \dim(g^*\xo f) = \dim(f^*\xo g).
\]
The left-hand side of the desired formula becomes
\[\sum_{C, C\to[g] B} \frac{\dim(g) \dim(f^*\xo g)}{\dmo(C)} \hspace{0.5cm}\superequals{Prop~\ref{prop:maindimensionformula}} \hspace{0.5cm} \dim(f^*) ~= ~\dim(f).\qedhere
\]
\end{proof}

\begin{appcorollary}[Total factorization of the identity on a 1-morphism] \label{cor:formulabasis}
Let $A\to[f] B$ be a $1$-morphism in a spherical prefusion $2$-category $\tc{C}$. Then
\[
\sum_{C,C\to[g]B, A\to[h] C} \frac{\dim(g) \dim(h)}{\dmo(C)} \sum_{\gamma \in \Hom_{\tc{C}}(g\xo h ,f)} \gamma \xt \widehat{\gamma}= \It_f
\]
where the left sum is over representatives of the simple objects $C$ and simple $1$-morphisms $g$ and $h$, the right sum is over a basis $\{\gamma\}$ of the vector space $\Hom_{\tc{C}}(g\xo h, f)$, and $\{\widehat{\gamma}\}$ denotes the dual basis of $\Hom_{\tc{C}}(f, g\xo h)$ with respect to the pairing $\langle\cdot, \cdot \rangle$ (see Definition~\ref{def:pairing}).
\end{appcorollary}
\begin{proof} First note that $\sum_{\gamma} \gamma\xt \widehat{\gamma}$ is independent of the choice of basis. Let $\{s_i\}_{i \in I}$ be a set of representative simple $1$-morphisms $s_i:A\to B$. For each $i\in I$, let $\{ \alpha^i_j\}_{j \in J_i}$ be a basis of $\Hom_{\tc{C}}(s_i, f)$ and let $\{\beta^i_k\}_{k \in K_i}$ be a basis of $\Hom_{\tc{C}}(g\xo h, s_i)$. By local semisimplicity, the set
\begin{equation}\nonumber\mathcal{B}:=\sqcup_{i \in I} \{\alpha^i_j\xt\beta^i_k \in \Hom_{\tc{C}}(g\xo h, f)~|~ j \in J_i, k \in K_i\}
\end{equation}
forms a basis of $\Hom_{\tc{C}}(g\xo h, f)$. 
Let $\{\widehat{\alpha}^i_j \in \Hom_{\tc{C}}(f, s_i)\}_{j\in J_i}$ and $\{\widehat{\beta}^i_k \in \Hom_{\tc{C}}(s_i, g\xo h)\}_{k \in K_i}$ be dual bases to $\{\alpha^i_j\}_j$ and $\{\beta^i_k\}_k$ under the usual non-degenerate `simple $1$-morphism pairings':
\begin{calign}\nonumber\Hom_{\tc{C}}(f,s_i) \otimes \Hom_{\tc{C}}(s_i,f)\to k & \Hom_{\tc{C}}(g\xo h, s_i) \otimes \Hom_{\tc{C}}(s_i, g\xo h)\to k  \\\nonumber
\mu\otimes \nu \mapsto \langle \mu\xt \nu\rangle & \mu\otimes \nu \mapsto \langle \mu\xt \nu\rangle
\end{calign}
Direct calculation shows that the set 
\[\sqcup_{i \in I}\bigg\{\frac{1}{\dim(s_i)}~\widehat{\beta}^i_k\xt\widehat{\alpha}^i_j\bigg\}_{j \in J_i, k \in K_i}
\]
is a dual basis of the basis~$\mathcal{B}$ with respect to the `sphere-pairing' $\Hom_{\tc{C}}(g\xo h, f) \otimes \Hom_{\tc{C}}(f, g\xo h) \to k$ from Definition~\ref{def:pairing}. With this choice of basis, the left-hand side of the desired equation becomes
\[ \sum_{C,C\to[g]B, A\to[h] C} \frac{\dim(g) \dim(h)}{\dmo(C)}\sum_{i \in I}\frac{1}{\dim(s_i)}\sum_{j \in J_i, k\in K_i} \alpha^i_j \xt \beta^i_k \xt\widehat{\beta}^i_k \xt\widehat{\alpha}^i_j.
\]
Observing that $\sum_{k \in K_i} \beta^i_k \xt \widehat{\beta}^i_k = \dim\left( \Hom_{\tc{C}}(s_i, g\xo h)\right)~ \It_{s_i}$, this expression becomes
\[\sum_{i \in I, j\in J_i}\frac{1}{\dim(s_i)} \alpha^i_j \xt \widehat{\alpha}^i_j~\Bigg(\sum_{C, C\to[g] B, A\to[h] C} \frac{\dim(g)\dim(h)}{\dmo(C)} \dim\left( \Hom_{\tc{C}}(s_i, g\xo h)\right)\Bigg).
\]
By Corollary~\ref{cor:formulaHom} and the fact that for $1$-morphism $k,h$ in a locally semisimple $2$-category $\Hom_{\tc{C}}(h,k) \iso \Hom_{\tc{C}}(k,h)$, the expression in parenthesis is $\dim(s_i)$. Hence, the expression in question reduces to
\[\sum_{i \in I, j \in J_i} \alpha^i_j \xt \widehat{\alpha}^i_j = \It_{f}.\qedhere
\]
\end{proof}

\newpage

\addtocontents{toc}{\protect\vspace{8pt}}

\addtocontents{toc}{\SkipTocEntry}
\bibliographystyle{initalphaSort}
\bibliography{ftc}

\end{document}